\documentclass[a4paper,leqno,11pt]{amsart}

\usepackage{lmodern}
\usepackage{amsmath, amsfonts, amssymb, amsthm}
\usepackage{ esint }
\usepackage{mathtools}
\usepackage{bbm}
\usepackage{xparse}
\usepackage[inline]{enumitem}
\usepackage{verbatim}
\usepackage{mathrsfs}
\usepackage{tikz-cd}
\usepackage{mathrsfs}
\usepackage[margin=2.54cm]{geometry}
\usepackage[normalem]{ulem}

\usepackage{float}
\usepackage{stmaryrd}

\newtheorem{thm}{Theorem}[section]
\newtheorem{lem}[thm]{Lemma}
\newtheorem{cor}[thm]{Corollary}
\newtheorem{prop}[thm]{Proposition}
\theoremstyle{definition}
\newtheorem{definition}[thm]{Definition}
\newtheorem{rem}[thm]{Remark}

\def\Xint#1{\mathchoice
{\XXint\displaystyle\textstyle{#1}}%
{\XXint\textstyle\scriptstyle{#1}}%
{\XXint\scriptstyle\scriptscriptstyle{#1}}%
{\XXint\scriptscriptstyle\scriptscriptstyle{#1}}%
\!\int}
\def\XXint#1#2#3{{\setbox0=\hbox{$#1{#2#3}{\int}$ }
\vcenter{\hbox{$#2#3$ }}\kern-.57\wd0}}

\def\dashint{\Xint-}
\usepackage{graphicx} % for \rotatebox macro
\def\Yint#1{\mathchoice
    {\YYint\displaystyle\textstyle{#1}}%
    {\YYint\textstyle\scriptstyle{#1}}%
    {\YYint\scriptstyle\scriptscriptstyle{#1}}%
    {\YYint\scriptscriptstyle\scriptscriptstyle{#1}}%
      \!\iint}
\def\YYint#1#2#3{{\setbox0=\hbox{$#1{#2#3}{\iint}$}
    \vcenter{\hbox{$#2#3$}}\kern-.51\wd0}}
\def\longdash{{-}\mkern-3.5mu{-}} 
\def\fiint{\Yint\longdash}

\newcommand*\dd{\mathop{}\!\mathrm{d}}

\newcommand{\Rn}{\mathbb{R}^{n}}
\newcommand{\R}{\mathbb{R}}

\newcommand{\N}{\mathbb{N}}
\newcommand{\C}{\mathbb{C}}

\newcommand{\abs}[1]{\left| #1\right|}
\newcommand{\norm}[1]{\left\lVert #1\right\rVert}
\newcommand{\meas}[1]{\left\lvert #1\right\rvert}
\newcommand{\ind}[1]{\mathbbm{1}_{#1}}
\newcommand{\RNum}[1]{%
  \textup{\uppercase\expandafter{\romannumeral#1}}%
}
\newcommand{\ball}{\textup{B}}

\newcommand{\Hn}{\mathbb{R}^{1+n}_{+}}
\renewcommand{\Re}[1]{\operatorname{Re}{{#1}}}
\newcommand{\diam}[1]{\textup{diam}(#1)}
\newcommand{\dist}{\textup{dist}}

\newcommand{\clos}[1]{\overline{#1}}
\newcommand{\inner}[2]{\left\langle #1 , #2 \right\rangle_{\textup{L}^{2}}}

\newcommand{\supp}[1]{\textup{supp}\left(#1\right)}
\newcommand{\eps}{\varepsilon}

\newcommand{\sgn}[1]{\textup{sgn}(#1)}

\newcommand{\Ell}[1]{\mathrm{L}^{#1}(\Rn)}
\newcommand{\hlmax}[1]{\mathcal{M}\left(#1\right)}

\newcommand{\dom}[1]{\mathsf{D}(#1)}
\newcommand{\ran}[1]{\mathsf{R}(#1)}
\newcommand{\kernel}[1]{\mathsf{N}(#1)}
\newcommand{\El}[1]{\mathrm{L}^{#1}}
\newcommand{\Dz}[1]{\textup{H}_{V}^{#1}}
\newcommand{\BMO}{\textup{BMO}}
\newcommand{\av}[1]{\textup{av}_{#1}}
\newcommand{\bb}[1]{\mathbb{#1}}
\newcommand{\set}[1]{\left\{#1\right\}}
\newcommand{\loc}{\textup{loc}}
\newcommand{\smcp}{C^{\infty}_{c}(\Rn)}
\newcommand{\dvg}{\textup{div}}
\newcommand{\invH}{\dot{H_{0}}^{\mkern-9mu -1/2}}
\newcommand{\dotH}{\dot{H_{0}}^{\mkern-8mu 1/2}}
\DeclareDocumentCommand{\tent}{o m}{\textup{T}^{#2}\IfValueT{#1}{{#1}^}}

\numberwithin{equation}{section}
\usepackage[colorlinks,citecolor=red,linkcolor=blue]{hyperref}  
\author{Arnaud Dumont}
\author{Andrew J. Morris}
\address{A. Dumont, School of Mathematics, University of Birmingham, Edgbaston, B15 2TT, UK}
\email{axd461@bham.ac.uk}
\address{A. J. Morris, School of Mathematics, University of Birmingham, Edgbaston, B15 2TT, UK}
\email{a.morris.2@bham.ac.uk}
\date{\today}
\thanks{This work was supported by the Royal Society (IES$\backslash$R3$\backslash$193232).}
\thanks{No data were created or analysed in this study.}


\begin{document}
\title[Boundary value problems for singular Schrödinger equations]{Boundary value problems and Hardy spaces for singular Schrödinger equations with block structure
}

\setcounter{tocdepth}{1}

\begin{abstract}
We obtain Riesz transform bounds and characterise operator-adapted Hardy spaces to solve boundary value problems for singular Schrödinger equations $-\textup{div}(A\nabla u)+aVu=0$ in the upper half-space $\Hn$ with boundary dimension $n\geq 3$. The coefficients $(A,a,V)$ are assumed to be independent of the transversal direction to the boundary, and consist of a complex-elliptic pair $(A,a)$ that is bounded and measurable with a certain block structure, and a non-negative singular potential $V$ in the reverse Hölder class $\textup{RH}^{q}(\Rn)$ for $q\geq \max\set{\frac{n}{2},2}$. This block structure is significant because it allows for coefficients that are not symmetric but for which $\El{2}(\R^n)$-solvability persists due to recently obtained Kato square root type estimates. We find extrapolation intervals for exponents $p$ around 2 on which the Dirichlet problem is well-posed for boundary data in $\El{p}(\R^n)$, and the associated Regularity problem is well-posed for boundary data in Sobolev spaces $\dot{\mathcal{V}}^{1,p}(\R^n)$ that are adapted to the potential $V$, when $p>1$.  
The well-posedness of these Dirichlet problems and related estimates then allow us to solve the corresponding Neumann problem with boundary data in $\Ell{p}$. The results permit boundary data in the Dziuba\`{n}ski–Zienkiewicz Hardy space $\textup{H}^1_{V}(\R^n)$ and adapted Hardy--Sobolev spaces $\dot{\textup{H}}^{1,p}_{V}(\Rn)$ when $p\leq 1$. We also obtain comparability of square functions and nontangential maximal functions for the solutions with their boundary data.
\end{abstract}
\maketitle
\tableofcontents

\section{Introduction}
We consider singular Schrödinger equations $-\textup{div}(A\nabla u)+\widetilde{V}u=0$ in the upper half-space $\Hn:=\set{(t,x)\in\R \times\Rn : t>0 }$, with boundary $\partial\R^{1+n}_+\cong \Rn$, which is the prototypical domain for understanding solvability of associated boundary value problems on Lipschitz domains. The consideration of $t$-independent coefficients $A(t,x)=A(x)$ and $\widetilde{V}(t,x)=
\widetilde{V}(x)$ is also significant for allowing minimal regularity coefficients, as the necessary transversal regularity determined by Caffarelli--Fabes--Kenig in~\cite{CFK} is then typically reached via perturbation. Amid recent efforts to characterise solvability for such equations when the principal coefficient matrix $A$ is not symmetric and the potential $\widetilde{V}\equiv0$, results for block structure coefficients 
\begin{equation}\label{block structure of the coefficients}
    A=
    \begin{bmatrix}
        A_{\perp \perp } & 0\\
        0 & A_{\parallel\parallel}
    \end{bmatrix}
    \text{, where $A_{\perp \perp }:\Rn\to \C$ and $A_{\parallel\parallel}:\Rn\to \C^{n\times n}$},
\end{equation}
have been indispensable (e.g., \cite{DHP,HKMP}). In that case, solvability for $\El{2}$-data and extrapolation to $\El{p}$-type data are both deeply connected to the resolution of the Kato square root conjecture by Auscher--Hofmann--Lacey--McIntosh--Tchamitchian in~\cite{Auscher_Hofmann_Lacey_McIntosh_Tchamitchian2002}. A comprehensive treatment of the state of the art can be found in the monograph by Auscher--Egert~\cite{AuscherEgert}, which unifies and improves a variety of techniques to treat  Dirichlet, Regularity and Neumann problems for many boundary data spaces, including those of $\El{p}$-type, in ranges likely to be the best possible.

These boundary value problems are far less understood in the presence of a non-zero singular potential $\widetilde{V}\in\El{1}_{\loc}(\Rn)$ when the Schrödinger equation can be expressed in the form
\begin{equation}\label{general div form schrodinger equ}
    -\sum_{i=0}^{n}\sum_{j=0}^{n}\partial_{i}\left(A_{ij}\partial_{j}u\right)\left(t,x\right) +\widetilde{V}(x)u(t,x)=0,
\end{equation}
where $\partial_{0}:=\partial_{t}$ and $\left(\partial_{1},\ldots,\partial_{n}\right)=\nabla_{x}$ on $\Hn$. The $\El{2}$-solvability for Regularity and Neumann problems was recently obtained by Morris--Turner in~\cite{MorrisTurner} for a class of reverse Hölder potentials. The approach once again exploited the connection with the Kato conjecture but the results showed how the presence of a singular potential fundamentally changes the approach. This paper extrapolates those results to obtain $\El{p}$-type solvability for Dirichlet, Regularity and Neumann problems, utilising 
tools for singular potentials from~\cite{MorrisTurner} to extend the extrapolation framework in~\cite{AuscherEgert}.

More specifically, the principal coefficients $A$ are assumed to be $t$-independent and measurable with the block structure~\eqref{block structure of the coefficients}. The potential $\widetilde{V}$ is assumed to be $t$-independent and given by $\widetilde{V}(t,\cdot)=aV$ for some measurable function $a:\Rn\to \C$ and non-negative function $V\in\El{1}_{\loc}(\Rn)$. This allows us to express~\eqref{general div form schrodinger equ} in the form
\begin{equation}\label{block coefficient schrodinger equation simplified}
    \mathscr{H}u:=-\partial_{t}\left(A_{\perp\perp}\partial_{t} u\right) -\dvg_{x}\left({A_{\parallel\parallel}\nabla_{x}u}\right) +aVu=0.
\end{equation}
We assume that there are constants $0<\lambda\leq \Lambda<\infty$ with the bound
\begin{equation}\label{boundedness of coefficients}
    \max\left\{\norm{A}_{\El{\infty}(\Rn;\mathcal{L}(\C^{n+1}))}, \norm{a}_{\El{\infty}(\Rn)}\right\}\leq \Lambda
\end{equation}
and ellipticity
\begin{align}\label{ellipticity assumption}
\begin{split}
    \Re\int_{\Rn} 
    \bigg(A(x)\begin{bmatrix}
        f(x)\\
        \nabla_{x}g(x)
    \end{bmatrix}\cdot\overline{\begin{bmatrix}
        f(x)\\
        \nabla_{x}g(x)
    \end{bmatrix}}
    &+a(x)V(x)\abs{g(x)}^{2}\bigg)\dd x \\
    &\geq\lambda \int_{\Rn}\bigg(\abs{f(x)}^{2}+\abs{\nabla_{x}g(x)}^{2}+V(x)\abs{g(x)}^2\bigg) \dd x
    \end{split}
\end{align}
for all $f$, $g\in C^{\infty}_{c}(\Rn)$. Finally, we assume that the potential $V$ belongs to the reverse Hölder class $\textup{RH}^{q}(\Rn)$ for some $q\geq \max\{\frac{n}{2},2\}$. This means that there is $C>0$ such that
\begin{equation}\label{reverse Holder ineq for V first occurence}
    \left(\dashint_{B}V^{q}\right)^{\frac{1}{q}}\leq C\dashint_{B}V
\end{equation}
for every ball $B\subset \Rn$, where the left-hand side of \eqref{reverse Holder ineq for V first occurence} is understood as $\norm{V}_{\El{\infty}(B)}$ when $q=\infty$. 

There is a growing literature devoted to the study of boundary value problems for such perturbations of divergence form elliptic equations (and systems) with complex coefficients. One of the earliest results in this direction was obtained by Shen \cite{Shen_Neumann_SchrodingerBVP}, who used the method of layer potentials to prove well-posedness of the $\El{p}$-Neumann problem for the Schrödinger equation $-\Delta u + Vu=0 $ above a Lipschitz graph for $p\in(1,2]$ and potentials $V\in \textup{RH}^{\infty}(\Rn)$. This was improved by Tao--Wang \cite{Tao_Wang_2004} to include potentials $V\in\textup{RH}^{n}(\Rn)$. Tao \cite{Tao_Regularity_Problems} has also considered the corresponding Regularity problem. More recently, Morris--Turner \cite{MorrisTurner} proved well-posedness of the $\El{2}$-Regularity and Neumann problems for the Schrödinger equation \eqref{block coefficient schrodinger equation simplified} on the upper half-space $\Hn$ for potentials $V\in\textup{RH}^{q}(\Rn)$ when $q\geq \max\{\frac{n}{2},2\}$ and $n\geq3$. The introduction of \cite{MorrisTurner} contains a more detailed review of the literature concerning square integrable boundary data. For $\El{p}$-data in a related context, we also highlight the recent series of papers by Bortz--Hofmann--Luna~Garc\'{i}a--Mayboroda--Poggi \cite{Bortz_Hofmann_Garcia_Mayboroda_Poggi_PART1,Bortz_Hofmann_Garcia_Mayboroda_Poggi_PART2}, which develops a theory of abstract layer potentials to prove well-posedness of the $\El{p}$-Dirichlet, Regularity and Neumann problems on the upper half-space for equations with $t$-independent coefficients and lower-order terms in certain Lebesgue-critical spaces. Finally, for real, symmetric and $t$-independent coefficients in the region above a Lipschitz graph, we mention the recent work of Geng--Xu \cite{geng2024schrodingerequationsbinftypotentials} proving unique solvability of the $\El{p}$-Regularity and $\El{p}$-Neumann problems for Schrödinger equations with potentials $V\in\textup{RH}^{\infty}(\R^n)$. 

Solutions of \eqref{block coefficient schrodinger equation simplified} have to be understood in a weak sense, which we will now make precise. Let $\Omega\subset\R^{1+n}$ be an arbitrary open set. Following \cite{MorrisTurner}, we utilise the $V$-adapted Sobolev space $\mathcal{V}^{1,2}_{\textup{loc}}(\Omega)$, defined as the space of functions $u\in \textup{W}^{1,2}_{\loc}(\Omega)$ such that $V^{1/2}u\in\El{2}_{\loc}(\Omega).$ 
These spaces are studied in Section~\ref{section on definition of Sobolev spaces adapted to singular potentials}. For $g\in\El{2}_{\loc}(\Omega)$, we say that a function $u\in\mathcal{V}^{1,2}_{\textup{loc}}(\Omega)$ is a weak solution of $\mathscr{H}u=g$ in $\Omega$ if
\begin{equation}\label{definition of weak solutions in general open subset of R1+n}
    \iint_{\Omega} \left(A \begin{bmatrix}
    \partial_{t}u\\
    \nabla_{x}u
\end{bmatrix}\cdot \overline{\begin{bmatrix}
    \partial_{t}\varphi\\
    \nabla_{x}\varphi
\end{bmatrix}} +aVu\overline{\varphi} \right)\dd t\dd x = \iint_{\Omega}g\overline{\varphi}\dd t\dd x
\end{equation}
for all smooth compactly supported  functions $\varphi\in C^{\infty}_{c}(\Omega)$. 

In order to pose the relevant boundary value problems on $\Hn$, we need to introduce the following maximal function, which measures the size of solutions on conical nontangential approach regions. For a function $F\in\El{2}_{\loc}(\Hn;\C^{m})$, we consider the nontangential maximal function $N_{*}(F): \Rn\to [0,\infty]$  defined as
\begin{equation*}
    N_{*}(F)(x):=\sup_{t>0}\left(\fiint_{W(t,x)}\abs{F(s,y)}^{2}\dd s\dd y\right)^{1/2}
\end{equation*}
for all $x\in\Rn$, where $W(t,x):=(\frac{t}{2},2t)\times B(x,t)$ is a Whitney box (or cylinder) of scale $t>0$ above the ball $B(x,t):=\set{y\in\Rn : \abs{y-x}<t}$. We also require the conical square function $S(F):\Rn\to[0,\infty]$ defined as 
\begin{equation*}
    S(F)(x):=\left(\iint_{\Gamma_{1}(x)}\abs{F(s,y)}^{2}\frac{\dd s\dd y}{s^{1+n}}\right)^{1/2}
\end{equation*}
for all $x\in\Rn$, where $\Gamma_{a}(x):=\set{(s,y)\in\Hn : \abs{x-y}<as}$ is a cone of aperture $a>0$ with vertex at $x$.

Our main results require a systematic treatment and new characterisations for the $\El{2}$-based versions of the Hardy spaces $\textup{H}^{p}_{V}(\Rn)$, for $p\in(0,\infty)$, introduced by Dziuba\'nski--Zienkiewicz~\cite{Dziubanski_Hp_Spaces} and adapted to the Schrödinger operator $H_{0}:=-\Delta +V$ on the boundary $\partial\Hn \cong \Rn $. This is the focus of Section~\ref{section on Hardy spaces adapted to schrodinger operators}. In particular, we consider the pre-completed versions of these spaces $\textup{H}^{p}_{V,\textup{pre}}(\Rn)$ and in Section~\ref{subsubsection with definition of completion H^{1}_{V}} we also prove that the completion of $\textup{H}^{1}_{V,\textup{pre}}(\Rn)$ can be realised as a subspace of $\Ell{1}$. In Theorem~\ref{square function characterisation of Dziubanski Hardy spaces}, we prove a new square function characterisation for $\textup{H}^{p}_{V,\textup{pre}}(\Rn)$ in the full range $p\in(\frac{n}{n+1},\infty)$. This enables us to obtain the interpolation result for these spaces in Lemma~\ref{lemma notion of HpV boundedness interpolates}, which is crucial for adapting the extrapolation framework in \cite{AuscherEgert}.

We may now state the boundary value problems treated in this paper. For the Dirichlet problem, we consider boundary data $f\in\Ell{p}$ for $p\in(1,\infty)$, and $f\in A_{\perp\perp}^{-1}\textup{H}^{1}_{V}(\Rn)$ for the endpoint case $p=1$ (see \eqref{properties of b(x)} below).
For $p\in[1,\infty)$, the Dirichlet problem $\left(\mathcal{D}\right)^{\mathscr{H}}_{p}$ is to find a function $u\in\mathcal{V}^{1,2}_{\loc}(\Hn)$ such that
\begin{equation*}
    \left(\mathcal{D}\right)^{\mathscr{H}}_{p}\hspace{0.5cm}\left\{\begin{array}{l}
        \mathscr{H}u=0  \text{ weakly in } \Hn,\\
        N_{*}(u)\in\Ell{p},\\
        \lim_{t\to 0^{+}}\fiint_{W(t ,x)}\abs{u(s,y)-f(x)}\dd s\dd y = 0 \text{ for a.e. } x\in\Rn.
    \end{array}\right.
\end{equation*}
The convergence of the solution $u$ towards the boundary data $f$ required here is a weak type of non-tangential convergence. Indeed, note that $\Gamma_{1}(x)\subseteq \cup_{t>0}W(t,x)\subseteq \Gamma_{2}(x)$ for all $x\in\Rn$.

For the Regularity problem, we distinguish the cases $p>1$ and $p\leq 1$. For $p\in(1,\infty)$, we consider boundary data $f\in\dot{\mathcal{V}}^{1,p}(\Rn)$, meaning that $f\in\El{1}_{\loc}(\Rn)$ is such that 
\[
\nabla_{\mu}f:=\begin{bmatrix}
    \nabla_{x}f\\
    V^{1/2}f
\end{bmatrix}\in\El{p}(\Rn;\C^{n+1}).
\]
These homogeneous $V$-adapted Sobolev spaces are defined and studied in Section~\ref{section on definition of Sobolev spaces adapted to singular potentials}.  The density of smooth compactly supported functions $C^{\infty}_{c}(\Rn)$ in $\dot{\mathcal{V}}^{1,p}(\Rn)$ is crucial for our results, and when $p=n$ we find that it relies on the full strength of the improved Fefferman--Phong inequality obtained by Auscher--Ben Ali in \cite[Lemma 2.1]{Auscher-BenAli} (see Propositions~\ref{improved Fefferman--Phong inequality} and~\ref{density lemma in homogeneous V adapted Sobolev spaces}).

For $p\in(\frac{n}{n+1},1]$, we consider boundary data $f\in\dot{\textup{H}}^{1,p}_{V}(\Rn)$, which is a Hardy--Sobolev space defined as a completion of the normed space 
\[
\dot{\textup{H}}^{1,p}_{V,\textup{pre}}(\Rn):=\set{f\in\dom{H_{0}^{1/2}} : \norm{f}_{\dot{\textup{H}}^{1,p}_{V}} := \norm{H_{0}^{1/2}f}_{\textup{H}^{p}_{V}} < \infty}.
\]
The spaces $\dot{\textup{H}}^{1,p}_{V,\textup{pre}}(\Rn)$ are studied in Section~\ref{section on hardy sobolev spaces adapted to schrodinger operators}, and we construct their completion $\dot{\textup{H}}^{1,p}_{V}(\Rn)$ as a subspace of $\Ell{p^{*}}$ in Section~\ref{subsection definition of the completion of DZiubanski H^{1,p} spaces}, where $p^{*}$ is the upper Sobolev exponent of $p$ in dimension $n$ (see Section~\ref{subsubsection on definition of exponents and order relations}).
For $p\in(\frac{n}{n+1},\infty)$, the Regularity problem $\left(\mathcal{R}\right)^{\mathscr{H}}_{p}$ is to find a function $u\in\mathcal{V}_{\loc}^{1,2}(\Hn)$ such that
\begin{equation*}
    \left(\mathcal{R}\right)^{\mathscr{H}}_{p}\hspace{0.5cm}\left\{\begin{array}{l}
        \mathscr{H}u=0  \text{ weakly in } \Hn,\\
        N_{*}(\nabla_{\mu}u)\in\Ell{p},\\
        \lim_{t\to 0^{+}}\fiint_{W(t ,x)}\abs{u(s,y)-f(x)}\dd s\dd y = 0 \text{ for a.e. } x\in\Rn.
    \end{array}\right.
\end{equation*}

We also study the corresponding Neumann problem $\left(\mathcal{N}\right)^{\mathscr{H}}_{p}$ for $p\in[1,\infty)$. Given boundary data $g\in\Ell{p}$ if $p>1$, or $g\in\textup{H}^{1}_{V}(\Rn)$ if $p=1$, the Neumann problem $\left(\mathcal{N}\right)^{\mathscr{H}}_{p}$ is to find a function $u\in\mathcal{V}_{\loc}^{1,2}(\Hn)$ such that
\begin{equation*}
\left(\mathcal{N}\right)^{\mathscr{H}}_{p}\hspace{0.5cm}\left\{\begin{array}{l}
        \mathscr{H}u=0  \text{ weakly in } \Hn,\\
        N_{*}(\nabla_{\mu}u)\in\Ell{p},\\
        \lim_{t\to 0^{+}}\fiint_{W(t ,x)}\abs{(A_{\perp\perp}\partial_{t}u)(s,y)-g(x)}\dd s\dd y = 0 \text{ for a.e. } x\in\Rn.
    \end{array}\right.
\end{equation*}

The boundary value problems $\left(\mathcal{D}\right)^{\mathscr{H}}_{p}$, $\left(\mathcal{R}\right)^{\mathscr{H}}_{p}$, and $\left(\mathcal{N}\right)^{\mathscr{H}}_{p}$ are called
\emph{solvable} if for each boundary datum there exists at least one solution $u$ with the respective
properties specified above. In the case $V\not\equiv0$, they are called \emph{uniquely posed} if for each boundary datum there exists at most one such solution $u$ in $\mathcal{V}^{1,2}_{\loc}(\Hn)$. In the case $V\equiv0$, it is instead required that there exists at most one such solution $u$ in $W^{1,2}_{\loc}(\Hn)$ modulo constants. The boundary
value problems are called \emph{well-posed} if they are solvable and uniquely posed. 

We also note that \cite[Theorem 1.1]{MorrisTurner} implies that the Regularity problem $\left(\mathcal{R}\right)^{\mathscr{H}}_{2}$ is well-posed if $V\not\equiv 0$ is such that $V\in\textup{RH}^{q}(\Rn)$ for $q\geq \max\set{\frac{n}{2},2}$ and $n\geq 3$. While the type of convergence of the solution towards the boundary data obtained in the aforementioned reference is \emph{a priori} different to that considered here, it can be shown to imply the convergence required in $\left(\mathcal{R}\right)^{\mathscr{H}}_{2}$. We provide details at the start of Section~\ref{subsection on existence for the regularity problem}. 

Before stating the main results of the paper, we shall give an overview of the method that we will employ to solve the boundary value problems $\left(\mathcal{D}\right)^{\mathscr{H}}_{p}$, $\left(\mathcal{R}\right)^{\mathscr{H}}_{p}$ and $\left(\mathcal{N}\right)^{\mathscr{H}}_{p}$, adapting the machinery of Auscher--Egert \cite{AuscherEgert} to treat the Schrödinger equation. 

The first step is the following reduction.
Taking $g=0$ in the ellipticity assumption \eqref{ellipticity assumption}, we see that $\Re{A_{\perp\perp}(x)}\geq \lambda>0$ for almost every $x\in\Rn$. As a consequence, $b(x):=A_{\perp\perp}^{-1}(x)$ exists with
\begin{equation}\label{properties of b(x)}
    \abs{b(x)}\leq \lambda^{-1} \text{ and } \Re{b(x)}\geq \lambda\Lambda^{-2}
\end{equation}
for almost every $x\in\Rn$. The $t$-independence of $A_{\perp\perp}$ shows that the equation~\eqref{block coefficient schrodinger equation simplified} can be written in the form $\partial_{t}^{2}u=Hu$, where $H$ is the differential operator formally defined as
\begin{equation*}
    H:=-b \dvg_{x}\left({A_{\parallel\parallel}\nabla_{x}}\right) +baV
\end{equation*}
on the boundary $\partial\Hn \cong \Rn $. This suggests that solutions $u$ to the boundary value problems can be constructed, at least when $f\in\Ell{2}$, using the holomorphic functional calculus for the operator $H$ on $\Ell{2}$ to define the Poisson semigroup solution $u(t,\cdot):=e^{-tH^{1/2}}f$ for $t>0$. This will be made precise in Theorem~\ref{Poisson semigroup extension is a weak solution}.

Some operator theoretic properties of the operator $H$ on $\El{2}(\Rn)$ are needed to support this strategy.  In the case $V\equiv 0$, the solution of the Kato square root conjecture is essential for solving the Regularity problem for boundary data in $\dot{\textup{W}}^{1,2}(\R^n)$. Those results were extended recently to Schrödinger operators by Morris--Turner in \cite{MorrisTurner}. They used the first-order approach, originally developed for boundary value problems by Auscher--Axelsson--McIntosh in~\cite{AAM_2008}, to prove quadratic estimates for so-called $DB$-type operators (see \cite[Theorem~3.2]{MorrisTurner}) and applied those to obtain well-posedness results for $\El{2}$-Regularity and Neumann boundary value problems for the Schrödinger equation~\eqref{general div form schrodinger equ} on the upper half-space $\Hn$ when $n\geq 3$.
Importantly for our purposes, these quadratic estimates imply that $H$ has a bounded holomorphic functional calculus on $\Ell{2}$, with square root domain $\dom{H^{1/2}}=\mathcal{V}^{1,2}(\Rn)$ and estimate
\begin{equation}\label{first non rigourous apparition of the Kato square root type result for schrodinger operator}
    \norm{H^{1/2}f}_{\El{2}}\eqsim \norm{\nabla_{\mu}f}_{\El{2}}
\end{equation}
for all $f\in\mathcal{V}^{1,2}(\Rn)$, when $V\in\textup{RH}^{q}(\Rn)$ for some $q\geq \max\{\frac{n}{2},2\}$ (see~\cite[Corollary 3.21]{MorrisTurner}). 

Four critical numbers $p_{\pm}(H)$ and $q_{\pm}(H)$ govern the $\El{p}$-extrapolation of solvability for~\eqref{block coefficient schrodinger equation simplified}. 
Analogous to \cite[Section~4]{AuscherEgert}, we define $(p_{-}(H),p_{+}(H))$ and $(q_{-}(H),q_{+}(H))$ to be the largest open intervals of exponents $p\in\left(\frac{n}{n+1},\infty\right)$ for which the respective family
\[
\set{A_{\perp\perp}(1+t^{2}H)^{-1}A_{\perp\perp}^{-1}}_{t>0}
\quad\text{and}\quad 
\set{t\nabla_{\mu}(1+t^{2}H)^{-1}A_{\perp\perp}^{-1}}_{t>0}
\]
is uniformly bounded on $\textup{H}^{p}_{V}(\Rn)$ (recall that $\textup{H}^{p}_{V}(\Rn)=\Ell{p}$ when $p>1$).
These intervals are studied in Section~\ref{section on critical numbers and their properties}. We note that they both contain $p=2$, whilst estimates for the endpoints in some prototypical cases are obtained in Remark \ref{remark about the value of critical numbers in some particular cases}. It is also important to note that the use of the letter $q$ in both $q_{\pm}(H)$ and $V\in\textup{RH}^{q}(\R^n)$ is an unfortunate consequence of the literature, which we keep for consistency.

One of the main tools used in \cite{AuscherEgert} to establish the non-tangential maximal function and conical square function $\El{p}$-bounds for Poisson semigroup solutions is a class of $H$-adapted Hardy spaces $\bb{H}^{s,p}_{H}(\Rn)$ defined for $s\in\R$ and $p\in(0,\infty)$. The introduction of \cite{Duong_Li_2013} provides a useful overview of the history of such operator-adapted spaces. 

An essential step in the extrapolation theory for $\El{p}$-type solvability is then the identification of the $H$-adapted Hardy spaces $\bb{H}^{p}_{H}(\Rn):=\bb{H}^{0,p}_{H}(\Rn)$ and $\bb{H}^{1,p}_{H}(\Rn)$ with “concrete" function spaces for boundary data, such as $\El{p}(\Rn)$ and $\dot{\mathcal{V}}^{1,p}(\Rn)$ when $p>1$. To identify these spaces when $p\leq 1$, we introduce a new class of molecules for $\textup{H}^{p}_{V,\textup{pre}}(\Rn)$ which have supports concentrated on cubes $Q\subset\Rn$ of arbitrary size. This is in contrast to the $\textup{H}^{p}_{V}$-molecules considered previously, such as in \cite[Section~3.2]{BJL_T1_Schrodinger}, which typically restrict the size of such cubes according to the critical radius function in \eqref{definition of Shen's critical radius function}. Instead, we impose a cancellation condition determined only by the scaling in the Fefferman--Phong inequality (see Definition~\ref{molecule Dziubanski hardy space} and Proposition~\ref{improved Fefferman--Phong inequality}). This substitute is vital to the proof of Lemma~\ref{abstract molecules are concrete dziubanski molecules}. An important ingredient for these identifications when $p>1$ is the $\El{p}$-boundedness of the \emph{Riesz transform} $R_H := \nabla_{\mu}H^{-1/2}$. The square-root estimate \eqref{first non rigourous apparition of the Kato square root type result for schrodinger operator} shows that this operator has a bounded extension mapping from $\Ell{2}$ into $\El{2}(\Rn;\C^{n+1})$. The $\El{p}$-extrapolation of these bounds is investigated in Section~\ref{section on riesz transform bounds}. Our main result in this direction is the following theorem. 
\begin{thm}\label{thm boundedness of Riesz transforms on Lp full range}
    If $n\geq 3$, $q\geq \max\{\frac{n}{2},2\}$ and $V\in \textup{RH}^{q}(\Rn)$, then the Riesz transform $R_{H}$ extends to a bounded operator on  $\Ell{p}$ for all $p\in \left(p_{-}(H) ,q_{+}(H)\right)\cap(1,2q]$.
\end{thm}
This result is a significant extension of~\cite[Theorem 7.3]{AuscherEgert}. 
The $\El{p}$-bounds for $R_{H}$ in the range $p\in (1,2]$ follow from a modification of the arguments used therein to treat the case $V\equiv 0$. 
The $\El{p}$-bounds in the range $p\in (2,\infty)$ are more subtle, however, because difficulties arise from the absence of the conservation property $(1+t^2H)^{-1}1=1$, which is exploited when $V\equiv 0$ (see ~\cite[Corollary 5.4]{AuscherEgert}) but fails when $V \not\equiv 0$. These difficulties materialise as an additional term that we estimate by combining the local Fefferman--Phong inequality (see Proposition~\ref{improved Fefferman--Phong inequality}), with a novel “cancellation" bound~\eqref{cancellation bound} that controls averages of $t\nabla_{\mu}(1+t^{2}H)^{-k}(1)$ on balls of radius $t$ for sufficiently large integers $k\geq1$. This is achieved by utilising an extension of the aforementioned conservation property that holds even when $V\not\equiv0$ (see~Lemma~\ref{cancellation lemma}) as well as certain off-diagonal estimates for $\set{t\nabla_{\mu}(1+t^{2}H)^{-k}}_{t>0}$ (see Lemma~\ref{ode resolvents lemma for riesz transform bounds}.(ii)). The later estimates don't appear to have a precedent in the literature. Altogether, this allows us to prove Theorem~\ref{thm boundedness of Riesz transforms on Lp full range} in Section~\ref{subsection on the proof of Lp bounds for riesz trsform}.

Once Theorem~\ref{thm boundedness of Riesz transforms on Lp full range} is at hand, we can proceed to identify the $H$-adapted Hardy spaces $\bb{H}^{p}_{H}(\Rn)$ and $\bb{H}^{1,p}_{H}(\Rn)$ in a range of exponents $p\in(\frac{n}{n+1},\infty)$ determined by the critical numbers $p_{\pm}(H)$ and $q_{\pm}(H)$. More specifically, if $V\in\textup{RH}^{q}(\Rn)$ for some $q>\max\{\frac{n}{2},2\}$ and $\delta:=2-\frac{n}{q}>0$, then we obtain the following identifications with equivalent $p$-(quasi)norms in Theorem~\ref{summary of identifications of H adapted Hardys spaces}:
\begin{equation*}
\begin{array}{ll}
 \bb{H}^{p}_{H}(\Rn)=b\textup{H}^{p}_{V,\textup{pre}}(\Rn), &  \textup{ if }p\in(p_{-}(H)\vee \frac{n}{n+\delta/2}\,,\,p_{+}(H));\\
 \bb{H}^{1,p}_{H}(\Rn)=\dot{\mathcal{V}}^{1,p}(\Rn)\cap\El{2}(\Rn), & \textup{ if }p\in (p_{-}(H)_{*}, q_{+}(H))\cap(1,2q]; \\  
\bb{H}^{1,p}_{H}(\Rn)\cap\mathcal{V}^{1,2}(\Rn)=\dot{\textup{H}}^{1,p}_{V,\textup{pre}}(\Rn),& \textup{ if }p\in(p_{-}(H)_{*}\vee \frac{n}{n+\delta/2}\, ,1]\cap \mathcal{I}(V).
    \end{array}
        \end{equation*}
Note that $p_{-}(H)_{*}$ is the lower Sobolev exponent of $p_{-}(H)$ in dimension $n$ (see Section~\ref{subsubsection on definition of exponents and order relations}).
We recall that the subscript “pre" in the notation $\textup{H}^{p}_{V,\textup{pre}}(\Rn)$ and $\dot{\textup{H}}^{1,p}_{V,\textup{pre}}(\Rn)$ is used to indicate that these are pre-completed spaces, which are subspaces of $\Ell{2}$ and $\mathcal{V}^{1,2}(\Rn)$, respectively. Moreover, we note that $\textup{H}^{p}_{V,\textup{pre}}(\Rn)=\Ell{p}\cap \Ell{2}$ with equivalent $p$-norms when $p>1$ (see Lemma~\ref{identification of Dziubanski hardy spaces for p>=1}).
 
These identifications provide an analogue of \cite[Theorem 9.7]{AuscherEgert}. A significant difference is the presence of two new endpoints in the intervals of identification, namely $\frac{n}{n+\delta/2}=(1+\frac{1}{n}-\frac{1}{2q})^{-1}$ and $2q$, which depend only on the reverse Hölder property of $V$  through the exponent $q$ and dimension $n$. Another distinction is the introduction of an additional subset of exponents $\mathcal{I}(V)$, which is used to obtain an identification of $\bb{H}^{1,p}_{H}(\Rn)$ when $p\leq 1$. An exponent $p\in \left(\frac{n}{n+1},1\right]$ is defined to belong to $\mathcal{I}(V)$ when
\begin{equation}\label{reverse Riesz transform bound from definition of I(V)}
    \norm{\nabla_{\mu}f}_{\textup{H}^{p}_{V}}\gtrsim \norm{(-\Delta 
 +V)^{1/2}f}_{\textup{H}^{p}_{V}} 
\end{equation}
for all $f\in\mathcal{V}^{1,2}(\Rn)$ with the property that each of the $n+1$ components of $\nabla_{\mu}f$ is in $\textup{H}^{p}_{V,\textup{pre}}(\Rn)$. This set allows us to circumvent the need to establish atomic decompositions of $\nabla_{\mu}f$ with atoms belonging to the range of $\nabla_{\mu}$. Instead, we establish Theorem~\ref{thm atomic decomposition of hardy sobolev spaces} as a substitute, after having defined $\dot{\textup{H}}^{1,p}_{V}(\Rn)$ accordingly using the fractional power $H_{0}^{1/2}$ in place of $\nabla_{\mu}$, when $p\leq 1$. These reverse Riesz transform bounds on $\textup{H}^{p}_{V}(\Rn)$ are studied in Section~\ref{subsection on reverse riesz transform bounds on Dziubanski hardy space for no coefficients case}, where a non-trivial class of potentials $V$ satisfying \eqref{reverse Riesz transform bound from definition of I(V)} is exhibited (see Theorem~\ref{full boundedness criteria for adjoint riesz transforms on hardy dziubanski spaces} and Section~\ref{subsubsection on the S alpha smoothness class of potentials}).
It is important to note, however, that the set $\mathcal{I}(V)$ is not required for the \textit{well-posedness} of the Regularity problem in Theorem~\ref{existence and uniqueness result Regularity Lp} below. This is because the only relevant embedding for the associated existence result is the continuous inclusion $\dot{\textup{H}}^{1,p}_{V,\textup{pre}}(\Rn)\subseteq\bb{H}^{1,p}_{H}(\Rn)\cap\mathcal{V}^{1,2}(\Rn)$, which is shown to hold when ${p\in(p_{-}(H)_{*}\vee \frac{n}{n+1},1]}$ in Lemma~\ref{Lemma continuous inclusion V adapted hardy sobolev space for p<=1}. 

Once Theorem~\ref{summary of identifications of H adapted Hardys spaces} is established, we obtain estimates for Poisson semigroup solutions when the boundary data is in $\El{2}(\Rn)\cap \textup{X}(\Rn)$, for each relevant boundary data space $\textup{X}(\Rn)$, in Section~\ref{section estimates towards the dirichelt and regularity problems}. In contrast to \cite[Sections 17.1 and 17.2]{AuscherEgert}, significant  difficulties arise when proving the reverse nontangential maximal function estimates when $p\leq 1$ and these ultimately restrict our consideration to potentials $V\in\textup{RH}^{\infty}(\Rn)$ in that case
(see Propositions \ref{estimate non tangential maximal function on poisson extension solution} and \ref{proposition non tangential max function estimate regularity problem for p<1}). This is because the presence of the potential $V$ implies that the maximal operator used in the definition of the Hardy space $\textup{H}^{p}_{V,\textup{pre}}(\Rn)$ may no longer be associated with convolution operators and must instead be treated using the functional calculus of $H_{0}=-\Delta +V$. We circumvent this challenge by combining the finite propagation speed property of $H_{0}$ (in the sense \cite[Definition 3.3]{HLMMY_HardySpacesDaviesGaffney}) with pointwise bounds on its fundamental solution (see Lemma \ref{lemma L^infty estimate nabla mu from L^infty estimate on H_0}).

We then extend the estimates of Section~\ref{section estimates towards the dirichelt and regularity problems} by density, utilising the approach from \cite[Section~17.3]{AuscherEgert}, to prove solvability of $\left(\mathcal{D}\right)^{\mathscr{H}}_{p}$ and $\left(\mathcal{R}\right)^{\mathscr{H}}_{p}$ for general boundary data in Section~\ref{section on existence for the dirichlet and regularit problems}. To prove uniqueness for the Dirichlet problems, we follow the method from Chapters 20 and 21 of \cite{AuscherEgert}, which originates in \cite{Auscher_Egert_Uniqueness}. This is the content of Section~\ref{section on Uniqueness for the Dirichlet and Regularity problems}. Finally, we treat the Neumann problem in Section \ref{section proof well posedness neumann}.

Our main result for the Dirichlet problem is the following theorem.
\begin{thm}\label{existence and uniqueness result Dirichlet Lp}
Let $n\geq 3$, $q\geq\max\{\frac{n}{2},2\}$ and $V\in\textup{RH}^{q}(\Rn)$. Suppose that $p\geq 1$ and
\[
p_{-}(H)<p<p_{+}(H)^{*}.
\]
If $p>1$, then 
$\left(\mathcal{D}\right)^{\mathscr{H}}_{p}$ is well-posed and the solution $u$ for data $f$ in $\Ell{p}$ satisfies
\[
\norm{N_{*}(u)}_{\El{p}}\eqsim \norm{A_{\perp\perp}f}_{\El{p}}\eqsim \norm{S\left(t\nabla_{\mu}u\right)}_{\El{p}}.
\]
If $p=1$, then $\left(\mathcal{D}\right)^{\mathscr{H}}_{1}$ is well-posed and the solution $u$ for data $f$ in $\textup{H}^{1}_{V}(\Rn)$ satisfies
\[
\norm{N_{*}(u)}_{\El{1}}\lesssim \norm{A_{\perp\perp}f}_{\textup{H}^{1}_{V}}\eqsim \norm{S\left(t\nabla_{\mu}u\right)}_{\El{1}}.
\]
Moreover, the reverse nontangential maximal function bound holds when $V\in\textup{RH}^{\infty}(\Rn)$.
\end{thm}

We obtain the existence and further properties of solutions to the Dirichlet problem, including non-tangential convergence of $\El{2}$-averages to the boundary data, and uniform $\textup{H}^{p}_{V}$-estimates, in Theorem~\ref{existence result for Dirichlet problem again} below. The uniqueness is contained in Theorem~\ref{uniqueness result for dirichlet problem}.

Our main result for the Regularity problem is the following theorem.
\begin{thm}\label{existence and uniqueness result Regularity Lp}
Let $n\geq 3$, $q\geq\max\{\frac{n}{2},2\}$ and $V\in\textup{RH}^{q}(\Rn)$. Suppose that 
\[
p_{-}(H)_{*}<p<q_{+}(H) \quad\text{and}\quad \frac{n}{n+1}<p\leq 2q.
\]
If $p>1$, then $\left(\mathcal{R}\right)^{\mathscr{H}}_{p}$ is well-posed and the solution $u$ for data $f$ in $\dot{\mathcal{V}}^{1,p}(\Rn)$ satisfies 
\[
\norm{N_{*}\left(\nabla_{\mu}u\right)}_{\El{p}} \eqsim \norm{\nabla_{\mu}f}_{\El{p}}\eqsim \norm{S(t\nabla_{\mu}\partial_{t}u)}_{\El{p}}.
\]
If $p\leq 1$, then $\left(\mathcal{R}\right)^{\mathscr{H}}_{p}$ is well-posed and the solution $u$ for data $f$ in $\dot{\textup{H}}^{1,p}_{V}(\Rn)$ satisfies 
\[
\norm{N_{*}\left(\nabla_{\mu}u\right)}_{\El{p}}\lesssim\norm{f}_{\dot{\textup{H}}^{1,p}_{V}}
\quad\text{and}\quad
\norm{S(t\nabla_{\mu}\partial_{t}u)}_{\El{p}}\lesssim\norm{f}_{\dot{\textup{H}}^{1,p}_{V}}.
\]
Moreover, the reverse square function bound holds when ${p\in\left(p_{-}(H)_{*}\vee \frac{n}{n+\delta/2}, 1\right]\cap\mathcal{I}(V)}$, whilst the reverse nontangential maximal function bound holds when $p\in\mathcal{I}(V)$ and $V\in\textup{RH}^{\infty}(\Rn)$.
\end{thm}

We obtain the existence and further properties of solutions to the Regularity problem in Theorems~\ref{existence and uniqueness result Regularity Lp again} and~\ref{existence and uniqueness result Regularity for p<1} below. The uniqueness is contained in Theorem~\ref{uniqueness result for regularity problem}.

Our main result for the Neumann problem is the following theorem.
\begin{thm}\label{existence and uniqueness result Neumann Lp}
Let $n\geq 3$, $q\geq\max\{\frac{n}{2},2\}$ and $V\in\textup{RH}^{q}(\Rn)$. Suppose that  
\[
p_{-}(H)<p<q_{+}(H) \quad \text{and}\quad 1\leq p\leq 2q.
\]
If $p>1$, then $\left(\mathcal{N}\right)^{\mathscr{H}}_{p}$ is well-posed and the solution $u$ for data $g$ in $\Ell{p}$ satisfies 
\[
\norm{N_{*}\left(\nabla_{\mu}u\right)}_{\El{p}} \eqsim \norm{g}_{\El{p}}\eqsim \norm{S(t\nabla_{\mu}\partial_{t}u)}_{\El{p}}.
\]
If $p=1\in\mathcal{I}(V)$, then $\left(\mathcal{N}\right)^{\mathscr{H}}_{1}$ admits a solution $u$ for data $g$ in $\textup{H}^{1}_{V}(\Rn)$ satisfying 
\[
\norm{N_{*}\left(\nabla_{\mu}u\right)}_{\El{1}}\lesssim\norm{g}_{\textup{H}^{1}_{V}}\eqsim 
\norm{S(t\nabla_{\mu}\partial_{t}u)}_{\El{1}}.
\]
Moreover, the reverse nontangential maximal function bound holds when $V\in\textup{RH}^{\infty}(\Rn)$.
\end{thm}

We deduce these results for the Neumann problem as an application of Theorems~\ref{existence and uniqueness result Dirichlet Lp} and \ref{existence and uniqueness result Regularity Lp} in Section~\ref{section proof well posedness neumann} below. Our method provides a new self-contained proof even in the case $V\equiv0$ treated by Auscher--Egert~\cite{AuscherEgert}, since it no longer relies on the results in~\cite{Auscher_Mourgoglou_BoundaryValueProblems} and \cite{AuscherStahlhut_FunctionalCalculus_Dirac}.
The open question of uniqueness for $\left(\mathcal{N}\right)^{\mathscr{H}}_{1}$ is discussed in Remark \ref{remark uniqueness Neumann p=1} after the proof of Theorem \ref{existence and uniqueness result Neumann Lp}.

We stress that throughout the paper $V$ will always denote a \textit{non-negative} locally integrable function. Theorems \ref{thm boundedness of Riesz transforms on Lp full range}, \ref{existence and uniqueness result Dirichlet Lp}, \ref{existence and uniqueness result Regularity Lp} and \ref{existence and uniqueness result Neumann Lp} are proved in \cite{AuscherEgert} in the case $V\equiv 0$. This means that we can restrict our attention to the case $V\not\equiv 0$, or equivalently  $\int_{\Rn}V\in(0,\infty]$,
which often allows us to avoid working modulo constants.

The paper is organised as follows. In Section~\ref{section on preliminaries}, we introduce preliminary material that will be used throughout the paper. Section~\ref{subsection on reverse Holder potentials} collects some important properties of reverse Hölder potentials $V$, whilst  Sobolev spaces adapted to singular potentials $V$ are studied in Section~\ref{section on definition of Sobolev spaces adapted to singular potentials}. The rest of Section~\ref{section on preliminaries} serves as a brief overview of the holomorphic functional calculus for the (bi)sectorial operators that we will consider, along with the construction of the corresponding adapted Hardy spaces. In Section~\ref{section on Hardy spaces adapted to schrodinger operators}, we introduce the Hardy and Hardy--Sobolev spaces $\textup{H}^{p}_{V}(\Rn)$ and $\dot{\textup{H}}^{1,p}_{V}(\Rn)$, which are adapted to the Schrödinger operator $H_{0}:=-\Delta +V$. Topics covered include atomic decompositions (Sections \ref{section on atomic decomposition of hardy dziubanski spaces} and \ref{section on atomic decompositions of Hardy sobolev spaces for schrodinger operators}), interpolation, and duality (Section~\ref{section on families of operators and bddness on dziubanski hardy spaces}). Section~\ref{section on critical numbers and their properties} is concerned with properties of the critical numbers $p_{\pm}(H)$ and $q_{\pm}(H)$, as well as presenting a class of potentials $V$ for which the set $\mathcal{I}(V)$ is non-empty. Theorem~\ref{thm boundedness of Riesz transforms on Lp full range} is proved in Section~\ref{section on riesz transform bounds}, and the identifications of $\bb{H}^{p}_{H}(\Rn)$ and $\bb{H}^{1,p}_{H}(\Rn)$ are established in Section~\ref{section on identification of H adapted Hardy spaces}. Some elementary properties and interior estimates for weak solutions of $\mathscr{H}u=0$ are collected in Section~\ref{section on properties of weak solutions and interior estimates}, before proving the required estimates for Poisson semigroup solutions in Section~\ref{section estimates towards the dirichelt and regularity problems}. Finally, Theorems~\ref{existence and uniqueness result Dirichlet Lp} and~\ref{existence and uniqueness result Regularity Lp} follow from Section~\ref{section on existence for the dirichlet and regularit problems} (existence and estimates) and Section~\ref{section on Uniqueness for the Dirichlet and Regularity problems} (uniqueness), and Theorem \ref{existence and uniqueness result Neumann Lp} is proved in Section \ref{section proof well posedness neumann}.

\subsection*{Acknowledgements}
We would like to thank Pascal Auscher, Tim Böhnlein and Moritz Egert for insightful discussions that have contributed to this paper. This research was conducted at the University of Birmingham with further support provided by the Royal Society (IES$\backslash$R3$\backslash$193232) and we thank the Technical University of Darmstadt for hosting our associated research visit.

\section{Preliminaries}\label{section on preliminaries}
In this first section, we introduce the notation that we will be using, the definitions and properties of the function spaces that we need, as well as some operator-theoretic objects needed to carry out the strategy outlined in the introduction. 

\subsection{Notation and function spaces}
In this section, we set out the notation that we will use throughout the text. We also recall the definitions of some well-known function spaces. 
\subsubsection{Exponents and order relations}\label{subsubsection on definition of exponents and order relations}

For $p\in\left[1,\infty\right]$, we denote the Hölder conjugate $p'\in\left[1,\infty\right]$ such that $\frac{1}{p'}+\frac{1}{p}=1$, with the convention that $\frac{1}{\infty}=0$. For $p_{0},p_{1}\in(0,\infty]$ and an interpolating index $\theta\in[0,1]$, we define $[p_{0},p_{1}]_{\theta}\in (0,\infty]$ by the relation $\frac{1}{[p_{0},p_{1}]_{\theta}}:=\frac{1-\theta}{p_{0}}+\frac{\theta}{p_1}$. For $p\in\left(0,\infty\right)$, we let 
$p^{*}\in\left(0,\infty\right]$ be the Sobolev exponent for $\R^n$ defined as $p^{*}:=\infty$ if $p\geq n$, and $p^{*}:=\frac{np}{n-p}$ if $p<n$. For $p\in\left(0,\infty\right]$, we let $p_{*}\in\left(0,\infty\right)$ be the lower Sobolev exponent defined as $p_{*}:=\frac{np}{n+p}$. 

We shall also be using the notation $a\vee b := \max\set{a,b}$ and $a\wedge b:= \min\set{a,b}$ for $a,b\in\R$. We say that $x\lesssim y$ if there exists $C>0$, independent of $x$ and $y$, such that $x\leq Cy$. We define $x\gtrsim y$ in a similar way, and use the notation $x\eqsim y $ to express that both $x\lesssim y$ and $x\gtrsim y$ hold.
\subsubsection{Geometry in $\Rn$}\label{subsubsection on geometry in Rn}
For all $x\in\Rn$, we let $\abs{x}:=\left(\sum_{i=1}^{n}\abs{x_{i}}^{2}\right)^{1/2}$ denote the Euclidean norm of $x$. For any $x\in \Rn$ and $r>0$, we let 
\begin{equation*}
    B(x,r):=\set{y\in\Rn : \abs{y-x}<r}
\end{equation*}
be the ball centred at $x$ with radius $r$. The distance between two subsets $E,F\subseteq \Rn$ is $\dist(E,F):=\inf\set{\abs{x-y} : x\in E, y\in F}$. For cubes $Q\subset\Rn$ or balls $B\subset\Rn$, and any real $r>0$, the set $rQ$ (or $rB$) denotes the cube (or the ball) with same centre as $Q$, and with sidelength (radius) equal to $r$ times the sidelength of $Q$ (or the radius of $B$). For a given cube $Q\subset\Rn$, we let $\ell(Q)$ denote its sidelength. For cubes $Q\subset\Rn$ (or balls $B\subset\Rn$), and integers $j\geq 1$, consider the annular regions $C_{j}(Q)$ (or $C_{j}(B)$), defined as $C_{1}(Q):=4Q$ and $C_{j}(Q):=2^{j+1}Q\setminus 2^{j}Q$ for all $j\geq 2$. A collection $\set{A_{k}}_{k\in \mathcal{I}}$ of subsets of $\Rn$ is said to have bounded overlap if $\sum_{k\in\mathcal{I}}\ind{A_{k}}(x)\lesssim 1$ for all $x\in\Rn$ (we let $\ind{A}$ denote the indicator function of a set $A$).

\subsubsection{Operators and quasinormed spaces}
Given quasinormed spaces $X$ and $Y$, we let $\mathcal{L}(X,Y)$ denote the space of all bounded linear operators $T:X\to Y$ equipped with the usual operator quasinorm  $\norm{T}_{\mathcal{L}(X,Y)}$. The domain, range and null-space of an operator $T$ are denoted respectively as $\dom{T}$, $\ran{T}$ and $\kernel{T}$. We also consider the intersection $X\cap Y$ with the quasinorm $\|z\|_{X\cap Y}:=\max\set{\|z\|_{X},\|z\|_{Y}}$ for all $z\in X\cap Y$.

\subsubsection{Lebesgue spaces and distribution spaces}\label{subsubsection lebesgue spaces and distributions}
We shall occasionally use the notation $\El{p}_{c}(\Rn)$ for the space of compactly supported functions in $\Ell{p}:=\El{p}(\Rn,\C)$. If the Lebesgue measure of $A$, denoted $\meas{A}$, is finite, we let $\textup{av}_{A}f:=\left(f\right)_{A}:=\dashint_{A}f:=\frac{1}{\meas{A}}\int_{A}f(x)\dd x$ denote the average of a measurable function $f$ over the set $A$. 

For an open subset $O\subseteq \Rn,$ we let $C^{\infty}_{c}(O)$ denote the space of smooth functions with compact support in $O$. The space of Schwartz functions on $\Rn$ is denoted by $\mathcal{S}(\Rn)$. Both spaces are equipped with the usual Fréchet topology, and their topological duals are the distribution spaces denoted by $\mathcal{D}'(\Rn)$ and $\mathcal{S}'(\Rn)$, respectively.

For $1\leq p\leq \infty$, the Sobolev space $\textup{W}^{1,p}(O)$ is the space of all functions $f\in\El{p}(O)$ with distributional gradient $\nabla f\in\El{p}(O;\C^{n})$, equipped with the usual norm
\begin{equation*}
    \norm{f}_{\textup{W}^{1,p}(O)}:=\left(\norm{f}^{p}_{\El{p}(O)} + \norm{\nabla f}_{\El{p}(O)}^{p}  
 \right)^{1/p}.
\end{equation*}
The corresponding local Sobolev spaces are denoted by $\textup{W}_{\loc}^{1,p}(O)$. The homogeneous Sobolev space $\dot{\textup{W}}^{1,p}(O)$ is the space of all distributions $f\in\mathcal{D}'(O)$ (or equivalently functions $f\in\El{1}_{\loc}(O)$) such that $\nabla_{x}f\in\El{p}(O;\C^{n})$, equipped with the seminorm $\norm{f}_{\dot{\textup{W}}^{1,p}}:=\norm{\nabla f}_{\El{p}(O)}$.

\subsubsection{Tent spaces $\tent{s,p}$}\label{section on tent spaces}
These function spaces have been introduced by Coifman--Meyer--Stein \cite{Coifman_Meyer_Stein_TentSpaces}. We refer the reader to \cite{Amenta_tent_spaces_doubling} and \cite{Amenta_tent_spaces2018} for detailed proofs. For $x\in\Rn$, we introduce the cone of unit aperture with vertex at $x$, $\Gamma(x):=\set{(t,y)\in\Hn : \abs{x-y}<t}$.

For a measurable function $F:\Hn \to \C$, we define the corresponding standard conical square function $S(F):\Rn\to [0,\infty]$ by
\begin{equation*}
    S(F)(x):=\left(\iint_{\Gamma(x)}\abs{F(t,y)}^{2}\frac{\dd t\dd y}{t^{1+n}}\right)^{1/2}
\end{equation*}
for all $x\in\Rn$. For $s\in\R$ and $p\in(0,\infty)$, the tent space $\tent{s,p}$ consists of all functions $F\in\El{2}_{\loc}(\Hn)$ such that $\norm{F}_{\tent{s,p}}:=\norm{S\left(t^{-s}F\right)}_{\Ell{p}}<\infty$.
The maps $\norm{\cdot}_{\tent{s,p}}$ are quasinorms (norms when $p\geq 1$), and all tent spaces are quasi Banach spaces, or Banach spaces when $p\geq 1$ (see \cite[Proposition 3.5]{Amenta_tent_spaces_doubling}). Their topology is finer than that of $\El{2}_{\loc}(\Hn)$.
When $s=0$, we simply write $\tent{0,p}=:\tent{p}$. The space $\tent{p}$ coincides with the space $\tent{p}_{2}$ introduced in \cite{Coifman_Meyer_Stein_TentSpaces}.

It follows from Fubini's theorem that $\tent{2}$ and $\El{2}(\Hn;\frac{\dd x\dd t}{t})$ are equal as sets, with equivalent norms. When $1<p<\infty$, the $\El{2}(\Hn;\frac{\dd x\dd t}{t})$ inner-product 
\begin{equation*}
    \langle F,G\rangle_{\El{2}(\Hn;\frac{\dd x\dd t}{t})}:=\iint_{\Hn} F(t,x)\clos{G(t,x)}\frac{\dd x\dd t}{t}
\end{equation*}
identifies the dual of $\tent{s,p}$ with $\tent{-s,p'}$, see \cite[Proposition 1.9]{Amenta_tent_spaces2018}. In particular, there is a $C>0$ such that 
\begin{equation}\label{estimate duality in tent spaces}
    \abs{\langle F,G\rangle_{\El{2}(\Hn;\frac{\dd x\dd t}{t})}}\leq C\norm{F}_{\tent{s,p}}\norm{G}_{\tent{-s,p'}}
\end{equation}
for all $F,G\in\El{2}(\Hn,\frac{\dd x\dd t}{t})$.

When $s=0$ and $0<p\leq 1$, the tent spaces $\tent{p}$ admit atomic decompositions. A $\tent{p}$-atom associated with a cube $Q\subset\Rn$ is a measurable function $A:\Hn\to \C$ with support in $Q\times (0,\ell(Q))$ such that 
$\left(\int_{0}^{\ell(Q)}\int_{Q}\abs{A(s,y)}^{2}\frac{\dd y\dd s}{s}\right)^{1/2}\leq \ell(Q)^{\frac{n}{2}-\frac{n}{p}}$. 
The atomic decomposition \cite[Proposition 5]{Coifman_Meyer_Stein_TentSpaces} says that there is $C\geq 0$ (depending only on $n$ and $p$) such that every $F\in\tent{p}$ can be written as $F=\sum_{i=0}^{\infty}\lambda_{i}A_{i}$ with unconditional convergence in $\El{2}_{\loc}(\Hn)$, where each $A_{i}$ is a $\tent{p}$-atom and $\norm{(\lambda_{i})_{i}}_{\ell^{p}}\leq C\norm{F}_{\tent{p}}$. Moreover, if $F\in\tent{p}\cap\tent{2}$, then the atomic decomposition also converges in $\tent{2}$ and thus in $\El{2}(\Hn,\frac{\dd x\dd t}{t})$. The unconditional convergence in $\El{2}_{\loc}(\Hn)$ is not explicitly stated in \cite[Proposition 5]{Coifman_Meyer_Stein_TentSpaces}, but follows from the construction. Indeed, the decomposition has the form $F=\sum_{i\geq 0}F\ind{\Delta_{i}}$, where $\set{\Delta_{i}}_{i\geq 0}$ is a family of disjoint measurable subsets of $\Hn$ such that $\Hn=\bigcup_{i\geq 0}\Delta_{i}$, see equation~(4.5) in \cite{Coifman_Meyer_Stein_TentSpaces}. This also justifies convergence of the atomic decomposition of $F\in\tent{p}$ in the $\tent{p}$ norm.

For $p\in(0,\infty)$, we shall occasionally need the tent space $\tent{0,p}_{\infty}$ defined as the set of functions $u\in\El{2}_{\loc}(\Hn)$ with finite (quasi)norm $\norm{u}_{\tent{0,p}_{\infty}}:=\norm{N_{*}(u)}_{\El{p}}<\infty$. Again, these are (quasi)Banach spaces whose topologies are finer than that of $\El{2}_{\loc}(\Hn)$.

\subsubsection{Fefferman--Stein Hardy spaces $\textup{H}^{p}(\Rn)$}\label{subsubsection on hardy spaces}
For $p\in(0,1]$, we let $\textup{H}^{p}(\Rn)$ denote the real Hardy space of Fefferman--Stein (see, e.g., \cite[Chapter \RNum{3}]{Stein_HarmonicAnalysis_93}). It is useful to think of these spaces in terms of atomic decompositions. We restrict ourselves to the range $\frac{n}{n+1}<p\leq 1$. An $\El{\infty}$-atom for $\textup{H}^{p}$ is a measurable function $a:\Rn\to\C$ supported in a cube $Q\subset\Rn$, such that $\norm{a}_{\El{\infty}}\leq \meas{Q}^{-1/p}$ and $\int_{\Rn}a=0$. 
The atomic decomposition of the space $\textup{H}^{p}$ (see \cite[Section~\RNum{3}.3.2]{Stein_HarmonicAnalysis_93}) states that every $f\in\textup{H}^{p}(\Rn)$ can be decomposed as $f=\sum_{k=1}^{\infty}\lambda_{k}a_{k}$,
where the sum converges absolutely in $\textup{H}^{p}$, the $a_k$ are $\El{\infty}$-atoms for $\textup{H}^{p}$, and the scalars $\set{\lambda_{k}}_{k\geq 1}$ satisfy
$\sum_{k=1}^{\infty}\abs{\lambda_{k}}^{p}\eqsim \norm{f}^{p}_{\textup{H}^{p}}$. It is useful to note that if $f\in\textup{H}^{p}(\Rn)\cap\Ell{2}$, then the decomposition can be taken to also converge in the $\El{2}$ norm (see \cite[Section~2.4]{AuscherEgert} and the references therein).

\subsection{Reverse Hölder weights}\label{subsection on reverse Holder potentials}

 For $n\geq 1$ and $q\in(1,\infty)$, the reverse Hölder class $\textup{RH}^{q}(\Rn)$ is the vector space consisting of  locally integrable functions $V:\Rn\to \left[0,\infty\right)$ for which there exists $C\geq 0$ such that
\begin{equation}\label{reverse Holder ineq for V}
\left(\dashint_{B}V^{q}\right)^{1/q}\leq C\dashint_{B}V
\end{equation}
for every ball $B\subset \Rn$. The infimum of all such $C\geq 0$ will be denoted by $\llbracket V\rrbracket_{q}$. The reverse Hölder class $\textup{RH}^{\infty}(\Rn)$ is defined by replacing the left-hand side of (\ref{reverse Holder ineq for V}) with $\norm{V}_{\El{\infty}(B)}$. Note that replacing balls with cubes, or indeed any family of sets with bounded eccentricity, gives the same class $\textup{RH}^{q}(\Rn)$ by item (ii) of Lemma~\ref{self-improvement property of reverse holder weights} below.
It follows from Hölder's inequality that $\textup{RH}^{q}(\Rn)\subseteq \textup{RH}^{p}(\Rn)$ for all $1<p\leq q\leq \infty$. As examples of reverse Hölder weights, we mention that $x\mapsto\abs{x}^{-\alpha}$ belongs to $\textup{RH}^{q}(\Rn)$ for all $\alpha>0$ such that $1<q<\frac{n}{\alpha}$, and that for any polynomial $P:\Rn\to \C$ and $a>0$, it holds that $\abs{P}^{a}\in\textup{RH}^{\infty}(\Rn)$ (see Proposition~\ref{polynomials are RH^{infty}}).
The self-improvement, doubling and two other well-known properties of reverse Hölder weights are collected below. 

\begin{lem}\label{self-improvement property of reverse holder weights}
    If $n\geq 1$, $q>1$ and $V\in\textup{RH}^{q}(\Rn)$, then the following properties hold:
    \begin{enumerate}[label=\emph{(\roman*)}]
        \item There exists $\eps>0$ such that $V\in\textup{RH}^{q+\eps}(\Rn)$.
        \item There exists $\kappa>0$ such that for all balls $B\subset\Rn$ it holds that 
    $\int_{2B}V \leq \kappa\int_{B}V$.
    \item For each $s\in(0,1)$ it holds that $V^{s}\in\textup{RH}^{\frac{1}{s}}(\Rn)$.
    \item There exists $C>0$ such that for all $x\in\Rn$ and all $0<r<R$ it holds that 
    \begin{equation*}
        \frac{1}{r^{n-2}}\int_{B(x,r)}V\leq C\left(\frac{R}{r}\right)^{\frac{n}{q}-2}\frac{1}{R^{n-2}}\int_{B(x,R)}V.
    \end{equation*}   
    \end{enumerate} 
\end{lem}
\begin{proof}
    See 6.6(a) in Section~6 of \cite[Chapter \RNum{5}]{Stein_HarmonicAnalysis_93} for (i), Section~1 of the same reference for (ii), \cite[Proposition 11.1]{Auscher-BenAli} for (iii), and \cite[Lemma 1.2]{Shen_Schrodinger} for (iv).
\end{proof}

The following singularity property for reverse H\"older weights $V$ that are not equal to $0$ almost everywhere, henceforth denoted $V\not\equiv 0$, is essential in the proof of Proposition~\ref{embedding homogeneous V spaces in L ^p* for p<n} below.
\begin{lem}\label{non integrability property of reverse holder potentials}
    If $n\geq 1$, $q>1$, $V\in\textup{RH}^{q}(\Rn)$, $V\not\equiv 0$ and $0<p<q$, then $\norm{V}_{\El{p}}=\infty$.
\end{lem}
\begin{proof}
    Since $V\not\equiv 0$, there are $x_0\in\Rn$ and $r_{0}>0$ such that $\int_{B(x_0,r_{0})}V>0$. Therefore, for $R>r_{0}$, item (iv) from Lemma~\ref{self-improvement property of reverse holder weights} gives a $c>0$ such that
    \begin{equation*}
        \frac{1}{R^{n-2}}\int_{B(x_0,R)}V\geq c^{-1}\left(\frac{r_{0}}{R}\right)^{\frac{n}{q}-2}\frac{1}{r_{0}^{n-2}}\int_{B(x_0,r_{0})}V =c_{0}R^{2-\frac{n}{q}},
    \end{equation*}
    where $c_{0}:=c^{-1}r_{0}^{\frac{n}{q}-n}\displaystyle\int_{B(x_0,r_0)}V$. If $p\geq 1$, then we obtain
    \begin{equation*}
        c_{0}R^{-\frac{n}{q}}\leq \frac{1}{R^{n}}\int_{B(x_0,R)}V \lesssim \left(\dashint_{B(x_0,R)}V^p\right)^{\frac{1}{p}}.
    \end{equation*}
If $p\in(0,1)$, then item (iii) of Lemma~\ref{self-improvement property of reverse holder weights} shows that $V^{p}\in\textup{RH}^{1/p}(\Rn)$, hence
    \begin{equation*}
        c_{0}R^{-\frac{n}{q}}\lesssim\dashint_{B(x_0,R)}V\lesssim \left(\dashint_{B(x_0,R)}V^{p}\right)^{1/p}.
    \end{equation*}
    In either case, this proves that there are $x_{0}\in\Rn$, $C>0$ and $r_{0}>0$ such that
    \begin{equation*}
        R^{n\left(\frac{1}{p}-\frac{1}{q}\right)}\leq C\norm{V}_{\El{p}\left(B(x_0,R)\right)}
    \end{equation*}
    for all $R>r_{0}$. The conclusion follows by letting $R\to\infty$.
\end{proof}
 
Now suppose that $n\geq 3$, $q\geq \frac{n}{2}$ and $V\in\textup{RH}^{q}(\Rn)$. Shen \cite{Shen_Schrodinger} introduced the \emph{critical radius function} ${\rho(\cdot,V) :\Rn\to[0,\infty]}$ defined by 
\begin{equation}\label{definition of Shen's critical radius function}
    \rho(x,V):=\sup\set{r>0 : r^{2}\dashint_{B(x,r)}V(y)\dd y \leq 1}
\end{equation}
for all $x\in\Rn$. It follows from items (i) and (iv) of Lemma~\ref{self-improvement property of reverse holder weights} that $\rho(\cdot,V)$ is well-defined since the set defining it is not empty. Moreover, we have $0<\rho(x,V)<\infty$ for all $x\in\Rn$ when $V\not\equiv 0$. It is customary to denote $m(x,V):=\rho(x,V)^{-1}$ for all $x\in\Rn$. When it is clear from the context, we shall abbreviate $\rho(x):=\rho(x,V)$, and similarly for $m(x,V)$. We shall often use the elementary properties of $\rho(\cdot ,V)$ listed in \cite[Lemma 1.4]{Shen_Schrodinger}. Important for us is the “slowly-varying" property, whereby for every $C>0$, whenever $|y-x|\leq C\rho(x)$ it holds that
\begin{equation}\label{comparability property for critical radius function}
\rho(x)\eqsim \rho(y),
\end{equation}
where the implicit constant depends on $C$.
Also, applying item (iv) of Lemma~\ref{self-improvement property of reverse holder weights} with $R=\rho(x)$ and setting $\delta= 2-\frac{n}{q}$ shows that 
\begin{equation}\label{Shen eqn control of integral of potential by the critical radius function}
r^{2}\dashint_{B(x,r)}V\lesssim \left(\frac{r}{\rho(x)}\right)^{\delta}
\end{equation}
for all $x\in\Rn$ and $r\in (0,\rho(x)]$.

The following result is stated without proof in \cite[Estimate (0.8)]{Shen_1996}, and in \cite[Section 0]{Shen_Schrodinger} for non-negative polynomials $P$ and $p=1$. We provide a proof for completeness. 
\begin{prop}\label{polynomials are RH^{infty}}
    Let $n\geq 1$ and $p\in(0,\infty)$. If $P:\Rn\to\C$ is a polynomial of degree $d\geq 0$, then $\abs{P}^{p}\in\textup{RH}^{\infty}(\Rn)$ and 
    \begin{equation}\label{equivalence critical radius function modulus polynomial}
        m(x_{0},\abs{P}^{p})\eqsim \sum_{\abs{\alpha}\leq d}\abs{\partial^{\alpha}P(x_{0})}^{\frac{p}{\abs{\alpha}p+2}}
    \end{equation}
    for all $x_{0}\in\Rn$, where the implicit constants only depend on $d$, $n$ and $p$.
\end{prop}
\begin{proof}
    Let $d\geq 0$ be a fixed integer, and let $\textup{X}$ be the finite dimensional complex vector space of all polynomials $Q:\Rn\to \C$ of degree at most $d$. Observe that the mappings 
    \begin{equation*}
         Q\mapsto 
        \begin{cases}
            \dashint_{B(0,1)}\abs{Q}^{p}, & \text{ if }p\in(0,1);\\
            \left(\dashint_{B(0,1)}\abs{Q}^{p}\right)^{1/p}, & \text{ if }p\in [1,\infty),
        \end{cases}
    \end{equation*}
    and 
    \begin{equation*}
        Q:=\sum_{\abs{\alpha}\leq d}a_{\alpha}x^{\alpha} \mapsto 
        \begin{cases}
            \sum_{\abs{\alpha}\leq d}\abs{a_{\alpha}}^{p},& \text{ if }p\in(0,1);\\
            \left(\sum_{\abs{\alpha}\leq d}\abs{a_{\alpha}}\right)^{1/p}, & \text{ if } p\in [1,\infty),
        \end{cases}
    \end{equation*}
both define norms on $\textup{X}$, and are therefore equivalent. Combining this with a change of variable, this implies that there exists $C\geq 1$ (depending only on $d$, $n$ and $p$) such that for any polynomial $Q(x):=\sum_{\abs{\alpha}\leq d}a_{\alpha}x^{\alpha}$ it holds that
    \begin{equation}\label{two sided estimate polynomial and coefficients}
    C^{-1}\sum_{\abs{\alpha}\leq d}\abs{a_{\alpha}}^{p}r^{\abs{\alpha}p}\leq \dashint_{B(0,r)}\abs{Q}^{p}\leq C\sum_{\abs{\alpha}\leq d}\abs{a_{\alpha}}^{p}r^{\abs{\alpha}p}
    \end{equation}
    for all $r>0$, and therefore $\norm{\abs{Q}^{p}}_{\El{\infty}(B(0,r))}\lesssim \sum_{\abs{\alpha}\leq d}\abs{a_{\alpha}}^{p}r^{\abs{\alpha}p}\lesssim\dashint_{B(0,r)}\abs{Q}^{p}$.
    
    Now, given $x_{0}\in\Rn$ and a polynomial $P:\Rn\to\C$ of degree $d$, we note that $Q(x):=P(x_{0}+x)$, $x\in\Rn$, defines a polynomial of degree $d$, and therefore 
    \begin{equation*}
        \norm{\abs{P}^{p}}_{\El{\infty}(B(x_{0},r))}=\norm{\abs{Q}^{p}}_{\El{\infty}(B(0,r))}\lesssim \dashint_{B(0,r)}\abs{Q}^{p}=\dashint_{B(x_{0},r)}\abs{P}^{p},
    \end{equation*}
    which means that $\abs{P}^{p}\in\textup{RH}^{\infty}(\Rn)$.
    
    For a polynomial $Q(x)=\sum_{\abs{\alpha}\leq d}a_{\alpha}x^{\alpha}$, we let $M_{p}(x_{0},Q):=\sum_{\abs{\alpha}\leq d}\abs{\partial^{\alpha}Q(x_{0})}^{\frac{p}{\abs{\alpha}p+2}}$ for all $x_{0}\in\Rn$, and note that 
    \begin{equation}\label{equivalence M function coefficients polynomial}
    M_{p}(0,Q)=\sum_{\abs{\alpha}\leq d}\abs{\partial^{\alpha}Q(0)}^{\frac{p}{\abs{\alpha}p+2}}=\sum_{\abs{\alpha}\leq d}(\alpha! \abs{a_{\alpha}})^{\frac{p}{\abs{\alpha}p+2}}\eqsim \sum_{\abs{\alpha}\leq d}\abs{a_{\alpha}}^{\frac{p}{\abs{\alpha}p+2}}.
    \end{equation}
    Moreover, observe that the definition~\eqref{definition of Shen's critical radius function} and the two-sided estimate \eqref{two sided estimate polynomial and coefficients} show that
    \begin{equation*}
        \rho(0,\abs{Q}^{p})=\sup\set{r>0 : \sum_{\abs{\alpha}\leq d}\abs{a_{\alpha}}^{p}r^{\abs{\alpha}p+2}\leq C},
    \end{equation*}
    and therefore $\rho(0,\abs{Q}^p)\eqsim \min_{\abs{\alpha}\leq d}\abs{a_{\alpha}}^{-\frac{p}{\abs{\alpha}p+2}}$ (with the understanding that $0^{-\frac{p}{\abs{\alpha}p+2}}=\infty$). Combined with \eqref{equivalence M function coefficients polynomial}, this shows that $m(0,\abs{Q}^{p})\eqsim \max_{\abs{\alpha}\leq d}\abs{a_{\alpha}}^{\frac{p}{\abs{\alpha}p+2}}\eqsim M_{p}(0,Q)$. As above, for any given $x_{0}\in\Rn$, applying this with $Q(x):=P(x_{0}+x)$ proves \eqref{equivalence critical radius function modulus polynomial}.
\end{proof}
\subsection{Sobolev spaces adapted to singular potentials}\label{section on definition of Sobolev spaces adapted to singular potentials}
We consider a class of function spaces obtained by adapting the usual Sobolev spaces $\textup{W}^{1,p}\left(\Omega\right)$ for $p\in\left[1,\infty\right)$ and $n\geq 1$ to account for non-negative potentials $V\in \El{1}_{\loc}(\Rn)$ 
on open sets $\Omega\subseteq \Rn$. If $f\in\El{1}_{\textup{loc}}(\Omega)$ has a weak (distributional) derivative $\nabla_{x}f=\left(\partial_{1}f,\ldots,\partial_{n}f\right)\in\El{1}_{\textup{loc}}(\Omega;\C^{n})$, then $\nabla_{\mu}f:\Omega\to\C^{n+1}$ is the measurable function defined for all $x\in\Omega$ by
\begin{equation*}
    \nabla_{\mu}f(x)=\begin{bmatrix}
        \nabla_{x}f(x)\\
        V(x)^{1/2}f(x)
    \end{bmatrix}.
\end{equation*}
The adapted Sobolev space $\mathcal{V}^{1,p}\left(\Omega\right)$ is  defined to be the set 
\begin{equation}\label{definition of inhomogeneous V adapted sobolev spaces}
\mathcal{V}^{1,p}\left(\Omega\right):=\set{f\in\El{p}(\Omega) : \nabla_{\mu}f\in \El{p}\left(\Omega;\C^{n+1}\right)}
\end{equation}
with the norm
$\norm{f}_{\mathcal{V}^{1,p}(\Omega)}:=(\norm{f}_{\El{p}(\Omega)}^{p} + \norm{\nabla_{\mu}f}_{\El{p}(\Omega,\C^{n+1})}^{p})^{1/p}$
for all $f\in \mathcal{V}^{1,p}(\Omega)$. 
It follows from \cite[Lemma 2.1]{MorrisTurner} that $\mathcal{V}^{1,p}(\Omega)$ is a Banach space and that $\mathcal{V}^{1,2}(\Omega)$ is a Hilbert space with the inner-product 
$\left\langle f,g\right\rangle_{\mathcal{V}^{1,2}(\Omega)}:=\int_{\Omega}(f\overline{g} + \nabla_{\mu}f\cdot\overline{\nabla_{\mu}g})$
for all $f,g\in\mathcal{V}^{1,2}(\Omega)$. We shall also consider the local variants $\mathcal{V}^{1,p}_{\loc}(\Omega)$ defined as the space of functions $f\in\textup{W}^{1,p}_{\loc}(\Omega)$ for which $V^{1/2}f\in\El{p}_{\loc}(\Omega)$  equipped with the natural locally convex topology. We also define the homogeneous adapted Sobolev space $\dot{\mathcal{V}}^{1,p}(\Rn)$ as the set 
\begin{equation}\label{definition of homogeneous V adapted Sobolev spaces}
    \dot{\mathcal{V}}^{1,p}(\Rn):=\set{f\in\El{1}_{\text{loc}}(\Rn) : \nabla_{\mu}f\in\El{p}\left(\Rn; \C^{n+1}\right)}
\end{equation}
equipped with the norm
$\norm{f}_{\dot{\mathcal{V}}^{1,p}(\R^n)}:=\norm{\nabla_{\mu}f}_{\El{p}(\Rn;\C^{n+1})}$
for all $f\in\dot{\mathcal{V}}^{1,p}(\Rn)$ when $V\not\equiv 0$. %This is indeed a norm since if $\norm{f}_{\dot{\mathcal{V}}^{1,p}=0}$, then $\norm{\nabla_{x}f}_{\El{p}}=0$, hence $f$ is constant, and that constant is zero since $\norm{V^{1/2}f}_{\El{p}}=0$.

\begin{rem}\label{remark about different definition of V^{11,2} in Morris Turner}
    We note that when $n\geq3$, the space $\dot{\mathcal{V}}^{1,2}(\Rn)$ given by \eqref{definition of homogeneous V adapted Sobolev spaces} may be larger than that given by (2.5) in \cite{MorrisTurner}, as here we require $f$ to be in $\El{1}_{\loc}(\R^n)$ instead of $\El{2^*}(\R^n)$. Proposition~\ref{embedding homogeneous V spaces in L ^p* for p<n} below shows that these definitions coincide when  $V\in\textup{RH}^{q}(\Rn)$ for $q\geq \frac{n}{2}$ and $n\geq 3$, which is the only case when we will need to use these spaces. The definition of $\dot{\mathcal{V}}^{1,p}(\Rn)$ given by \eqref{definition of homogeneous V adapted Sobolev spaces} is chosen here because it is more suited to our purposes when $p\neq 2$.
\end{rem}

The following properties of $\dot{\mathcal{V}}^{1,p}(\Rn)$ are important considerations given the remark above.

\begin{lem}\label{completeness lemma for homogeneous V adapted sob spaces}
    Let $n\geq1$ and $V\in\El{1}_{\loc}(\Rn)$ be non-negative with $V\not\equiv 0$. If $p\in[1,\infty)$, then  $\dot{\mathcal{V}}^{1,p}(\Rn)$ is a Banach space. In addition, $\dot{\mathcal{V}}^{1,2}(\Rn)$ is a Hilbert space with the inner-product 
$    \langle f,g\rangle_{\dot{\mathcal{V}}^{1,2}(\Rn)}:=\langle \nabla_{\mu}f,\nabla_{\mu}g\rangle_{\El{2}(\Rn;\C^{n+1})}$
for all $f,g\in\dot{\mathcal{V}}^{1,2}(\Rn)$.
\end{lem}
\begin{proof}
    It suffices to prove completeness with respect to the norm $\norm{\cdot}_{\dot{\mathcal{V}}^{1,p}(\Rn)}.$ We begin with a general observation. Let $Q_0\subset\Rn$ be an arbitrary cube. Then, if $f\in\El{1}_{\loc}(\Rn)$ has a weak gradient $\nabla_{x}f\in\El{p}(\Rn;\C^{n})$, the Poincaré inequality lets us estimate, for all larger cubes $Q\supseteq Q_0$,
    \begin{align*}
        \int_{Q}\abs{f-(f)_{Q_0}}&\leq \left(1+\frac{\meas{Q}}{\meas{Q_0}}\right)\int_{Q}\abs{f-(f)_Q}
        \leq c_{n}\left(1+\frac{\meas{Q}}{\meas{Q_0}}\right)l(Q)\int_{Q}\abs{\nabla_{x}f}\\
        &\leq c_{n}\left(1+\frac{\meas{Q}}{\meas{Q_0}}\right)l(Q)\meas{Q}^{1-\frac{1}{p}}\norm{\nabla_{x}f}_{\El{p}}.
    \end{align*}
    We can now proceed with the proof. Let  $(f_k)_{k\geq 1}$ be a Cauchy sequence in $\left(\dot{\mathcal{V}}^{1,p},\norm{\cdot}_{\dot{\mathcal{V}}^{1,p}}\right)$. Since $(\nabla_{x}f_{k})_{k\geq 1}$ forms a Cauchy sequence in $\El{p}$, the observation above implies that the sequence $(g_k)_{k\geq 1}:=\left(f_{k}-\left(f_{k}\right)_{Q_0}\right)_{k\geq 1}$ is Cauchy in $\El{1}_{\loc}(\Rn)$. Let $g\in\El{1}_{\loc}(\Rn)$ such that $g_{k}\to g$ in $\El{1}_{\loc}(\Rn)$ as $k\to\infty$. Since the sequence $(\nabla_{x}f_k)_{k\geq 1}=(\nabla_{x}g_k)_{k\geq 1}$ is Cauchy in $\El{p}$, there is $F\in\El{p}(\Rn;\C^{n})$ such that $\nabla_{x}f_k=\nabla_{x}g_k\to F$
    in $\El{p}$ as $k\to\infty$. A standard argument (see, e.g., \cite[Lemma 2.1]{MorrisTurner}) now implies that $g$ has a weak gradient in $\El{p}$, given by $\nabla_{x}g=F$. Moreover, since the sequence $\left(V^{1/2}f_k\right)_{k\geq 1}$ is Cauchy in $\El{p}$, there is a function $G\in\El{p}$ such that $V^{1/2}f_k\to G$ in $\El{p}$, as $k\to\infty$. In particular, there is a subsequence $(k_l)_{l\geq 1}\subset\N$ such that $V(x)^{1/2}f_{k_l}(x)\to G(x)$ as $l\to\infty$, for almost every $x\in\Rn$. Then, we can extract a further subsequence $(k_{l_m})_{m\geq 1}\subset\N$ such that $g_{k_{l_m}}(x)\to g(x)$ as $m\to\infty$ for almost every $x\in\Rn$. As a consequence, we obtain that
    \begin{equation*}
        V(x)^{1/2}\left(f_{k_{l_m}}\right)_{Q_0}=V(x)^{1/2}f_{k_{l_m}}(x)-V(x)^{1/2}g_{k_{l_m}}(x)\to G(x)-V(x)^{1/2}g(x)
    \end{equation*}
    as $m\to\infty$, for almost every $x\in\Rn$. Since $V\neq 0$ on a set of positive measure, this implies that the subsequence of averages $\left(\left(f_{k_{l_m}}\right)_{Q_0}\right)_{m\geq 1}$ converges as $m\to\infty$, towards a constant that we shall denote $c$. Consequently, the subsequence $\left(f_{k_{l_{m}}}\right)_{m\geq 1}\subset \El{1}_{\loc}(\Rn)$ converges towards $\tilde{g}:=g+c$ in $\El{1}_{\loc}(\Rn)$, and
    \begin{equation*}
        V(x)^{1/2}f_{k_{l_{m}}}(x)=V(x)^{1/2}\left(g_{k_{l_{m}}}(x) + \left(f_{k_{l_m}}\right)_{Q_0}    \right)\to V(x)^{1/2}\tilde{g}(x)
    \end{equation*}
    as $m\to\infty$, for almost every $x\in\Rn$. Hence, $V^{1/2}\tilde{g}=G\in\El{p}$, and $V^{1/2}f_{k}\to V^{1/2}\tilde{g}$ in $\El{p}$, as $k\to\infty$. Finally, we have $\nabla_{x}\tilde{g}=\nabla_{x}g=F\in\El{p}(\Rn;\C^{n})$, and $\nabla_{x}f_{k}\to \nabla_{x}\tilde{g}$ in $\El{p}$ as $k\to\infty$. This proves that $\tilde{g}\in\dot{\mathcal{V}}^{1,p}(\Rn)$, and that $f_{k}\to \tilde{g}$ in $\dot{\mathcal{V}}^{1,p}$, hence $\dot{\mathcal{V}}^{1,p}$ is complete.
\end{proof}

The remainder of this section is dedicated to proving a density result for the homogeneous spaces $\dot{\mathcal{V}}^{1,p}(\Rn)$ which is analogous to that recorded below for $\mathcal{V}^{1,p}(\Rn)$ from \cite[Lemma 2.2]{MorrisTurner}.

\begin{lem}\label{density lemma adapted sobolev spaces Morris Turner}
    If $n\geq 1$, $p\in[1,\infty)$ and $V^{p/2}\in\El{1}_{\loc}(\Rn)$, then $\smcp$ is dense in $\mathcal{V}^{1,p}(\Rn)$. Moreover, if $p_0,p_1\in [1,p]$, then $\smcp$ is dense in $\mathcal{V}^{1,p_0}(\Rn)\cap \mathcal{V}^{1,p_1}(\Rn)$.
\end{lem}

Such a density result for $\dot{\mathcal{V}}^{1,p}(\R^n)$ cannot hold under the mere assumption that $V^{p/2}\in\El{1}_{\loc}(\Rn)$. Indeed, if $n\geq 3$, $p\in[1,n)$, and we choose $V\not\equiv 0$ such that $V^{p/2}\in\Ell{1}$, then clearly $V^{p/2}\in\El{1}_{\loc}(\Rn)$ (but $V\not\in\textup{RH}^{\frac{n}{2}}$ by Lemma~\ref{non integrability property of reverse holder potentials}). Let $\textup{X}$ denote the closure of $C^{\infty}_{c}(\Rn)$ in $\dot{\mathcal{V}}^{1,p}(\Rn)$. It follows from the Sobolev embedding theorem and the fact that $V\not\equiv 0$ that $\textup{X}\subseteq\El{p*}$. However, $\ind{\Rn}\in\dot{\mathcal{V}}^{1,p}(\Rn)$, but $\ind{\Rn}\notin\Ell{p^{*}}$. This shows that $C^{\infty}_{c}(\Rn)$ is not dense in $\dot{\mathcal{V}}^{1,p}(\R^n)$.

The following result will allow us to prove in Proposition~\ref{density lemma in homogeneous V adapted Sobolev spaces} below that it does hold when $n\geq 3$, $q\geq \frac{n}{2}$ and $V\in\textup{RH}^{q}(\Rn)$, which is sufficient for our main results.
\begin{prop}\label{embedding homogeneous V spaces in L ^p* for p<n}
Suppose that either $n\geq 3$ and $q\geq \frac{n}{2}$, or $n=2$ and $q>1$. If $V\in\textup{RH}^{q}(\Rn)$, $V\!\not\equiv 0$ and $p\in[1,n)$, then $\dot{\mathcal{V}}^{1,p}(\Rn)\! \subseteq \El{p^{*}}\!(\Rn)$ and 
    $\norm{f}_{\El{p^{*}}\!}\! \lesssim \norm{\nabla_x f}_{\El{p}}$
   for all $f\in \dot{\mathcal{V}}^{1,p}(\Rn)$.
\end{prop}
\begin{proof}
    Since $f\in\El{1}_{\loc}(\Rn)$ with $\nabla_{x}f\in\El{p}$, there is a sequence $(f_{k})_{k\geq 1}\subset\smcp$ such that $\nabla_{x}f_{k}\to \nabla_{x}f$ in $\El{p}$ as $k\to\infty$. In fact, a standard regularisation argument shows that we can assume that $f\in C^{\infty}(\Rn)\cap\dot{\textup{W}}^{1,p}(\Rn)$. Next, for $R>0$, consider the annulus $C_{R}:=\set{x\in\Rn : R<\abs{x}<2R}$ and the bump function $\eta_{R}(x):=\eta\left(\frac{x}{R}\right)$, where $\eta\in C^{\infty}_{c}(B(0,2))$ is such that $\eta\geq 0$ and $\eta=1$ on $B(0,1)$. The sequence $f_{k}:=(f-(f)_{C_{k}})\eta_{k}$, $k\geq 1$ has the required property (see, e.g., \cite[Theorem~4]{Hajłasz1995}). 
    
    The Sobolev embedding theorem applied to $f_{k}-f_{j}$, $j,k\geq 1$ implies that the sequence $(f_{k})_{k\geq 1}$ is Cauchy in $\El{p^{*}}$. If $g\in\El{p^{*}}$ is the limit of this sequence in $\El{p^{*}}$, then $\nabla_{x}g=\nabla_{x}f$ in the distributional sense, hence $f=g+c$ for some $c\in\C$. It suffices to prove that $c=0$. We trivially have $cV^{1/2}=V^{1/2}f-V^{1/2}(f-c)$, hence for all balls $B\subset \Rn$ it holds that 
    \begin{align}\label{intermediate estimate triangle and holder}
        \abs{c}\norm{V^{1/2}}_{\El{p}(B)}\leq \norm{V^{1/2}f}_{\El{p}(B)}+\norm{V^{1/2}}_{\El{n}(B)}\norm{f-c}_{\El{p^{*}}(B)},
    \end{align}
    by Hölder's inequality. If $n\geq 3$, then since $V\in\textup{RH}^{\frac{n}{2}}(\Rn)$  we have 
    \begin{equation*}
        \norm{V^{1/2}}_{\El{n}(B)}=\left(\int_{B}V^{n/2}\right)^{\frac{1}{n}}\lesssim \meas{B}^{\frac{1}{n}}\left(\dashint_{B}V\right)^{\frac{1}{2}}.
    \end{equation*}
    The same estimate trivially holds for $n=2$. Now, if $p\in[2,n)$, then $p/2\geq 1$, so by Hölder's inequality, 
    \begin{align}\label{reverse holder property of V^1/2 for Sobolev}
        \norm{V^{1/2}}_{\El{n}(B)}&\lesssim \meas{B}^{\frac{1}{n}}\left(\dashint_{B}V\right)^{1/2}\leq \meas{B}^{\frac{1}{n}-\frac{1}{p}}\norm{V^{1/2}}_{\El{p}(B)}.
    \end{align}
    If $p\in[1,2)$, then $p/2\in (0,1)$, so by Lemma~\ref{self-improvement property of reverse holder weights}.(iii) we have $V^{p/2}\in\textup{RH}^{\frac{2}{p}}$. We therefore obtain the same estimate as \eqref{reverse holder property of V^1/2 for Sobolev} up to a multiplicative constant. Combining \eqref{intermediate estimate triangle and holder} and \eqref{reverse holder property of V^1/2 for Sobolev}, we get
    \begin{align*}
        \abs{c}\norm{V^{1/2}}_{\El{p}(B)}\leq \norm{V^{1/2}f}_{\El{p}} + \meas{B}^{\frac{1}{n}-\frac{1}{p}}\norm{V^{1/2}}_{\El{p}(B)}\norm{g}_{\El{p^{*}}}.
    \end{align*}
    If $B\subset\Rn$ is sufficiently large to have $||V^{1/2}||_{\El{p}(B)}>0$, this yields the estimate
    \begin{equation*}
        \abs{c}\lesssim \frac{\norm{V^{1/2}f}_{\El{p}}}{\norm{V^{1/2}}_{\El{p}(B)}} + \meas{B}^{\frac{1}{n}-\frac{1}{p}}\norm{g}_{\El{p^{*}}}.
    \end{equation*}
    Finally, note that $\norm{V^{1/2}}_{\El{p}(B)}=\norm{V}^{1/2}_{\El{p/2}(B)}$. If $n\geq 3$, then $V\in\textup{RH}^{\frac{n}{2}}$, with $\frac{p}{2}<\frac{n}{2}$, hence Lemma~\ref{non integrability property of reverse holder potentials} implies that $c=0$ by letting the radius of $B$ tend to infinity. Similarly, if $n=2$, then $V\in\textup{RH}^{q}$ for $q>1$ and it follows that $0<\frac{p}{2}<\frac{n}{2}=1<q$, which allows us to conclude using Lemma~\ref{non integrability property of reverse holder potentials}.
\end{proof}

We also need a Fefferman--Phong inequality proved by Auscher--Ben Ali in \cite[Lemma 2.1]{Auscher-BenAli}.
\begin{prop}[Fefferman--Phong]\label{improved Fefferman--Phong inequality}
    If $n\geq 1$, $p\in[1,\infty)$,  $q\in(1,\infty)$ and $V^{p/2}\in\textup{RH}^{q}(\Rn)$, then there exists $\beta\in(0,1)$ such that for all cubes $Q\subset \Rn$ and $f\in\mathcal{V}^{1,p}_{\loc}(\Rn)$ it holds that
    \begin{equation*}
        m_{\beta}\left(l(Q)^{p}\textup{av}_{Q}V^{\frac{p}{2}}\right)\int_{Q}\abs{f}^{p}\lesssim l(Q)^{p}\int_{Q}\abs{\nabla_{\mu}f}^{p},
    \end{equation*}
    where $m_{\beta}(x):=x$ for $0\leq x\leq 1$, and $m_{\beta}(x):=x^{\beta}$ for $x\geq 1$. 
\end{prop}
\begin{proof}
    We follow the arguments in the proof \cite[Proposition 2.3]{MorrisTurner} to extend \cite[Lemma 2.1]{Auscher-BenAli}.
\end{proof}

We can now prove the following density results for the homogeneous spaces $\dot{\mathcal{V}}^{1,p}(\Rn)$.

\begin{prop}\label{density lemma in homogeneous V adapted Sobolev spaces}
If $V\not\equiv 0$ and either \emph{(i)}, \emph{(ii)} or \emph{(iii)} holds, then $\smcp$ is dense in $\dot{\mathcal{V}}^{1,p}(\Rn)$:
        \begin{enumerate}[label=\emph{(\roman*)}]
            \item $n\geq 3$, $q\geq \frac{n}{2}$, $p\in[1,n]$ and  $V\in\textup{RH}^{q}(\Rn)$;
            \item $n\in\set{1,2}$, $q>1$, $p\in[1,n]$ and $V\in\textup{RH}^{q}(\Rn)$;
            \item $n\geq 1$, $p\in(n,\infty)$ and $V^{p/2}\in\El{1}_{\loc}(\Rn)$.
        \end{enumerate}
Moreover, if $p_0,p_1\in[1,\infty)$ and either \emph{(i)}, \emph{(ii)} or \emph{(iii)} hold with both $p=p_{0}$ and $p=p_{1}$, then $\smcp$ is dense in $\dot{\mathcal{V}}^{1,p_0}(\Rn) \cap \dot{\mathcal{V}}^{1,p_1}(\Rn)$. 
\end{prop}
\begin{proof}
    We will closely follow the argument of \cite[Lemma 2.2]{MorrisTurner}. Let us first remark that our assumptions imply that $\smcp\subset \dot{\mathcal{V}}^{1,p}(\Rn)$ for all $1\leq p<\infty$. Indeed, since $V\in\textup{RH}^{q}$ implies that $V\in\El{q}_{\loc}(\Rn)$, this always yields $V^{1/2}\in\El{n}_{\loc}(\Rn)$. 
    
    Let $f\in\dot{\mathcal{V}}^{1,p}(\Rn)$. By approximating separately the real and imaginary parts of $f$, we may assume that $f$ is real valued. For each $N\geq 1$, we consider the truncation
    \begin{equation*}
        f_{N}(x):=\begin{cases}
            N, & \text{ if }f(x)>N;\\
            f(x), & \text{ if }\abs{f(x)}\leq N;\\
            -N, & \text{ if }f(x)<-N,
        \end{cases}
    \end{equation*}
    for all $x\in\Rn$. It is well-known that $f_{N}\in \dot{\textup{W}}^{1,p}(\Rn)$, with $\abs{\nabla_{x}f_{N}(x)}\leq \abs{\nabla_{x}f(x)}$ for almost every $x\in\Rn$. Moreover, $\abs{f_{N}(x)}\leq \abs{f(x)}$ for all $x\in\Rn$, hence $f_{N}\in\dot{\mathcal{V}}^{1,p}(\Rn)\cap\Ell{\infty}$, with $\abs{\nabla_{\mu}f_{N}(x)}\leq \abs{\nabla_{\mu}f(x)}$ for almost every $x\in\Rn$. Moreover, it is known that $\nabla_{x}f_{N}\to\nabla_{x}f$ in $\El{p}$ as $N\to\infty$. Similarly, by dominated convergence we deduce that $V^{1/2}f_{N}\to V^{1/2}f$ in $\El{p}$ as $N\to\infty$. Consequently, $f_{N}\to f$ in $\dot{\mathcal{V}}^{1,p}(\Rn)$ as $N\to\infty$. 
    
    Next, let $\eta\in\smcp$ be $[0,1]$-valued, with support in $B(0,2)$, and such that $\eta(x)=1$ for $\abs{x}\leq 1$. Then, for all $N\geq 1$ and $R\geq 1$, we consider the approximant 
    $f_{N,R}(x):=\eta\left({x}/{R}\right)f_{N}(x)$
    for all $x\in\Rn$. First, observe that $V^{1/2}f_{N,R}\to V^{1/2}f_{N}$ in $\Ell{p}$ as $R\to\infty$, by dominated convergence. Then, the product rule shows that $f_{N,R}\in\dot{\mathcal{V}}^{1,p}(\Rn)$, with
    \begin{equation*}
        \nabla_{x}f_{N,R}-\nabla_{x}f_{N}=R^{-1}\left(\nabla_{x}\eta\right)\left(\frac{\cdot}{R}\right)f_{N} + \left(\eta\left(\frac{\cdot}{R}\right)-1\right)\nabla_{x}f_{N}.
    \end{equation*}
    Again, dominated convergence shows that $\left(\eta\left(\frac{\cdot}{R}\right)-1\right)\nabla_{x}f_{N}\to 0$ in $\El{p}$ as $R\to\infty$. The first term needs to be handled differently depending on the value of $p$. 

    First, if $1\leq p< n$ (and thus $n\geq 2$), Proposition~\ref{embedding homogeneous V spaces in L ^p* for p<n} shows that $f\in\Ell{p^{*}}$. Therefore, we also have $f_{N}\in \Ell{p^{*}}$. As a consequence, Hölder's inequality implies that
    \begin{align*}
        \norm{R^{-1}\left(\nabla_{x}\eta\right)\left(\frac{\cdot}{R}\right)f_{N}}_{\El{p}}&\leq R^{-1}\norm{f_{N}}_{\El{p^{*}}(R\leq \abs{x}\leq 2R)}\left(\int_{R\leq \abs{x}\leq 2R}\abs{\nabla_{x}\eta\left(\frac{x}{R}\right)}^{n}\dd x\right)^{1/n}\\
        &\lesssim \norm{f_{N}}_{\El{p^{*}}(R\leq \abs{x}\leq 2R)} \to 0
    \end{align*}
    as $R\to\infty$. 
    
Next, if $p>n$, then we simply obtain
    \begin{align*}
        \norm{R^{-1}\left(\nabla_{x}\eta\right)\left(\frac{\cdot}{R}\right)f_{N}}_{\El{p}}\lesssim R^{\frac{n}{p}-1}\norm{f_{N}}_{\Ell{\infty}}\norm{\nabla_{x}\eta}_{\Ell{p}}\to 0
    \end{align*}
    as $R\to\infty$, since $\frac{n}{p}-1<0$.
Finally, if $p=n$, we note that our assumptions imply that there is always some $r>1$ such that $V^{\frac{n}{2}}\in\textup{RH}^{r}$. Indeed, if $n\geq 3$, then $V\in\textup{RH}^{\frac{n}{2}+\eps}$ for some $\eps>0$ and we can take $r=\frac{n/2 +\eps}{n/2}$. If $n=2$ and $q>1$ is such that $V\in\textup{RH}^{q}$, then we may take $r=q$. Similarly, if $n=1$ and $V\in\textup{RH}^{q}$, we may take $r=2$ thanks to Lemma~\ref{self-improvement property of reverse holder weights}.(iii). We may therefore apply the Fefferman--Phong inequality (Proposition~\ref{improved Fefferman--Phong inequality}) to $f_{N}\in\mathcal{V}^{1,n}_{\loc}(\Rn)$, to find some $\beta\in(0,1)$ and $C>0$ such that
    \begin{equation}\label{intermediate step fefferman phong ineq}
        m_{\beta}\left(\int_{Q}V^{n/2}\right)\int_{Q}\abs{f_{N}}^{n} \leq C l(Q)^{n}\int_{Q}\abs{\nabla_{\mu}f_{N}}^{n}
    \end{equation}
    for all cubes $Q\subset \Rn$ and all $N\geq 1$. Therefore, if $\set{Q_{R}}_{R\geq 1}$ is a family of cubes such that $B(0,2R)\subseteq Q_{R}$ for all $R\geq 1$, Lemma~\ref{non integrability property of reverse holder potentials} applied to $w:=V\in\textup{RH}^{\frac{n}{2}+\eps}$ shows that $\int_{Q_{R}}V^{n/2}>1$ for all sufficiently large $R$. Applying (\ref{intermediate step fefferman phong ineq}) with $Q:=Q_{R}$, we get
    \begin{align*}
    \norm{R^{-1}\left(\nabla_{x}\eta\right)\left(\frac{\cdot}{R}\right)f_{N}}_{\El{n}}
    \leq \left(\int_{Q_{R}}\abs{f_{N}}^{n}\right)^{1/n}
    \lesssim \norm{\nabla_{\mu}f_{N}}_{\El{n}}\left(\int_{B(0,2R)}V^{n/2}\right)^{-\frac{\beta}{n}}.
\end{align*}
    Finally, Lemma~\ref{non integrability property of reverse holder potentials} shows that the right-hand side tends to $0$ as $R\to\infty$. This proves that, for any $p\in [1,\infty)$ and $N\geq 1$, $f_{N,R}\to f_{N}$ in $\dot{\mathcal{V}}^{1,p}(\Rn)$ as $R\to\infty$. 
    We can now conclude using the mollification argument used in the proof of \cite[Lemma 2.2]{MorrisTurner}. The second part of the lemma follows, since the construction of the approximating sequence does not depend on $p$.
\end{proof}

The approximating sequence used in the proof of Proposition~\ref{density lemma in homogeneous V adapted Sobolev spaces} above is the same as that constructed in the proof of \cite[Lemma 2.2]{MorrisTurner}. Moreover, by dominated convergence and well-known properties of mollifiers, it is easily shown to converge to the original function $f$ in every $\El{p}$ space that $f$ belongs to. We therefore obtain the following corollary.
\begin{cor}\label{density in mixed V adapted sobolev spaces}
Suppose that $V\not\equiv 0$ and either \emph{(i)}, \emph{(ii)} or \emph{(iii)} from Proposition~\emph{\ref{density lemma in homogeneous V adapted Sobolev spaces}} holds:
\begin{enumerate}[label=\emph{(\arabic*)}]
    \item If $r\in[1,\infty)$ and $V^{r/2}\in\El{1}_{\loc}(\Rn)$, then $\smcp$ is dense in $\dot{\mathcal{V}}^{1,p}(\Rn) \cap \mathcal{V}^{1,r}(\Rn)$.
    \item If $r\in[1,\infty)$, then $\smcp$ is dense in $\dot{\mathcal{V}}^{1,p}(\Rn) \cap \El{r}(\Rn)$.
\end{enumerate} 
\end{cor}
\subsection{Holomorphic functional calculus and Hardy spaces for (bi)sectorial operators}\label{section operators and their functional calculus}
\subsubsection{Sectorial and bisectorial operators}\label{section on sectorial and bisectorial operators }
We consider the open sectors defined as
\begin{align*}
    S_{\omega}^{+}&:=\set{z\in\C : z\neq 0 \text{ and }\abs{\arg{z}}< \omega}
\end{align*}
for any $\omega\in\left(0,\pi\right]$. We also let $S_{0}^{+}:=(0,\infty)$, even though it is not open. We can then define the open bisector $S_{\mu}:=S_{\mu}^{+}\cup \left(-S_{\mu}^{+}\right)$ for $\mu\in\left[0,\frac{\pi}{2}\right]$. We say that a (necessarily closed) operator $T$ on a Hilbert space $\mathcal{H}$ is sectorial of type $S_{\omega}^{+}$, or of angle $\omega\in\left[0,\pi\right)$, if the spectrum $\sigma(T)\subseteq \clos{S^{+}_{\omega}}$ and for each $\mu\in(\omega,\pi)$, there is $C_{\mu}>0$ such that 
$
    \norm{\left(z-T\right)^{-1}}_{\mathcal{L}(\mathcal{H})}\leq C_{\mu}\abs{z}^{-1}
$
for all $z\in\C\setminus \clos{S_{\mu}^{+}}$. Bisectorial operators of type $S_{\omega}$, for $\omega\in[0,\frac{\pi}{2})$, are defined similarly, upon replacing sectors with bisectors. We refer the reader to \cite[Section~2.1]{HaaseFunctionalCalculusBook} for some well-known properties of these operators.

\subsubsection{Classes of holomorphic functions on sectors}\label{subsection on classes of holomorphic functions}
Let $\mu\in (0,\pi)$. We say that a holomorphic function $\varphi:S^{+}_{\mu}\to\C$ belongs to $\mathcal{F}(S_{\mu}^{+})$ if there are $C,s>0$ such that $\abs{\varphi(z)}\leq C\left(\abs{z}^{s} + \abs{z}^{-s}\right)$ for all $z\in S^{+}_{\mu}$. For $\alpha,\beta>0$, we say that $\varphi\in\Psi^{\alpha}_{\beta}(S_{\mu}^{+})$ if there is a $C>0$ such that 
$
    \abs{\varphi(z)}\leq C\left(\abs{z}^{\beta}\wedge \abs{z}^{-\alpha}\right)
$
for all $z\in S^{+}_{\mu}$. We shall also write $\textup{H}^{\infty}(S^{+}_{\mu})$ for the class of all bounded holomorphic functions on $S^{+}_{\mu}$. We also consider the classes $\textup{H}^{\infty}_{0}(S_{\mu}^{+}):=\bigcup_{\alpha,\beta>0} \Psi_{\beta}^{\alpha}(S_{\mu}^{+})$ and $\Psi^{\infty}_{\infty}(S_{\mu}^{+}) := \bigcap_{\alpha,\beta>0} \Psi_{\beta}^{\alpha}(S_{\mu}^{+})$. We shall in general suppress the reference to the sector when it is irrelevant or implicit. We note that $\textup{H}^{\infty}, \textup{H}^{\infty}_{0}$ and $\Psi^{\infty}_{\infty}$ are commutative algebras of functions under pointwise multiplication. Similarly, we define the classes $\Psi^{\alpha}_{\beta}(S_{\omega})$, $\textup{H}^{\infty}(S_{\omega}),  \textup{H}^{\infty}_{0}(S_{\omega})$ and $\Psi^{\infty}_{\infty}(S_{\omega})$ on bisectors $S_{\omega}$ for $0<\omega<\frac{\pi}{2}$.

\subsubsection{Holomorphic functional calculus}\label{section on the holomorphic functional calculus}
In this section, we briefly describe the holomorphic functional calculus for (bi)sectorial operators. Good references include \cite{HaaseFunctionalCalculusBook} (for sectorial operators) and \cite[Chapter 3]{Egert_Thesis} (focusing on bisectorial operators).

 Let $\omega\in[0,\pi)$ and let $T$ be an injective sectorial operator of type $S^{+}_{\omega}$ on a Hilbert space $\mathcal{H}$. Then, $T$ has dense domain and dense range. For each $\mu\in(\omega,\pi]$, there is a map sending any $\varphi\in\mathcal{F}(S_{\mu}^{+})$ to a closed operator $\varphi(T)\in\mathcal{C}(\mathcal{H})$. This map is called the $\mathcal{F}(S_{\mu}^{+})$-functional calculus for $T$. See \cite[Theorem A.1]{MorrisTurner} for a summary of some important properties of this mapping. We mention that if $\varphi\in\textup{H}^{\infty}_{0}$, then $\varphi(T)\in\mathcal{L}(\mathcal{H})$ and has the usual Cauchy integral representation.
 If $T$ is a sectorial operator of type $S^{+}_{\omega}$ on $\mathcal{H}$ which is not necessarily injective, then we can consider its injective part $T|_{\clos{\ran{T}}}$, which is defined as the restriction of $T$ to the closure of its range. It is an injective sectorial operator of type $S^{+}_{\omega}$ on the Hilbert space $\clos{\ran{T}}$. If $\mu\in(\omega , \pi]$ and $\varphi\in\mathcal{F}(S^{+}_{\mu})$, we \emph{define} $\varphi(T):=\varphi(T|_{\clos{\ran{T}}})\in\mathcal{C}(\clos{\ran{T}})$.
 
 We say that $T$ has a bounded $\textup{H}^{\infty}$-calculus of angle $\mu\in(\omega,\pi)$ on $\clos{\ran{T}}$ if there is $C\geq 0$ such that $\norm{\varphi(T)}_{\mathcal{L}(\clos{\ran{T}})}\leq C\norm{\varphi}_{\infty}$ for all $\varphi\in\textup{H}^{\infty}(S^{+}_{\mu})$.
Let $\varphi\in\textup{H}^{\infty}_{0}(S^{+}_{\mu})$ for some $\mu\in(\omega,\pi)$. We say that $T$ satisfies quadratic estimates with auxiliary function $\varphi$ if there is $C\geq 1$ such that $C^{-1}\norm{f}_{\mathcal{H}}^{2}\leq \int_{0}^{\infty}\norm{\varphi(tT)f}_{\mathcal{H}}^{2}\frac{\dd t}{t}\leq C\norm{f}_{\mathcal{H}}^{2}$ for all $f\in\clos{\ran{T}}$.
The usual modifications can be made to treat bisectorial operators. We also note that a positive self-adjoint operator is sectorial of type $S_0^+$ and has a bounded $\textup{H}^{\infty}$-calculus that is compatible with its Borel functional calculus.

The equivalence between bounded $\textup{H}^{\infty}$-calculus and quadratic estimates is a fundamental theorem due to McIntosh \cite{McIntosh_HInfty_Calculus}; see also \cite[Theorem 7.3.1]{HaaseFunctionalCalculusBook}. 
\begin{thm}[McIntosh's theorem]\label{McIntosh Theorem}
Let $T$ be a sectorial operator of angle $\omega_{T}\in [0,\pi)$ on a Hilbert space $\mathcal{H}$. Then, the following statements are equivalent.
\begin{enumerate}[label=\emph{(\roman*)}]
    \item The operator $T$ satisfies quadratic estimates with some (equivalently any non-trivial) auxiliary function $\varphi\in\textup{H}^{\infty}_{0}(S^{+}_{\mu})$ and some $\mu\in(\omega_{T},\pi)$,
    \item The operator $T$ has a bounded $\textup{H}^{\infty}$-calculus of some (equivalently arbitrary) angle $\mu\in(\omega_{T},\pi)$ on $\clos{\ran{T}}$.
\end{enumerate}
This also holds for bisectorial operators $T$, upon replacing sectors with bisectors, and $\pi $ with $\frac{\pi}{2}$.
\end{thm}
We shall also need the following useful tool, see \cite[Theorem 5.2.6]{HaaseFunctionalCalculusBook} for a proof.
\begin{thm}[Calderón--McIntosh reproducing formula]\label{Calderon McIntosh reproducing formula}
    Let $T$ be a sectorial operator of angle $\omega_{T}\in [0,\pi)$ on a Hilbert space $\mathcal{H}$. Let $\varphi\in\textup{H}^{\infty}_{0}(S^{+}_{\mu})$ be such that $\int_{0}^{\infty}\varphi(t)\frac{\dd t}{t}=1$. Then, for all $f\in\clos{\ran{T}}$ it holds that  
    \begin{equation*}
        \lim_{\eps\to 0}\int_{\eps}^{1/\eps}\varphi(tT)f\frac{\dd t}{t}=f,
    \end{equation*}
    with convergence in the norm $\norm{\cdot}_{\mathcal{H}}$. 
    The result also holds for bisectorial operators upon replacing sectors with bisectors and provided $\varphi\in\textup{H}^{\infty}_{0}( S_{\mu})$ satisfies $\int_{0}^{\infty}\varphi(\pm t)\frac{\dd t}{t}=1$.
\end{thm}

\subsubsection{Adjoints and similarity}
Let $T$ be a sectorial operator of angle $\omega_{T}\in[0,\pi)$ on a Hilbert space $\mathcal{H}$. Since $T$ is densely defined, it has an adjoint operator $T^{*}$. As is well-known, $T^{*}$ is also a sectorial operator of angle $\omega_{T}$. Moreover, $T^{*}$ is injective if and only if $T$ is injective. 

For a meromorphic function $\varphi:S^{+}_{\mu}\to\C$ and some $\mu\in(0,\pi)$, the complex conjugate of $\varphi$ is the function $\varphi^{*}:S^{+}_{\mu}\to\C$, defined as $\varphi^{*}(z):=\clos{\varphi(\clos{z})}$ for all $z\in S^{+}_{\mu}$. The operation of complex conjugation preserves the classes of holomorphic functions introduced in Section~\ref{subsection on classes of holomorphic functions}. For all $\mu\in(\omega_{T},\pi)$ and $\varphi\in\textup{H}^{\infty}_{0}(S^{+}_{\mu})$, the adjoint of the bounded operator $\varphi(T)\in\mathcal{L}(\mathcal{H})$ is $\varphi(T)^{*}=\varphi^{*}(T^{*})\in\mathcal{L}(\mathcal{H})$. The same result also holds for bisectorial operators up to the usual modifications.

\subsubsection{Off-diagonal estimates}\label{section on off diagonal estimates}
We will be working with operators which cannot generally be represented as integral operators which have pointwise integral kernel decay estimates. As substitutes for such estimates we will use the weaker \emph{off-diagonal estimates} (or Davies--Gaffney estimates).  The following definition is taken verbatim from \cite[Definition 4.5]{AuscherEgert}. 
\begin{definition}
Let $p,r\in [1,\infty]$, let $\Omega\subseteq \C\setminus\{0\}$, and let $W_1$ and  $W_2$ be finite-dimensional Hilbert spaces. A family $\set{T(z)}_{z\in\Omega}$ of bounded linear operators mapping $\El{2}(\Rn;W_1)$ into $\El{2}(\Rn;W_2)$ is said to satisfy $\El{p}-\El{r}$ off-diagonal estimates of order $\gamma>0$ if 
\begin{equation*}
    \norm{\ind{F}T(z)\left(\ind{E}f\right)}_{\El{r}}\lesssim \abs{z}^{\frac{n}{r}-\frac{n}{p}}\left(1+\frac{\dist(E,F)}{\abs{z}}\right)^{-\gamma}\norm{\ind{E}f}_{\El{p}}
\end{equation*}
for all measurable subsets $E,F\subseteq \Rn$, all $z\in\Omega$ and all $f\in\El{2}\cap\El{p}$. If there is a $c>0$ such that the stronger estimate
\begin{equation*}
    \norm{\ind{F}T(z)\left(\ind{E}f\right)}_{\El{r}}\lesssim \abs{z}^{\frac{n}{r}-\frac{n}{p}}\exp{\left(-c\frac{\dist(E,F)}{\abs{z}}\right)}\norm{\ind{E}f}_{\El{p}}
\end{equation*}
holds, then the family is said to satisfy $\El{p}-\El{r}$ off-diagonal estimates of exponential order. When $p=r$, we simply speak of $\El{p}$ off-diagonal estimates.
\end{definition}

\subsubsection{Hardy spaces $\mathbb{H}_{T}^{s,p}$}\label{subsection on operator adapted hardy spaces for general operators with standard assumptions}
In this section, we recall the definition of Hardy spaces adapted to a given (bi)sectorial operator, following the construction outlined in \cite[Chapter 8]{AuscherEgert}. 
Let $W$ be a finite dimensional Hilbert space, and let $T$ be an operator on $\El{2}(\Rn;W)$ satisfying the following standard assumptions:
\begin{enumerate}[label=(\roman*)]
    \item $T$ is a bisectorial operator on $\El{2}(\Rn;W)$, of angle $\omega_{T}\in\left[0,\pi /2\right)$,
    \item $T$ has a bounded $\textup{H}^{\infty}$-calculus on $\clos{\ran{T}}$, 
    \item  The family $\set{\left(1+itT\right)^{-1}}_{t\in\R\setminus\set{0}}$ satisfies $\El{2}$ off-diagonal estimates of arbitrarily large order.
\end{enumerate}
For sectorial operators, we replace $\frac{\pi}{2}$ by $\pi$ in (i), and require instead that $\set{\left(1+t^{2}T\right)^{-1}}_{t> 0}$ satisfies $\El{2}$ off-diagonal estimates in (iii).
For an arbitrary (bi)sectorial operator $T$ satisfying these standard assumptions, the $T$-adapted spaces are constructed as follows. For a given $\varphi\in \textup{H}^{\infty}$ on any appropriate (bi)sector, we consider the associated extension operator
\begin{equation}\label{extension operator for definition of operator adapted spaces}
    \bb{Q}_{\varphi,T}:\clos{\ran{T}}\longrightarrow \El{\infty}((0,\infty); \Ell{2}),
\end{equation}
defined for all $f\in\clos{\ran{T}}$ and $t>0$ as $\left(\bb{Q}_{\varphi, T}f\right)(t):=\varphi\left(tT\right)f$ if $T$ is bisectorial, or $\left(\bb{Q}_{\varphi, T}f\right)(t):=\varphi\left(t^{2}T\right)f$ if $T$ is sectorial.
For $p\in(0,\infty)$ and $s\in\R$, the pre Hardy--Sobolev space $\bb{H}^{s,p}_{\varphi ,T}$ of smoothness $s$ and integrability $p$, adapted to $T$, is then defined to be the set
\begin{equation*}
    \bb{H}^{s,p}_{\varphi ,T}:=\set{f\in\clos{\ran{T}} \,:\, \bb{Q}_{\varphi, T}f\in\text{T}^{s,p}\cap \text{T}^{0,2}}
\end{equation*}
equipped with the quasinorm  $\norm{f}_{\bb{H}^{s,p}_{\varphi ,T}}:=\norm{\bb{Q}_{\varphi ,T}f}_{\text{T}^{s,p}}$.
Notice that if $\varphi\in \textup{H}^{\infty}_{0}$ is non-degenerate, then $\bb{Q}_{\varphi ,T}$ is defined on $\El{2}$, and for all $f\in\El{2}$ there is an equivalence $\norm{\bb{Q}_{\varphi ,T}f}_{\tent{0,2}}\eqsim \norm{f}_{\El{2}}$ by Fubini's theorem and that of McIntosh.
As a consequence, the condition $\bb{Q}_{\psi , H}f\in\tent{0,2}$ is redundant in the definition of $\bb{H}^{s,p}_{\varphi , T}$ in that case. Moreover, this also shows that $\bb{H}^{0,2}_{\varphi ,T}=\clos{\ran{T}}$ with equivalent norms. Finally, we mention that, up to equivalent norms, the space $\bb{H}^{s,p}_{\varphi ,T}$ does not depend on the particular choice of $\varphi\in \textup{H}^{\infty}$ as long as it non-degenerate and has enough decay at $0$ and at $\infty$. In fact, we need $\varphi\in\Psi^{\tau}_{\sigma}$ with the following conditions on the decay parameters:
\begin{table}[ht]
\caption{}
    \begin{center}
\begin{tabular}{ l|c|c }
      &$T$ is bisectorial & $T$ is sectorial\\
      \hline
     If  $p\leq 2$ & $\tau > -s +\abs{\frac{n}{2}-\frac{n}{p}}$ and $\sigma >s$ & $\tau > -s/2 +\abs{\frac{n}{4}-\frac{n}{2p}}$ and $\sigma >s/2$\\ 
     \hline
     If $p\geq 2$ & $\tau> -s$ and $\sigma >s +\abs{\frac{n}{2}-\frac{n}{p}}$ & $\tau> -s/2$ and $\sigma >s/2 +\abs{\frac{n}{4}-\frac{n}{2p}}$\\
     \hline
    \end{tabular}
    \end{center}
    \label{decay parameters conditions bisectorial Hardy adapted space}
    \end{table}
    
See \cite[Proposition 8.2]{AuscherEgert} and references therein.
We denote $\bb{H}^{s,p}_{T}:=\bb{H}^{s,p}_{\varphi,T}$ for any suitable choice of $\varphi\in \textup{H}^{\infty}$. When $s=0$, we simply write $\bb{H}^{p}_{T}:=\bb{H}^{0,p}_{T}$.
\subsubsection{Molecular decompositions for $\bb{H}^{p}_{T}$}\label{subsubsection on molecular characterisation of op adapted hardy spaces}
If $p\leq 1$, the adapted Hardy space $\bb{H}^{p}_{T}$ admits a molecular decomposition \cite[Theorem 8.17]{AuscherEgert}.  In fact, provided $M$ is large enough depending on $p$ and $n$, for any $\eps>0$, an arbitrary function $f\in\bb{H}^{p}_{T}$ admits a decomposition $f=\sum_{k\geq 1}\lambda_{k}m_{k}$ with unconditional convergence in $\El{2}$, where the $m_{k}$ are so-called $(\bb{H}^{p}_{T},\eps,M)$-molecules and $\norm{(\lambda_{k})_{k\geq 1}}_{\ell^{p}(\N)}\eqsim \norm{f}_{\bb{H}^{p}_{T}}$. The $(\bb{H}^{p}_{T},\eps,M)$-molecules are defined in \cite[Definition 8.14]{AuscherEgert}. For $M=1$, their definition is recalled in Lemma~\ref{abstract molecules are concrete dziubanski molecules} below.

\subsubsection{Extension and contraction operators associated to an injective sectorial operator}\label{section on extension and contraction operators for general inj sectorial op}
Let $T$ be an injective sectorial operator on $\El{2}(\Rn)$ of type $S_{\omega}^{+}$ for some $\omega\in[0,\pi)$. Let us assume that $T$ has a bounded $\textup{H}^{\infty}$-calculus on $\Ell{2}$. For $\mu\in(\omega , \pi)$ and any non-trivial auxiliary function $\psi\in\textup{H}^{\infty}_{0}(S^{+}_{\mu})$, we observe that (since $\clos{\ran{T}}=\Ell{2}$) the extension operator $\bb{Q}_{\psi,T}$ from \eqref{extension operator for definition of operator adapted spaces} has the following mapping property:
\begin{equation*}
    \bb{Q}_{\psi,T}: \Ell{2}\rightarrow \El{2}((0,\infty),\frac{\dd t}{t}; \El{2}).
\end{equation*} 
This operator is bounded by McIntosh's theorem since $T$ has a bounded $\textup{H}^{\infty}$-calculus on $\Ell{2}$. Indeed, for $f\in\Ell{2}$ it holds that  
    \begin{align*}
        \norm{\bb{Q}_{\psi,T}f}_{\El{2}(0,\infty,\frac{\dd t}{t}; \El{2})}^{2}=\int_{0}^{\infty}\norm{\psi(t^{2}T)f}_{\El{2}}^{2}\frac{\dd t}{t}\eqsim \norm{f}_{\El{2}}^{2}.
    \end{align*} This implies that $\bb{Q}_{\psi,T}$ has a bounded adjoint $\left(\bb{Q}_{\psi,T}\right)^{*}: \El{2}((0,\infty),\frac{\dd t}{t};\El{2})\to \El{2}$, referred to as a \emph{contraction operator}. Let $\bb{C}_{\psi,T}:=(\bb{Q}_{\psi^{*},T^{*}})^{*}$. Then, for all $F\in\El{2}((0,\infty),\frac{\dd t}{t};\El{2})$ and $g\in\El{2}$, the operator $\bb{C}_{\psi,T}$ acts as
    \begin{equation*}
        \langle \bb{C}_{\psi,T}F , g\rangle_{\El{2}}=\lim_{\eps\to 0} \langle\int_{\eps}^{\frac{1}{\eps}}\psi(t^{2}T)\left[F(t)\right]\frac{\dd t}{t}, g \rangle_{\El{2}},
    \end{equation*}
and the limit exists by McIntosh's theorem. It follows from the Calderón--McIntosh reproducing formula that the constant $c:=\left(\int_{0}^{\infty}\psi^{2}(t^{2})\frac{dt}{t}\right)^{-1}$ is such that $f=c\bb{C}_{\psi,T}(\bb{Q}_{\psi,T}f)$ for all $f\in\Ell{2}.$

\subsection{Schrödinger operators $H$ and first-order operators $BD$ and $DB$}\label{subsection on schrodinger operators H and DB BD}
Here we consider $n\geq 3$ and non-negative $V\in \El{1}_{\loc}(\R^n)$. 
The bound \eqref{boundedness of coefficients} and ellipticity (\ref{ellipticity assumption}) imply that the sesquilinear form $h:\mathcal{V}^{1,2}(\Rn)\times \mathcal{V}^{1,2}(\Rn) \to\C$ on $\Ell{2}$, defined for all $f,g\in\mathcal{V}^{1,2}\left(\Rn\right)$ by
\begin{equation}\label{form method def of H}
    h(f,g):=\left\langle A_{\parallel\parallel} \nabla_{x}f , \nabla_{x}g\right\rangle_{\El{2}(\Rn;\C^{n})} + \left\langle a V^{1/2}f , V^{1/2}g\right\rangle_{\Ell{2}},
\end{equation}
is densely defined, closed, accretive and continuous. Following \cite[Section~3]{MorrisTurner}, there is a unique maximal accretive operator $\mathsf{H}:\dom{\mathsf{H}}\subseteq \Ell{2} \to \Ell{2}$ such that $h(f,\varphi)=\left\langle \mathsf{H}f, \varphi\right\rangle_{\Ell{2}}$ for all $\varphi\in\mathcal{V}^{1,2}(\Rn)$ and $f\in\dom{\mathsf{H}}\subseteq\mathcal{V}^{1,2}(\Rn)$. We then define $H:=b\mathsf{H}$ with $\dom{H}:=\dom{\mathsf{H}}$, since $b:=A_{\perp\perp}^{-1}\in\Ell{\infty}$, so that formally $H=-b\textup{div}_x(A\nabla_x)+baV$.

 We now consider the first-order operators of the type introduced by Auscher--McIntosh--Nahmod \cite{Auscher_McIntosh_Nahmod_97} from which the second-order operator $H$ can be constructed as in \cite[Section~3]{MorrisTurner}. Let $\nabla_{\mu}:\dom{\nabla_{\mu}}\subset\Ell{2}\longrightarrow \El{2}(\Rn ; \C^{n+1})$ denote the unbounded operator with domain $\dom{\nabla_{\mu}}:=\mathcal{V}^{1,2}(\Rn)$, and let $\nabla_{\mu}^{*}$ denote its adjoint. The self-adjoint operator $D:\dom{D}\subset \El{2}(\Rn;\C^{n+2})\longrightarrow \El{2}(\Rn;\C^{n+2})$ is defined as
\begin{equation*}
    Du := -\begin{bmatrix}
        0 & \nabla_{\mu}^{*}\\
        \nabla_{\mu} & 0 
    \end{bmatrix} 
    \begin{bmatrix}
        u_{\perp}\\
        \left(u_{\parallel}, u_{\mu}\right)
    \end{bmatrix}
    =-\begin{bmatrix}
    -\nabla_{\mu}^{*}\left(u_{\parallel}, u_{\mu}\right)\\
    -\nabla_{\mu}u_{\perp}
    \end{bmatrix}
\end{equation*}
for all $u=\left(u_{\perp} , u_{\parallel} , u_{\mu}\right) \in \El{2}(\Rn ; \C^{n+2})=\Ell{2}\oplus \El{2}(\Rn;\C^{n})\oplus \Ell{2}$ in the domain
\begin{equation*}
    \dom{D}:=\set{u\in \El{2}(\Rn ; \C^{n+2}) : u_{\perp}\in\mathcal{V}^{1,2}(\Rn) \text{ and } \left(u_{\parallel}, u_{\mu}\right)\in\dom{\nabla_{\mu}^{*}}}.
\end{equation*}
It follows from \cite[Lemma 3.1]{MorrisTurner} that the closure of the range of $D$ in $L^2(\Rn,\C^{n+2})$ is
\begin{align}\label{characterisation range D}
\begin{split}
    \clos{\ran{D}}&=\set{u\in\El{2}(\Rn; \C^{n+2}) : u_{\perp}\in\Ell{2} \text{ and} \left(u_{\parallel}, u_{\mu}\right) =\nabla_{\mu}g \text{ for  }g\in\dot{\mathcal{V}}^{1,2}(\Rn)\cap\Ell{2^{*}}}.
    \end{split}
\end{align}
The bounded multiplication operator $B:\El{2}(\Rn;\C^{n+2})\rightarrow\El{2}(\Rn;\C^{n+2})$  is defined pointwise by
\begin{equation}\label{definition of multiplication operator B associated to elliptic coefficients A}
    (Bu)(x):=B(x)u(x)
    \quad\text{ and }\quad
    B(x):=\begin{bmatrix}
        b(x) & 0 & 0\\
        0 & A_{\parallel\parallel}(x) & 0\\
        0 & 0 & a(x)
    \end{bmatrix}
\end{equation}
for all $u\in\El{2}(\Rn;\C^{n+2})$ and almost every $x\in\Rn$, and it is strictly accretive on $\ran{D}$ of angle
\[
\omega(B):=\sup_{v\in\ran{D}\setminus \{0\}}\abs{\arg\left\langle Bv ,v \right\rangle}
\in[0,\pi/2).
\]
The unbounded operators $BD$ and $DB$ are then understood on the domains $\dom{BD}=\dom{D}$ and $\dom{DB}=\set{u\in\El{2}(\Rn;\C^{n+2}) : Bu\in\dom{D}}$.
The operators $BD$ and $DB$ are densely defined bisectorial operators of type $S_{\omega(B)}$ on $\El{2}(\Rn;\C^{n+2})$ (see \cite[Section~3]{MorrisTurner} and the references therein).
Moreover, we have the topological direct sum decompositions
    \begin{equation*}\label{topological direct sum decompositions}
        \El{2}(\Rn;\C^{n+2})=\clos{\ran{DB}}\oplus\kernel{DB}=\clos{\ran{BD}}\oplus\kernel{BD}
    \end{equation*}
    and the following identities: $\ran{DB}=\ran{D}$, $\ran{BD}=B\ran{D}$, $B\kernel{DB}=\kernel{D}$ and $\kernel{BD}=\kernel{D}$.
This and the accretivity of $B$ imply that $B\clos{\ran{DB}}=B\clos{\ran{D}}=\clos{\ran{BD}}$, and that the restriction $B|_{\clos{\ran{DB}}} : \clos{\ran{DB}}\longrightarrow\clos{\ran{BD}}$ is an isomorphism.

The operators $\left(BD\right)^{2}$ and $\left(DB\right)^{2}$ are sectorial of type $S_{2\omega(B)}^{+}$ on $\El{2}(\Rn;\C^{n+2})$.
The block structure of the operators $B$ and $D$ shows that we can define operators $\widetilde{H}$, $M$ and $\widetilde{M}$
such that
\begin{equation}\label{identities square BD; H, M operators}
    (BD)^{2}=\begin{bmatrix}
        H & 0 \\
        0 & M
    \end{bmatrix} \quad\text{ and }\quad (DB)^{2}=\begin{bmatrix}
        \widetilde{H} & 0\\
        0 & \widetilde{M}
    \end{bmatrix}.
\end{equation}
The operators $H$ and $\widetilde{H}$ are injective sectorial operators of type $S_{2\omega(B)}^{+}$ on $\Ell{2}$, and $M$ and $\widetilde{M}$ are sectorial operators of type $S_{2\omega(B)}^{+}$ on $\El{2}(\Rn;\C^{n+1})$. Moreover, if $T\in\set{DB, BD}$, then bisectoriality of $T$ implies that $\kernel{T^{2}}=\kernel{T}$ and therefore that $\clos{\ran{T^{2}}}=\clos{\ran{T}}$ (see the bisectorial counterpart of \cite[Proposition 2.1.1 e)]{HaaseFunctionalCalculusBook}). As a consequence, (\ref{characterisation range D}) shows that 
\begin{equation}\label{identity closure of ranges of schrodinger operators widetilde}
    \clos{\ran{\widetilde{H}}}\times\clos{\ran{\widetilde{M}}}=\clos{\ran{(DB)^{2}}}=\clos{\ran{DB}}=\ran{D}=\Ell{2}\times \clos{\ran{\nabla_{\mu}}}
\end{equation}
and
\begin{equation}\label{identity closure of ranges of schrodinger operators}  \clos{\ran{H}}\times\clos{\ran{M}}=\clos{\ran{(BD)^{2}}}=\clos{\ran{BD}}=B\clos{\ran{D}}=\Ell{2}\times A\clos{\ran{\nabla_{\mu}}}.
\end{equation}

\subsubsection{Block functional calculus}
Let $\mu\in(2\omega(B),\pi)$, and $\psi\in \textup{H}^{\infty}(S_{\mu}^{+})$. Then, the function $\varphi(z):=\psi(z^{2})$ satisfies $\varphi\in \textup{H}^{\infty}(S_{\mu /2})$, and \cite[Theorem 3.2.20]{Egert_Thesis} shows that for $T\in\set{DB,BD}$ it holds that  $\varphi(T)=\psi(T^{2})$,
where the left-hand side is defined by the $\mathcal{F}(S_{\mu/2})$-functional calculus for the bisectorial operator $T$, and the right-hand side is constructed through the $\mathcal{F}(S^{+}_{\mu})$-functional calculus for the sectorial operator $T^{2}$. The block decomposition 
\eqref{identities square BD; H, M operators} is preserved by the functional calculus, with
\begin{equation}\label{decomposition property of BD into H and M}
    \varphi(BD)=\psi((BD)^{2})=\begin{bmatrix}
        \psi(H) & 0\\
        0 & \psi(M)
    \end{bmatrix},
\end{equation}
and a similar identity for $DB$. A central result obtained by Morris--Turner \cite{MorrisTurner} is the following. 
\begin{thm}[{\cite[Theorem 3.2]{MorrisTurner}}]\label{bounded H_infty functional calculus of H}
     Let $n\geq 3$, $q\geq \max\{\frac{n}{2},2\}$ and $V\in\textup{RH}^{q}(\Rn)$. If ${T\in \set{DB,BD}}$, then the bisectorial operator $T$ has a bounded $\textup{H}^{\infty}$-calculus of angle $\omega(B)$ on $\clos{\ran{T}}$. 
\end{thm}
It follows from Theorem~\ref{bounded H_infty functional calculus of H} and McIntosh's theorem (Theorem~\ref{McIntosh Theorem}) that the sectorial operators $H$ and $\widetilde{H}$ have a bounded $\textup{H}^{\infty}$-calculus of angle $2\omega(B)$ on $\Ell{2}$, whereas $M$ and $\widetilde{M}$ have a bounded $\textup{H}^{\infty}$-calculus of angle $2\omega(B)$ on $\clos{\ran{M}}$.

\subsubsection{Intertwining relations}\label{section Intertwining relations}
Let $\mathcal{H}:=\clos{\ran{D}}=\clos{\ran{DB}}$. We know that $B|_{\mathcal{H}}:\clos{\ran{DB}}\longrightarrow \clos{\ran{BD}}$ is an isomorphism. Consequently, we have $BD|_{\clos{\ran{BD}}}=B|_{\mathcal{H}}(DB)|_{\clos{\ran{DB}}}\left(B|_{\mathcal{H}}\right)^{-1}$ on $\clos{\ran{BD}}$. By similarity, see \cite[Proposition 3.2.10]{Egert_Thesis}), we deduce that
\begin{equation}\label{similarity DB and BD}
    \varphi(BD)=B\varphi(DB)\left(B|_{\mathcal{H}}\right)^{-1},
\end{equation}
as closed (bounded) operators on $\clos{\ran{BD}}$, whenever $\varphi\in \textup{H}^{\infty}(S_{\mu})$, $\omega(B)<\mu<\pi /2$. 
The following so-called \emph{intertwining relations} are obtained by following the proof of \cite[Lemma 3.6]{AuscherEgert}.
\begin{lem}[Intertwining relations]\label{intertwining relations}
    Let $\mu\in(2\omega(B),\pi)$, and $\psi\in \textup{H}^{\infty}(S_{\mu}^{+})$. Let $\varphi\in \textup{H}^{\infty}(S_{\mu /2})$ be defined as $\varphi(z):=\psi(z^{2})$ for all $z\in S_{\mu /2}$. Then,
    \begin{equation*}
        D\varphi(BD)u=\varphi(DB)Du
    \end{equation*}
    for all $u\in\dom{D}\cap \clos{\ran{BD}}$. Moreover, if $f_{\perp}\in\mathcal{V}^{1,2}(\Rn)$ and $(f_{\parallel},f_{\mu})\in\dom{\nabla_{\mu}^{*}}\cap\clos{\ran{M}}$, then
    \begin{align*}
        \nabla_{\mu}^{*}\psi(M)\begin{bmatrix}
            f_{\parallel}\\
            f_{\mu}
        \end{bmatrix} &= \psi(\widetilde{H})\nabla_{\mu}^{*}\begin{bmatrix}
            f_{\parallel}\\
            f_{\mu}
        \end{bmatrix},\\
        \nabla_{\mu}\psi(H)f_{\perp} &= \psi(\widetilde{M})\nabla_{\mu}f_{\perp}.
    \end{align*}
\end{lem}

\subsubsection{Adjoints and similarity}\label{subsubsection on adjoints and similarity}
It is useful to introduce the operator $H^{\sharp}:\dom{H^{\sharp}}\subset\Ell{2}\to\Ell{2}$, formally identified as
\begin{equation}\label{definition H sharp operator}
    H^{\sharp}:=-b^{*} \dvg_{x}\left({A_{\parallel\parallel}^{*}\nabla_{x}}\right) +b^{*}a^{*}V.
\end{equation}
Its coefficients have exactly the same ellipticity and boundedness properties as those of $H$. Note that $\dom{H^{*}}=\set{g\in\Ell{2}:b^{*}g\in\dom{H^{\sharp}}}$, and $H^{\sharp}\left(b^{*}g\right)=b^{*}H^{*}g$ for all $g\in\dom{H^{*}}$. In other words, the adjoint operator $H^{*}$ is similar to $H^{\sharp}$ under conjugation with $b^{*}$. The functional calculus preserves this similarity. In fact, if $\varphi\in\mathcal{F}(S_{\mu}^{+})$ for some $\mu\in(2\omega_{B},\pi)$, then $\varphi(H^{*})f=(b^{*})^{-1}\varphi(H^{\sharp})(b^{*}f)$
for all $f\in\Ell{2}$ (see \cite[Proposition 3.2.10]{Egert_Thesis}). In the same way, the operator $\widetilde{H}$ is similar to $H$ under conjugation with $b$.

\subsubsection{Mapping properties of adapted Hardy spaces}\label{subsubsection on mapping properties ofa dapted hardy spaces}
Let $n\geq 3$, $q\geq \max\{\frac{n}{2},2\}$ and $V\in\textup{RH}^{q}(\Rn)$. Let $T\in\set{BD,DB}$. We obtain the uniform bound
$\norm{\left(1+itT\right)^{-1}f}_{\El{2}}\lesssim \norm{f}_{\El{2}}$ for all $t\in\R$ and $f\in\El{2}(\Rn;\C^{n+2})$ because $T$ is  bisectorial. It follows from \cite[Proposition 3.7]{MorrisTurner} and the references therein that the family $\set{\left(1+itT\right)^{-1}}_{t\in\R\setminus\set{0}}$ satisfies $\El{2}$ off-diagonal estimates of exponential order.
Following the argument of the proof of \cite[Corollary 3.13]{AuscherEgert}, this implies that for any $S\in\set{H,\widetilde{H},M,\widetilde{M}}$, the families $\set{(1+t^{2}S)^{-1}}_{t>0}$ and $\set{t\nabla_{\mu}\left(1+t^{2}H\right)^{-1}}_{t>0}$ both satisfy $\El{2}$ off-diagonal estimates of exponential order. With Theorem~\ref{bounded H_infty functional calculus of H}, this shows that $T$ and $S$ satisfy the standard assumptions from Section~\ref{subsection on operator adapted hardy spaces for general operators with standard assumptions} (on $\El{2}(\Rn;\C^{n+2})$ and $\Ell{2}$, respectively). This means that the adapted Hardy spaces $\bb{H}^{s,p}_{T}$ and $\bb{H}^{s,p}_{S}$ can be constructed as in Section~\ref{subsection on operator adapted hardy spaces for general operators with standard assumptions} above.

This rest of this section is concerned with establishing the following important diagram connecting these different adapted Hardy spaces; compare with \cite[Figure 9.1]{AuscherEgert}.
\begin{figure}[H]
    \begin{tikzcd}
    \bb{H}^{p}_{BD}\cap\ran{[BD]} &[-2em] = &[-2em] \bb{H}^{p}_{H}\cap\ran{H^{1/2}}&[-1em]\oplus&[-1em] \bb{H}^{p}_{M}\cap\ran{M^{1/2}}\\    
    \bb{H}^{1,p}_{BD}\cap\dom{D} \arrow{u}{[BD]} \arrow{d}{D}&[-2em]=&\bb{H}^{1,p}_{H}\cap\dom{H^{1/2}}\arrow{u}{H^{1/2}}\arrow[drr,pos=0.1,"-\nabla_{\mu}"]&\oplus&\bb{H}^{1,p}_{M}\cap\dom{M^{1/2}}\arrow{u}{M^{1/2}}\arrow[dll,pos=0.9, "-\nabla_{\mu}^{*}"]\\ 
    \bb{H}^{p}_{DB}\cap\ran{D} \arrow{d}{B}&[-2em]=& \bb{H}^{p}_{\widetilde{H}}\cap\ran{\nabla_{\mu}^{*}}\arrow{d}{b} &\oplus& \bb{H}^{p}_{\widetilde{M}}\cap\ran{\nabla_{\mu}}\arrow{d}{A}\\
    \bb{H}^{p}_{BD}\cap\ran{BD} &[-2em]=& \bb{H}^{p}_{H}\cap\ran{b\nabla_{\mu}^{*}} &\oplus& \bb{H}^{p}_{M}\cap\ran{A\nabla_{\mu}}
  \end{tikzcd}
  \caption{Relations between various adapted Hardy spaces}
  \label{figure relations between adapted Hardy spaces}
\end{figure}
All spaces appearing on the diagram are intersections of some operator-adapted Hardy space with one of its dense subsets, for both the $\El{2}$ norm and the corresponding Hardy space norm. Moreover, each arrow is a bijection which is bounded from above and below for the correponding $p$-(quasi)norms.

We shall follow the construction outlined in \cite[Section~8.6]{AuscherEgert}, starting with the first column.
Let $T\in\set{BD,DB}$. We know that $T$ is a bisectorial operator of angle $\omega(B)\in\left[0,\pi/2\right)$ on $\El{2}(\Rn;\C^{n+2})$, with a bounded $\textup{H}^{\infty}$-calculus on $\clos{\ran{T}}$. It follows from equation~(4.8) in \cite{MorrisTurner} that the operators $T$ and $[T]$ have the same domain and the same range. Let $p>0$ and $s\in\R$. By following the argument of \cite[Corollary 8.11]{AuscherEgert}, we also get that they are bijections mapping $\bb{H}^{s+1,p}_{T}\cap \dom{T}$ to $\bb{H}^{s,p}_{T}\cap\ran{T}$ and satisfying 
    \begin{equation}\label{elementary regularity shift}
        \norm{Tf}_{\bb{H}^{s,p}_{T}}\eqsim \norm{f}_{\bb{H}^{s+1,p}_{T}}\eqsim \norm{\left[T\right]f}_{\bb{H}^{s,p}_{T}}
    \end{equation}
    for all $f\in \bb{H}^{s+1,p}_{T}\cap \dom{T}$. Moreover, we may follow the proof of \cite[Lemma 8.34]{AuscherEgert} and use the intertwining relations from Section~\ref{section Intertwining relations} to obtain that the map ${B:\bb{H}_{DB}^{s,p}\cap \ran{D}\to \bb{H}^{s,p}_{BD}\cap \ran{BD}}$
    obtained by restriction of the operator $B$, is bijective and $\norm{Bf}_{\bb{H}^{s,p}_{BD}}\eqsim \norm{f}_{\bb{H}^{s,p}_{DB}}$ for all $f\in\bb{H}_{DB}^{s,p}\cap\ran{D}$.
 We can also adapt the so-called regularity shift from \cite[Proposition 5.6]{Amenta_Auscher_2018elliptic} to deduce that the map $D:\bb{H}^{s+1,p}_{BD}\cap\dom{D}\to \bb{H}^{s,p}_{DB}\cap\ran{D}$ is bijective and satisfies $\norm{Df}_{\bb{H}^{s,p}_{DB}}\eqsim \norm{f}_{\bb{H}^{s+1,p}_{BD}}$ for all $f\in\bb{H}^{s+1,p}_{BD}\cap\dom{D}$. 
     This is because the proofs of \cite[Lemma 8.34]{AuscherEgert} and \cite[Proposition 5.6]{Amenta_Auscher_2018elliptic} extend  to the operators $BD$ and $DB$ considered here once the  following local coercivity inequality is established.
\begin{lem}\label{local coercivity inequality BD}
    Let $u\in\dom{D}$. Then, for all $t>0$ and all $x\in\Rn$ it holds that 
    \begin{equation*}
        \int_{B(x,t)}\abs{Du}^{2}\lesssim \int_{B(x,2t)}\abs{BDu}^{2} + t^{-2}\int_{B(x,2t)}\abs{u}^{2}.
    \end{equation*}
\end{lem}
\begin{proof}
     We follow the argument of the proof of \cite[Lemma 5.14]{AuscherStahlhut_FunctionalCalculus_Dirac}. Let $u\in\dom{D}$, and let $B\subset\Rn$ be a ball of radius $r>0$. Let $\eta\in\smcp$ such that $0\leq \eta\leq 1$, $\eta=1$ on $B$ and $\supp{\eta}\subset 2B$ with $\norm{\nabla_{x} \eta}_{\Ell{\infty}}\lesssim r^{-1}$. The proof of \cite[Proposition 3.7]{MorrisTurner} shows that $\eta u \in\dom{D}$, with 
        $D(\eta u)=\eta Du + \begin{bmatrix}
            \left(\nabla_{x}\eta\right)\cdot u_{\parallel}\\
            -\left(\nabla_{x}\eta\right)u_{\perp}\\
            0
        \end{bmatrix}.$
    This allows to conclude as in the proof of \cite[Lemma 5.14]{AuscherStahlhut_FunctionalCalculus_Dirac}.
\end{proof}

Since $\dom{D}=\dom{BD}$, $\ran{D}=\ran{DB}$ and $\ran{[BD]}=\ran{BD}$, \cite[Lemma 8.7]{AuscherEgert} shows that each space appearing in the first column of Figure \ref{figure relations between adapted Hardy spaces} is the intersection of an operator-adapted Hardy space $\bb{H}^{s,p}_{T}$ with one if its dense subsets. Density holds for both the $\bb{H}^{s,p}_{T}$ norm and the $\El{2}$ norm.
Following the argument of \cite[Section~8.6]{AuscherEgert}, we also have the decomposition $\bb{H}^{s,p}_{BD}=\bb{H}^{s,p}_{H}\oplus\bb{H}^{s,p}_{M}$ with the estimate
    $\norm{u}_{\bb{H}^{s,p}_{BD}}\eqsim \norm{ u_{\perp}}_{\bb{H}^{s,p}_{H}} + \norm{(u_{\parallel},u_{\mu})}_{\bb{H}^{s,p}_{M}}$. 
Similarly, we have the decomposition $\bb{H}^{s,p}_{DB}=\bb{H}^{s,p}_{\widetilde{H}}\oplus\bb{H}^{s,p}_{\widetilde{M}}$. Splitting each space in the first column according to the corresponding decomposition, we obtain the complete Figure \ref{figure relations between adapted Hardy spaces}.

\section{Hardy spaces adapted to Schrödinger operators}\label{section on Hardy spaces adapted to schrodinger operators}
The Fefferman--Stein Hardy spaces $\textup{H}^{p}(\Rn)$ (see \cite[Chapter \RNum{3}]{Stein_HarmonicAnalysis_93}) are well-known to be adapted to the Laplace operator $-\Delta$. In this section, we consider Hardy spaces adapted to the Schrödinger operator $H_{0}:=-\Delta +V$ and introduced by Dziuba\`{n}ski and Zienkiewicz in~\cite{Dziubanski_Hp_Spaces}. We begin by considering $n\geq 1$ and non-negative $V\in\El{1}_{\loc}(\Rn)$. The operator $H_{0}:=-\Delta +V$ is associated to the sesquilinear form $h_{0}$ defined by setting $A_{\parallel\parallel}\equiv 1$ and $a\equiv 1$ in the form $h$ defined by \eqref{form method def of H}.
    The form $h_{0}$ is symmetric, densely defined, accretive and closed. The properties of the self-adjoint operator $H_{0}$  below are well-known, but a proof is included for completeness. Note that $-\Delta$, defined  by setting $V\equiv 0$, generates the classical heat semigroup $\set{e^{t\Delta}}_{t\geq 0}$ given by
  $e^{t\Delta}f=p_{t}\ast f$, where $p_{t}(x):=(4\pi t)^{-n/2}\exp{\left(-\frac{\abs{x}^{2}}{4t}\right)}$, for all $f\in\Ell{2}$, $x\in\R^n$ and $t>0$.

\begin{lem}\label{properties of schrodinger semigroup on L2} If $n\geq 1$ and $V\in\El{1}_{\loc}(\Rn)$ is non-negative, then $-H_{0}$ generates a holomorphic semigroup $\set{e^{-zH_{0}}}_{\Re{z}>0}$ on $\Ell{2}$, and the restriction $\set{e^{-tH_{0}}}_{t>0}$ is real, positive and dominated by the classical heat-semigroup, in the sense that for all $t\geq 0$ and all $f\in\Ell{2}$ it holds that
\begin{equation}\label{domination of Schrodinger semigroup by Heat semigroup}
    \abs{e^{-tH_{0}}f}\leq e^{t\Delta}\abs{f}
\end{equation}
pointwise almost everywhere on $\Rn$.
\end{lem}
\begin{proof}
    Since $H_{0}$ is self-adjoint and positive, it is sectorial of type $S_{0}^{+}$ and thus $-H_{0}$ generates a bounded holomorphic semigroup by \cite[Theorem 1.54]{OuhabazHeatEquonDomains}.
    The semigroup $\set{e^{-tH_{0}}}_{t>0}$ is real and  positive by Proposition~2.5 and Theorem~2.7 in~ \cite{OuhabazHeatEquonDomains}. It is elementary to check, using well-known properties of the Sobolev space $\textup{W}^{1,2}(\Rn)$, that for any real-valued $u\in\mathcal{V}^{1,2}(\Rn)$, it holds that   $\abs{u}\in\mathcal{V}^{1,2}(\Rn)$, with $h_{0}(\abs{u},\abs{u})=h_{0}(u,u)$. 
    The domination property (\ref{domination of Schrodinger semigroup by Heat semigroup}) then follows from Proposition~2.23 and Theorem~2.24 in~\cite{OuhabazHeatEquonDomains}. 
\end{proof}
If $p\in[1,\infty]$, then it follows from (\ref{domination of Schrodinger semigroup by Heat semigroup}) and Young's convolution inequality that  
\begin{align}\label{uniform boundedness heat semigroup schrodinger on Lp for all p}
    \norm{e^{-tH_{0}}f}_{\El{p}}&\leq \norm{\abs{f}\ast p_{t}}_{\El{p}}\leq \norm{f}_{\El{p}}\norm{p_t}_{\El{1}}\lesssim \norm{f}_{\El{p}}
\end{align}
for all $f\in\El{2}\cap\El{p}$ and $t>0$. In addition, for all $f\in\El{1}\cap\El{2}$ and $t>0$ we have 
\begin{align*}
    \norm{e^{-tH_{0}}f}_{\El{\infty}}\leq \norm{\abs{f}\ast p_{t}}_{\El{\infty}}\leq \norm{p_{t}}_{\El{\infty}}\norm{f}_{\El{1}}.
\end{align*}

It follows that the operators $e^{-tH_0}$ are integral operators (see, e.g., \cite[Theorem 1.3]{Arendt_Bukhvalov_1994}). Namely, for each $t>0$ there is a bounded measurable function ${k_{t}:\Rn\times \Rn\to \R}$ such that 
\begin{equation}\label{integral operator representation of semigroup}
    (e^{-tH_0}f)(x)=\int_{\Rn}k_{t}(x,y)f(y)\dd y
\end{equation}
for all $f\in\Ell{2}$ and almost every $x\in\Rn$. The self-adjointness of the semigroup $\set{e^{-tH_{0}}}_{t>0}$ implies that $k_{t}(x,y)=\clos{k_{t}(y,x)}$ for almost every $(x,y)\in\Rn\times \Rn.$
Moreover, testing the domination property (\ref{domination of Schrodinger semigroup by Heat semigroup}) and the positivity of the semigroup with (\ref{integral operator representation of semigroup}) and $f:={\meas{B(\eps,z)}}^{-1}\ind{B(\eps,z)}$ for arbitrary $z\in\Rn$ and $\eps\to 0^{+}$, we obtain by Lebesgue's differentiation theorem that
\begin{align}\label{classical gaussian estimates for schrodinger semigroup}
    0\leq k_{t}(x,y)\leq p_{t}(x-y) = (4\pi t)^{-n/2}\exp{\left(-\frac{\abs{x-y}^{2}}{4t}\right)}
\end{align}
for almost every $(x,y)\in\Rn\times \Rn.$ These Gaussian estimates are strengthened below.
\begin{lem}\label{improved gaussian bound on heat kernel of schrodinger operator}
If $n\geq 3$, $q>\frac{n}{2}$, $V\in\textup{RH}^{q}(\Rn)$, $\delta=2-\frac{n}{q}$ and $\sigma\in (0,\delta\wedge 1)$, then:
\begin{enumerate}[label=\emph{(\roman*)}]
    \item There exists $c>0$, and for each $N>0$ there exists $C_{N}\geq 0$, such that 
    \begin{equation*}
        0\leq k_{t}(x,y)\leq C_{N}t^{-\frac{n}{2}}e^{-c\frac{\abs{x-y}^{2}}{t}}\left(1+\frac{t^{1/2}}{\rho(x)}+\frac{t^{1/2}}{\rho(y)}\right)^{-N}
    \end{equation*}
    for all $x,y\in\Rn$ and all $t>0$. 
    \item There exists $c>0$, and for each $N>0$ there exists $C_{N}\geq 0$, such that 
    \begin{equation*}
        \abs{k_{t}(x,y)-k_{t}(x,z)}\leq C_{N}\left(\frac{\abs{y-z}}{t^{1/2}}\right)^{\sigma}t^{-n/2}e^{-c\frac{\abs{x-y}^{2}}{t}}\left(1+\frac{t^{1/2}}{\rho(x)}+\frac{t^{1/2}}{\rho(y)}\right)^{-N}
    \end{equation*}
    for all $x,y,z\in\Rn$ such that $\abs{y-z}<t^{1/2}$.
    \item There exists $c>0$ and $C\geq 0$ such that
    \begin{equation*}
        \abs{k_{t}(x,y)-k_{t}(x,z)}\leq C\left(\frac{\abs{y-z}}{t^{1/2}}\right)^{\sigma}t^{-n/2}e^{-c\frac{\abs{x-y}^{2}}{t}}
    \end{equation*}
    for all $t>0$ and all $x,y,z\in\Rn$ such that $\abs{x-z}\geq \frac{1}{2}\abs{x-y}$.
    \end{enumerate}
\end{lem}
\begin{proof}
    The estimate (i) can be found in \cite[Theorem 1]{Kurata1999}.
The Hölder continuity (ii) can be found in \cite[Proposition 4.11]{Dziubanski2003}. The third item is a combination of the previous two.
\end{proof}
\subsection{Dziuba\`{n}ski--Zienkiewicz Hardy spaces $\textup{H}^{p}_{V}$}\label{section on DZ Hardy spaces}

We now return to consider $n\geq 1$ and non-negative $V\in\El{1}_{\loc}(\Rn)$. For $f\in\Ell{2}$, we define the maximal function
\begin{align*}
    \mathcal{M}_{V}f :=\sup_{t>0}\abs{e^{-tH_{0}}f}
\end{align*}
as a lattice supremum to ensure measurability of $\mathcal{M}_{V}f$ (see, e.g., \cite[Lemma 2.6.1]{meyer-Nieberg-1991banach_lattices} for details). For $0<p\leq\infty$, this allows us to define the (quasi)normed space 
\begin{equation*}
    \textup{H}^{p}_{V, \text{pre}}\left(\Rn\right):=\set{f\in\Ell{2} : \mathcal{M}_{V}f\in\Ell{p}}
\end{equation*}
with the (quasi)norm $\norm{f}_{\text{H}^{p}_{V}}:=\norm{\mathcal{M}_{V}f}_{\El{p}}$ for all $f\in\textup{H}^{p}_{V, \text{pre}}\left(\Rn\right)$. To see that this is indeed a (quasi)norm, observe that if $f\in\Ell{2}$ and $\norm{f}_{\Dz{p}}=0$, then $\abs{\left(e^{-tH_{0}}f\right)(x)}\leq\left(\mathcal{M}_{V}f\right)(x)=0$ for almost every $x\in\Rn$ and all $t>0$, so $\norm{e^{-tH_{0}}f}_{\El{2}}=0$ for all $t>0$, and the limit as $t\to 0$ shows that $f=0$ in $\Ell{2}$. The Dziuba\`{n}ski--Zienkiewicz Hardy space $\textup{H}^{p}_{V}\left(\Rn\right)$ can be defined as a completion of $\textup{H}^{p}_{V, \text{pre}}(\R^n)$. We will construct a completion of $\textup{H}^{1}_{V}(\Rn)$ in Section~\ref{subsubsection with definition of completion H^{1}_{V}} after establishing the following preliminaries.
\subsubsection{Properties for $p\geq 1$}

At the level of the pre-completed spaces $\textup{H}^{p}_{V, \text{pre}}(\R^n)$, we have the following result.
\begin{lem}\label{identification of Dziubanski hardy spaces for p>=1}
    Let $n\geq 1$ and $V\in\El{1}_{\loc}(\Rn)$ be nonnegative. Then $\textup{H}^{p}_{V, \textup{pre}}(\Rn)=\Ell{p}\cap\Ell{2}$ with equivalent $p$-norms for all $p\in(1,\infty]$. Moreover, if $f\in\textup{H}^{1}_{V, \textup{pre}}(\Rn)$, then $f\in\Ell{1}$ with $\norm{f}_{\El{1}}\leq \norm{f}_{\textup{H}^{1}_{V}}$. 
\end{lem}
\begin{proof}
    Let $1<p\leq \infty$ and $f\in\Ell{2}\cap\Ell{p}$. The domination property (\ref{domination of Schrodinger semigroup by Heat semigroup}) implies that 
    \begin{equation}\label{domination of schrodinger maximal function by hardy littlewood max function}
        \left(\mathcal{M}_{V}f\right)(x)=\sup_{t>0}\abs{\left(e^{-tH_0}f\right)(x)}\leq \sup_{t>0}\left(\abs{f}\ast p_{t}\right)(x)\lesssim \hlmax{\abs{f}}(x)
    \end{equation}
    for almost every $x\in\Rn$, where $\mathcal{M}$ is the Hardy--Littlewood maximal operator (see \cite[Chapter \RNum{3}, Theorem 2 (a)]{SteinSingularIntegrals} for the last estimate). Since $p>1$, the maximal theorem implies that $f\in\Dz{p}$, with $\norm{f}_{\Dz{p}}=\norm{\mathcal{M}_{V}f}_{\El{p}}\lesssim \norm{f}_{\El{p}}.$
    Conversely, if $f\in\textup{H}^{p}_{V, \text{pre}}(\Rn)$ then for all $t>0$ we have $e^{-tH_{0}}f\in\Ell{p}$, with $\norm{e^{-tH_{0}}f}_{\El{p}}\leq \norm{\mathcal{M}_{V}f}_{\El{p}}=\norm{f}_{\Dz{p}}$. Moreover, since $f\in\Ell{2}$, it holds that  $e^{-tH_{0}}f\to f$ in $\Ell{2}$ as $t\to 0^{+}$, by strong continuity of the semigroup. In particular, for every $\varphi\in\El{2}\cap\El{p'}$ we have $\langle f,\varphi\rangle_{\El{2}}=\lim_{t\to 0^{+}}\langle e^{-tH_{0}}f , \varphi \rangle_{\El{2}}$.
    Since in addition we have
    \begin{align*}
        \abs{\langle e^{-tH_{0}}f , \varphi \rangle_{\El{2}}}\leq \norm{e^{-tH_{0}}f}_{\El{p}}\norm{\varphi}_{\El{p'}}\leq \norm{f}_{\Dz{p}}\norm{\varphi}_{\El{p'}}
    \end{align*}
    for all $t>0$, we obtain that $\abs{\langle f,\varphi\rangle_{\El{2}}}\leq \norm{f}_{\Dz{p}}\norm{\varphi}_{\El{p'}}$ for arbitrary $\varphi\in\El{2}\cap\El{p'}$. This implies that $f\in\Ell{p}$, with $\norm{f}_{\El{p}}\leq \norm{f}_{\Dz{p}}$.

    Now let $f\in\textup{H}^{1}_{V, \text{pre}}(\Rn)$. The previous argument gives
        $\abs{\langle f,\varphi\rangle_{\El{2}}}\leq \norm{f}_{\textup{H}^{1}_{V}}\norm{\varphi}_{\El{\infty}}$
    for all $\varphi\in\El{2}\cap\El{\infty}$. For each $N\geq 1$, we apply this estimate with $\varphi:=\varphi_{N}\in\El{2}\cap\El{\infty}$, where 
    \begin{equation*}
        \varphi_{N}(x):=
        \begin{cases}
            \frac{f(x)}{\abs{f(x)}}\ind{B(0,N)}(x) & \text{ if }f(x)\neq 0,\\
            0 & \text{ if }f(x)=0.
        \end{cases}
    \end{equation*}
    This yields $\int_{B(0,N)}\abs{f(x)}\dd x\leq \norm{f}_{\textup{H}^{1}_{V}}\norm{\varphi_{N}}_{\El{\infty}}\leq \norm{f}_{\textup{H}^{1}_{V}}$ and the result follows.
\end{proof}
\begin{rem}
    If $V\equiv 0$, then $\textup{H}^{p}_{V, \textup{pre}}(\Rn)=\textup{H}^{p}\left(\Rn\right)\cap\Ell{2}$ for all $p\in\left(0,\infty\right]$. Indeed, in this case $H_{0}=-\Delta$, and for $f\in\Ell{2}$ and almost every $x\in\Rn$ it holds that 
$\left(\mathcal{M}_{V}f\right)(x)=\sup_{t>0}\abs{\left(f\ast p_{t}\right)(x)}$, so the fact that $p_{1}\in\mathcal{S}(\Rn)$ yields the claim (see \cite[Chapter \RNum{3}]{Stein_HarmonicAnalysis_93}).
\end{rem}
\subsubsection{The space $\textup{H}^{1}_{V}$}\label{subsubsection with definition of completion H^{1}_{V}}
We now define the Banach space $\textup{H}^{1}_{V}(\Rn)$ that was mentioned in the introduction.
Lemma~\ref{identification of Dziubanski hardy spaces for p>=1} shows that there is a continuous embedding $\textup{H}^{1}_{V,\textup{pre}}(\Rn)\subseteq \Ell{1}$. \emph{The completion} of $\textup{H}^{1}_{V,\textup{pre}}(\Rn)$ in $\Ell{1}$, in the sense of \cite[Definition 2.1]{AMM_2015_CalderonReproducingFormulas}, will be denoted by $\textup{H}^{1}_{V}(\Rn)$. Its existence follows from \cite[Proposition 2.2]{AMM_2015_CalderonReproducingFormulas} and the fact that every Cauchy sequence $(f_{k})_{k\geq 1}$ in $\textup{H}^{1}_{V,\textup{pre}}(\Rn)$ that converges to $0$ in $\El{1}$ satisfies $f_{k}\to 0$ in $\textup{H}^{1}_{V,\textup{pre}}(\Rn)$. This follows from the $\El{1}$-boundedness of the semigroup $\set{e^{-tH_{0}}}_{t>0}$ (see \eqref{uniform boundedness heat semigroup schrodinger on Lp for all p} above). In fact, for every $t>0$ and $k\geq 1$, since $\lim_{m\to\infty}f_{m}=0$ in $\El{1}$, this implies that 
\begin{align*}
    e^{-tH_{0}}f_{k}=\lim_{m\to\infty} e^{-tH_{0}}(f_k-f_m),
\end{align*}
with convergence in $\Ell{1}$. Up to extraction of a subsequence, we get that 
\begin{align*}
    \abs{(e^{-tH_{0}}f_{k})(x)}=\liminf_{m\to\infty} \abs{e^{-tH_{0}}(f_k-f_m)(x)}\leq \liminf_{m\to\infty} \mathcal{M}_{V}(f_k-f_m)(x)
\end{align*}
for all $t>0$, $k\geq 1$ and almost every $x\in\Rn$. Taking the supremum over $t>0$, integrating over $\Rn$, and finally using Fatou's lemma, we deduce that
\begin{equation*}
    \norm{f_{k}}_{\textup{H}^{1}_{V}}\leq \liminf_{m\to \infty}\norm{f_k-f_m}_{\textup{H}^{1}_{V}}
\end{equation*}
for all $k\geq 1$. The conclusion follows from the Cauchy property of $(f_{k})_{k\geq 1}$ in $\textup{H}^{1}_{V,\textup{pre}}(\Rn)$. For completeness we spell out that $\textup{H}^{1}_{V}(\Rn)$ consists of all $f\in\Ell{1}$ for which there is a Cauchy sequence
$(f_{k})_{k\geq 1}$ in $\textup{H}^{1}_{V,\textup{pre}}(\Rn)$ such that $(f_{k})_{k\geq1}$ converges to $f$ in $\Ell{1}$, equipped with the (compatibly defined) norm $\norm{f}_{\textup{H}^{1}_{V}}:=\lim_{k\to\infty}\norm{f_k}_{\textup{H}^{1}_{V}}$. If $n\geq 3$, $q>\frac{n}{2}$ and $V\in\textup{RH}^{q}(\Rn)$, then it also follows from Corollary~\ref{thm continuous inclusion of classical hardy into Dziubanski hardy space} that the classical Hardy space $\textup{H}^{1}(\Rn)$ is continuously embedded into $\textup{H}^{1}_{V}(\Rn)$.
\subsubsection{Equivalent norms on $\textup{H}^{p}_{V,\textup{pre}}$}\label{subsubsection on equivalent characterisations of hardy dziubaski spaces}
We will use an equivalent characterisation of $\textup{H}^{p}_{V,\textup{pre}}$ obtained by Dekel--Kerkyacharian--Kyriazis--Petrushev in \cite{DKKP}. For $n\geq 1$ and non-negative $V\in\El{1}_{\loc}(\Rn)$, we know that $H_{0}$ is a non-negative self-adjoint operator on $\Ell{2}$, mapping real-valued functions to real-valued functions, and the semigroup $\set{e^{-tH_{0}}}_{t>0}$ satisfies the Gaussian bound (\ref{classical gaussian estimates for schrodinger semigroup}), so $H_{0}$ satisfies the hypotheses (H1) and (H2) of \cite[Section~1]{DKKP}.
Let $\varphi\in\mathcal{S}(\R)$ be a real-valued and even Schwartz function. For $f\in\Ell{2}$, we consider the maximal function $\mathcal{M}_{V}(f,\varphi):\Rn\to[0,\infty]$ defined as
$
    \mathcal{M}_{V}(f,\varphi):=\sup_{t>0}\abs{\varphi(tH_{0}^{1/2})f},
$
where the supremum is understood as a lattice supremum (as at the start of Section \ref{section on DZ Hardy spaces}), and  $\varphi(tH_{0}^{1/2})\in\mathcal{L}(\El{2})$ is defined by the Borel functional calculus for the positive self-adjoint operator $H_{0}^{1/2}$ (see Section~\ref{section on the holomorphic functional calculus}). The equivalent characterisation that we shall need is the following result.
\begin{lem}[{\cite[Theorem 3.7]{DKKP}}]\label{lemma equivalent characterisation Hp_V by DKKP}
    Let $p\in(0,1]$ and $\varphi\in\mathcal{S}(\R)$ be real-valued, even, and such that $\varphi(0)\neq 0$. If $f\in\Ell{2}$, then $f\in\textup{H}^{p}_{V,\textup{pre}}(\Rn)$ if and only if $\mathcal{M}_{V}(f,\varphi)\in\Ell{p}$. Moreover, $\norm{f}_{\textup{H}^{p}_{V}}\eqsim \norm{\mathcal{M}_{V}(f,\varphi)}_{\El{p}}$ for all $f\in\textup{H}^{p}_{V,\textup{pre}}(\Rn)$.
\end{lem}

\subsubsection{Atomic decompositions}\label{section on atomic decomposition of hardy dziubanski spaces}
Lemma~\ref{identification of Dziubanski hardy spaces for p>=1} above shows that those Hardy spaces enjoy properties similar to those satisfied by the Fefferman--Stein Hardy spaces. Another very important property that these spaces share with the classical Hardy spaces $\textup{H}^{p}(\Rn)$ for $p\leq 1$ is their atomic decomposition. In this section, unless specified otherwise, we fix $n\geq 3$, $q>\frac{n}{2}$, the potential $V\in\textup{RH}^{q}(\Rn)$, and we let $\delta:=2-\frac{n}{q}>0$.
The definition of the atoms in the spaces $\textup{H}^{p}_{V}(\Rn)$ is based upon the important critical radius function mentioned in subsection \ref{subsection on reverse Holder potentials}.
\begin{definition}\label{definition Dziubanski atoms for V}
Let $p\in(\frac{n}{n+ 1\wedge\delta},1]$. An $\El{\infty}$-atom for $\textup{H}^{p}_{V}(\Rn)$ is a function ${a:\Rn\to\C}$ associated to a ball $B=B(x,r)$ such that
\begin{enumerate}[label=(\roman*)]
    \item $r\leq \rho(x,V),$
    \item $\supp{a}\subseteq \clos{B},$
    \item $\norm{a}_{\El{\infty}}\leq \meas{B}^{-1/p},$
    \item $\displaystyle\int_{\Rn}a(y)\dd y =0$ if $r\leq \frac{1}{4}\rho(x,V)$.
\end{enumerate}
\end{definition}
It can be shown that any such atom $a$ belongs to $\textup{H}^{p}_{V,\textup{pre}}(\Rn)$, and satisfies $\norm{a}_{\text{H}^{p}_{V}}\leq C$, where $C$ only depends on $n, p, q$ and $\llbracket V\rrbracket_{q}$ (see \cite[Section~4]{Dziubanski_Hp_Spaces}).

We note that if $f\in\Ell{2}$, $(a_k)_{k\geq 1}\subset \Ell{2}$ and $(\lambda_{k})_{k\geq 1}\subset \C$ are such that $f=\sum_{k\geq 1}\lambda_{k}a_{k}$ with convergence in $\Ell{2}$, then 
\begin{equation}\label{reduction to uniform atomic estimate for Hardy Dziubanski spaces}
    \norm{f}_{\Dz{p}}^{p}\leq \sum_{k\geq 1}\abs{\lambda_{k}}^{p}\norm{a_{k}}_{\Dz{p}}^{p}.
\end{equation}
Indeed, for all $t>0$ we can use $\El{2}$-boundedness of $e^{-tH_{0}}$ to obtain $e^{-tH_{0}}f = \sum_{k\geq 1}\lambda_{k}e^{-tH_{0}}a_k$
with convergence in $\El{2}$. Using the triangle inequality and taking the (lattice) supremum over all $t>0$ implies that $(\mathcal{M}_{V}f)(x)\leq \sum_{k\geq 1}\abs{\lambda_{k}}(\mathcal{M}_{V}a_k)(x)$
for almost every $x\in\Rn$. Raising this inequality to the power $p\leq 1$, and integrating with respect to $x\in\Rn$, we get \eqref{reduction to uniform atomic estimate for Hardy Dziubanski spaces}. In particular, if $\set{a_{k}}_{k}$ is a collection of atoms for $\textup{H}^{p}_{V}$, and $(\lambda_{k})_{k\geq 1}\in \ell^{p}(\N)$, then $f\in\textup{H}^{p}_{V,\textup{pre}}$ with $\norm{f}_{\textup{H}^{p}_{V}}\lesssim \norm{(\lambda_{k})_{k}}_{\ell^{p}}$.

The atomic decomposition for the spaces $\textup{H}^{p}_{V}$ was obtained by Dziuba\`{n}ski--Zienkiewicz \cite[Theorem 1.11]{Dziubanski_Hp_Spaces}. It shows that every function $f\in\textup{H}^{p}_{V,\textup{pre}}\cap\El{\infty}_{c}$ admits such a decomposition.
\begin{thm}\label{Dziubanski atomic decomposition L^2 local bounded with cpct support}
    Let $p\in(\frac{n}{n+1\wedge\delta},1]$, and let $f\in\textup{H}^{p}_{V,\textup{pre}}(\Rn)\cap\El{\infty}_{c}(\Rn)$. There exists a sequence $\set{a_{j}}_{j\geq 1}$ of $\El{\infty}$-atoms for $\textup{H}^{p}_{V}$, and numbers $(\lambda_{j})_{j\geq 1}\in\ell^{p}(\N)$ such that $f=\sum_{j\geq 1}\lambda_{j}a_{j}$ with convergence in $\Ell{2}$. Moreover, $\norm{f}_{\textup{H}^{p}_{V}}\eqsim\norm{\set{\lambda_{j}}_{j\geq 1}}_{\ell^{p}(\N)}$.
\end{thm}
While the result \cite[Theorem 1.11]{Dziubanski_Hp_Spaces} doesn't mention any particular notion of convergence, inspection of the proof reveals that one can obtain $\El{2}$ convergence of the atomic decomposition, because if $h^{p}(\Rn)$, $0<p\leq 1$, denote the local Hardy spaces of Goldberg \cite{Goldberg_Local_Hardy_Spaces}, then any function $f\in h^{p}(\Rn)\cap\Ell{2}$ admits an $\El{2}$-converging atomic decomposition for $h^{p}$ in the sense of \cite[Lemma 5]{Goldberg_Local_Hardy_Spaces}, as follows from the proof and the corresponding result for $\textup{H}^{p}(\Rn)\cap\Ell{2}$.

Another, more “abstract", atomic decomposition for the Hardy spaces $\Dz{p}$ was obtained in \cite{DKKP} for general non-negative self-adjoint operators with certain Gaussian heat kernel bounds. If $n\geq 1$ and $V\in\El{1}_{\loc}(\Rn)$ is non-negative, then the heat kernel for the non-negative self-adjoint operator $H_{0}:=-\Delta +V$ on $\Ell{2}$ satisfies \eqref{classical gaussian estimates for schrodinger semigroup}, so the main assumptions (H1)--(H2) in \cite[Section~1]{DKKP} are satisfied. The authors introduced the following type of atoms for such operators.
\begin{definition}\label{definition abstract atoms DKKP}
    Let $p\in(0,1]$ and $m :=  \lfloor \frac{n}{2p}\rfloor + 1$. A function $a:\Rn\to\C$ is called an abstract $\textup{H}^{p}_{V}$-atom associated with the operator $H_{0}$ if there exists a function $b\in\dom{H_ {0}^{m}}$ and a ball $B\subset\Rn$ of radius $r_{B}>0$ such that
    \begin{enumerate}[label=(\roman*)]
        \item $a= H_{0}^{m}b$,
        \item $\supp{b}\subseteq B$ (hence $\supp{H_{0}^{k}b}\subseteq B$ for all $k\in\set{0,\cdots ,m}$),
        \item $\norm{H_{0}^{k}b}_{\El{\infty}}\leq r_{B}^{2(m-k)}\meas{B}^{-1/p}$ for all $k\in\set{0,\cdots ,m}$.
    \end{enumerate}
\end{definition}
A particular case of \cite[Theorem 1.4]{DKKP} is the following atomic decomposition result.
\begin{thm}\label{abstract atomic decomposition DKKP}
    Let $p\in(0,1]$, $n\geq 1$ and $V\in\El{1}_{\loc}(\Rn)$ be non-negative. If $f\in\textup{H}^{p}_{V,\textup{pre}}(\Rn)$, then there exists a sequence $\set{a_{j}}_{j\geq 1}$ of abstract $\textup{H}^{p}_{V}$-atoms associated with $H_{0}$, and numbers $\set{\lambda_{j}}_{j\geq 1}\in\ell^{p}(\N)$ such that $f=\sum_{j\geq 1}\lambda_{j}a_{j}$ with convergence in $\Ell{2}$ and in the $\Dz{p}$ norm. Moreover, $\norm{f}_{\textup{H}^{p}_{V}}\eqsim\norm{\set{\lambda_{j}}_{j}}_{\ell^{p}(\N)}$.
\end{thm}
\begin{rem}\label{remark about estimate atoms implies HpV estimate} Since the abstract $\textup{H}^{p}_{V}$-atoms associated with $H_{0}$ are in $\El{\infty}_{c}(\Rn)\cap\textup{H}^{p}_{V,\textup{pre}}(\Rn)$, we can combine Theorems \ref{Dziubanski atomic decomposition L^2 local bounded with cpct support} and \ref{abstract atomic decomposition DKKP} to obtain that for any bounded operator $T\in\mathcal{L}(\El{2})$ satisfying a uniform bound $\norm{Ta}_{\textup{H}^{p}_{V}}\lesssim 1$ for all $\El{\infty}$-atoms $a$ for $\textup{H}^{p}_{V}(\Rn)$ (in the sense of Definition \ref{definition Dziubanski atoms for V}), it follows that $\norm{Tf}_{\textup{H}^{p}_{V}}\lesssim \norm{f}_{\textup{H}^{p}_{V}}$ for all $f\in\textup{H}^{p}_{V,\textup{pre}}(\Rn)$ (recall \eqref{reduction to uniform atomic estimate for Hardy Dziubanski spaces} above).
\end{rem}
\subsubsection{Molecular decompositions}
Molecules adapted to the spaces $\textup{H}^{p}_{V}$ have already been introduced in \cite[Section~3.2]{BJL_T1_Schrodinger}. We modify their definition in order to get rid of any size conditions on the cube (or ball, in the context of the aforementioned reference) associated to the molecule. This seems like a reasonable definition in view of the proof of Lemma~\ref{abstract molecules are concrete dziubanski molecules} below.
\begin{definition}\label{molecule Dziubanski hardy space}
    Let $n\geq 1$ and $p\in(0,1]$. A function $m:\Rn\to\C$ is called a $(p,\eps)$-molecule associated to a cube $Q\subset\Rn$ if there exists $\eps>0$ such that 
    \begin{enumerate}[label=(\roman*)]
        \item $\norm{m}_{\El{2}(C_{j}(Q))}\leq (2^{j}\ell(Q))^{\frac{n}{2}-\frac{n}{p}}2^{-j\eps}$ for all $j\geq 1$,
        \item $\abs{\int_{\Rn}m(x)\dd x}\leq \meas{Q}^{1-\frac{1}{p}}\min\set{1, \left(\ell{(Q)}^{2}\dashint_{Q}V\right)^{1/2}}$.
    \end{enumerate}
\end{definition}
Note that the first condition automatically implies that $m\in\Ell{1}$, hence the second condition makes sense. 
The following result adapts \cite[Lemma 3.9]{BJL_T1_Schrodinger} by incorporating the changes made in the definition of the molecules. 
\begin{thm}\label{molecules are in Dziubanski hardy space with uniform bound}
    Let $n\geq 3$, $ q\in(\frac{n}{2},\infty]$, $V\in\textup{RH}^{q}(\Rn)$, and let $\delta= 2-\frac{n}{q}>0$. If $p\in (\frac{n}{n+\frac{\delta}{2}},1]$, then every $(p,\eps)$-molecule belongs to $\textup{H}^{p}_{V,\textup{pre}}(\Rn)$, with $\norm{m}_{\Dz{p}}\lesssim 1$.
\end{thm}
\begin{proof}
We follow the argument of the proof of \cite[Lemma 3.9]{BJL_T1_Schrodinger}.
    For $j\geq 1$ and $x\in\Rn$, let $a_{j}(x):=\left(m(x)-\dashint_{C_{j}(Q)}m\right)\ind{C_{j}(Q)}(x)$.
    Hence, $\supp{a_{j}}\subseteq C_{j}(Q)$ and $\int_{\Rn}a_{j}=0$ for all $j\geq 1$.
    We introduce $\alpha_{k}:=\int_{C_{k}(Q)} m$ and $\chi_{k}:=\frac{1}{\meas{C_{k}(Q)}}\ind{C_{k}(Q)}$. 
    Then, for all $j\geq 1$, let $N_{j}:=\sum_{k=j}^{\infty}\alpha_{k}$. We let $b_{j}:=N_{j+1}(\chi_{j+1}-\chi_{j})$ for all $j\geq 1$, and obtain that $\supp{b_{j}}\subseteq C_{j}(Q)\cup C_{j+1}(Q)$ and $\int_{\Rn}b_{j}=0$ for all $j\geq 1$.
    We can estimate 
    \begin{align}\label{L2 size estimate for the submolecules a_j}
        \norm{a_{j}}_{\El{2}}&\leq \norm{m}_{\El{2}(C_{j}(Q))} + \abs{\dashint_{C_{j}(Q)}m}\meas{C_{j}(Q)}^{1/2}
        \leq 2\norm{m}_{\El{2}(C_{j}(Q))}\lesssim (2^{j}\ell(Q))^{\frac{n}{2}-\frac{n}{p}}2^{-j\eps}.
    \end{align}
Similarly, we get $\norm{b_j}_{\El{2}}\lesssim 2^{-j\eps}(2^{j}\ell(Q))^{\frac{n}{2}-\frac{n}{p}}$. Denoting $a:=\chi_{1}\int_{\Rn}m = \frac{\ind{4Q}}{\meas{4Q}}\left(\int_{\Rn}m\right)$, we have that $m=a + \sum_{j=1}^{\infty}(a_j+b_j)$ with convergence in $\Ell{2}$. Consequently, it follows from (\ref{reduction to uniform atomic estimate for Hardy Dziubanski spaces}) that 
\begin{equation*}
    \norm{\mathcal{M}_{V}(m)}_{\El{p}}^{p}\leq \norm{\mathcal{M}_{V}(a)}_{\El{p}}^{p} + \sum_{j=1}^{\infty}\norm{\mathcal{M}_{V}(a_j)}_{\El{p}}^{p} + \sum_{j=1}^{\infty}\norm{\mathcal{M}_{V}(b_j)}_{\El{p}}^{p}.
\end{equation*}
It remains to bound each term in the right-hand side. Let us start with the terms involving $a_j$. Since $p\leq 2$, we can use Hölder's inequality to get the following local bound for all $j\geq 1$. 
\begin{align}\label{local bound on generalized atom a_j hardy space}
    \norm{\mathcal{M}_{V}a_{j}}_{\El{p}(2^{j+2}Q)}&\leq \norm{\mathcal{M}_{V}a_{j}}_{\El{2}(2^{j+2}Q)}\meas{2^{j+2}Q}^{\frac{1}{p}-\frac{1}{2}}\lesssim \norm{a_{j}}_{\El{2}}(2^{j+2}\ell(Q))^{\frac{n}{p}-\frac{n}{2}}\lesssim 2^{-j\eps},
\end{align}
where we have used the boundedness of the maximal operator $\mathcal{M}_{V}$ on $\Ell{2}$ (which follows from \eqref{domination of schrodinger maximal function by hardy littlewood max function}) and the estimate (\ref{L2 size estimate for the submolecules a_j}). For the global part, we take advantage of the cancellation $\int_{\Rn}a_{j}=0$ and support property $\supp{a_{j}}\subseteq C_{j}(Q)$ to write for all $t>0$ and $x\in\Rn\setminus 2^{j+2}Q$,
\begin{align*}
    (e^{-tH_0}a_j)(x)=\int_{\Rn}k_{t}(x,y)a_{j}(y)\dd y=\int_{C_{j}(Q)}\left(k_{t}(x,y)-k_{t}(x,x_Q)\right)a_{j}(y)\dd y,
\end{align*}
where $x_Q\in\Rn$ is the centre of the cube $Q\subset \Rn$. 
Since by assumption $\frac{n}{n+\delta/2}<p\leq 1$, there is some $\tilde{\delta}<\delta\leq 2$ such that $\frac{n}{n+\tilde{\delta}/2}<p\leq 1$. The choice $\sigma:=\frac{\tilde{\delta}}{2}$ ensures that $\sigma\in (0,\delta\wedge 1)$ and that $\frac{n}{n+\sigma}<p$ (even if $q=\infty$ and thus $\delta=2$). 
Since for all $y\in C_{j}(Q)$ it holds that  $\abs{x-y}\geq \frac{1}{2}\abs{x-x_Q}$, item (iii) of Lemma~\ref{improved gaussian bound on heat kernel of schrodinger operator} implies that there is $c>0$ such that
\begin{align*}
    \abs{(e^{-tH_0}a_j)(x)}
    &\lesssim \int_{C_{j}(Q)}\left(\frac{\abs{y-x_Q}}{t^{1/2}}\right)^{\sigma}t^{-n/2}e^{-c\frac{\abs{x-x_{Q}}^{2}}{t}}\abs{a_j(y)}\dd y\lesssim \frac{(2^{j}\ell(Q))^{\sigma}}{\abs{x-x_Q}^{n+\sigma}}2^{-j\eps}(2^{j}\ell(Q))^{n-\frac{n}{p}}.
\end{align*}
Using that $n-(n+\sigma)p <0$, it follows that 
\begin{align*}
    \norm{\mathcal{M}_{V}a_{j}}_{\El{p}(\Rn\setminus 2^{j+2}Q)}\lesssim (2^{j}\ell(Q))^{\sigma}2^{-j\eps}(2^{j}\ell(Q))^{n-\frac{n}{p}}(2^{j}\ell(Q))^{-(n+\sigma) +\frac{n}{p}}=2^{-j\eps}.
\end{align*}
Combining this with the local bound (\ref{local bound on generalized atom a_j hardy space}) we obtain that $\norm{a_j}_{\Dz{p}}\lesssim 2^{-j\eps}$ for all $j\geq 1$. The exact same argument can be applied to the functions $b_j$ to obtain the bounds $\norm{b_j}_{\Dz{p}}\lesssim 2^{-j\eps}$ for all $j\geq 1$. 

It therefore remains to estimate $\norm{\mathcal{M}_{V}a}_{\El{p}}$. Let us consider a ball $B\subset\Rn$ centred at $x_B=x_Q$, and such that $\frac{1}{\sqrt{n}} B\subseteq 4Q\subseteq B$. We shall now treat two different cases. The first case is when $r_{B}\leq \rho(x_B)$. In this situation, we can first obtain a local bound 
\begin{align*}
    \norm{\mathcal{M}_{V}a}_{\El{p}(8Q)}&\leq \norm{\mathcal{M}_{V}a}_{\El{2}}\meas{8Q}^{\frac{1}{p}-\frac{1}{2}}\lesssim \norm{a}_{\El{2}}\meas{Q}^{\frac{1}{p}-\frac{1}{2}}\lesssim \abs{\int_{\Rn} m}\meas{Q}^{\frac{1}{p}-1}\lesssim 1,
\end{align*}
where we have used the cancellation condition on $m$ in the last inequality. To obtain the global bound, we then apply item (i) from Lemma~\ref{improved gaussian bound on heat kernel of schrodinger operator} with $N=\frac{\delta}{2}$ to get for all $t>0$ and $x\in\Rn\setminus 8Q$,
\begin{align*}
    (e^{-tH_0}a)(x)&\lesssim \meas{Q}^{-1/p}\left(\ell(Q)^{2}\dashint_{Q}V\right)^{1/2}\int_{4Q}t^{-n/2}e^{-c\frac{\abs{x-y}^{2}}{t}}\left(\frac{t^{1/2}}{\rho(y)}\right)^{-\frac{\delta}{2}}\dd y\\
    &\lesssim t^{-(n+\frac{\delta}{2})/2}e^{-\tilde{c}\frac{\abs{x-x_Q}^{2}}{t}}\meas{Q}^{-1/p}\left(r_{B}^{2}\dashint_{B}V\right)^{1/2}\int_{B}\rho(y)^{\frac{\delta}{2}}\dd y\\
    &\lesssim \abs{x-x_B}^{-n-\frac{\delta}{2}}\meas{B}^{-1/p}\left(\frac{r_B}{\rho(x_B)}\right)^{\delta/2}\rho(x_B)^{\frac{\delta}{2}}\meas{B}\;\lesssim \; r_{B}^{\frac{\delta}{2}}\meas{B}^{1-\frac{1}{p}}\abs{x-x_B}^{-n-\frac{\delta}{2}},
\end{align*}
where we have used \eqref{Shen eqn control of integral of potential by the critical radius function} and the fact that $\rho(y)\eqsim \rho(x_B)$ for $y\in B$ since $r_{B}\leq \rho(x_B)$. We can now use that $n-p(n+\frac{\delta}{2})<0$ to get $\norm{\mathcal{M}_{V}a}_{\El{p}(\Rn\setminus 8Q)}\lesssim 1$.

Finally we treat the case $r_{B}>\rho(x_B)$. The reference \cite[Proposition 5]{DGMTZ_2005} provides a countable collection of points $(x_{\alpha})_{\alpha\in\N}$ such that 
\begin{equation}\label{covering of Rn by critical balls}
\Rn=\bigcup_{\alpha\in\N}B(x_{\alpha}, \rho(x_{\alpha})),
\end{equation}
with bounded overlap (see Section~\ref{subsubsection on geometry in Rn} for a definition). It is shown in \cite[Lemma 1.4]{Shen_Schrodinger} that there are $C>0$ and $k_{0}>0$ such that 
    \[m(x)\leq C(1+\abs{x-y}m(y))^{k_{0}}m(y)\]
    for all $x,y\in\Rn$. This property implies that there is $c\geq 0$ such that for all $\alpha\in\N$, $x\in\Rn$ and $R>\rho(x)$ such that $B(x,R)\cap B(x_{\alpha},\rho(x_{\alpha})))\neq \emptyset$, it holds that  $B(x_{\alpha},\rho(x_{\alpha}))\subseteq B(x,cR)$.
    
We can use the covering \eqref{covering of Rn by critical balls} to decompose the function $a=\frac{1}{\meas{4Q}}(\int m)\ind{4Q}$ into a finite linear combination of atoms supported in balls of the form $B(x_{\alpha},\rho(x_\alpha))$. Let us write $B_{\alpha}:=B(x_{\alpha},\rho(x_{\alpha}))$ for each $\alpha\in\N$. Let us also define $B_{1}':=B_{1}$, and $B_{n}':=B_{n}\setminus\left(\bigcup_{k=1}^{n-1}B_{k}\right)$ for $n\geq 2$. Then we can decompose $a=\displaystyle\sum_{\substack{{\alpha\in\N}\\{B_{\alpha}'\cap 4Q\neq \emptyset}}}\lambda_{\alpha}a_{\alpha}$, 
where $\lambda_{\alpha}:=\meas{B_{\alpha}}^{1/p}\norm{a\ind{B_{\alpha}}}_{\Ell{\infty}}\neq 0$ and $a_{\alpha}:=\lambda_{\alpha}^{-1}a\ind{B_{\alpha}'}$. It trivially follows that each $a_{\alpha}$ is supported in $B_{\alpha}$, with $\norm{a_{\alpha}}_{\Ell{\infty}}\leq \meas{B_{\alpha}}^{-1/p}$. Moreover, since for each $\alpha\in\N$ such that $B_{\alpha}\cap 4Q\neq \emptyset$, it holds that $B_{\alpha}\cap B\neq \emptyset$, we get that $B_{\alpha}\subset cB$, since $r_{B}>\rho(x_B)$. 
Let us denote $A_{Q}:=\set{\alpha\in\N : B_{\alpha}'\cap 4Q\neq \emptyset}$.
The bounded overlap of the balls $\set{B_{\alpha}}_{\alpha\in\N}$ implies that $\displaystyle\sum_{\alpha\in A_{Q}}\meas{B_{\alpha}}\lesssim \meas{\bigcup_{\substack{{\alpha\in\N}\\{B_{\alpha}\cap 4Q\neq \emptyset}}}B_{\alpha}}\lesssim \meas{cB}\eqsim \meas{B}$.
As a consequence, the bound $\abs{\int m}\leq \meas{Q}^{1-1/p}$ implies that 
\begin{align}\label{estimate sum of lambdas in finite atomic decomposition}
\begin{split}
    \left(\sum_{\alpha\in A_{Q}}\abs{\lambda_{\alpha}}^{p}\right)^{1/p}&\lesssim \norm{a}_{\El{\infty}}\left(\sum_{\substack{{\alpha\in\N}\\{B_{\alpha}\cap 4Q\neq \emptyset}}}\meas{B_{\alpha}}\right)^{1/p}\lesssim \abs{\int m}\meas{B}^{-1}\meas{B}^{1/p}\lesssim 1.
    \end{split}
\end{align}
Notice that since 
 $a=\sum_{\alpha\in A_{Q}}\lambda_{\alpha}a_{\alpha}$ and $p\leq 1$, we have $\norm{a}_{\Dz{p}}^{p}\leq \sum_{\alpha\in A_{Q}}\abs{\lambda_{\alpha}}^{p}\norm{a_{\alpha}}^{p}_{\Dz{p}}$.
The estimate (\ref{estimate sum of lambdas in finite atomic decomposition}) shows that it suffices to show that $\norm{a_{\alpha}}_{\Dz{p}}\lesssim 1$ for all $\alpha\in A_{Q}$. As before, we get the local bound
$\norm{\mathcal{M}_{V}a_{\alpha}}_{\El{p}(2B_{\alpha})}\leq \norm{\mathcal{M}_{V}a_{\alpha}}_{\El{2}(2B_{\alpha})}\meas{2B_{\alpha}}^{\frac{1}{p}-\frac{1}{2}}\lesssim 1$.
For the global bound, we note that for all $t>0$, $x\in\Rn\setminus 2B_{\alpha}$ and $y\in B_{\alpha}$ it holds that  $\abs{x-y}\geq \frac{1}{2}\abs{x-x_{\alpha}}$ and $\rho(y)\eqsim \rho(x_{\alpha})$. Consequently, we can use as before the estimate (i) from Lemma~\ref{improved gaussian bound on heat kernel of schrodinger operator} with an arbitrary $N\geq 1$ to obtain
\begin{align*}
    \abs{(e^{-tH_{0}}a_{\alpha})(x)}\leq \int_{B_{\alpha}}k_{t}(x,y)\abs{a_{\alpha}(y)}\dd y
    \lesssim \abs{x-x_{\alpha}}^{-n-N}\meas{B_{\alpha}}^{1-\frac{1}{p}}\rho(x_{\alpha})^{N}.
\end{align*}
Since $B_{\alpha}$ has radius $\rho(x_{\alpha})$, we obtain that $\norm{\mathcal{M}_{V}a_{\alpha}}_{\El{p}(\Rn\setminus 2B_{\alpha})}\lesssim 1$, and this finishes the proof.
\end{proof}
We remark that Theorem~\ref{molecules are in Dziubanski hardy space with uniform bound} implies the following result.
\begin{cor}\label{thm continuous inclusion of classical hardy into Dziubanski hardy space}
    Let $n\geq 3$, $q\in(\frac{n}{2},\infty]$, $V\in\textup{RH}^{q}(\Rn)$, and let $\delta=2-\frac{n}{q}>0$. For all $p\in \left(\frac{n}{n+\delta/2}, 1\right]$, there is a continuous inclusion $\textup{H}^{p}(\Rn)\cap\Ell{2}\subseteq \textup{H}^{p}_{V,\textup{pre}}(\Rn)$ for the respective $p$-quasinorms.
\end{cor}
\begin{proof}
    Let $a\in\El{\infty}_{c}(\Rn)$ be an atom for $\textup{H}^{p}(\Rn)$ associated to a cube $Q\subseteq \Rn$ (see Section~\ref{subsubsection on hardy spaces}). For an arbitrary $\eps>0$, it is clear that $a$ is a $(p,\eps)$-molecule for $\textup{H}^{p}_{V,\textup{pre}}$, associated with $Q$.  Consequently, we have a uniform bound of the form $\norm{a}_{\textup{H}^{p}_{V}}\lesssim 1$. The conclusion follows from the fact that the atomic decomposition of any function $f\in\textup{H}^{p}(\Rn)\cap\Ell{2}$ converges in the $\El{2}$ norm (as well as in the $\textup{H}^{p}(\Rn)$ norm). 
\end{proof}

\subsubsection{Square function characterisations}\label{section interpolation Dziubanski Hardy spaces}
We shall need to interpolate between the Hardy spaces $\textup{H}^{p}_{V,\textup{pre}}(\Rn)$, for $\frac{n}{n+1}<p<\infty$. This is made possible by the following square function characterisation. 

Let us consider the auxiliary function $\psi_{0}(z):=ze^{-z}$. For $f\in\Ell{2}$, we consider the square function $\mathcal{S}_{0}f : \Rn \to [0,\infty]$ defined as 
\begin{equation*}
    (\mathcal{S}_{0}f)(x):=(S_{\psi_{0},H_{0}}f)(x)=\left(\iint_{\Gamma(x)}{\abs{(\psi_{0}(t^{2}H_{0})f)(y)}}^{2}\frac{\dd y\dd t}{t^{n+1}}\right)^{1/2}
\end{equation*}
for all $x\in\Rn$. As usual, it follows from Fubini's theorem and that of McIntosh that $\norm{\mathcal{S}_{0}f}_{\El{2}}\lesssim\norm{f}_{\El{2}}$ for all $f\in\El{2}$. The following theorem gives a characterisation of the maximal Hardy space $\Dz{p}$ in terms of the square function operator $\mathcal{S}_{0}$.
\begin{thm}\label{square function characterisation of Dziubanski Hardy spaces}
    Let $n\geq 1$ and let $V\in\El{1}_{\loc}(\Rn)$ be non-negative. If $p\in(\frac{n}{n+1},\infty)$ and $f\in\Ell{2}$, then $f\in\textup{H}^{p}_{V,\textup{pre}}(\Rn)$ if and only if $\mathcal{S}_{0}f\in\Ell{p}$. Moreover, $\norm{f}_{\Dz{p}}\eqsim \norm{\mathcal{S}_{0}f}_{\El{p}}$ for all $f\in\textup{H}^{p}_{V,\textup{pre}}(\Rn)$.
\end{thm}

\begin{proof}
We first recall that $\textup{H}^{p}_{V,\textup{pre}}=\El{p}\cap\El{2}$ with equivalent $p$-norms (Lemma~\ref{identification of Dziubanski hardy spaces for p>=1}), and that $H_{0}$ is a non-negative self-adjoint operator with a heat-semigroup $\set{e^{-tH_{0}}}_{t>0}$  satisfying the Gaussian estimates \eqref{classical gaussian estimates for schrodinger semigroup}.
    The result for $p\in (1,\infty)$ is stated and proved (in a more general form) in \cite[Theorem 1.1]{Gong-Yan_Weighted_area_integral_estimates}. Note that the lower bounds can be obtained from the upper bounds by duality (as in the proof of Lemma~\ref{continuous inclusions q>2 abstract bisectorial adapted hardy spaces} below). The result is also mentioned in \cite[Section~1, p. 465]{SONG_YAN_2016}, based on the earlier work of Auscher \cite{Auscher_2007}.
    
    Now assume that $\frac{n}{n+1}<p\leq 1$, and denote by $\bb{Q}_{0}:=\bb{Q}_{\psi_{0},H_{0}}:\Ell{2}\to \El{2}((0,\infty),\frac{\dd t}{t};\El{2})$
and $\bb{C}_{0}:=\bb{C}_{\psi_{0},H_{0}}:\El{2}((0,\infty),\frac{\dd t}{t};\El{2})\to \Ell{2}$,
    the (bounded) extension and contraction operators associated with $\psi_{0}$ and $H_{0}$, as defined in Section~\ref{section on extension and contraction operators for general inj sectorial op}. Since $\psi^{*}_{0}=\psi_{0}$ and $H_{0}$ is self-adjoint, we know that $\bb{Q}_{0}^{*}=\bb{C}_{0}$ (see Section~\ref{section on extension and contraction operators for general inj sectorial op}). Moreover, there is a $c>0$ such that $f=c\bb{C}_{0}\bb{Q}_{0}f$ for all $f\in\Ell{2}.$

Let $f\in\El{2}$ be such that $\mathcal{S}_{0}f\in\Ell{p}$. The definition of the square function $\mathcal{S}_{0}f$ precisely means that the function $\bb{Q}_{0}f$, in addition to being in $\tent{2}$, belongs to the tent space $\tent{p}$, with 
$\norm{\bb{Q}_{0}f}_{\tent{p}}=\norm{\mathcal{S}_{0}f}_{\El{p}}$.
    The atomic decomposition in tent spaces (see Section~\ref{section on tent spaces}) implies that there exists a sequence of $\tent{p}$ atoms $\set{A_{k}}_{k\geq 1}$ and a sequence of scalars $\set{\lambda_{k}}_{k\geq 1}\subset\C$ such that $\bb{Q}_{0}f=\sum_{k\geq 1}\lambda_{k}A_{k}$
    with convergence in $\El{2}_{\loc}(\Hn)$ and in $\tent{2}$.  Moreover, we have $\left(\sum_{k\geq 1 }\abs{\lambda_{k}}^{p}\right)^{1/p}\lesssim \norm{\bb{Q}_{0}f}_{\tent{p}}$. 
    By boundedness of $\bb{C}_{0}$ we get $f=c\bb{C}_{0}(\bb{Q}_{0}f)=c\sum_{k\geq 1}\lambda_{k}\bb{C}_{0}(A_{k})$ with convergence in $\El{2}$. By (\ref{reduction to uniform atomic estimate for Hardy Dziubanski spaces}), it therefore remains to prove that $\bb{C}_{0}(A_k)\in\textup{H}^{p}_{V,\textup{pre}}$ for all $k\geq 1$, with $\norm{\bb{C}_{0}(A_k)}_{\Dz{p}}\lesssim 1$ uniformly for all $k$. Let $A$ be an arbitrary $\tent{p}$-atom associated to a cube $Q\subset\Rn$, and let $\alpha:=\bb{C}_{0}(A)$. We first estimate $\norm{\alpha}_{\El{2}}$ by duality. For arbitrary $g\in\Ell{2}$, we can use the fact that $\psi_{0}(t^{2}H_{0})$ is a bounded self-adjoint operator on $\El{2}$, the support property of $A$, the Cauchy--Schwarz inequality and McIntosh's theorem to get
    \begin{align*}
        \abs{\langle \alpha, g\rangle_{\El{2}}}&\leq \int_{0}^{\ell(Q)}\abs{\langle \psi_{0}(t^{2}H_{0})(A(t,\cdot)),g\rangle} \frac{\dd t}{t}=\int_{0}^{\ell(Q)}\abs{\langle A(t,\cdot),\psi_{0}(t^{2}H_{0})g\rangle} \frac{\dd t}{t}\\
        &\leq \left(\iint_{\Hn}\abs{A(t,x)}^{2}\frac{\dd x\dd t}{t}\right)^{1/2}\left(\int_{0}^{\infty}\norm{\psi_{0}(t^{2}H_{0})g}_{\El{2}}^{2}\frac{\dd t}{t}\right)^{1/2}\lesssim \meas{Q}^{\frac{1}{2}-\frac{1}{p}}\norm{g}_{\El{2}}.
    \end{align*}
    This implies that $\norm{\alpha}_{\El{2}}\lesssim \meas{Q}^{\frac{1}{2}-\frac{1}{p}}$. By Hölder's inequality and $\El{2}$-boundedness of the maximal operator $\mathcal{M}_{V}$, we obtain the local bound $\norm{\mathcal{M}_{V}\alpha}_{\El{p}(4Q)}\lesssim 1$. To obtain the global bound, let us note that for fixed $s>0$ and arbitrary $g\in\Ell{2}$ it holds that 
    \begin{align*}
        \langle e^{-s^{2}H_{0}}\alpha , g\rangle_{\El{2}}=\langle \bb{C}_{0}(A) ,e^{-s^{2}H_{0}}g\rangle_{\El{2}}=\lim_{\eps\to 0}\langle \int_{\eps}^{1/\eps}\psi_{0}(t^{2}H_{0})(A(t,\cdot))\frac{\dd t}{t},e^{-s^{2}H_{0}}g\rangle_{\El{2}}.
    \end{align*}
    Now, for fixed $\eps>0$, we can use the fact that the operators $\psi_{0}(s^{2}H_{0})=s^{2}H_{0}e^{-s^{2}H_{0}}\in\mathcal{L}(\Ell{2})$ are integral operators for all $s>0$ (see \cite[Proposition 4]{DGMTZ_2005}), with associated kernels ${\mathcal{Q}_{s}:\Rn\times \Rn\to \R}$ to write
    \begin{align*}
        &\abs{\langle \int_{\eps}^{1/\eps}\psi_{0}(t^{2}H_{0})(A(t,\cdot))\frac{\dd t}{t},e^{-s^{2}H_{0}}g\rangle_{\El{2}}}=\abs{\langle \int_{\eps}^{1/\eps}e^{-s^{2}H_{0}}\psi_{0}(t^{2}H_{0})(A(t,\cdot))\frac{\dd t}{t},g\rangle_{\El{2}}}\\
        &=\abs{\langle \int_{\eps}^{1/\eps}\frac{t^{2}}{t^{2}+s^{2}}\psi_{0}((s^{2}+t^{2})H_{0})(A(t,\cdot))\frac{\dd t}{t} , g\rangle_{\El{2}}}\\
        &\leq \int_{\Rn}\left(\int_{0}^{\ell(Q)}\int_{Q}\frac{t^{2}}{t^{2}+s^{2}}\abs{\mathcal{Q}_{(t^{2}+s^{2})^{\frac{1}{2}}}(x,y)}\abs{A(t,y)}\frac{\dd y\dd t}{t}\right)\abs{g(x)}\dd x.
    \end{align*}
    Since $g\in\smcp$ can be chosen arbitrarily (and in particular can be any translation of a standard mollifier), this implies that for almost every $x\in\Rn$ it holds that 
    \begin{align*}
        \abs{\left(e^{-s^{2}H_{0}}\alpha\right)(x)}&\leq \int_{0}^{\ell(Q)}\int_{Q}\frac{t^{2}}{t^{2}+s^{2}}\abs{\mathcal{Q}_{(t^{2}+s^{2})^{\frac{1}{2}}}(x,y)}\abs{A(t,y)}\frac{\dd y\dd t}{t}\\
        &\lesssim \meas{Q}^{\frac{1}{2}-\frac{1}{p}}\left(\int_{0}^{\ell(Q)}\int_{Q}\abs{\frac{t^{2}}{t^{2}+s^{2}}\mathcal{Q}_{(t^{2}+s^{2})^{\frac{1}{2}}}(x,y)}^{2}\frac{\dd y\dd t}{t}\right)^{1/2},
    \end{align*}
    where we have used the Cauchy--Schwarz inequality and the fact that $A$ is a $\tent{p}$-atom.
    
    Now let $x\in\Rn\setminus 4Q$ and let $x_{Q}$ be the centre of the cube $Q$. Then, for $y\in Q$ it holds that  $\abs{x-y}\eqsim \abs{x-x_{Q}}$. The integral kernel $\mathcal{Q}_{t}(x,y)$ satisfies the following Gaussian bounds:
    \begin{equation*}
        \abs{\mathcal{Q}_{t}(x,y)}\lesssim t^{-n}e^{-\frac{\abs{x-y}^{2}}{ct^{2}}}
    \end{equation*}
    for almost every $x,y\in\Rn$ and $t>0$ (see \cite[Proposition 4]{DGMTZ_2005}). Consequently, for arbitrary $N>0$ we have 
    \begin{align*}
        \abs{\frac{t^{2}}{t^{2}+s^{2}}\mathcal{Q}_{(t^{2}+s^{2})^{1/2}}(x,y)}^{2}&\lesssim t^{4}(t^{2}+s^{2})^{-n-2}\exp\left(-\frac{2\abs{x-y}^{2}}{c(t^{2}+s^{2})}\right)\\
        &\lesssim t^{4}(t^{2}+s^{2})^{N-n-2}\abs{x-x_{Q}}^{-2N}\leq t^{2N-2n}\abs{x-x_{Q}}^{-2N},
    \end{align*}
    provided $N-n-2<0$. We can then estimate
    \begin{align*}
        \abs{(e^{-s^{2}H_{0}}\alpha)(x)}&\lesssim \meas{Q}^{\frac{1}{2}-\frac{1}{p}}\abs{x-x_{Q}}^{-N}\left(\int_{0}^{\ell(Q)}\int_{Q}t^{2N-2n-1}\dd y\dd t\right)^{1/2}\\
        &\lesssim \meas{Q}^{1-\frac{1}{p}}\abs{x-x_{Q}}^{-N} \ell(Q)^{N-n}=\meas{Q}^{-\frac{1}{p}}\abs{x-x_{Q}}^{-N}\ell(Q)^{N},
    \end{align*}
    provided that $N-n>0$. As $s>0$ was arbitrary, we have $(\mathcal{M}_{V}\alpha) (x)\lesssim \abs{x-x_{Q}}^{-N}\ell(Q)^{N-\frac{n}{p}}$ for almost every $x\in\Rn\setminus 4Q$. If $n<Np$, then $\left(\int_{\Rn\setminus 4Q}\abs{x-x_{Q}}^{-Np}\dd x\right)^{1/p}\lesssim \ell(Q)^{\frac{n}{p}-N}$. Consequently, we may take $N:=n+1$ and obtain
    \begin{align*}
        \norm{\mathcal{M}_{V}\alpha}_{\El{p}}\leq \norm{\mathcal{M}_{V}\alpha}_{\El{p}(4Q)}+\norm{\mathcal{M}_{V}\alpha}_{\El{p}(\Rn\setminus {4Q})}\lesssim 1.
    \end{align*}
    This implies that $f\in\textup{H}^{p}_{V,\textup{pre}}(\Rn)$, with $\norm{f}_{\Dz{p}}\lesssim \left(\sum_{k\geq 1}\abs{\lambda_{k}}^{p}\right)^{1/p}\lesssim \norm{\mathcal{S}_{0}f}_{\El{p}}$.

    Let us now turn to the converse implication. Let $f\in\textup{H}^{p}_{V,\textup{pre}}(\Rn)$. We need to show that $\mathcal{S}_{0}f\in\Ell{p}$ with $\norm{\mathcal{S}_{0}f}_{\El{p}}\lesssim\norm{f}_{\Dz{p}}$. It suffices to prove a uniform bound $\norm{\mathcal{S}_{0}a}_{\El{p}}\lesssim 1$ for an arbitrary abstract atom $a$ for $H_{0}$ associated to a ball $B\subset \Rn$ of radius $r_{B}>0$.
    As usual, we use Hölder's inequality and the $\El{2}$-boundedness of $\mathcal{S}_{0}$ to obtain the local estimate $\norm{\mathcal{S}_{0}a}_{\El{p}(4B)}\leq \norm{\mathcal{S}_{0}a}_{\El{2}}\meas{4B}^{\frac{1}{p}-\frac{1}{2}}\lesssim \norm{a}_{\El{2}}\meas{B}^{\frac{1}{p}-\frac{1}{2}}\lesssim 1$.
    To estimate the remaining global term $\norm{\mathcal{S}_{0}a}_{\El{p}(\Rn\setminus 4B)}^{p}$, we will obtain a pointwise bound on $(\mathcal{S}_{0}a)(x)$ for all $x\in\Rn\setminus 4B$. We do so by writing $a=H_{0}^{m}b$ for some $b\in\dom{H_{0}^{m}}$ with $m=\lfloor \frac{n}{2p}\rfloor +1$ and $\supp{b}\subseteq B$. It follows that for all $t>0$ it holds that 
    \begin{align*}
        \psi_{0}(t^{2}H_{0})a=t^{2}H_{0}^{m+1}e^{-t^{2}H_{0}}b=t^{-2m}(t^{2}H_{0})^{m+1}e^{-t^{2}H_{0}}=:t^{-2m}P_{t}^{m}b,
    \end{align*}
    where $P_{t}^{m}$ is an integral operator with kernel $\mathcal{P}_{t}^{m}:\Rn\times\Rn\to \R$ satisfying the estimate
    \begin{equation*}
        \abs{\mathcal{P}_{t}^{m}(x,y)}\lesssim_{\gamma,m} t^{-n}\left(1+\frac{\abs{x-y}}{t}\right)^{-\gamma} 
    \end{equation*}
    for all $\gamma >0$ and almost every $x,y\in\Rn$ (see \cite[Theorem 2.3]{DKKP}). Consequently, we can estimate
    \begin{align*}
        \abs{(\psi_{0}(t^{2}H_{0})a)(y)}\lesssim t^{-2m-n}\int_{B}\left(1+\frac{\abs{y-z}}{t}\right)^{-\gamma}\abs{b(z)}\dd z.
    \end{align*}
    We now treat two different cases, depending on the size of $t>0$. If $t>\frac{1}{2}\abs{x-x_{B}}$, we crudely estimate 
    \begin{align*}
        \abs{(\psi_{0}(t^{2}H_{0})a)(y)}\lesssim t^{-2m-n}\int_{B}\abs{b}\leq \meas{B}^{1-\frac{1}{p}}r_{B}^{2m}t^{-2m-n}.
    \end{align*}
    If $t\leq \frac{1}{2}\abs{x-x_{B}}$, then, since $x\in\Rn\setminus 4B$ and $\abs{x-y}<t$, it holds that  $\abs{y-z}\geq \frac{1}{4}\abs{x-x_{B}}$ for all $z\in B$. Consequently we can estimate
    \begin{align*}
        \abs{(\psi_{0}(t^{2}H_{0})a)(y)}&\lesssim \meas{B}^{1-\frac{1}{p}}r_{B}^{2m}t^{-2m-n}\left(1+\frac{\abs{x-x_{B}}}{4t}\right)^{-\gamma}\lesssim \meas{B}^{1-\frac{1}{p}}r_{B}^{2m}t^{\gamma-2m-n}\abs{x-x_{B}}^{-\gamma},
    \end{align*}
    for all $\gamma>0$. We now split the range of integration at $t=\frac{1}{2}\abs{x-x_{B}}$ to obtain
    \begin{align*}
        (\mathcal{S}_{0}a)(x) 
        &\lesssim \meas{B}^{1-\frac{1}{p}}r_{B}^{2m}\abs{x-x_{B}}^{-\gamma}\left(\iint_{\substack{\abs{x-y}<t \\ 0<t\leq \frac{1}{2}\abs{x-x_{B}}}}  t^{2\gamma-4m-2n}\frac{\dd y\dd t}{t^{n+1}}\right)^{1/2}\\
        &+ \meas{B}^{1-\frac{1}{p}}r_{B}^{2m}\left(\iint_{\substack{\abs{x-y}<t \\ t>\frac{1}{2}\abs{x-x_{B}}}}  t^{-4m-2n}\frac{\dd y\dd t}{t^{n+1}}\right)^{1/2}\\
        &\lesssim \meas{B}^{1-\frac{1}{p}}r_{B}^{2m}\abs{x-x_{B}}^{-\gamma}\left(\int_{0}^{\frac{1}{2}\abs{x-x_{B}}}t^{2\gamma-4m-2n-1}\dd t\right)^{1/2}\\
        &+ \meas{B}^{1-\frac{1}{p}}r_{B}^{2m}\left(\int_{\frac{1}{2}\abs{x-x_{B}}}^{\infty}t^{-4m-2n-1}\dd t\right)^{1/2}\lesssim \meas{B}^{1-\frac{1}{p}}r_{B}^{2m}\abs{x-x_{B}}^{-2m-n},
    \end{align*}
    provided we choose $\gamma>2m+n$ to make the integral converge. Since $-p(2m+n)+n<0$ by definition of $m$, we can integrate this to the power $p$ to obtain $\norm{\mathcal{S}_{0}a}_{\El{p}(\Rn\setminus 4B)}\lesssim 1$.
     This shows that $\norm{\mathcal{S}_{0}a}_{\El{p}}\lesssim 1$ for arbitrary abstract atoms $a$ for $H_{0}$. 
     
     Now, for $f\in\textup{H}^{p}_{V,\textup{pre}}(\Rn)$, there exists an $\El{2}$-convergent atomic decomposition $f=\sum_{k\geq 1}\lambda_{k}a_k$, where $a_{k}$ are abstract atoms associated with $H_{0}$ and $\set{\lambda_{k}}_{k\geq 1}\subset\C$ satisfies $\left(\sum_{k\geq 1}\abs{\lambda_{k}}^{p}\right)^{1/p}\lesssim \norm{f}_{\Dz{p}}$. For each $n\geq 1$, let $f_{n}:=\sum_{k=1}^{n}\lambda_{k}a_k$. Then, by subadditivity of $\mathcal{S}_{0}$ and the fact that $p\leq 1$, we obtain $\norm{\mathcal{S}_{0}f_{n}}^{p}_{\El{p}}\leq \sum_{k=1}^{n}\abs{\lambda_{k}}^{p}\norm{\mathcal{S}_{0}a_{k}}^{p}_{\El{p}}\lesssim \sum_{k=1}^{n}\abs{\lambda_{k}}^{p}$.
    This gives a uniform bound $\norm{\mathcal{S}_{0}f_{n}}_{\El{p}}\lesssim \norm{f}_{\Dz{p}}$ for $n\geq 1$. Since $f_{n}\to f$ in $\El{2}$, for all $(t,x)\in\Hn$ it holds that  
    \begin{align*}
        \int_{\abs{x-y}<t}\abs{(\psi_{0}(t^{2}H_{0})f)(y)}^{2}\dd y=\lim_{n\to\infty}\int_{\abs{x-y}<t}\abs{(\psi_{0}(t^{2}H_{0})f_{n})(y)}^{2}\dd y.
    \end{align*}
    Applying Fatou's lemma twice we obtain $\norm{\mathcal{S}_{0}f}_{\El{p}}^{p}\leq \liminf_{n\to\infty}\norm{\mathcal{S}_{0}f_{n}}_{\El{p}}^{p}\lesssim \norm{f}_{\Dz{p}}^{p}$.
\end{proof}
Let us remark that the proof of Theorem~\ref{square function characterisation of Dziubanski Hardy spaces} shows that, for $\frac{n}{n+1}<p<\infty$, the bounded operators $\bb{Q}_{0}:\Ell{2}\to \tent{2}$ and $\bb{C}_{0}:\tent{2}\to \Ell{2}$ are such that their restrictions $\bb{Q}_{0}|_{\textup{H}^{p}_{V,\textup{pre}}}: \textup{H}^{p}_{V,\textup{pre}}(\Rn)\to\tent{p}\cap\tent{2}$ and $\bb{C}_{0}|_{\tent{p}\cap\tent{2}}: \tent{p}\cap\tent{2}\to\textup{H}^{p}_{V,\textup{pre}}(\Rn)$ 
are bounded for the respective $p$-quasinorms. To be precise, it holds that 
\begin{equation}\label{boundedness of the canonical extension operator from Hardy to tent}
    \norm{\bb{Q}_{0}f}_{\tent{p}}\eqsim \norm{f}_{\textup{H}^{p}_{V}}
\end{equation}
for all $f\in\textup{H}^{p}_{V,\textup{pre}}(\Rn)$, and 
\begin{equation}\label{boundedness of the canonical contraction operator from tent to Hardy}
    \norm{\bb{C}_{0}F}_{\textup{H}^{p}_{V}}\lesssim \norm{F}_{\tent{p}}
\end{equation}
for all $F\in\tent{p}\cap\tent{2}$. Indeed, the two-sided estimate (\ref{boundedness of the canonical extension operator from Hardy to tent}) is precisely the statement of Theorem~\ref{square function characterisation of Dziubanski Hardy spaces}, and the estimate (\ref{boundedness of the canonical contraction operator from tent to Hardy}) for $\frac{n}{n+1}<p\leq 1$ follows immediately from the estimate $\norm{\bb{C}_{0}(A)}_{\textup{H}^{p}_{V}}\lesssim 1$ for all $\tent{p}$-atoms $A$, that was established in the proof of Theorem~\ref{square function characterisation of Dziubanski Hardy spaces}. 
The estimate (\ref{boundedness of the canonical contraction operator from tent to Hardy}) for $1<p<\infty$ can be obtained by duality from (\ref{boundedness of the canonical extension operator from Hardy to tent}), using that $\bb{C}_{0}^{*}=\bb{Q}_{0}$. Indeed, by \cite[Theorem 2]{Coifman_Meyer_Stein_TentSpaces}, for $1<p<\infty$, $F\in\tent{p}\cap\tent{2}$ and $g\in\El{2}\cap\El{p'}$, we have
\begin{align*}
    \abs{\langle \bb{C}_{0}F , g\rangle_{\El{2}}}=\abs{\langle F , \bb{Q}_{0}g\rangle_{\El{2}((0,\infty),\frac{\dd t}{t};\El{2})}}
\lesssim\norm{F}_{\tent{p}}\norm{\bb{Q}_{0}g}_{\tent{p'}}\lesssim\norm{F}_{\tent{p}}\norm{g}_{\El{p'}}.
\end{align*}

\subsection{Hardy--Sobolev spaces $\dot{\textup{H}}^{1,p}_{V}$}\label{section on hardy sobolev spaces adapted to schrodinger operators}
We introduce the class of homogeneous Hardy--Sobolev spaces adapted to the potential $V$ that were mentioned in the introduction. We can work in the generality of $n\geq 1$ and a non-negative $V\in\El{1}_{\loc}(\Rn)$.
\subsubsection{Definition and properties}
For $p\in (0,1]$ and $n\geq 1$, we let $\dot{\textup{H}}^{1,p}_{V,\textup{pre}}(\Rn)$ be the space of all functions $f\in\mathcal{V}^{1,2}(\Rn)=\dom{H_{0}^{1/2}}$ such that $H_{0}^{1/2}f\in\textup{H}^{p}_{V,\textup{pre}}(\Rn)$. The space $\dot{\textup{H}}^{1,p}_{V,\textup{pre}}(\Rn)$ is equipped with the quasinorm
$\norm{f}_{\dot{\textup{H}}^{1,p}_{V}}:=\norm{H_{0}^{1/2}f}_{\Dz{p}}$ (note that $H_{0}^{1/2}$ is injective).

These spaces may not be complete. We will see in the next section that a completion can sometimes be realised as a subspace of some $\El{p}$ space. For this we require the following proposition.
\begin{prop}\label{continuous embedding V adapted Hardy sobolev space in Ell^{p*}}
    Let $n\geq 3$, $p\in (\frac{n}{n+1},1]$, and $f\in \dot{\textup{H}}^{1,p}_{V,\textup{pre}}(\Rn)$. Then $f\in\Ell{p^{*}}$, with the estimate
    \begin{equation*}
        \norm{f}_{\El{p^{*}}}\lesssim \norm{f}_{\dot{\textup{H}}^{1,p}_{V}}.
    \end{equation*}
\end{prop}

\begin{proof}
    Let us observe that for $p\in (\frac{n}{n+1},1]$ it holds that  $p^{*}\in (1,\infty)$. 
    We require the following observation.
    For $1<r<\infty$, arbitrary $0<\mu<\pi$, and $\psi\in\textup{H}^{\infty}(S^{+}_{\mu})$, it holds that  
    \begin{equation}\label{lemma bounded H infty functional calculus in all Lp spaces for H_0}
        \norm{\psi(H_{0})f}_{\El{r}}\lesssim \norm{\psi}_{\El{\infty}(S^{+}_{\mu})}\norm{f}_{\El{r}}
    \end{equation}
    for all $f\in\El{2}\cap\El{r}$. This follows from \cite[Theorem 3.4]{Duong_Robinson_Semigroup_Kernels}. In fact, choose $\varphi\in(0,\frac{\pi}{2})$ such that $\frac{\pi}{2}-\mu<\varphi<\frac{\pi}{2}.$ Lemma~\ref{properties of schrodinger semigroup on L2} shows that $H_{0}$ generates a holomorphic semigroup $\set{e^{-zH_{0}}}_{\Re{z}>0}$, which is uniformly bounded in the sector $S^{+}_{\varphi}$. Moreover, the heat-semigroup $\set{e^{-tH_{0}}}_{t>0}$ has an integral kernel satisfying the Gaussian bounds \eqref{classical gaussian estimates for schrodinger semigroup}. Finally, $H_{0}$ has a bounded $\textup{H}^{\infty}(S^{+}_{\mu})$-functional calculus on $\Ell{2}$. The estimate \eqref{lemma bounded H infty functional calculus in all Lp spaces for H_0} then follows from \cite[Theorem 3.4]{Duong_Robinson_Semigroup_Kernels}.
    
    Let now $f\in\textup{H}^{p}_{V,\textup{pre}}(\Rn)\cap \ran{H_{0}^{1/2}}$. Then we can use the atomic decomposition of Theorem~\ref{abstract atomic decomposition DKKP} to write $f=\sum_{i=1}^{\infty}\lambda_{i}a_{i}$ with convergence in $\El{2}$, where all $a_i$ are abstract atoms for $H_{0}$, thus $a_{i}\in\ran{H_{0}^{1/2}}$ for all $i\geq 1$. Note that  $H_{0}^{-1/2}:\ran{H_{0}^{1/2}}\to \Ell{2}$ is bounded from $\Ell{2}$ to $\Ell{2^{*}}$, as follows from a Sobolev embedding in dimension $n\geq 3$. 
    Since $f\in\ran{H_{0}^{1/2}}$ and $f=\sum_{i=1}^{\infty}\lambda_{i}a_i$ with convergence in $\El{2}$, this implies that
$H_{0}^{-1/2}f=\sum_{i=1}^{\infty}\lambda_{i}H_{0}^{-1/2}a_{i}$ with convergence in $\El{2^{*}}$.
We claim that $\norm{H_{0}^{-1/2}a_{i}}_{\El{p^{*}}}\lesssim 1$, uniformly for all $i\geq 1.$

 Let $a\in\ran{H_{0}^{1/2}}$ be an abstract atom associated to a ball $B\subset \Rn$ of radius $r_{B}>0$. It follows that there is a function $b\in\dom{H_{0}}$ with support in $B$ and such that $a=H_{0}b$, with $\norm{b}_{\El{\infty}}\leq r_{B}^{2}\meas{B}^{-1/p}$. We can use the Calderón--McIntosh reproducing formula for the injective sectorial operator $H_{0}$ and the function $\psi(z)=\sqrt{z}e^{-z}$ to get that
\begin{equation*}
        H_{0}^{-1/2}a=c\lim_{\eps\to 0}\int_{\eps}^{\frac{1}{\eps}}e^{-tH_0}a \frac{\dd t}{t^{1/2}}=c\lim_{\eps\to 0}\left(\int_{\eps}^{r_{B}^{2}}e^{-tH_{0}}a \frac{\dd t}{t^{1/2}} + \int_{r_{B}^{2}}^{\frac{1}{\eps}}e^{-tH_{0}}H_{0}b \frac{\dd t}{t^{1/2}}\right),
    \end{equation*}
    with convergence in the $\El{2}$ norm. 
    To estimate the first term, we use Minkowski's inequality and uniform $\El{p^{*}}$-boundedness of the semigroup $\set{e^{-tH_0}}_{t>0}$ to get
    \begin{align*}
        \norm{\int_{\eps}^{r_{B}^{2}}e^{-tH_{0}}a \frac{\dd t}{t^{1/2}}}_{\El{p^{*}}}&\leq \int_{\eps}^{r_{B}^{2}}\norm{e^{-tH_0}a}_{\El{p^{*}}}\frac{\dd t}{t^{1/2}}\lesssim \norm{a}_{\El{p^{*}}}\int_{0}^{r_{B}^{2}}\frac{\dd t}{t^{1/2}}= r_{B}\norm{a}_{\El{p^{*}}}\\
        &\leq r_{B}\norm{a}_{\El{\infty}}\meas{B}^{1/p^{*}}\lesssim \meas{B}^{1/n}\meas{B}^{-1/p +1/p^{*}}=1,
    \end{align*}
    for all $\eps>0$. Similarly, for the second term, we use that \eqref{lemma bounded H infty functional calculus in all Lp spaces for H_0} implies uniform $\El{p^{*}}$-boundedness of the family $\set{tH_{0}e^{-tH_0}}_{t>0}$ to get 
    \begin{align*}
        \norm{\int_{r_{B}^{2}}^{\frac{1}{\eps}}e^{-tH_{0}}H_{0}b \frac{\dd t}{t^{1/2}}}_{\El{p^{*}}}&\leq \int_{r_{B}^{2}}^{\frac{1}{\eps}}\norm{tH_{0}e^{-tH_0}b}_{\El{p^{*}}}\frac{\dd t}{t^{3/2}}\lesssim \norm{b}_{\El{p^{*}}}\int_{r_{B}^{2}}^{\infty}\frac{\dd t}{t^{3/2}}\\
        &\lesssim \norm{b}_{\El{\infty}}\meas{B}^{1/p^{*}} r_{B}^{-1}\leq r_{B}\meas{B}^{-1/p}\meas{B}^{1/p^{*}}\eqsim 1.
    \end{align*}
    Since the bound doesn't depend on $\eps>0$, it follows that $H_{0}^{-1/2}a\in\El{p^{*}}$ with ${\norm{H_{0}^{-1/2}a}_{\El{p^{*}}}\lesssim 1}$.
    Moreover, as $\set{\lambda_{i}}_{i\geq 1}\in\ell^{p}$ and $p\leq 1$, this implies that $\set{\sum_{i=1}^{N}\lambda_{i}H_{0}^{-1/2}a_{i}}_{N\geq 1}$ is a Cauchy sequence in $\El{p^{*}}$. It follows that $H_{0}^{-1/2}f\in\Ell{p^{*}}$ and that $H_{0}^{-1/2}f=\sum_{i=1}^{\infty}\lambda_{i}H_{0}^{-1/2}a_{i}$ with convergence in $\El{p^{*}}$. Consequently, $\norm{H_{0}^{-1/2}f}_{\El{p^{*}}}\lesssim \left(\sum_{i=1}^{\infty}\abs{\lambda_{i}}^{p}\right)^{1/p}\lesssim \norm{f}_{\Dz{p}}$.
    Since $f\in\textup{H}^{p}_{V,\textup{pre}}(\Rn)\cap\ran{H_{0}^{1/2}}$ was arbitrary, this concludes the proof.
\end{proof}
\subsubsection{Completion}\label{subsection definition of the completion of DZiubanski H^{1,p} spaces}
As in Section~\ref{subsubsection with definition of completion H^{1}_{V}}, we can use the continuous embedding $\dot{\textup{H}}^{1,p}_{V,\textup{pre}}(\Rn)\subseteq\Ell{p^{*}}$ from Proposition~\ref{continuous embedding V adapted Hardy sobolev space in Ell^{p*}} and \cite[Proposition 2.2]{AMM_2015_CalderonReproducingFormulas} (which equally applies to quasinormed spaces) to construct (for $n\geq 3$ and $\frac{n}{n+1}<p\leq 1$) the completion $\dot{\textup{H}}^{1,p}_{V}(\Rn)$ of $\dot{\textup{H}}^{1,p}_{V,\textup{pre}}(\Rn)$ in $\Ell{p^{*}}$. Its existence follows from the following lemma.
\begin{lem}\label{compatibility lemma between V adapted Hardy sobolev spaces and Ell^p spaces}
Let $n\geq 3$, $V\in\El{1}_{\loc}(\Rn)$ and $p\in(\frac{n}{n+1},1]$. Let $(f_{k})_{k\geq 1}\subset \dot{\textup{H}}^{1,p}_{V,\textup{pre}}(\Rn)$ be a Cauchy sequence in $\dot{\textup{H}}^{1,p}_{V,\textup{pre}}(\Rn)$. If $f_{k}\to 0$ in $\El{p^{*}}$, then $f_{k}\to 0$ in $\dot{\textup{H}}^{1,p}_{V,\textup{pre}}(\Rn)$.
\end{lem}
\begin{proof}
 We know from \eqref{lemma bounded H infty functional calculus in all Lp spaces for H_0} with $\psi(z)=\sqrt{z}e^{-z}$ that $tH_{0}^{1/2}e^{-t^{2}H_{0}}\in\mathcal{L}(\El{2})$ is actually $\El{p^{*}}$-bounded for each $t>0$. Consequently, for all $k\geq 1$ and $t>0$ it holds that  
    \begin{align*}
        e^{-t^{2}H_{0}}\left(H_{0}^{1/2}f_{k}\right)=\lim_{m\to\infty}e^{-t^{2}H_{0}}\left(H_{0}^{1/2}f_{k}-H_{0}^{1/2}f_{m}\right),
    \end{align*}
    with convergence in the $\El{p^{*}}$ norm. 
    The conclusion then follows from the argument outlined in Section~\ref{subsubsection with definition of completion H^{1}_{V}}.
\end{proof}

Again, for completeness, we spell out that $\dot{\textup{H}}^{1,p}_{V}(\Rn)$ consists of all $f\in\Ell{p^{*}}$ for which there exists a Cauchy sequence $(f_k)_{k\geq 1}$ in $\dot{\textup{H}}^{1,p}_{V,\textup{pre}}(\Rn)$ that converges to $f$ in $\Ell{p^{*}}$, and is equipped with the (compatibly defined) quasinorm $\norm{f}_{\dot{\textup{H}}^{1,p}_{V}}:=\lim_{k\to\infty}\norm{f_{k}}_{\dot{\textup{H}}^{1,p}_{V}}$.
\subsubsection{Atomic decompositions}\label{section on atomic decompositions of Hardy sobolev spaces for schrodinger operators}
As for Hardy spaces, the easiest way to deal with Hardy--Sobolev spaces is through the use of a suitable atomic decomposition. We introduce the following definition; compare with \cite[Definition 8.30]{AuscherEgert}. We can work in the generality of non-negative $V\in\El{1}_{\loc}(\Rn)$ and $n\geq 1$.
\begin{definition}\label{definition of L2 atoms for hardy sobolev spaces schrodinger}
    Let $p\in(\frac{n}{n+1},1]$. An $\El{2}$-atom for $\dot{\textup{H}}^{1,p}_{V,\textup{pre}}(\Rn)$ associated to a cube $Q\subset\Rn$ is a function $m\in\dom{H_{0}^{1/2}}=\mathcal{V}^{1,2}(\Rn)$ such that $\supp{m}\subseteq Q$ and $\norm{H_{0}^{1/2}m}_{\El{2}}\leq \meas{Q}^{\frac{1}{2}-\frac{1}{p}}$.
\end{definition}
The following result is the required atomic decomposition, to be compared with \cite[Proposition 8.31]{AuscherEgert}.
\begin{thm}\label{thm atomic decomposition of hardy sobolev spaces}
    Let $p\in(\frac{n}{n+1},1]$ and $f\in \dot{\textup{H}}^{1,p}_{V,\textup{pre}}(\Rn)$. There is a collection $\set{m_i}_{i\geq 1}$ of $\El{2}$-atoms for $\dot{\textup{H}}^{1,p}_{V,\textup{pre}}(\Rn)$ and complex numbers $\set{\lambda_{i}}_{i\geq 1}\in\ell^{p}(\N)$ such that $\norm{\set{\lambda_{i}}_{i\geq 1}}_{\ell^{p}}\lesssim \norm{f}_{\dot{\textup{H}}^{1,p}_{V}}$ and 
        ${H_{0}^{1/2}f=\sum_{i\geq 1}\lambda_{i}H_{0}^{1/2}m_i}$ with convergence in $\El{2}$. 
\end{thm}
\begin{proof}
Let us pick an even function $\eta\in \mathcal{S}(\R)$ such that $\supp{\hat{\eta}}\subseteq (-c_{0},c_{0})$, where $c_0$ is the constant from \cite[Lemma 3.5]{HLMMY_HardySpacesDaviesGaffney}. Since $H_{0}$ is a positive self-adjoint operator on $\El{2}$ satisfying Davies--Gaffney estimates, \cite[Lemma 3.5]{HLMMY_HardySpacesDaviesGaffney} shows that for all $t>0$, the operator $\eta(tH_{0}^{1/2})\in\mathcal{L}(\El{2})$ has the finite propagation speed property
    \begin{equation}\label{compactness support integral kernel Davies gaffney estimates}
        \supp{\eta(tH_{0}^{1/2})f}\subseteq \supp{f} + \clos{B(0,t)},
    \end{equation}
for all $f\in\El{2}_{c}(\Rn)$. Let $\varphi(z):=z\eta(z)$. Using the Borel functional calculus for the positive self-adjoint operator $H_{0}^{1/2}$ (see Section~\ref{section on the holomorphic functional calculus}), we can define a bounded extension operator ${\bb{Q}^{\eta} : \El{2}\to \El{2}(0,\infty,\El{2};\frac{\dd t}{t})\simeq \tent{2}}$, as in Section~\ref{section on extension and contraction operators for general inj sectorial op}, by $(\bb{Q}^{\eta}f)(t)=\varphi(tH_{0}^{1/2})f$ for all $t\in(0,\infty)$. Indeed, the boundedness of the operator $\bb{Q}^{\eta}$ follows from the spectral theorem, because if $E:\mathscr{B}(\R)\to \mathcal{L}(\Ell{2})$ is the resolution of the identity (or spectral measure) on $(\R,\mathscr{B}(\R))$ for the self-adjoint operator $H_{0}^{1/2}$ (see \cite[Theorem 13.30]{Rudin1991_functional_analysis}) then for $f\in\Ell{2}$ we have
    \begin{align*}
        \int_{0}^{\infty}\norm{(\bb{Q}^{\eta}f)(t)}_{\El{2}}^{2}\frac{\dd t}{t}&=\int_{0}^{\infty}\left(\int_{\R}\abs{tx\eta(tx)}^{2}dE_{f,f}(x)\right)\frac{\dd t}{t}=\int_{\R}\left(\int_{0}^{\infty}\abs{tx\eta(tx)}^{2}\frac{\dd t}{t}\right)dE_{f,f}(x)\\
        &\leq\left(\int_{0}^{\infty}\abs{s\eta(s)}^{2}\frac{\dd s}{s}\right)\int_{\R}dE_{f,f}(x)\lesssim E_{f,f}(\R)=\norm{f}^{2}_{\El{2}}
    \end{align*}
by changing variables ($s=tx$) and properties of $E$ (see, e.g., \cite[Theorem 13.24]{Rudin1991_functional_analysis}). The operator $\bb{Q}^{\eta}$ then has a bounded adjoint $\bb{C}_{\eta}:=\left({\bb{Q}^{\eta}}\right)^{*}:\El{2}(0,\infty,\El{2};\frac{\dd t}{t})\to\El{2}$.
We use the Borel functional calculus, which is compatible with the holomorphic functional calculus for self-adjoint operators, as $\varphi$ may not have a holomorphic extension to a sector with enough decay at infinity.
    
     As in Section~\ref{section on extension and contraction operators for general inj sectorial op}, for all $g\in\El{2}$ and $F\in\El{2}(0,\infty,\El{2};\frac{\dd t}{t})=\tent{2}$ it holds that 
    \begin{align}\label{property of the contraction operator weakly converging integral}
        \langle \bb{C}_{\eta}F , g \rangle_{\El{2}} =\lim_{\eps\to 0^{+}}\langle \int_{\eps}^{\frac{1}{\eps}}\varphi(tH_{0}^{1/2})F(t)\frac{\dd t}{t} ,g \rangle_{\El{2}}.
    \end{align}
    Let us consider the auxiliary holomorphic function $\psi(z)=z^{2}e^{-z^{2}}$. 
    Since $H_{0}$ is an injective self-adjoint operator, we can apply the Calderón--McIntosh reproducing formula from \cite[Proposition 4.2]{AMM_2015_CalderonReproducingFormulas} to obtain a $c>0$ such that for all $g\in\El{2}$,
    \begin{equation}\label{reproducing formula to prove the atomic decomposition of Hardy--sobolev spaces}
        g=c\lim_{\eps\to 0^{+}} \int_{\eps}^{1/\eps} \varphi(tH_{0}^{1/2})\psi(tH_{0}^{1/2})g\frac{\dd t}{t},
    \end{equation}
    with convergence in $\Ell{2}.$ Note however that, since $\varphi$ is odd, we have $c\int_{0}^{\infty}\varphi(t)\psi(t)\frac{\dd t}{t}=1$, but $c\int_{0}^{\infty}\varphi(-t)\psi(-t)\frac{\dd t}{t}=-1$. Inspection of the proof of \cite[Proposition 4.2]{AMM_2015_CalderonReproducingFormulas} reveals that the result still holds since $H_{0}$ is non-negative.
    As in the proof of Theorem~\ref{square function characterisation of Dziubanski Hardy spaces}, we let $(\bb{Q}_{0}g)(t):=\psi(tH_{0}^{1/2})g\in\El{2}$ for all $t>0$. It follows from (\ref{property of the contraction operator weakly converging integral}) and (\ref{reproducing formula to prove the atomic decomposition of Hardy--sobolev spaces}) that $c\bb{C}_{\eta}\bb{Q}_{0}=\textup{Id}_{\El{2}}$.
    
    We now have all the ingredients to start the proof. If $f\in\mathcal{V}^{1,2}(\Rn)$ and $H_{0}^{1/2}f\in\textup{H}^{p}_{V,\textup{pre}}$, then we have 
$H_{0}^{1/2}f=c\bb{C}_{\eta}\left(\bb{Q}_{0}\left(H_{0}^{1/2}f\right)\right)$.
    Theorem~\ref{square function characterisation of Dziubanski Hardy spaces} shows that $\bb{Q}_{0}\left(H_{0}^{1/2}f\right)\in\tent{p}\cap\tent{2}$, with $\norm{\bb{Q}_{0}\left(H_{0}^{1/2}f\right)}_{\tent{p}}\lesssim \norm{H_{0}^{1/2}f}_{\Dz{p}}$. By the atomic decomposition in tent spaces, there is a decomposition $\bb{Q}_{0}\left(H_{0}^{1/2}f\right)=\sum_{k\geq 1}\lambda_{k} A_{k}$
    with convergence in the $\tent{2}$ norm, where the $A_{k}$ are $\tent{p}$-atoms and $\norm{\set{\lambda_k}_{k}}_{\ell^{p}}\lesssim \norm{H_{0}^{1/2}f}_{\Dz{p}}=\norm{f}_{\dot{\textup{H}}^{1,p}_{V}}$. It follows from boundedness of $\bb{C}_{\eta}$ that $H_{0}^{1/2}f= c\sum_{k\geq 1}\lambda_{k} \C_{\eta}\left( A_k\right)$
    with convergence in $\El{2}$. Let $A\in\tent{2}$ be a $\tent{p}$-atom associated with a cube $Q\subset\Rn$. We consider a function $m\in\El{2}$ defined as the following absolutely convergent Bochner integral in $\Ell{2}$:
    \begin{align*}
        m:=\int_{0}^{\infty}\eta(tH_{0}^{1/2})A(t,\cdot)\dd t = \int_{0}^{\ell(Q)}\eta(tH_{0}^{1/2})A(t,\cdot)\dd t.
    \end{align*}
    We claim that $m\in\dom{H_{0}^{1/2}}$ and that $H_{0}^{1/2}m=\bb{C}_{\eta}\left(A\right)$. Indeed, if $g\in\dom{H_{0}^{1/2}}$ is arbitrary, and $\eps>0$ is fixed, then
    \begin{align*}
        \langle \int_{\eps}^{1/\eps}\varphi(tH_{0}^{1/2})A(t,\cdot)\frac{\dd t}{t} ,g \rangle_{\El{2}} &= \langle \int_{\eps}^{1/\eps}H_{0}^{1/2}\eta(tH_{0}^{1/2})A(t,\cdot)\dd t ,g \rangle_{\El{2}}\\
        &=\langle \int_{\eps}^{1/\eps}\eta(tH_{0}^{1/2})A(t,\cdot)\dd t , H_{0}^{1/2}g \rangle_{\El{2}},
    \end{align*}
    where we have used Hille's theorem (see \cite{Hille_Phillips} or \cite[Theorem A.37]{HaaseLectureNotes}), followed by the self-adjointness of $H_{0}^{1/2}$ in the last equality. Letting $\eps\to 0$, we obtain $\langle \bb{C}_{\eta}\left(A\right) , g\rangle_{\El{2}} = \langle m , H_{0}^{1/2}g\rangle_{\El{2}}$.
    Since $H_{0}^{1/2}$ is self-adjoint, this implies that $m\in\dom{\left(H_{0}^{1/2}\right)^{*}}=\dom{H_{0}^{1/2}}$, with $H_{0}^{1/2}m=\bb{C}_{\eta}(A)$. It follows that $\norm{H_{0}^{1/2}m}_{\El{2}}=\norm{\bb{C}_{\eta}(A)}_{\El{2}}\lesssim \norm{A}_{\tent{2}}\lesssim \meas{Q}^{\frac{1}{2}-\frac{1}{p}}.$ Finally, the support property $\supp{m}\subseteq 2Q$ follows from (\ref{compactness support integral kernel Davies gaffney estimates}) because it implies that 
    \begin{align*}
        \supp{\eta(tH_{0}^{1/2})A(t,\cdot)}&\subseteq \supp{A(t,\cdot)}+B(0,t)\subseteq Q+ B(0,\ell(Q)) \subseteq 2Q
    \end{align*}
    for all $t\in (0,\ell(Q))$. Applying this construction to $A:=A_{k}$ for all $k\geq 1$ and hiding all the multiplicative constants in the coefficients $\set{\lambda_{k}}_{k\geq 1}$ concludes the proof.
\end{proof}

\subsection{Interpolation, duality and iteration of bounds for families of operators}\label{section on families of operators and bddness on dziubanski hardy spaces}
In analogy with \cite[Definition 4.1]{AuscherEgert}, we introduce the following definition. Again, here we can work with the assumption of a non-negative $V\in\El{1}_{\loc}(\Rn)$ and $n\geq 1$.
\begin{definition}\label{boundedness family operators}
    Let $\Omega \subset \C\setminus \{0\}$ and let $W_{1}, W_{2}$ be finite dimensional Hilbert spaces. Let $\set{T(z)}_{z\in\Omega}$ be a family of bounded operators mapping $\El{2}(\Rn;W_{1})$ into $\El{2}(\Rn;W_2)$. Let $a_{i}\in\El{\infty}(\Rn;\mathcal{L}\left(W_{i}\right))$, $i\in\{1,2\}$, such that $a_{i}(x)$ is invertible for almost every $x\in\Rn$ and $a_{i}^{-1}\in\El{\infty}(\Rn;\mathcal{L}\left(W_{i}\right))$. 
    Let $0<p\leq q \leq\infty$. The family $\set{T(z)}_{z\in\Omega}$ is said to be $a_{1}\textup{H}^{p}_{V}-a_{2}\textup{H}^{q}_{V}$-bounded if 
    \begin{align}\label{boundedness of family operators on Dziubanski hardy spaces}
        \norm{a_{2}^{-1}T(z)\left(a_{1}f\right)}_{\textup{H}^{q}_{V}}\lesssim \abs{z}^{\frac{n}{q}-\frac{n}{p}}\norm{f}_{\textup{H}^{p}_{V}}
    \end{align}
    for all $f\in\textup{H}_{V,\textup{pre}}^{p}$ and all $z\in\Omega$.
\end{definition}
We speak of $a\textup{H}_{V}^{p}$-boundedness when $a_{1}=a_{2}=a$ and $p=q$. 
Observe that when $p$ and $q$ are both larger than $1$, this definition coincides with the notion of $a_{1}\textup{H}^{p}-a_{2}\textup{H}^{q}$-boundedness of \cite[Definition 4.1]{AuscherEgert}, which is simply $\El{p}-\El{q}$-boundedness, by Lemma~\ref{identification of Dziubanski hardy spaces for p>=1}.  In \eqref{boundedness of family operators on Dziubanski hardy spaces} and in the rest of this paper, for any integer $N\geq 1$ and $p\in(0,\infty)$, we let $\textup{H}^{p}_{V,\textup{pre}}(\Rn;\C^{N}):=(\textup{H}^{p}_{V,\textup{pre}}(\Rn))^{N}$, with the associated (quasi)norm $\norm{F}_{\textup{H}^{p}_{V}}:=\sum_{k=1}^{N}\norm{F_k}_{\textup{H}^{p}_{V}}$ for $F=(F_{1},\ldots,F_{N})\in\textup{H}^{p}_{V,\textup{pre}}(\Rn;\C^{N})$. 

In the context of Definition \ref{boundedness family operators}, the notion of $a_{1}\textup{H}^{p}_{V}-a_{2}\textup{H}^{q}_{V}$-boundedness interpolates in the expected way. In fact, we have the following lemma. Recall the notation $[p_0,p_1]_{\theta}$ introduced in Section~\ref{subsubsection on definition of exponents and order relations}.
\begin{lem}\label{lemma notion of HpV boundedness interpolates}
    If $\frac{n}{n+1}<p_{j}\leq q_{j}<\infty$, $j\in\set{0,1}$, and the family $\set{T(z)}_{z\in\Omega}$ is $a_{1}\textup{H}^{p_j}_{V}-a_{2}\textup{H}^{q_j}_{V}$-bounded for $j\in\set{0,1}$, then $\set{T(z)}_{z\in\Omega}$ is $a_{1}\textup{H}^{p_{\theta}}_{V}-a_{2}\textup{H}^{q_{\theta}}_{V}$-bounded for all $\theta\in [0,1]$, where $p_{\theta}=[p_0,p_1]_{\theta}$ and $q_{\theta}=[q_0,q_1]_{\theta}$. 
\end{lem}
\begin{proof}
    We use the fact that the notion of $\tent{p}-\tent{q}$-boundedness interpolates. We also use the operators $\bb{C}_{0}:\tent{2}\to\El{2}$ and $\bb{Q}_{0}:\El{2}\to\tent{2}$ constructed above. In particular, we use their boundedness properties $\bb{Q}_{0}|_{\textup{H}^{p}_{V,\textup{pre}}}\in\mathcal{L}(\textup{H}^{p}_{V,\textup{pre}}, \tent{p})$ and $\bb{C}_{0}|_{\tent{p}\cap\tent{2}}\in\mathcal{L}(\tent{p},\textup{H}^{p}_{V,\textup{pre}})$ for all $\frac{n}{n+1}<p<\infty$ and the identity $c\bb{C}_{0}\bb{Q}_{0}=\textup{Id}_{\El{2}}$ for some $c>0$. To be more precise, the family $\set{\bb{Q}_{0}a_{2}^{-1}T(z)a_{1}\bb{C}_{0}}_{z\in\Omega}$ is both $\tent{p_0}-\tent{q_0}$-bounded and $\tent{p_1}-\tent{q_1}$-bounded. By complex interpolation in tent spaces (see \cite[Proposition 6]{Coifman_Meyer_Stein_TentSpaces}), it follows that the family is $\tent{p_{\theta}}-\tent{q_{\theta}}$-bounded. Consequently, for $f\in\textup{H}^{p_{\theta}}_{V,\textup{pre}}(\Rn)$ we have $\bb{Q}_{0}f\in\tent{p_{\theta}}\cap\tent{2}$ and we can estimate
    \begin{align*}
        \norm{a_{2}^{-1}T(z)a_{1}f}_{\Dz{q_{\theta}}}&\eqsim\norm{\bb{C}_{0}\bb{Q}_{0}a_{2}^{-1}T(z)a_{1}f}_{\Dz{q_{\theta}}}\lesssim \norm{\bb{Q}_{0}a_{2}^{-1}T(z)a_{1}f}_{\tent{q_{\theta}}}\\
        &\eqsim\norm{\bb{Q}_{0}a_{2}^{-1}T(z)a_{1}\bb{C}_{0}\bb{Q}_{0}f}_{\tent{q_{\theta}}}\lesssim \abs{z}^{\frac{n}{q_{\theta}}-\frac{n}{p_{\theta}}}\norm{\bb{Q}_{0}f}_{\tent{p_{\theta}}}\lesssim \abs{z}^{\frac{n}{q_{\theta}}-\frac{n}{p_{\theta}}}\norm{f}_{\Dz{p_{\theta}}}.\qedhere
    \end{align*}
\end{proof}
We have the following substitute for \cite[Lemma 4.4]{AuscherEgert} for families of operators on the Hardy spaces $\textup{H}^{p}_{V}$.
\begin{prop}\label{bound on power of family}
    Let $\set{T(z)}_{z\in\Omega}$ be a family of operators as in Definition \emph{\ref{boundedness family operators}}, with $a_{1}=a_{2}=a$ and $W_1 = W_2$. Let $p,\varrho \in (0,2)$ such that $\set{T(z)}_{z\in\Omega}$ is $a\textup{H}^{p}_{V}$-bounded and $a\textup{H}^{\varrho}_{V}-\El{2}$-bounded. Then, for each $q\in (p,2)$, there exists an integer $\beta\geq 1$ (depending on $p$, $q$ and $\varrho$) such that $\set{T(z)^{\beta}}_{z\in\Omega}$ is $a\textup{H}^{q}_{V}-\El{2}$-bounded.
\end{prop}
\begin{proof}
     We follow \cite[Lemma 4.4]{AuscherEgert} verbatim, using that the notion of $a\textup{H}^{q}_{V}$-boundedness interpolates in the expected way (see Lemma~\ref{lemma notion of HpV boundedness interpolates}).
\end{proof}

In the same vein, the notion of $a_{1}\textup{H}^{p}_{V}-a_{2}\textup{H}^{q}_{V}$-boundedness dualises in the expected way. The duals of the Hardy spaces $\textup{H}^{p}_{V,\textup{pre}}(\Rn)$ can be identified with a class of spaces that we introduce next.
\subsubsection{Campanato spaces $\BMO_{V}^{\alpha}$ and duality}\label{subsection on campanato spaces for schrodinger operators}
Let $n\geq 3$, $q\geq \frac{n}{2}$ and $V\in\textup{RH}^{q}(\Rn)$. For $\alpha\in[0,1]$, the Campanato space $\BMO_{V}^{\alpha}(\Rn)$ is defined as the set of all functions $f\in\El{1}_{\loc}(\Rn)$ for which there exists $C\geq 0$ such that 
\begin{equation}\label{Campanato space condition 1}
    \dashint_{B}\abs{f-f_{B}}\dd x\leq C\meas{B}^{\frac{\alpha}{n}}
\end{equation}
for all balls $B\subset\Rn$, and
\begin{equation}\label{Campanato space condition 2}
    \dashint_{B}\abs{f}\dd x\leq C\meas{B}^{\frac{\alpha}{n}}
\end{equation}
for all balls $B=B(x_B,r_B)$ such that $r_{B}\geq \rho(x_{B})$. For such a function $f$, we let 
\begin{equation*}
    \norm{f}_{\BMO_{V}^{\alpha}}:=\inf\set{C\geq 0 : (\ref{Campanato space condition 1}) \textup{ and } (\ref{Campanato space condition 2}) \textup{ hold}}.
\end{equation*}
It is clear that $\BMO_{V}^{\alpha}(\Rn)$ is a linear subspace of $\El{1}_{\loc}(\Rn)$, and that $\norm{\cdot}_{\BMO_{V}^{\alpha}}$ is a norm on $\BMO_{V}^{\alpha}(\Rn)$ when $V\not \equiv 0$ (so  there is at least one ball $B(x_B,r_B)\subset\Rn$ such that $r_B\geq \rho(x_B)$). 

We shall need the following alternative characterisation of the spaces $\BMO_{V}^{\alpha}(\Rn)$ for $\alpha\in(0,1]$. 
    We let $\Lambda^{\alpha}_{V}(\Rn)$ be the space of $\alpha$-Hölder continuous functions $f:\Rn\to\C$ for which $\abs{f(x)}\lesssim \rho(x)^{\alpha}$ for all $x\in\Rn$. It is equipped with the norm 
\[\norm{f}_{\Lambda^{\alpha}_{V}}:=\sup_{\substack{x,y\in\Rn\\x\neq y}}\frac{\abs{f(x)-f(y)}}{\abs{x-y}^{\alpha}} + \norm{\rho^{-\alpha}f}_{\El{\infty}}.\]
 When $\alpha\in(0,1]$, it follows from \cite[Proposition 4]{BHS2008} that $\Lambda^{\alpha}_{V}(\Rn)=\BMO_{V}^{\alpha}(\Rn)$ with equivalent norms. Accordingly, we define $\Lambda^{0}_{V}(\Rn):=\BMO_{V}^{0}(\Rn)$ with the corresponding norm.

The duality property $\left(\textup{H}^{p}_{V}\right)^{*}=\BMO_{V}^{n\left(\frac{1}{p}-1\right)}$ is well known, and follows from the atomic decomposition of the Hardy spaces $\textup{H}^{p}_{V}$ from Section~\ref{section on atomic decomposition of hardy dziubanski spaces}, as is shown in \cite[Theorem 2.1]{Yang_Yang_Zhou_2010}. As we work with subspaces of $\El{2}$, we state an $\El{2}$ version, which is proved similarly.
\begin{prop}\label{duality DZiubanski hardy space and Campanato spaces}
    Let $n\geq 3$, $q>\frac{n}{2}$, $V\in\textup{RH}^{q}(\Rn)$, $p\in(\frac{n}{n+\delta\wedge 1},1]$ and $\alpha:=n\left(\frac{1}{p}-1\right)$. If $f\in\BMO_{V}^{\alpha}\cap\El{2}$ and $g\in\textup{H}^{p}_{V,\textup{pre}}(\Rn)$, then $\abs{\langle f,g\rangle_{\El{2}}}\leq \norm{f}_{\BMO_{V}^{\alpha}}\norm{g}_{\textup{H}^{p}_{V}}$.
    Conversely, if $f\in\Ell{2}$ and there exists $C\geq 0$ such that $\abs{\langle f,g\rangle_{\El{2}}}\leq C\norm{g}_{\textup{H}^{p}_{V}}$ for all $g\in\textup{H}^{p}_{V,\textup{pre}}(\Rn)$, then $f\in\BMO_{V}^{\alpha}$ with $\norm{f}_{\BMO_{V}^{\alpha}}\leq C$.
\end{prop}

We next introduce the following useful notion. Let $0<p<\infty$ and $\alpha\in [0,1]$. In the context of Definition~\ref{boundedness family operators}, we shall refer to a family $\set{T(z)}_{z\in\Omega}$ being $a_{1}\textup{H}^{p}_{V}-a_{2}\Lambda^{\alpha}_{V}$-bounded if 
\begin{align*}
        \norm{a_{2}^{-1}T(z)\left(a_{1}f\right)}_{\Lambda^{\alpha}_{V}}\lesssim \abs{z}^{-\alpha-\frac{n}{p}}\norm{f}_{\textup{H}^{p}_{V}}
    \end{align*}
    for all $f\in\textup{H}_{V,\textup{pre}}^{p}\left(\Rn\right)$ and all $z\in\Omega$. We can now state the duality result we had in mind. This is the analogue of \cite[Lemma 4.3]{AuscherEgert} in our context. 
\begin{lem}\label{dualisation of boundedness of families in Lp spaces and Hp spaces-campanato}
    For $p,r\in[1,\infty]$, a family $\set{T(z)}_{z\in\Omega}$ is $\El{p}-\El{r}$-bounded if and only if the dual family $\set{T(z)^{*}}_{z\in\Omega}$ is $\El{r'}-\El{p'}$-bounded. 
    Moreover, if $n\geq 3$, $q\geq \frac{n}{2}$, $V\in\textup{RH}^{q}(\Rn)$ and $\delta=2-\frac{n}{q}$, then for $p\in(\frac{n}{n+\delta\wedge 1},1]$, $r\in[1,\infty]$ and a $a_{1}\textup{H}^{p}_{V}-\El{r}$-bounded family $\set{T(z)}_{z\in\Omega}$, the dual family $\set{T(z)^{*}}_{z\in\Omega}$ is $\El{r'}-{a_{1}^{*}}^{-1}\Lambda_{V}^{n\left(\frac{1}{p}-1\right)}$-bounded. 
\end{lem}
\begin{proof}
    The first part is elementary. The second part follows from Proposition~\ref{duality DZiubanski hardy space and Campanato spaces}.
\end{proof}

\section{Critical numbers $p_{\pm}(H)$, $q_{\pm}(H)$ and their properties}\label{section on critical numbers and their properties}

Let us recall that the bounded accretive function $b$ is defined as $b(x):=A_{\perp\perp}(x)^{-1}$ for almost every $x\in\Rn$. We will work with $n\geq 1$ and the general assumption that $V\in\El{1}_{\loc}(\Rn)$ is non-negative. We shall consider the sets $\mathcal{J}(H)$ and $\mathcal{N}(H)$ defined as follows:
\begin{align*}
    \mathcal{J}(H)&:=\set{p\in\left(\frac{n}{n+1},\infty\right)\; :\;\set{\left(1+t^{2}H\right)^{-1}}_{t>0} \textup{ is } b\textup{H}^{p}_{V} \textup{-bounded} };\\
    \mathcal{N}(H)&:=\set{p\in\left(\frac{n}{n+1},\infty\right) \;:\; \set{t\nabla_{\mu}\left(1+t^{2}H\right)^{-1}}_{t>0} \textup{ is } b\textup{H}^{p}_{V}-\textup{H}^{p}_{V} \textup{-bounded} }.
\end{align*}
These sets are intervals containing $p=2$ with non-empty interior (see Section~\ref{subsubsection on mapping properties ofa dapted hardy spaces}, Lemma~\ref{lemma notion of HpV boundedness interpolates} and Corollary~\ref{improvement on property of interval J(H)}). The endpoints of $\mathcal{J}(H)$ will be denoted by $p_{\pm}(H)$, and those of $\mathcal{N}(H)$ will be denoted by $q_{\pm}(H)$, analogous to \cite[Definition 6.1]{AuscherEgert}. These four numbers will be called the critical numbers. 

\subsection{Resolvents bounds}
In this subsection, we establish some important properties satisfied by the critical numbers. 
We start by mentioning the following lemma for later reference. Its proof is identical to that of \cite[Lemma 7.2]{AuscherEgert}.
    \begin{lem}\label{density lemma range dom L^p}
        If $p\in\mathcal{J}(H)\cap (1,\infty)$ and $k\geq 1$, then the space $\Ell{p}\cap \dom{H^{k}}\cap\ran{H^{k}}$ is dense in $\Ell{p}\cap\Ell{2}$.
    \end{lem}

In contrast with \cite[Theorem 6.2]{AuscherEgert}, it is not clear wether $p_{-}(H)=q_{-}(H)$ when $V\not\equiv 0$. Nevertheless, Proposition~\ref{properties of critical numbers proposition} will be enough for our purposes. The following two lemmas will be needed for its proof. The first lemma is adapted from \cite[Lemma 6.3]{AuscherEgert} and follows here by observing that the $\El{2}$-boundedness of the family $\set{t\nabla_{\mu}\left(1+t^{2}H\right)^{-1}}_{t>0}$ implies the $\El{2}$-boundedness of $\set{t\nabla_{x}\left(1+t^{2}H\right)^{-1}}_{t>0}$.
\begin{lem}\label{property J(H)}
   If $n\geq 3$, then   $\left(2_{*},2^{*}\right)\subset\mathcal{J}(H)$. Moreover, $\set{\left(1+t^{2}H\right)^{-1}}_{t>0}$ is $\El{2}-\El{p}$-bounded and $\El{p'}-\El{2}$-bounded for every $p\in[2,2^{*}]$. 
\end{lem}

The second lemma is adapted from \cite[Lemma 6.5]{AuscherEgert}, and is proved in exactly the same way.
\begin{lem}\label{decreasing powers in boundedness of resolvents}
    Let $\frac{n}{n+1}<p\leq r<\infty$ such that $r>1$ and $\frac{n}{p}-\frac{n}{r}<1$. If there is an integer $\beta\geq 1$ such that the family $\set{t\nabla_{\mu}\left(1+t^{2}H\right)^{-\beta -1}}_{t>0}$ is $b\textup{H}^{p}_{V}-\El{r}$-bounded, then $\set{t\nabla_{\mu}\left(1+t^{2}H\right)^{-\alpha}}_{t>0}$ is also $b\textup{H}^{p}_{V}-\El{r}$-bounded for all integers $1\leq \alpha\leq \beta$.
\end{lem}
We can now state and prove the main result of this section. We follow the argument used in the proof of \cite[Theorem 6.2]{AuscherEgert}.
\begin{prop}\label{properties of critical numbers proposition}
    Let $n\geq 3$. The critical numbers satisfy
    \begin{align*}
        p_{-}(H)\vee 1 = q_{-}(H)\vee 1\hspace{1cm}\textup{ and }\hspace{1cm}
        p_{+}(H)\geq q_{+}(H)^{*}.
    \end{align*}
    \end{prop}
    \begin{proof}
        Let $\rho\in\mathcal{N}(H)\cap(1,\infty)$. We treat two cases:
        \begin{enumerate}[wide, labelwidth=!, labelindent=0pt, label={\textit{{Case} \arabic*.}}]
        \item Suppose that $\rho<n$. This implies that the family $\set{t\nabla_{x}\left(1+t^{2}H\right)^{-1}}_{t>0}$ is $\El{\rho}$-bounded. As a consequence, a Sobolev embedding shows that
        \begin{align*}
            \norm{\left(1+t^{2}H\right)^{-1}f}_{\El{\rho^{*}}}&\lesssim \norm{\nabla_{x}\left(1+t^{2}H\right)^{-1}f}_{\El{\rho}}\lesssim t^{-1}\norm{f}_{\El{\rho}}=t^{\frac{n}{\rho^{*}}-\frac{n}{\rho}}\norm{f}_{\El{\rho}}\
        \end{align*}
        for all $f\in\El{2}\cap\El{\rho}$ and all $t>0$. This means that the family $\set{\left(1+t^{2}H\right)^{-1}}_{t>0}$ is $\El{\rho}-\El{\rho^{*}}$-bounded. Since it also satisfies $\El{2}$ off-diagonal estimates of exponential order, \cite[Lemma 4.14]{AuscherEgert} shows that it satisfies $\El{[\rho,2]_{\theta}}-\El{[\rho^{*},2]_{\theta}}$ off-diagonal estimates of exponential order, for all $\theta\in(0,1)$. Then, \cite[Lemma 4.7]{AuscherEgert} implies that $\set{\left(1+t^{2}H\right)^{-1}}_{t>0}$ is $\El{[\rho,2]_{\theta}}$-bounded and $\El{[\rho^{*},2]_{\theta}}$-bounded for all $\theta\in(0,1)$. This means that $[\rho,2]_{\theta},[\rho^{*},2]_{\theta}\in\mathcal{J}(H)$ for all $\theta\in(0,1)$, which implies that $p_{-}(H)\leq \rho \leq \rho^{*}\leq p_{+}(H).$
\vspace{6pt}
        \item Suppose that $\rho\geq n$. Then, since $\mathcal{N}(H)$ is an interval containing $2$, it follows that $[2,n]\subset\mathcal{N}(H)$. We may therefore apply Case 1 to any $\tilde{\rho}\in\mathcal{N}(H)\cap [2,n)$ in place of $\rho$, to obtain that $\tilde{\rho}^{*}\leq p_{+}(H)$ for all $\tilde{\rho}\in[2,n)$. Since $\tilde{\rho}^{*}\to\infty$ as $\tilde{\rho}\to n$, we have $p_{+}(H)=\infty=\rho^{*}$, since $\rho\geq n$.
        Moreover, $\rho\geq n > 2 \geq  p_{-}(H)$, hence $p_{-}(H)\leq \rho \leq \rho^{*}\leq p_{+}(H)$.
        
        Altogether, Case 1 and 2 show that        \begin{equation}\label{inequality exponents critical numbers}
            p_{-}(H)\leq \rho \leq \rho^{*}\leq p_{+}(H)
        \end{equation}
        for every $\rho\in\mathcal{N}(H)\cap (1,\infty)$. 
        The last inequality in \eqref{inequality exponents critical numbers} shows that $q_{+}(H)^{*}\leq p_{+}(H).$
        In addition, the first inequality in (\ref{inequality exponents critical numbers}) shows that $p_{-}(H)\leq q_{-}(H)\vee 1$.
        
        Now let $p\in\mathcal{J}(H)\cap (1,2)$. Lemma~\ref{property J(H)} shows that $\set{\left(1+t^{2}H\right)^{-1}}_{t>0}$ is $\El{\rho}-\El{2}$-bounded for some $\rho\in (1,2)$ (we may take $\rho =2_{*}$ since $n\geq 3$). We can therefore apply Proposition~\ref{bound on power of family} to find, for each $q\in (p,2)$, an integer $\beta\geq 1$ such that $\set{(1+t^{2}H)^{-\beta}}_{t>0}$ is $\El{q}-\El{2}$-bounded. By composition with the $\El{2}$-bounded family $\set{t\nabla_{\mu}\left(1+t^{2}H\right)^{-1}}_{t>0}$, we obtain $\El{q}-\El{2}$-boundedness of the family $\set{t\nabla_{\mu}\left(1+t^{2}H\right)^{-\beta -1}}_{t>0}$. The later family also satisfies $\El{2}$ off-diagonal estimates of exponential order (see \cite[Lemma 4.6]{AuscherEgert}). Therefore, \cite[Lemma 4.14]{AuscherEgert} shows that it has $\El{[q,2]_{\theta}}-\El{2}$ off-diagonal estimates of exponential order for every $\theta\in(0,1)$. Consequently, \cite[Lemma 4.7]{AuscherEgert} shows that $\set{t\nabla_{\mu}\left(1+t^{2}H\right)^{-\beta -1}}_{t>0}$ is $\El{[q,2]_{\theta}}$-bounded for each $\theta \in (0,1)$. We may finally apply Lemma~\ref{decreasing powers in boundedness of resolvents} to deduce that $[q,2]_{\theta}\in\mathcal{N}(H)$ for all $\theta\in(0,1)$. Since $q\in\left(p,2\right)$ was arbitrary, this implies that $(p,2]\subset \mathcal{N}(H)$, and therefore $q_{-}(H)\leq p$. The later inequality also trivially holds for all $p\in[2,\infty)$, thus $q_{-}(H)\leq p$ for all $p\in\mathcal{J}(H)\cap(1,\infty)$, hence $q_{-}(H)\leq p_{-}(H)\vee 1$. Altogether, this shows that $q_{-}(H)\vee 1 =p_{-}(H)\vee 1$.
        \end{enumerate}
    \end{proof}
    Following the proof of \cite[Proposition 6.7]{AuscherEgert}, we can use the previous result along with \cite[Lemma 5.5]{MorrisTurner} in order to improve on the conclusions of Lemma~\ref{property J(H)}. This yields the following result, which provides a worst-case estimate for the critical numbers.
    \begin{cor}\label{improvement on property of interval J(H)}
        Let $n\geq 3$, $q>1$ and $V\in\textup{RH}^{q}(\Rn)$. There exists $\eps>0$ such that $q_{+}
        (H)\geq 2+\eps$ and $\left[2_{*}-\eps,2^{*}+\eps\right]\subseteq \mathcal{J}(H)$.
    \end{cor}
    The critical numbers can be effectively estimated in some simple cases that we list in the following remark.
\begin{rem}\label{remark about the value of critical numbers in some particular cases}
     Let $n\geq 3$, $q\geq \frac{n}{2}$, $V\in\textup{RH}^{q}(\Rn)$, and let $\delta=2-\frac{n}{q}$. It follows from \cite[Theorem 1.5]{BJL_T1_Schrodinger} and the estimate \eqref{uniform boundedness heat semigroup schrodinger on Lp for all p} that $p_{-}(H_{0})\leq \frac{n}{n+1\wedge\delta}$ and that $p_{+}(H_{0})=\infty$. In addition, it follows from Lemma~\ref{converse implication for boundedness of riesz transform and critical numbers} below and Theorems 0.5 and 5.10 in \cite{Shen_Schrodinger} that $q_{+}(H_{0})\geq 2q$. 
     
     More generally, if $n\geq 1$, $V\in\El{1}_{\loc}(\Rn)$ is non-negative, $A_{\perp\perp}=a=1$ and $A_{\parallel\parallel}$ is real, symmetric, bounded and strongly elliptic, then $H$ generates a heat semigroup $\set{e^{-tH}}_{t>0}$ satisfying Gaussian bounds. Indeed, the argument of the proof of Lemma~\ref{properties of schrodinger semigroup on L2} shows that $\set{e^{-tH}}_{t>0}$ is dominated by $\set{e^{-tL}}_{t>0}$, where $L:=-\dvg_{x}(A_{\parallel\parallel}\nabla_{x})$. The claim then follows from \cite[Corollary 3.2.8]{Davies_Heat_Kernels_Spectral_Theory}. It follows that $p_{-}(H)\leq 1$ and $p_{+}(H)=\infty$.
\end{rem}
\subsection{The set $\mathcal{I}(V)$ and reverse Riesz transform bounds on $\textup{H}^{p}_{V,\textup{pre}}$}\label{subsection on reverse riesz transform bounds on Dziubanski hardy space for no coefficients case}
In order to identify the adapted Hardy spaces $\bb{H}^{1,p}_{H}$ for $p\leq 1$, more specifically for Proposition~\ref{lemma continuous inclusions abstract hardy spaces for p<=1} below, we shall need an additional set of exponents, $\mathcal{I}(V)$, defined as
\begin{align*}
    \mathcal{I}(V):=\set{p\in \left(\frac{n}{n+1},1\right] : \norm{\nabla_{\mu}f}_{\textup{H}^{p}_{V}}\gtrsim \norm{f}_{\dot{\textup{H}}^{1,p}_{V}}
         \textup{ for all }f\in\mathcal{V}^{1,2}(\Rn)}.
\end{align*}
We will see in Lemma~\ref{characterisation of I(V) in terms of boundedness of the adjoint of Riesz transform} that an exponent $p$ belongs to $\mathcal{I}(V)$ provided a reverse Riesz transform bound holds on the Hardy space $\textup{H}^{p}_{V}(\Rn)$ (recall that $\norm{f}_{\dot{\textup{H}}^{1,p}_{V}}:=\norm{H_{0}^{1/2}f}_{\textup{H}^{p}_{V}}$). 

We now let $n\geq 1$ and $V$ satisfy either (i), (ii) or (iii) of Proposition~\ref{density lemma in homogeneous V adapted Sobolev spaces} for $p=2$, so in particular $\mathcal{V}^{1,2}(\Rn)$ is dense in $\dot{\mathcal{V}}^{1,2}(\Rn)$.
Recall that $H_{0}=-\Delta +V$ is the self-adjoint operator on $\Ell{2}$ defined at the start of Section~\ref{section on Hardy spaces adapted to schrodinger operators}. By Kato's second representation theorem (see, e.g., \cite[Theorem 8.1]{OuhabazHeatEquonDomains}), we have $\dom{H_{0}^{1/2}}=\mathcal{V}^{1,2}(\Rn)$ with $\norm{H_{0}^{1/2}f}_{\El{2}}=\norm{\nabla_{\mu}f}_{\El{2}}$ for all $f\in\mathcal{V}^{1,2}(\Rn).$  By density, the operator $H_{0}^{1/2}$ extends to an isomorphism $\dotH : \dot{\mathcal{V}}^{1,2}(\Rn)\to \Ell{2}$ such that 
\begin{equation}\label{kato estimate identity for H_0}
\norm{\dotH f}_{\El{2}}=\norm{\nabla_{\mu}f}_{\El{2}}
\end{equation}
for all $f\in\dot{\mathcal{V}}^{1,2}(\Rn)$.
Indeed, \eqref{kato estimate identity for H_0} shows that $\dotH$ is injective and has closed range. By functional calculus, we have $\ran{H_{0}}\subseteq \ran{H_{0}^{1/2}}\subseteq \ran{\dotH}\subseteq \Ell{2}$. Since $H_{0}$ is self-adjoint, we have $\clos{\ran{H_{0}}}=\Ell{2}$. It follows that $\dotH$ is surjective, whence an isomorphism.
The inverse of this isomorphism shall be denoted $\invH : \Ell{2}\to\dot{\mathcal{V}}^{1,2}(\Rn)$.

The Riesz transform $\mathcal{R}_{0}:\Ell{2}\to \El{2}(\Rn;\C^{n+1})$ is the bounded operator $\mathcal{R}_{0}:=\!\!\nabla_{\mu}(\dotH)^{-1}$. Let us note that $\mathcal{R}_{0}=\nabla_{\mu}H_{0}^{-1/2}$ on $\ran{H_{0}^{1/2}}$. It follows from \eqref{kato estimate identity for H_0} that $\mathcal{R}_{0}$ is an isometry, i.e. it satisfies $\norm{\mathcal{R}_{0}f}_{\El{2}}^{2}=\norm{f}_{\El{2}}^{2}$ for all $f\in\Ell{2}.$ By the polarisation identity, it follows that the adjoint of $\mathcal{R}_{0}$ is a bounded operator $\mathcal{R}_{0}^{*}:\El{2}(\Rn ;\C^{n+1})\to\Ell{2}$ such that $\mathcal{R}_{0}^{*}\mathcal{R}_{0}=\textup{Id}_{\Ell{2}}$. 
We have the following elementary observation.
\begin{lem}\label{characterisation of I(V) in terms of boundedness of the adjoint of Riesz transform}
    Let $p\in \left(\frac{n}{n+1},1\right]$. The exponent $p$ belongs to $\mathcal{I}(V)$ if and only if $\norm{\mathcal{R}_{0}^{*}g}_{\textup{H}^{p}_{V}}\lesssim \norm{g}_{\textup{H}^{p}_{V}}$ for all $g\in(\textup{H}^{p}_{V,\textup{pre}})^{n+1}\cap\ran{\nabla_{\mu}}$.
\end{lem}
\begin{proof}
    Let $p\in(\frac{n}{n+1},1]$ be such that $\norm{\mathcal{R}_{0}^{*}g}_{\textup{H}^{p}_{V}}\lesssim \norm{g}_{\textup{H}^{p}_{V}}$ for all $g\in(\textup{H}^{p}_{V,\textup{pre}})^{n+1}\cap\ran{\nabla_{\mu}}$. Then, if $f\in\mathcal{V}^{1,2}(\Rn)$ is such that $\nabla_{\mu}f\in(\textup{H}^{p}_{V,\textup{pre}})^{n+1}$, we note that $\mathcal{R}_{0}(H_{0}^{1/2}f)=\nabla_{\mu}f$, and therefore
    \begin{align*}
        \norm{H_{0}^{1/2}f}_{\textup{H}^{p}_{V}}=\norm{\mathcal{R}_{0}^{*}\mathcal{R}_{0}(H_{0}^{1/2}f)}_{\textup{H}^{p}_{V}}\lesssim \norm{\mathcal{R}_{0}(H_{0}^{1/2}f)}_{\textup{H}^{p}_{V}}=\norm{\nabla_{\mu}f}_{\textup{H}^{p}_{V}}.
    \end{align*}
    Conversely, if $p\in\mathcal{I}(V)$, then for $g\in(\textup{H}^{p}_{V,\textup{pre}})^{n+1}\cap\ran{\nabla_{\mu}}$ we write $g=\nabla_{\mu}f$ for $f\in\mathcal{V}^{1,2}(\Rn)$ and obtain
    \begin{align*}
\norm{\mathcal{R}_{0}^{*}g}_{\textup{H}^{p}_{V}}=\norm{\mathcal{R}_{0}^{*}(\nabla_{\mu}f)}_{\textup{H}^{p}_{V}}=\norm{\mathcal{R}_{0}^{*}\mathcal{R}_{0}(H_{0}^{1/2}f)}_{\textup{H}^{p}_{V}}=\norm{H_{0}^{1/2}f}_{\textup{H}^{p}_{V}}\lesssim \norm{\nabla_{\mu}f}_{\textup{H}^{p}_{V}}=\norm{g}_{\textup{H}^{p}_{V}}. 
\hspace{1cm}\qedhere
    \end{align*}
\end{proof}
The following remarks indicate the importance of the set $\mathcal{I}(V)$.
\begin{rem}
    Let $n\geq 3$, $q>\frac{n}{2}$, $V\in\textup{RH}^{q}(\Rn)$, and $\delta=2-\frac{n}{q}$. It follows from \cite[Theorems 1.6 and 1.7]{BJL_T1_Schrodinger} and Corollary~\ref{thm continuous inclusion of classical hardy into Dziubanski hardy space} that the Riesz transform $\mathcal{R}_{0}$ is $\textup{H}^{p}_{V}$-bounded for all $p\in(\frac{n}{n+\frac{\delta}{2}},1]$. So, if $p\in\mathcal{I}(V)\cap (\frac{n}{n+\frac{\delta}{2}},1]$, there is an equivalence $\norm{H_{0}^{1/2}f}_{\textup{H}^{p}_{V}}\eqsim \norm{\nabla_{\mu}f}_{\textup{H}^{p}_{V}}$
    for all $f\in\mathcal{V}^{1,2}(\Rn)$.
\end{rem}
\begin{rem}
    Let $n\geq 1$ and $V\in\El{1}_{\loc}(\Rn)$ be non-negative. If $f\in\mathcal{V}^{1,2}(\Rn)$ and $\nabla_{\mu}f\in\left(\textup{H}^{p}_{V,\textup{pre}}\right)^{n+1}$ for some $p\in\mathcal{I}(V)$, then the atomic decomposition of Theorem~\ref{thm atomic decomposition of hardy sobolev spaces} can be applied to obtain some coefficients $\set{\lambda_{k}}_{k}\subset\C$ such that $\nabla_{\mu}f=\sum_{k}\lambda_{k}\nabla_{\mu}m_k$ in $\El{2}$, where the $\set{m_k}_{k}$ are $\El{2}$-atoms for $\dot{\textup{H}}^{1,p}_{V,\textup{pre}}(\Rn)$ in the sense of Definition \ref{definition of L2 atoms for hardy sobolev spaces schrodinger}, and $\norm{\set{\lambda_{k}}_{k}}_{\ell^{p}}\lesssim \norm{\nabla_{\mu}f}_{\textup{H}^{p}_{V}}$. 
\end{rem}
The main result of this section is the following boundedness criterion for $\mathcal{R}_{0}^{*}$. The class $S^{\alpha}$ is a “smoothness" class for potentials $V\in\textup{RH}^{\frac{n}{2}}(\Rn)$ which is introduced at the start of Section~\ref{subsection reverse riesz trsf bounds hardy 2} below, and shown to contain a non trivial class of potentials $V$ in Section~\ref{subsubsection on the S alpha smoothness class of potentials}. 
\begin{thm}\label{full boundedness criteria for adjoint riesz transforms on hardy dziubanski spaces}
    Let $n\geq 3$, $q>n$, $\alpha\in(0,1]$ and $V\in\textup{RH}^{q}(\Rn)\cap S^{\alpha}$. The adjoint $\mathcal{R}_{0}^{*}$ is $\textup{H}^{p}_{V}$-bounded for all $p\in(\frac{n}{n+\alpha\wedge \sigma},1]$, where $\sigma=1-\frac{n}{q}>0$. It follows that $\left(\frac{n}{n+\alpha\wedge\sigma} ,1\right] \subseteq \mathcal{I}(V)$.
\end{thm}
This theorem will be proved in two independent steps, namely Propositions \ref{thm boundedness adjoint Riesz transform 1} and \ref{thm boundedness adjoint Riesz transform 2} below, combined with Lemma~\ref{characterisation of I(V) in terms of boundedness of the adjoint of Riesz transform} to obtain the conclusion. The Riesz transform $\mathcal{R}_{0}$ can be decomposed into its components as $\mathcal{R}_{0}f:=\begin{bmatrix}
        \mathcal{R}_{\parallel}f\\
        \mathcal{R}_{\mu}f
    \end{bmatrix}$ where the operators $\mathcal{R}_{\parallel}:\Ell{2}\to\El{2}(\Rn;\C^{n})$ and $\mathcal{R}_{\mu}\in\mathcal{L}(\Ell{2})$ are both bounded. We can also further decompose the operator $\mathcal{R}_{\parallel}$ as $\mathcal{R}_{\parallel}=\left(\mathcal{R}_{\parallel,k}\right)_{k=1}^{n}$.
This implies that 
\begin{align*}
    \mathcal{R}_{0}^{*}\begin{bmatrix}
        f\\
        g
    \end{bmatrix}&=\mathcal{R}_{\parallel}^{*}f + \mathcal{R}_{\mu}^{*}g=\sum_{k=1}^{n}\mathcal{R}_{\parallel,k}^{*}f_k + \mathcal{R}_{\mu}^{*}g
\end{align*}
for all $f=\left(f_k\right)_{k=1}^{n}\in\El{2}(\Rn;\C^{n})$ and $g\in\Ell{2}$. For $p\in (\frac{n}{n+1},1]$, the boundedness of the adjoint $\mathcal{R}_{\parallel}^{*}$ on the Hardy space $\textup{H}^{p}_{V}$ is investigated in Section~\ref{subsection reverse riesz trsf bounds hardy 1}, while that of $\mathcal{R}_{\mu}^{*}$ is considered in Section~\ref{subsection reverse riesz trsf bounds hardy 2}.
\subsubsection{Boundedness of $\mathcal{R}_{\parallel}^{*}$ on $\textup{H}^{p}_{V,\textup{pre}}$}\label{subsection reverse riesz trsf bounds hardy 1}
In this section, we let $n\geq 3$, $q>n$ and $V\in\textup{RH}^{q}(\Rn)$. It follows from \cite[Theorem 0.8]{Shen_Schrodinger} that the operators $\mathcal{R}_{\parallel,k}$, and therefore $\mathcal{R}_{\parallel,k}^{*}$, are Calderón--Zygmund operators. This means that there are locally integrable kernels $K_{k}\in\El{1}_{\loc}((\Rn\times\Rn)\setminus \set{(x,y): x=y\in\Rn})$ such that 
\begin{align*}
    (\mathcal{R}_{\parallel,k}f)(x)=\int_{\Rn}K_{k}(x,y)f(y)\dd y\hspace{1cm}\textup{ for a.e. }x\in\Rn\setminus\supp{f}
\end{align*}
for all $f\in\El{\infty}_{c}(\Rn)$. The same representation also holds for the adjoints $\mathcal{R}_{\parallel,k}^{*}$ with the associated kernel $K_{k}^{*}(x,y):=K_{k}(y,x)$. The kernels satisfy the usual Calderón--Zygmund estimates, meaning that there are $C\geq 0$ and $\sigma>0$ such that 
\begin{equation}\label{cz estimates for kernel of riesz tranforms}
\left\{\begin{array}{l}
     \abs{K_{k}(x,y)}\lesssim{\abs{x-y}^{n}},  \\
     \abs{K_{k}(x+h,y)-K_{k}(x,y)}\lesssim \frac{\abs{h}^{\sigma}}{\abs{x-y}^{n+\sigma}},\\
     \abs{K_{k}(x,y+h)-K_{k}(x,y)}\lesssim \frac{\abs{h}^{\sigma}}{\abs{x-y}^{n+\sigma}},
\end{array}\right.
\end{equation}
for all $x,y,h\in\Rn$, $x\neq y$ and $\abs{h}<\frac{\abs{x-y}}{2}$, where $\sigma=1-\frac{n}{q}>0$.
To prove $\textup{H}^{p}_{V}$-boundedness of $\mathcal{R}_{\parallel, k}^{*}$, we shall apply the boundedness criterion \cite[Theorem 1.2]{BJL_T1_Schrodinger}. This criterion is concerned with generalised $H_{0}$-adapted Calderón--Zygmund operators, defined in \cite[Definition 1.1]{BJL_T1_Schrodinger}.

\begin{prop}\label{thm boundedness adjoint Riesz transform 1}
    Let $n\geq 3$, $q>n$, $V\in\textup{RH}^{q}(\Rn)$, and let $\sigma=1-\frac{n}{q}>0$. Then $\mathcal{R}_{\parallel,k}^{*}$, $k\in\set{1,\ldots,n}$, are $\textup{H}^{p}_{V}$-bounded for $p\in(\frac{n}{n+\sigma}, 1]$.
\end{prop}
\begin{proof}
Let $k\in\set{1,\cdots, n}$ be fixed. We claim that $\mathcal{R}_{\parallel,k}^{*}\in GCZO_{\rho}(\sigma,2)$ in the sense of \cite[Definition 1.1]{BJL_T1_Schrodinger}. We first check that condition (i) of \cite[Definition 1.1]{BJL_T1_Schrodinger} is satisfied. We have that for each $N>0$, there exists $C_{N}>0$ such that
\begin{align}\label{stronger estimate on the kernel of riesz transform for H_0 Shen}
    \abs{K_{k}(x,y)}\leq C_{N}(1+\abs{x-y}m(x))^{-N}\abs{x-y}^{-n}
\end{align}
for every $x,y\in\Rn$, $x\neq y$ (see the proof of \cite[Theorem 0.8]{Shen_Schrodinger}). 

Using \eqref{stronger estimate on the kernel of riesz transform for H_0 Shen}, we get that for $N>0$, $x_B\in\Rn$, $y\in B(x_B,\rho(x_B))$ and $R>2\rho(x_B)$, we have $\abs{K_{k}^{*}(x,y)}^{2}\lesssim \abs{x-y}^{-2n}\left(1+\frac{\abs{x-y}}{\rho(y)}\right)^{-2N}$.
Note that for $\abs{x-x_B}>R$ and $y\in B(x_B,\rho(x_B))$ we have $\abs{x_{B}-y}<\frac{\abs{x-x_{B}}}{2}$ and therefore $\abs{x-y}\geq \frac{\abs{x-x_{B}}}{2}>R/2$. Since also $\rho(y)\eqsim\rho(x_B)$ we get $\left(1+\frac{\abs{x-y}}{\rho(y)}\right)^{-2N}\lesssim \left(1+\frac{R}{\rho(x_B)}\right)^{-2N}$.
We therefore obtain
\begin{align*}
    \left(\int_{R<\abs{x-x_B}<2R}\abs{K_{k}^{*}(x,y)}^{2}\dd x\right)^{1/2}&\lesssim \left(\frac{R}{\rho(x_B)}\right)^{-N}\left(\int_{R<\abs{x-x_B}<2R}\abs{x-x_B}^{-2n}\dd x\right)^{1/2}\\
    &\lesssim \left(\frac{\rho(x_B)}{R}\right)^{N}R^{-\frac{n}{2}},
\end{align*}
which shows that condition (i) of \cite[Definition 1.1]{BJL_T1_Schrodinger} is satisfied. Next, we check that $K_{k}^{*}$ satisfies condition (ii) in \cite[Definition 1.1]{BJL_T1_Schrodinger}. If $B=B(x_B,r_B)$, $y\in B$ and $j\geq 1$ is an integer, then we note that for $\abs{x-x_B}>2^{j}r_{B}\geq 2r_B$ we have $\frac{\abs{x-x_{B}}}{2}>r_{B}>\abs{x_B-y}$. We can therefore use \eqref{cz estimates for kernel of riesz tranforms} to get
\begin{align*}
    &\left(\int_{2^{j}r_{B}<\abs{x-x_B}<2^{j+1}r_B} \abs{K^{*}_{k}(x,y)-K^{*}_{k}(x,x_{B})}^{2}\dd x\right)^{1/2}\\
    &\lesssim \left(\int_{2^{j}r_B<\abs{x-x_B}<2^{j+1}r_{B}} \left(\frac{\abs{y-x_B}^{\sigma}}{\abs{x-x_{B}}^{n+\sigma}}\right)^{2}\dd x\right)^{1/2}\lesssim \frac{r_{B}^{\sigma}}{(2^{j}r_{B})^{n+\sigma}}(2^{j}r_{B})^{n/2}\lesssim 2^{-j\sigma}\meas{2^{j}B}^{-1/2},
\end{align*}
as needed. These estimates show that $\mathcal{R}_{\parallel,k}^{*}\in GCZO_{\rho}(\sigma,2)$. The proof of \cite[Proposition 4.12]{Ma_etAl} shows that $\mathcal{R}_{\parallel,k}1$ satisfies the criterion of \cite[Theorem 1.2]{BJL_T1_Schrodinger} (the definition of $\mathcal{R}_{\parallel,k}1$ is made precise therein).
We conclude that for every $\textup{H}^{p}_{V}$-atom $a\in\El{\infty}_{c}(\Rn)$ it holds that  $\norm{\mathcal{R}_{\parallel,k}^{*}a}_{\textup{H}^{p}_{V}}\lesssim 1$. 
It follows from the atomic decompositions of Theorems \ref{Dziubanski atomic decomposition L^2 local bounded with cpct support} and \ref{abstract atomic decomposition DKKP} that $\norm{\mathcal{R}_{\parallel,k}^{*}f}_{\textup{H}^{p}_{V}}\lesssim \norm{f}_{\textup{H}^{p}_{V}}$ for all $f\in\textup{H}^{p}_{V,\textup{pre}}(\Rn)$.
\end{proof}
\subsubsection{Boundedness of $\mathcal{R}_{\mu}^{*}$ on $\textup{H}^{p}_{V,\textup{pre}}$}\label{subsection reverse riesz trsf bounds hardy 2}
We now turn to the $\textup{H}^{p}_{V}$-boundedness of the adjoint $\mathcal{R}_{\mu}^{*}$ for $p\in(\frac{n}{n+1},1]$. We let $n\geq 3$, $q\geq\frac{n}{2}$ and $V\in\textup{RH}^{q}(\Rn)$. We introduce a “smoothness" class for these potentials $V$. For a parameter $\alpha\in (0,1]$, we say that $V\in S^{\alpha}$ if
\begin{equation}\label{S alpha smoothness condition}
    \dashint_{B}\abs{V^{1/2}(x)-(V^{1/2})_{B}}\dd x\lesssim r_{B}^{\alpha}m(x_B,V)^{1+\alpha}
\end{equation}
for all balls $B=B(x_B,r_B)$ such that $r_B\leq \rho(x_B,V)$. 
We will exhibit some potentials that lie in $S^{\alpha}$ for some $\alpha\in(0,1]$ in Section~\ref{subsubsection on the S alpha smoothness class of potentials} below.

We obtain the following result, which is similar to \cite[Theorem 1.8]{BJL_T1_Schrodinger} for $s=\frac{1}{2}$, although the range of indices $p\in (\frac{n}{n+\alpha\wedge \sigma},1]$ obtained here is somewhat different. This is because we avoid their approach based on \cite[Lemma 5.5]{BJL_T1_Schrodinger}, which we could not get to work.
\begin{prop}\label{thm boundedness adjoint Riesz transform 2}
    Let $n\geq 3$, $q>n$, $\alpha\in (0,1]$, $V\in\textup{RH}^{q}(\Rn)\cap S^{\alpha}$ and let $\sigma:=1-\frac{n}{q}$. Then $\mathcal{R}_{\mu}^{*}$ is $\textup{H}^{p}_{V}$-bounded for all $p\in (\frac{n}{n+\alpha\wedge \sigma},1]$.
\end{prop}
\begin{proof}
    Since the interval $(\frac{n}{n+\alpha\wedge \sigma},1]$ is open at the left endpoint, we note that the case $q=\infty$ follows from the case $q<\infty$. We can therefore assume that $\delta=2-\frac{n}{q}<2$.
    Let $p\in (\frac{n}{n+\alpha\wedge \sigma},1]$. Since $p\in(\frac{n}{n+\delta/2},1]$, we can choose $\eta_{0},\gamma \in (0,1)$ sufficiently close to $1$ so that $p>\frac{n}{n+\frac{\delta}{2}+(\eta_{0}\gamma -1)}$. Let $a$ be an $\El{\infty}$-atom for $\textup{H}^{p}_{V}$ associated to a ball $B=B(x_B,r_B)\subset\Rn$ with $r_B\leq \rho(x_B)$. We have $\mathcal{R}_{\mu}^{*}a=\invH (V^{1/2}a)$. It follows from \cite[Section~4.7]{Ma_etAl} that $\invH $ is an integral operator with a symmetric kernel $K\in\El{1}_{\loc}(\Rn\times\Rn)$ such that for all $0<\eta<\delta\wedge 1=1$ and $N>0$ there is a $C_{N}\geq 0$ such that 
    \begin{align}\label{generalised calderon zygmund estimates for kernel of H_{0}^{-1/2}}
    \left\{\begin{array}{l}
        \abs{K(x,y)}\leq \frac{C_{N}}{\abs{x-y}^{n-1}}\left(1+\frac{\abs{x-y}}{\rho(x)} +\frac{\abs{x-y}}{\rho(y)}\right)^{-N}\\
         \abs{K(x,y)-K(x,z)}\leq C\frac{\abs{y-z}^{\eta}}{\abs{x-z}^{n-1+\eta}}
    \end{array}\right.
\end{align}
for all $x,y,z\in\Rn$, $x\neq y$ and $\abs{x-y}>2\abs{y-z}$. (Note it seems that $0<\delta<2-\frac{n}{q}$ should be $0<\delta<(2-\frac{n}{q})\wedge 1$ at the top of page 588 in \cite{Ma_etAl}.) These estimates show that $\invH $ is a $1$-Schrödinger–Calderón–Zygmund operator in the sense of \cite[Definition 3.1]{Ma_etAl}. In particular, this allows us to define $\invH (1)\in\El{2}_{\loc}(\Rn)$ with the property that for all $f\in\El{2}_{c}(\Rn)$, it holds that  $\invH f\in\Ell{1}$, and
\begin{equation}\label{duality property of H_0^{-1/2}1}
    \int_{\Rn}\invH f = \langle f , \invH 1 \rangle_{\El{2}}.
\end{equation}

We claim that there is $\beta>0$ such that $\invH V^{1/2}a$ is a $(p,2,\beta,\alpha\wedge \sigma)$-molecule for $H_{0}$ associated to the ball $B$ in the sense of \cite[Section~3.2]{BJL_T1_Schrodinger} (we shall only require $r_{B}\leq \rho(x_B)$ instead of $r_{B}\leq \frac{1}{2}\rho(x_B)$). Note that this is slightly different from the molecules introduced in Definition \ref{molecule Dziubanski hardy space}, because it requires a size condition on the ball associated with the molecule. To this aim, we split 
\begin{align}\label{decomposition of part of adjoint riesz trsf into two terms}
\begin{split}
    \invH V^{1/2}a&=\invH (V^{1/2}-(V^{1/2})_{B})a + (V^{1/2})_{B}\invH a\\
    &=\invH (\eta_{V}a) + (V^{1/2})_{B}\invH a=:\RNum{1}+\RNum{2},
    \end{split}
\end{align}
where we defined $\eta_{V}:=V^{1/2}-(V^{1/2})_{B}\in\El{2}_{\loc}(\Rn)$. We note that the assumption $V\in S^{\alpha}$ implies that $\norm{\eta_{V}}_{\El{1}(B)}\leq r_{B}^{\alpha}m(x_B)^{1+\alpha}\meas{B}$.
Since $r_B\leq \rho(x_B)$, it follows from Hölder's inequality and \eqref{Shen eqn control of integral of potential by the critical radius function} that $\left(V^{1/2}\right)_{B}\leq \left(\dashint_{B}V\right)^{1/2}\lesssim r_{B}^{-1}\left(\frac{r_B}{\rho(x_B)}\right)^{\delta /2}$.
We shall prove that both terms $\RNum{1}$ and $\RNum{2}$ in \eqref{decomposition of part of adjoint riesz trsf into two terms} are $(p,2,\beta,\alpha\wedge \sigma)$-molecules for $H_{0}$ associated to the ball $B$.

We first check the molecular decay (condition (b) in \cite[Definition 3.7]{BJL_T1_Schrodinger}), and start with term $\RNum{2}.$ We need to treat two cases depending on the cancellation property of the atom $a$. We first assume that $\frac{\rho(x_B)}{4}\leq r_{B}\leq \rho(x_B)$. Let $j\geq 2$, and consider $x\in C_{j}(B)$. Then, for $y\in B$, we have $\rho(y)\eqsim \rho(x_B)$ and $\abs{x-y}>\abs{x-x_B}(1-2^{-j})> \frac{\abs{x-x_B}}{2}$. It follows from \eqref{generalised calderon zygmund estimates for kernel of H_{0}^{-1/2}} that for all $N>0$ it holds that  
\begin{align*}
    \abs{(V^{1/2})_{B}(\invH a)(x)}\lesssim(V^{1/2})_{B}\meas{B}^{1-\frac{1}{p}}\abs{x-x_B}^{-(n-1)}\left(1+\frac{\abs{x-x_B}}{\rho(x_B)}\right)^{-N}
\end{align*}
for almost every $x\in\Rn\setminus 4B$. Consequently, for $j\geq 2$ it holds that  
\begin{align}
    \norm{(V^{1/2})_{B}(\invH a)}_{\El{2}(C_{j}(B))}&\lesssim (V^{1/2})_{B}\meas{B}^{1-\frac{1}{p}}\rho(x_B)^{N}\left(\int_{2^{j}r_B\leq \abs{z}<2^{j+1}r_B} \abs{z}^{-2(n-1)-2N}\dd z\right)^{1/2}\nonumber\\
    &\lesssim\left(\frac{r_{B}}{\rho(x_B)}\right)^{\frac{\delta}{2}-N}2^{-j(n-\frac{n}{p}+N-1)}\meas{2^{j}B}^{\frac{1}{2}-\frac{1}{p}}.\label{universal estimate on atoms}
\end{align}
Taking $N=2$, we get the required molecular decay since $n+1-\frac{n}{p}>0$ and $r_{B}\eqsim \rho(x_B)$. Next, if $r_B\leq \frac{\rho(x_B)}{4}$, we have the cancellation $\int_{\Rn} a(x)\dd x=0$. This allows to write
\begin{align*}
    (\invH a)(x)=\int_{B}K(x,y)a(y)\dd y=\int_{B} (K(x,y)-K(x,x_B))a(y)\dd y
\end{align*}
for almost every $x\in \Rn\setminus 4B$. If $x\in C_{j}(B)=2^{j+1}B\setminus 2^{j}B$ for $j\geq 2$ and $y\in B$, then $\abs{x-y}>2\abs{x_B-y}$. Since $\eta_{0}\in(0,1)$, the kernel bounds \eqref{generalised calderon zygmund estimates for kernel of H_{0}^{-1/2}} give
\begin{align*}
    \abs{(\invH a)(x)}&\lesssim \int_{B}\frac{\abs{y-x_B}^{\eta_{0}}}{\abs{x-x_B}^{n-1+\eta_{0}}}\abs{a(y)}\dd y\lesssim r_{B}^{\eta_{0}}\abs{x-x_B}^{-(n-1+\eta_{0})}\meas{B}^{1-\frac{1}{p}}
\end{align*}
for almost every $x\in\Rn\setminus 4B$. Integrating this over the annular region $C_{j}(B)$, we obtain
\begin{align}
    \norm{(V^{1/2})_{B}(\invH a)}_{\El{2}(C_{j}(B))}&\lesssim (V^{1/2})_{B}r_{B}^{\eta_{0}}\meas{B}^{1-\frac{1}{p}}\left(\int_{2^{j}r_B\leq\abs{x-x_B}\leq 2^{j+1}r_B} \!\!\!\!\!\!\!\abs{x-x_B}^{-2(n-1+\eta_{0})}\dd x\right)^{\!\!1/2}\nonumber\\
    &\lesssim \left(\frac{r_B}{\rho(x_B)}\right)^{\frac{\delta}{2}}2^{-j(n-\frac{n}{p}+\eta_{0}-1)}\meas{2^{j}B}^{\frac{1}{2}-\frac{1}{p}}.\label{estimate on atoms with cancellation}
\end{align}
Let us now consider the inequality \eqref{universal estimate on atoms} with $N=\frac{\delta}{2(1-\gamma)}$, and raise it to the power $1-\gamma$, and raise \eqref{estimate on atoms with cancellation} to the power $\gamma$. Multiplying the results, we obtain
\begin{align*}
    \norm{(V^{1/2})_{B}(\invH a)}_{\El{2}(C_{j}(B))}&\lesssim 2^{-j(n-\frac{n}{p}+\frac{\delta}{2}+\gamma\eta_{0}-1)}\meas{2^{j}B}^{\frac{1}{2}-\frac{1}{p}}.
\end{align*}
Since $n-\frac{n}{p}+\frac{\delta}{2}+\gamma\eta_{0}-1>0$ by our choice of $\gamma$ and $\eta_{0}$, this shows that term $\RNum{2}$ satisfies the molecular decay required for a $(p,2,\beta,\alpha\wedge \sigma)$-molecule, for some $\beta>0$.

For term $\RNum{1}$, we simply get that for almost every $x\in\Rn\setminus 4B$ and all $N>0$ it holds that, independently of any cancellation property of $a$,  
\begin{align*}
    \abs{\invH (\eta_{V}a)(x)}&\leq \int_{B}\abs{K(x,y)}\abs{\eta_{V}(y)}\abs{a(y)}\dd y\\
    &\lesssim \meas{B}^{1-1/p}r_{B}^{\alpha}m(x_B)^{1+\alpha}\abs{x-x_B}^{-(n-1)}\left(1+\frac{\abs{x-x_B}}{\rho(x_B)}\right)^{-N}.
\end{align*}
Consequently, we can estimate as above, for all $j\geq 2$,
\begin{align*}
    \norm{\invH (\eta_{V}a)}_{\El{2}(C_{j}(B))}&\lesssim \meas{B}^{1-1/p}r_{B}^{\alpha}m(x_B)^{1+\alpha}\rho(x_B)^{N}(2^{j}r_{B})^{-\frac{n}{2}+1-N}\\
    &\lesssim \left(\frac{r_B}{\rho(x_B)}\right)^{\alpha +1-N}2^{j(1-N-n+\frac{n}{p})}(2^{j}r_B)^{\frac{n}{2}-\frac{n}{p}}\lesssim 2^{-j(n+\alpha-\frac{n}{p})}\meas{2^{j}B}^{\frac{1}{2}-\frac{1}{p}},
\end{align*}
by choosing $N=1+\alpha.$ Since $n+\alpha-\frac{n}{p}>0$ by assumption, this shows that term $\RNum{1}$ satisfies the required molecular decay for a $(p,2,\beta,\alpha\wedge \sigma)$-molecule, for some $\beta>0$. This means that we get the required molecular decay for $\mathcal{R}_{\mu}^{*}a$ on $C_{j}(B)$ for $j\geq 2$. For $j=1$, we can simply use $\El{2}$-boundedness of $\mathcal{R}_{\mu}^{*}$ to get $\norm{\mathcal{R}_{\mu}^{*}a}_{\El{2}(4B)}\lesssim \norm{a}_{\El{2}}\lesssim \meas{B}^{\frac{1}{2}-\frac{1}{p}}$.

It remains to show that $\mathcal{R}_{\mu}^{*}a$ satisfies the global cancellation estimate for a $(p,2,\beta,\alpha\wedge \sigma)$-molecule (condition (c) in \cite[Definition 3.7]{BJL_T1_Schrodinger}). First, using the support property of $a$, an application of \eqref{duality property of H_0^{-1/2}1} yields
\begin{align*}
    \abs{\int_{\Rn}(\invH (\eta_{V}a))(x)\dd x}&=\abs{\langle \eta_{V}a , \invH 1\rangle_{\El{2}}}\leq \norm{\eta_{V}a}_{\El{1}(B)}\norm{\invH 1}_{\El{\infty}(B)}\\
    &\lesssim \meas{B}^{1-\frac{1}{p}}r_{B}^{\alpha}m(x_B)^{1+\alpha}\rho(x_B)=\left(\frac{r_B}{\rho(x_B)}\right)^{\alpha}\meas{B}^{1-\frac{1}{p}}.
\end{align*}
Here, we have used the fact that $\abs{(\invH 1)(x)}\lesssim \rho(x_{B})$ for almost every $x\in B$. This can be seen as follows. For $R:=\rho(x_{B})$, $N>1$ and a.e. $x\in B$, we have
\begin{align*}
    \abs{(\invH 1)(x)}&\leq \int_{B(x_{B},2R)}\abs{K(x,y)}\dd y + \int_{\Rn\setminus B(x_{B},2R)}\abs{K(x,y)}\dd y\\
    &\lesssim \int_{\abs{x-y}<R}\abs{x-y}^{-(n-1)}\dd y + \int_{\abs{x-y}\geq 3R}\abs{x-y}^{-(n-1)}\left(1+\frac{\abs{x-y}}{\rho(x)}\right)^{-N}\dd y\\
    &\lesssim R + \rho(x)^{N}\int_{\abs{x-y}\geq R} \abs{x-y}^{-(n-1)-N}\dd y\lesssim R +\rho(x)^{N}R^{-N+1}\lesssim \rho(x_{B}),
\end{align*}
since $\rho(x)\eqsim \rho(x_B)$ by \eqref{comparability property for critical radius function}. For the remaining term, we use \eqref{duality property of H_0^{-1/2}1} again to get
\begin{align*}
    \abs{\int_{\Rn}\invH a}
    &\leq \abs{\langle a, \invH 1 - (\invH 1)_{B} \rangle_{\El{2}}} + \abs{\langle a, (\invH 1)_{B}\rangle_{\El{2}}}\\
    &\leq \norm{a}_{\El{\infty}(B)}\norm{\invH 1 - (\invH 1)_{B}}_{\El{1}(B)} + \abs{(\invH 1)_{B}}\abs{\int_{\Rn}a}.
\end{align*}
A slight modification of the arguments used in the proof of \cite[Proposition 4.14]{Ma_etAl} (and, by extension, of the arguments of the proof of \cite[Proposition 4.6]{Ma_etAl}) show that for $0<\eta< 1\wedge\delta=1$ it holds that 
\begin{equation*}
    \frac{1}{\meas{B}^{1+\frac{1}{n}}}\int_{B}\abs{\invH (1) -(\invH 1)_{B}}\dd y\lesssim \left(\frac{r_B}{\rho(x_B)}\right)^{\eta -1}.
\end{equation*}
We point out that there seems to be an inaccuracy in the proof of this estimate in \cite{Ma_etAl}, which is due to a typo in the statement of \cite[Lemma 4.4]{Ma_etAl}. Indeed, “every $0<\delta<\delta_{0}$" should be “every $0<\delta < 1\wedge \delta_{0}$". This prevents the $t$-integral in the penultimate display of the proof of \cite[Proposition 4.14]{Ma_etAl} to converge. The issue can be fixed by splitting the integral with respect to $t$ in (4.29) of  \cite{Ma_etAl} into three parts: $0<t\leq 4r_{B}^{2}$, $4r_{B}^{2}<t\leq 4\rho(x_{B})^{2}$, and $t>4\rho(x_{B})^{2}$, similarly to what is done in the proof of \cite[Proposition 4.6]{Ma_etAl}. One must then use the full decay available in $t$ from \cite[Lemma 4.4]{Ma_etAl}, to treat the part $t>4\rho(x_{B})^{2}$. This can be achieved by picking $N=1$.

Moreover, $\int_{\Rn}a \neq 0$ only if $\frac{\rho(x_B)}{4}\leq r_B\leq \rho(x_B)$, in which case $\abs{(\invH 1)_{B}}\abs{\int_{\Rn}a}\lesssim \rho(x_B)\meas{B}^{1-\frac{1}{p}}$.
Consequently, if $r_{B}<\frac{\rho(x_B)}{4}$, we obtain 
\begin{align*}
    \abs{\int_{\Rn}(V^{1/2})_{B}\invH a}&\lesssim (V^{1/2})_{B}\meas{B}^{-1/p}\meas{B}^{1+\frac{1}{n}}\left(\frac{r_B}{\rho(x_B)}\right)^{\eta-1}\lesssim \meas{B}^{1-\frac{1}{p}}\left(\frac{r_B}{\rho(x_B)}\right)^{\eta+\frac{\delta}{2}-1}.
\end{align*}
On the other hand, if $\frac{\rho(x_B)}{4}\leq r_B\leq \rho(x_B)$, then we also have 
\begin{align*}
    \abs{\int_{\Rn}(V^{1/2})_{B}\invH a}&\lesssim \meas{B}^{1-\frac{1}{p}}\left(\frac{r_B}{\rho(x_B)}\right)^{\eta+\frac{\delta}{2}-1} + r_{B}^{-1}\left(\frac{r_B}{\rho(x_B)}\right)^{\delta/2}\rho(x_B)\meas{B}^{1-\frac{1}{p}}\\
    &\lesssim \meas{B}^{1-\frac{1}{p}}\left(\frac{r_B}{\rho(x_B)}\right)^{\eta+\frac{\delta}{2}-1}.
\end{align*}
We make the choice $\eta:=\sigma+1-\frac{\delta}{2}=\frac{\delta}{2}\in (0,1)$ to obtain the cancellation bound
\begin{align*}
    \abs{\int_{\Rn}\mathcal{R}_{\mu}^{*}a}&\lesssim \meas{B}^{1-\frac{1}{p}}\left[\left(\frac{r_B}{\rho(x_B)}\right)^{\alpha} + \left(\frac{r_B}{\rho(x_B)}\right)^{\sigma}\right]\lesssim \meas{B}^{1-\frac{1}{p}}\left(\frac{r_B}{\rho(x_B)}\right)^{\alpha\wedge \sigma}.
\end{align*}
This shows that $\mathcal{R}_{\mu}^{*}a$ satisfies the global cancellation estimate for a $(p,2,\beta,\alpha\wedge \sigma)$-molecule. It follows from (the proof of) \cite[Lemma 3.9]{BJL_T1_Schrodinger} that $\mathcal{R}_{\mu}^{*}a\in\textup{H}^{p}_{V}$ with $\norm{\mathcal{R}_{\mu}^{*}a}_{\textup{H}^{p}_{V}}\lesssim 1$. This concludes the proof by Remark \ref{remark about estimate atoms implies HpV estimate}.
\end{proof}

\subsubsection{The $S^{\alpha}$ class}\label{subsubsection on the S alpha smoothness class of potentials}
In this section we describe a collection of potentials $V$ that lie in the intersection $S^{\alpha}\cap\textup{RH}^{q}(\Rn)$ for some $\alpha\in (0,1]$ when $q>n\geq 1$. We recall that the class $S^{\alpha}$ was defined in \eqref{S alpha smoothness condition}.

Let $P:\Rn\to\C$ be a polynomial, and let $a\in(0,2]$. Note that $V:=\abs{P}^{a}\in\textup{RH}^{\infty}(\Rn)$ by Proposition \ref{polynomials are RH^{infty}}. We claim that $V\in S^{a/2}$, so $\left(\frac{n}{n+a/2},1\right]\subseteq \mathcal{I}(V)$ by Theorem~\ref{full boundedness criteria for adjoint riesz transforms on hardy dziubanski spaces}.
To see this, we first note that \eqref{equivalence critical radius function modulus polynomial} implies that $\abs{\nabla P(x)}\lesssim m(x,V)^{\frac{a+2}{a}}$ for all $x\in\Rn$. Also, note that $\abs{t^{p}-s^{p}}\leq \abs{t-s}^{p}$ for all $t,s\geq 0$ and $p\in(0,1]$. Let now $B=B(x_B,r_B)\subset\Rn$ be such that $r_B\leq \rho(x_B,V)$. Then, for all $x,y\in B$, the mean-value theorem and \eqref{comparability property for critical radius function} imply that there is a point $\xi\in B$ such that 
\begin{align*}
    \abs{V(x)^{1/2}-V(y)^{1/2}}&=\abs{\abs{P(x)}^{a/2}-\abs{P(y)}^{a/2}}\leq \abs{\abs{P(x)}-\abs{P(y)}}^{a/2}
    \leq \abs{P(x)-P(y)}^{a/2}\\
    &\leq \abs{x-y}^{a/2}\abs{\nabla P(\xi)}^{a/2}\lesssim r_{B}^{a/2}m(\xi, V)^{1+a/2}\eqsim r_{B}^{a/2}m(x_{B},V)^{1+a/2},
\end{align*}
which proves the claim.

\section{Riesz transform bounds and Theorem~\ref{thm boundedness of Riesz transforms on Lp full range}}\label{section on riesz transform bounds}
In this section, we consider $\El{p}$-boundedness of the Riesz transforms associated with the Schrödinger operator $H$. This is an important step towards the identification of the Hardy spaces $\bb{H}^{p}_{H}$ and $\bb{H}^{1,p}_{H}$ in the next section. The starting point for the definition of the Riesz transform is the Kato square root result obtained by Morris--Turner \cite{MorrisTurner}. They have shown that if $n\geq 3$, $q\geq \max\set{\frac{n}{2},2}$ and $V\in\textup{RH}^{q}(\Rn)$, then $\dom{H^{1/2}}=\mathcal{V}^{1,2}(\Rn)$ with the estimate
\begin{equation}\label{Kato estimate}
    \norm{H^{1/2}f}_{\El{2}}\eqsim \norm{\nabla_{\mu}f}_{\El{2}}=\norm{f}_{\dot{\mathcal{V}}^{1,2}}
\end{equation}
for all $f\in\mathcal{V}^{1,2}(\Rn).$ By density of $\mathcal{V}^{1,2}(\Rn)$ in $\dot{\mathcal{V}}^{1,2}(\Rn)$ (see Proposition~\ref{density lemma in homogeneous V adapted Sobolev spaces}), the operator $H^{1/2}$ has a bounded extension $\dot{H}^{1/2}:\dot{\mathcal{V}}^{1,2}(\Rn)\to \Ell{2}$ which still satisfies \eqref{Kato estimate}. Since $\clos{\ran{H}}=\Ell{2}$ by \eqref{identity closure of ranges of schrodinger operators}, the argument outlined at the start of Section~\ref{subsection on reverse riesz transform bounds on Dziubanski hardy space for no coefficients case} shows that this extension is an isomorphism. 
The inverse of this isomorphism shall be denoted $\dot{H}^{-1/2}: \Ell{2}\to\dot{\mathcal{V}}^{1,2}(\Rn)$. Consequently, the \emph{Riesz transform} $R_{H}:= \nabla_{\mu}\dot{H}^{-1/2}:\Ell{2}\to\El{2}\left(\Rn;\bb{C}^{n+1}\right)$ is a bounded mapping, satisfying $\norm{R_{H} f}_{\El{2}}\eqsim \norm{f}_{\El{2}}$ for all $f\in \Ell{2}$. Let us also note that $R_{H}=\nabla_{\mu}H^{-1/2}$ on $\ran{H^{1/2}}$.

In the rest of this section, we shall obtain boundedness of the operator $R_{H}$ on $\El{p}$ for a range of exponents $p\neq 2$ by adapting the method that was used in \cite[Chapter 7]{AuscherEgert}.

\subsection{A conservation property for resolvents of Schrödinger operators}
This section describes the analogue of the conservation property \cite[Corollary 5.4]{AuscherEgert} that holds in the presence of the singular potential $V$. The main result in this direction is Lemma~\ref{cancellation lemma} below. We shall first need the following elementary lemma. It is recorded at the start of the proof of \cite[Corollary 5.4]{AuscherEgert}, but since we will be using it several times, we include some details for the interested reader.

\begin{lem}\label{off-diag + cpct supp ==> integrable}
    Let $\Omega\subset\C\setminus\set{0}$, let $W_1$ and $W_2$ be finite dimensional Hilbert spaces, and let $\set{T(z)}_{z\in\Omega}$ be a family of linear operators mapping $\El{2}(\Rn;W_1)$ into $\El{2}(\Rn;W_2)$ that satisfies $\El{2}$ off-diagonal estimates of order $\gamma >\frac{n}{2}$. Then for all $f\in\El{2}_{c}(\Rn;W_1)$, it holds that  $T(z)f\in\El{1}(\Rn;W_2)$, locally uniformly in $z\in\Omega$.
\end{lem}
\begin{proof}
    Let $f$ be supported in a ball $B\subset \Rn$ of radius $r>0$. Using Hölder's inequality and the $\El{2}$ off-diagonal estimates satisfied by $\set{T(z)}_{z\in\Omega}$, we obtain for every $j\geq 2$ that 
    \begin{align*}
        \norm{T(z)f}_{\El{1}(C_{j}(B))}\leq \meas{C_{j}(B)}^{\frac{1}{2}}\norm{T(z)f}_{\El{2}(C_{j}(B))}\lesssim \abs{z}^{\gamma}r^{\frac{n}{2}-\gamma}2^{(\frac{n}{2}-\gamma)j}\norm{f}_{\El{2}}.
    \end{align*}
    We get $\norm{T(z)f}_{\El{1}(C_{1}(B))}\leq \meas{4B}^{1/2}\norm{T(z)f}_{\El{2}(4B)}\lesssim r^{n/2}\norm{f}_{\El{2}}$ by Hölder's inequality and the $\El{2}$-bounds for $T(z)$.
    Finally, since $\gamma> \frac{n}{2}$ we obtain
    \begin{align*}
        \norm{T(z)f}_{\El{1}}&\leq \sum_{j=1}^{\infty}\norm{T(z)f}_{\El{1}(C_{j}(B)))}\lesssim r^{n/2}\norm{f}_{\El{2}} + \left(\sum_{j=2}^{\infty}2^{(\frac{n}{2}-\gamma)j}\right)r^{\frac{n}{2}-\gamma}\abs{z}^{\gamma}\norm{f}_{\El{2}} <\infty.\qedhere
    \end{align*}
\end{proof}
We can now state and prove the main result of this section. This is to be compared with \cite[Corollary 5.4]{AuscherEgert}. We note that it holds for general $n\geq 1$ and non-negative $V\in\El{1}_{\loc}(\Rn)$.
\begin{lem}\label{cancellation lemma}
    If $f\in \El{2}_{c}(\Rn)$, then for every integer $k\geq 1$ and every $t>0$, it holds that both $H(1+t^{2}H)^{-k}(f)$ and $V(1+t^{2}H)^{-k}(f)$ belong to $\Ell{1}$, and 
    \begin{equation}\label{cancellation identity}
        \int_{\Rn}A_{\perp\perp}H(1+t^{2}H)^{-k}(f) = \int_{\Rn}aV(1+t^{2}H)^{-k}(f).
    \end{equation}
\end{lem}
\begin{proof}
    For every $t>0$ and every integer $k\geq 1$ we have 
    \begin{equation}\label{functional resolvent identity for telescopic}
        t^{2}H(1+t^{2}H)^{-k}=(1+t^{2}H)^{-(k-1)}-(1+t^{2}H)^{-k},
    \end{equation}
    with the convention that $(1+t^{2}H)^{0}:= 1$. By composition (when $k\geq 2$), we know that the family $\{(1+t^{2}H)^{-k}\}_{t>0}$ satisfies $\El{2}$ off-diagonal estimates of arbitrarily large order, for any $k\geq 1$. It follows from Lemma~\ref{off-diag + cpct supp ==> integrable} that $(1+t^{2}H)^{-k}(f)\in\Ell{1}$ for all $t>0$. Since also $f\in\Ell{1}$, we conclude from \eqref{functional resolvent identity for telescopic} that $t^{2}H(1+t^{2}H)^{-k}(f)\in\Ell{1}$ for all $k\geq 1$ and $t>0$.

    To prove that $V(1+t^{2}H)^{-k}(f)\in\Ell{1}$, we pick some ball $B\subset \Rn$ of radius $r>0$ such that $\supp{f}\subseteq B$, and we write
    \begin{align}
        \int_{\Rn} \abs{t^2 V(1+t^{2}H)^{-k}(f)} &= \sum_{j=1}^{\infty} \int_{C_{j}(B)}\abs{tV^{\frac{1}{2}}}\abs{t V^{\frac{1}{2}}(1+t^{2}H)^{-k}(f)}\nonumber\\
        &\leq \sum_{j=1}^{\infty} \norm{t V^{\frac{1}{2}}}_{\El{2}(C_{j}(B))}\norm{t \nabla_{\mu}(1+t^{2}H)^{-k}(f)}_{\El{2}(C_{j}(B))}.\label{finiteness of integral as a sum}
    \end{align}
    By the doubling property of $V$, we get $\norm{ V^{\frac{1}{2}}}_{\El{2}(C_{j}(B))}\lesssim \kappa^{\frac{j}{2}}\left(\int_{B}V\right)^{\frac{1}{2}}$ for each $j\geq 1$, where $\kappa$ denotes the doubling constant of Lemma~\ref{self-improvement property of reverse holder weights}.
    By composition (when $k\geq 2$) we know that the family $\{t \nabla_{\mu}(1+t^{2}H)^{-k}\}_{t>0}$ satisfies $\El{2}$ off-diagonal estimates of arbitrarily large order. Therefore, for $j\geq 2$ and all $\gamma >0$ we can write 
    \begin{align*}
        \norm{t \nabla_{\mu}(1+t^{2}H)^{-k}(f)}_{\El{2}(C_{j}(B))}\lesssim \left(1+\frac{\dist(B,C_{j}(B))}{t}\right)^{-\gamma}\norm{f}_{\El{2}}\lesssim \left(\frac{t}{r}\right)^{\gamma}2^{-j\gamma}\norm{f}_{\El{2}},
    \end{align*}
    while for $j=1$ we have the uniform $\El{2}$ estimate $\norm{t \nabla_{\mu}(1+t^{2}H)^{-k}(f)}_{\El{2}(C_{1}(B))}\lesssim \norm{f}_{\El{2}}$.
    Now it is only a matter of choosing $\gamma$ large enough so that $\sum_{j=1}^{\infty}2^{-j\gamma}\kappa^{\frac{j}{2}} <\infty$ to prove that the sum in (\ref{finiteness of integral as a sum}) is indeed finite.
    Let us finally prove the identity (\ref{cancellation identity}). Let $\{\chi_{r}\}_{r>0}$ be a family of smooth, compactly supported and $\left[0,1\right]$-valued functions such that $\chi_{r}(x)=1$ for all $\abs{x}\leq r$, $\supp{\chi_{r}}\subseteq \ball(0,2r)$ and such that $\norm{\nabla_{x}\chi_{r}}_{\El{\infty}}\lesssim r^{-1}$.
    Then, by dominated convergence we can write
    \begin{equation*}
        \int_{\Rn}A_{\perp\perp}H(1+t^{2}H)^{-k}(f) = \lim_{r\to\infty }\int_{\Rn} \chi_{r}A_{\perp\perp}H(1+t^{2}H)^{-k}(f) .
    \end{equation*}
    Let $r>1$ be fixed. Since $(1+t^{2}H)^{-k}(f)\in \dom{H}$ and $\chi_{r}\in C_{c}^{\infty}(\Rn)\subseteq \mathcal{V}^{1,2}(\Rn)$, we can use the form definition of $H$ to write
    \begin{align*}
        \int_{\Rn} \chi_{r}A_{\perp\perp}H(1+t^{2}H)^{-k}(f)&=\langle A_{\parallel\parallel}\nabla_{x}(1+t^{2}H)^{-k}(f), \nabla_{x}\chi_{r}\rangle_{\El{2}} + \langle aV(1+t^{2}H)^{-k}(f) , \chi_{r}\rangle_{\El{2}}.
    \end{align*}
    Now, using the off-diagonal estimates for the family $\{t\nabla_{\mu}(1+t^{2}H)^{-k}(f)\}_{t>0}$, the boundedness of $A_{\parallel\parallel}$, and the fact that $f\in\Ell{2}$ has compact support, we deduce from Lemma~\ref{off-diag + cpct supp ==> integrable} that $A_{\parallel\parallel}\nabla_{x}(1+t^{2}H)^{-k}(f)\in\Ell{1}$. 
    Since $\lim_{r\to\infty}\nabla_{x}\chi_{r}(x)=0$ for all $x\in\Rn$, and since $\norm{\nabla_{x}\chi_{r}}_{\El{\infty}}\lesssim r^{-1}\leq 1$, we can deduce by dominated convergence that 
    \begin{equation*}
        \lim_{r\to\infty} \langle A_{\parallel\parallel}\nabla_{x}(1+t^{2}H)^{-k}(f), \nabla_{x}\chi_{r}\rangle_{\El{2}} =0.
    \end{equation*}
    Similarly, since $aV(1+t^{2}H)^{-k}(f)\in\El{1},$ we obtain by dominated convergence that 
    \begin{equation*}
        \lim_{r\to\infty}\langle aV(1+t^{2}H)^{-k}(f) , \chi_{r}\rangle_{\El{2}} = \int_{\Rn}aV(1+t^{2}H)^{-k}(f).\qedhere
    \end{equation*}
\end{proof}
\subsection{Riesz transforms bounds on $\El{p}$ -- Proof of Theorem~\ref{thm boundedness of Riesz transforms on Lp full range}}\label{subsection on the proof of Lp bounds for riesz trsform}
Let us now turn to the boundedness properties of $R_{H}$. In this section, we let $n\geq 3$, $q\geq \max\set{\frac{n}{2},2}$ and $V\in\textup{RH}^{q}(\Rn)$. We first note that the arguments used in Sections 7.3 and 7.4 of \cite{AuscherEgert} can be used verbatim (upon replacing $L$ with $H$ and $\nabla_{x}$ with $\nabla_{\mu}$) to obtain the following result.
\begin{lem}\label{converse implication for boundedness of riesz transform and critical numbers}
   If $R_{H}$ is $\El{p}$-bounded for some $p\in (1,\infty)$, then $p\in[p_{-}(H), q_{+}(H)]$.
\end{lem}

We shall now prove Theorem~\ref{thm boundedness of Riesz transforms on Lp full range}. Let $\alpha\geq 1$ be an integer. We obtain a singular integral representation of $R_H$ in terms of powers of resolvents by following the approach that leads to equation~(7.3) in \cite{AuscherEgert}. In fact, by applying the Calderón--McIntosh reproducing formula to the injective sectorial operator $H$ with the auxiliary function $\psi_{\alpha}(z)=z^{3\alpha +1/2}(1+z)^{-9\alpha}$, we obtain that the truncated Riesz transforms
\begin{align}\label{truncated riesz trf def}
R_{H}^{\eps}:=
c_{\alpha}^{-1}\int_{\eps}^{\frac{1}{\eps}}t\nabla_{\mu}\left(1+t^{2}H\right)^{-3\alpha}\left(t^{2}H\right)^{3\alpha}\left(1+t^{2}H\right)^{-6\alpha}\frac{\dd t}{t}\in\mathcal{L}(\El{2}),
\end{align}
where $\eps>0$ and $c_{\alpha}=\int_{0}^{\infty}\psi_{\alpha}(t^{2})\frac{\dd t}{t}>0$, give the representation $R_{H}f=\lim_{\eps\to 0}R_{H}^{\eps}f$, with convergence in $\El{2}\left(\Rn;\bb{C}^{n+1}\right)$ for all $f\in\El{2}$.

Before starting the proof of Theorem~\ref{thm boundedness of Riesz transforms on Lp full range} we need two lemmas. The first is the analogue of \cite[Lemma 7.4]{AuscherEgert} in our setting. We note that it holds for general non-negative $V\in\El{1}_{\loc}(\Rn)$.
\begin{lem}\label{Obtaining L^2-L^q or L^p-L^2 odes for powers of resolvents}
    Let $p\in\mathcal{J}(H)\cap(1,\infty)$ and let $r$ be between $p$ and $2$ (i.e. either $r\in(p,2)$ or $r\in(2,p)$). There exists an integer $\beta\geq 1$ (depending on $p$, $r$ and $n$) with the following properties:
    \begin{enumerate}[label=\emph{(\roman*)}]
        \item If $p\leq 2$, then $\set{(1+t^{2}H)^{-\beta}}_{t>0}$ and $\set{t\nabla_{\mu}(1+t^{2}H)^{-\beta -1}}_{t>0}$ satisfy $\El{r}-\El{2}$ off-diagonal estimates of arbitrarily large order.
        \item If $p\geq 2$, then $\set{(1+t^{2}H)^{-\beta}}_{t>0}$ satisfies $\El{2}-\El{r}$ off-diagonal estimates of arbitrarily large order. 
    \end{enumerate}
    \end{lem}
    \begin{proof}
        This is proved as in \cite[Lemma 7.4]{AuscherEgert}, using Proposition~\ref{bound on power of family} instead of \cite[Lemma 4.4]{AuscherEgert} and Lemma \ref{property J(H)} instead of \cite[Lemma 6.3]{AuscherEgert}.
    \end{proof}
The following lemma will be crucial in the proof of Theorem~\ref{thm boundedness of Riesz transforms on Lp full range}.
\begin{lem}\label{ode resolvents lemma for riesz transform bounds}
    Let $p_{0}\in(2,q_{+}(H))$. There exists an integer $\omega\geq 2$ such that the following properties hold for all integers $k\geq 1$ and $m\geq \omega$:
    \begin{enumerate}[label=\emph{(\roman*)}]
    \item $\set{(1+t^{2}H)^{-(m-1)}}_{t>0}$ satisfies $\El{2}-\El{p_{0}}$ off-diagonal estimates of arbitrarily large order;
    \item  $\set{t\nabla_{\mu}(1+t^{2}H)^{-k}}_{t>0}$ satisfies $\El{p_{0}}$ and $\El{r}-\El{p_0}$ off-diagonal estimates of arbitrarily large order for all $r\in({p_{0}}_{*},p_{0}]$;
    \item  $\set{t\nabla_{\mu}(1+t^{2}H)^{-m}}_{t>0}$ satisfies $\El{2}-\El{p_{0}}$ off-diagonal estimates of arbitrarily large order.
    \end{enumerate}
\end{lem}
\begin{proof}
    It follows from Proposition~\ref{properties of critical numbers proposition} that $p_0\in (2,p_{+}(H))$, hence Lemma~\ref{Obtaining L^2-L^q or L^p-L^2 odes for powers of resolvents} shows that there exists an integer $\omega\geq 2$ such that the family $\{(1+t^{2}H)^{-(\omega -1)}\}_{t>0}$ satisfies $\El{2}-\El{p_0}$ off-diagonal estimates of arbitrarily large order. 
    Moreover, interpolation of the $\El{2}$ off-diagonal estimates for $\{(1+t^{2}H)^{-1}\}_{t>0}$ (recall the start of Section~\ref{subsubsection on mapping properties ofa dapted hardy spaces}) with its $\El{s}$-boundedness for some $p_{0}<s<p_{+}(H)$ implies that $\{(1+t^{2}H)^{-1}\}_{t>0}$ has $\El{p_0}$ off-diagonal estimates of arbitrarily large order (see \cite[Lemma 4.14]{AuscherEgert}). By composition (see \cite[Lemma 4.6]{AuscherEgert}), we obtain (i). 
    
    Similarly, interpolation of the $\El{2}$ off-diagonal estimates for $\{t\nabla_{\mu}(1+t^{2}H)^{-1}\}_{t>0}$ with its $\El{s}$-boundedness for some $p_{0}<s<q_{+}(H)$ implies that $\{t\nabla_{\mu}(1+t^{2}H)^{-1}\}_{t>0}$ has $\El{p_0}$ off-diagonal estimates of arbitrarily large order.
 By composition with those satisfied by $\{(1+t^{2}H)^{-1}\}_{t>0}$, we obtain the first part of (ii). 
 
Combining (by composition) (i) with the $\El{p_0}$ off-diagonal estimates of arbitrarily large order for the family $\{t\nabla_{\mu}(1+t^{2}H)^{-1}\}_{t>0}$ yields (iii).

Finally, we prove the second part of (ii). Let $m\geq \omega$ be arbitrary. It follows from (iii) and the first part of (ii) that both $\{t\nabla_{\mu}(1+t^{2}H)^{-m}\}_{t>0}$ and $\{t\nabla_{\mu}(1+t^{2}H)^{-m-1}\}_{t>0}$ are $\El{2}-\El{p_0}$-bounded and $\El{p_0}$-bounded. We now treat two different cases: 
\begin{enumerate}
    \item{
First, if $p_0\geq 2^{*}$, then we have $2\leq p_{0_{*}}<p_{0}$. By interpolation (see \cite[Lemma 4.13]{AuscherEgert}) of the $\El{2}- \El{p_0}$-boundedness of the family $\{t\nabla_{\mu}(1+t^{2}H)^{-m-1}\}_{t>0}$ with its $\El{p_0}$-boundedness, we deduce that this family is $\El{p_{0_{*}}}- \El{p_0}$-bounded. }

\item{If, on the contrary, $2<p_0<2^{*}$, then we have $2_{*}<p_{0_{*}}<2.$ By Lemma~\ref{property J(H)}, we deduce that $\{(1+t^{2}H)^{-1}\}_{t>0}$ is $\El{p_{0_{*}}}-\El{2}$-bounded. Then, by composition with the $\El{2}- \El{p_0}$-boundedness of the family $\{t\nabla_{\mu}(1+t^{2}H)^{-m}\}_{t>0}$, we obtain that the family $\{t\nabla_{\mu}(1+t^{2}H)^{-m-1}\}_{t>0}$ is $\El{p_{0_{*}}}- \El{p_0}$-bounded.}
\end{enumerate}
In conclusion, the family $\{t\nabla_{\mu}(1+t^{2}H)^{-m-1}\}_{t>0}$ is always $\El{p_{0_{*}}}- \El{p_0}$-bounded. By interpolation with the $\El{p_0}$-boundedness of this family, we deduce that $\{t\nabla_{\mu}(1+t^{2}H)^{-m-1}\}_{t>0}$ is $\El{r}-\El{p_0}$-bounded for all $r\in (p_{0_{*}},p_{0}]$.
 Now, note that for each $r\in (p_{0_{*}},p_{0}]$, we have $\frac{n}{n+1}<r\leq p_0 <\infty$ and $\frac{n}{r}-\frac{n}{p_0}<1$, hence Lemma~\ref{decreasing powers in boundedness of resolvents} shows that the family $\{t\nabla_{\mu}(1+t^{2}H)^{-k}\}_{t>0}$ is $\El{r}-\El{p_0}$-bounded for all integers $1\leq k\leq m+1$.
Finally, interpolating (see \cite[Lemma 4.14]{AuscherEgert}) with the $\El{p_0}-\El{p_0}$ off-diagonal estimates of arbitrarily large order satisfied by $\{t\nabla_{\mu}(1+t^{2}H)^{-k}\}_{t>0}$, we obtain the second part of (ii), and the proof is complete.
\end{proof}
We now have all the ingredients needed for the proof of Theorem~\ref{thm boundedness of Riesz transforms on Lp full range}.
\begin{proof}[Proof of Theorem~\ref{thm boundedness of Riesz transforms on Lp full range}]
The Riesz transform $R_H$ is $\El{p}$-bounded for all $p\in \left(p_{-}(H)\vee 1 , 2\right]$. This actually follows merely from the $\El{2}$-boundedness of $R_{H}$, and does not require any additional assumption on $V\in\El{1}_{\loc}(\Rn)$ other than non-negativity. This can be seen by following \cite[Section~7.1]{AuscherEgert} verbatim, replacing $L$ by $H$ and $\nabla_{x}$ by $\nabla_{\mu}$ appropriately. One should also use Lemma~\ref{Obtaining L^2-L^q or L^p-L^2 odes for powers of resolvents} instead of \cite[Lemma 7.4]{AuscherEgert}.

We now show that if $p\in\left(2,q_{+}(H)\right)$ and $p\leq 2q$, then $R_H$ is $\El{p}$-bounded. We proceed by adapting the proof in \cite[Section~7.2]{AuscherEgert}. There are several steps:
\begin{enumerate}[wide, labelwidth=!, labelindent=0pt, label={\textit{{Step} \arabic*.}}]
\item We start with some preliminaries. Let us first note that for integers $\alpha\geq 1$ we have
    \begin{align}
    \begin{split}\label{factored resolvents}
t\nabla_{\mu}\!\left(1+t^{2}H\right)^{\!-3\alpha}\!\!\left(t^{2}H\right)^{\!3\alpha}\!\left(1+t^{2}H\right)^{\!-6\alpha}
\!\!=t\nabla_{\mu}\left(1+t^{2}H\right)^{\!-1}\!\!\left(1-\!\left(1+t^{2}H\right)^{\!-1}\right)^{\!\!3\alpha}\!\!\!\left(1+t^{2}H\right)^{\!-6\alpha +1}.
\end{split}
    \end{align}
    Proposition~\ref{properties of critical numbers proposition} implies that $2<p<p^{*}<q_{+}(H)^{*}\leq p_{+}(H)$, 
    hence $p\in\mathcal{J}(H)$, and therefore both $\set{\left(1+t^{2}H\right)^{-1}}_{t>0}$ and $\set{t\nabla_{\mu}\left(1+t^{2}H\right)^{-1}}_{t>0}$ are $\El{p}$-bounded. The identity (\ref{factored resolvents}) above implies (by composition) that the family 
    \begin{equation}\label{kernel family Riesz tf}
\set{t\nabla_{\mu}\left(1+t^{2}H\right)^{-3\alpha}\left(t^{2}H\right)^{3\alpha}\left(1+t^{2}H\right)^{-6\alpha}}_{t>0}
    \end{equation}
    is $\El{p}$-bounded. As a consequence, $R_{H}^{\eps}$ is $\El{p}$-bounded with a bound depending on $\eps$. We will apply the Blunck--Kunstmann boundedness criterion stated in \cite[Proposition 7.6]{AuscherEgert} with $T=R_{H}^{\eps}$ for a parameter $\eps>0$, and a family $\set{A_{r}}_{r>0}\subset \mathcal{L}(\El{2})$, in order to obtain an $\El{p}$ bound which is independent of $\eps$. The parameter $\alpha\geq 1$ will have to be chosen large enough in the course of the proof.
    
    Using Lemma~\ref{self-improvement property of reverse holder weights}, we may pick $\beta>q$ such that $V\in\textup{RH}^{\beta}$. Then, we can find some $p_0\in (p,q_{+}(H))$ such that $p_{0}\leq 2\beta$. Since $\beta > \frac{n}{2}$, this automatically implies that $p_{0}<\beta^{*}$ (with the understanding that $\beta^{*}=\infty$ if $\beta\geq n$). Lemma~\ref{ode resolvents lemma for riesz transform bounds}.(i) shows that there exists an integer $\omega\geq 2$ such that the family $\{(1+t^{2}H)^{-(\omega -1)}\}_{t>0}$ satisfies $\El{2}- \El{p_0}$ off-diagonal estimates of arbitrarily large order. Let us assume from now on that $6\alpha\geq \omega $. 
Then, the identity (\ref{factored resolvents}) and the definition of $\omega$ imply (by composition and the choice of $\alpha$) that the family (\ref{kernel family Riesz tf}) is $\El{2}-\El{p_{0}}$-bounded. As a consequence, $R_{H}^{\eps}$ maps $\Ell{2}$ to $\Ell{p_{0}}$. 
    
Consider the function $\varphi(z):=\left(1-\left(1+z\right)^{-\omega}\right)^{3\alpha}$ for $z\in\C\setminus\set{-1}$. A binomial expansion shows that the approximating family $\set{A_{r}}_{r>0}:=\set{1-\varphi\left(r^{2}H\right)}_{r>0}\subset\mathcal{L}\left(\Ell{2}\right) $ satisfies
    \begin{equation}\label{binomial expansion of Ar}
        A_{r}=-\sum_{k=1}^{3\alpha}\binom{3\alpha}{k}(-1)^{k}\left(1+r^{2}H\right)^{-\omega k }.
    \end{equation}
    \vspace{6pt}
    \item Let us now check that both conditions in the Blunck--Kunstmann-type criterion \cite[Proposition 7.6]{AuscherEgert} are satisfied by the linear operator $R_{H}^{\eps}$ and the family $\set{A_{r}}_{r>0}$. Let $B\subset \Rn$ be a ball of radius $r>0$, $f\in\Ell{2}$ and $y\in B$. Let $\mathcal{M}$ denote the uncentred Hardy--Littlewood maximal operator for balls. These  conditions ask for a constant $C>0$, independent of $\eps$, $f$, $y$ and $B$, such that
    \begin{align}\label{first condition}
        \left(\dashint_{B}\abs{R_{H}^{\eps}\left(1-A_{r}\right)f}^{2}\right)^{1/2}\leq C\left(\hlmax{\abs{f}^{2}}\right)^{1/2}(y),
    \end{align}
and 
    \begin{equation}\label{second condition}
    \left(\dashint_{B}\abs{R_{H}^{\eps}A_r f}^{p_0}\right)^{1/p_0}\leq C \left(\hlmax{\abs{R_{H}^{\eps}f}^{2}}\right)^{1/2}(y).
\end{equation}
The proof of \eqref{first condition} is obtained by following Step 1 of \cite[Section~7.2]{AuscherEgert}, replacing $\nabla_{x}$ by $\nabla_{\mu}$ and $L$ by $H$.
The proof of \eqref{second condition} is based on the following claim: There are positive constants $\set{c_j}_{j\geq 1}$ such that for all $g\in\dot{\mathcal{V}}^{1,p_0}(\Rn)\cap\mathcal{V}^{1,2}(\Rn)$ it holds that 
\begin{align}\label{claim to be proven to obtain second condition for boundedness of riesz transform}
    \left(\dashint_{B}\abs{\nabla_{\mu}A_{r}g}^{p_0}\right)^{1/p_0}\lesssim \sum_{j\geq 1}c_{j}\left(\dashint_{2^{j+1}B}\abs{\nabla_{\mu}g}^{2}\right)^{1/2},
\end{align}
where $\sum_{j\geq 1}c_j <\infty$ and the implicit constant does not depend on the ball $B$. Assuming this for now (it will be proved in Step 3), for any $f\in \Ell{2}$ we can apply it with
\begin{equation*}
    g:=c_{\alpha}^{-1}\int_{\eps}^{1/\eps}t\left(1+t^{2}H\right)^{-3\alpha}\left(t^{2}H\right)^{3\alpha}\left(1+t^{2}H\right)^{-6\alpha}f \;\frac{\dd t}{t}
\end{equation*}
to obtain \eqref{second condition}. Indeed, by $\El{2}$-boundedness of $A_r$ and the definition of $R_{H}^{\eps}$, we obtain that $\nabla_{\mu}A_{r}g=R_{H}^{\eps}A_{r}f$, and that $\nabla_{\mu}g=R_{H}^{\eps}f\in \Ell{2}\cap\Ell{p_0}$, because $R_{H}^{\eps}$ maps $\El{2}$ to $\El{p_0}$, as observed in Step 1. Hence $g\in\dot{\mathcal{V}}^{1,p_0}\cap\mathcal{V}^{1,2}$, and the claim yields (\ref{second condition}). Altogether, this shows that both conditions in \cite[Proposition 7.6]{AuscherEgert} are satisfied, hence $\norm{R_{H}^{\eps}f}_{\El{p}}\lesssim \norm{f}_{\El{p}}$
for all $f\in\El{2}\cap\El{p}$, where the implicit constant doesn't depend on $\eps$. Since $R_{H}f=\lim_{\eps\to 0}R_{H}^{\eps}f$ in $\Ell{2}$, this concludes the proof.
\vspace{6pt}
\item We prove the claim (\ref{claim to be proven to obtain second condition for boundedness of riesz transform}). First, the expansion (\ref{binomial expansion of Ar}) shows that it suffices to prove the claim when $A_r$ is replaced by $\left(1+r^{2}H\right)^{-\omega k}$ for integers $1\leq k\leq 3\alpha$.
 It follows from Lemma~\ref{ode resolvents lemma for riesz transform bounds}.(ii) that the family $\{t\nabla_{\mu}(1+t^{2}H)^{-k}\}_{t>0}$ satisfies $\El{p_0}$ off-diagonal estimates of arbitrarily large order for all $k\geq 1$.
 By \cite[Proposition 5.1]{AuscherEgert}, for any integer $m\geq 1$, we can define $t\nabla_{\mu}(1+t^{2}H)^{-m}(1)$ as an element of $\El{p_0}_{\textup{loc}}(\Rn)$. To be clear, here we mean $1:=\ind{\Rn}$.
We claim that the following “cancellation" bound,
\begin{equation}\label{cancellation bound}
    \left(\dashint_{B}\abs{\nabla_{\mu}(1+r^{2}H)^{-m}(1)}^{p_0}\right)^{1/p_{0}}\lesssim \min\set{r^{-1},\left(\dashint_{B} V\right)^{\frac{1}{2}}},
\end{equation}
holds for all integers $m\geq \omega$, where the implicit constant does not depend on $B$. Note that this estimate is trivial when $V\equiv 0$.
Let us assume \eqref{cancellation bound} for now (it will be proved in Step 4) to conclude the proof. To this end, let $g\in\dot{\mathcal{V}}^{1,p_0}\cap\mathcal{V}^{1,2}$ and $1\leq k\leq 3\alpha$. We write
\begin{align*}
\left(\dashint_{B}\abs{\nabla_{\mu}\left(1+r^{2}H\right)^{-\omega k}g}^{p_0}\right)^{1/p_0}&\leq \left(\dashint_{B}\abs{\nabla_{\mu}\left(1+r^{2}H\right)^{-\omega k}\left(g-(g)_{B}\right)}^{p_0}\right)^{1/p_0}\\
    &+ \abs{(g)_B}\left(\dashint_{B}\abs{\nabla_{\mu}\left(1+r^{2}H\right)^{-\omega k}\left(1\right)}^{p_0}\right)^{1/p_0}=: \RNum{1} + \RNum{2}.
\end{align*}

For term $\RNum{1}$, we proceed as in \cite[Section~7.2]{AuscherEgert}, using the $\El{2}-\El{p_0}$ off-diagonal estimates for $\set{r\nabla_{\mu}\left(1+r^{2}H\right)^{-\omega k}}_{r>0}$ (see Lemma~\ref{ode resolvents lemma for riesz transform bounds}.(iii)) and Poincaré's inequality to obtain
\begin{align}\label{estimate first term}
    \RNum{1}=\left(\dashint_{B}\abs{\nabla_{\mu}\left(1+r^{2}H\right)^{-\omega k}\left(g-(g)_B\right)}^{p_0}\right)^{1/p_0}\lesssim \sum_{j=1}^{\infty}2^{-j\gamma}2^{\frac{3nj}{2}}\left(\dashint_{2^{j+1}B}\abs{\nabla_{x}g}^{2}\right)^{1/2}
\end{align}
for any $\gamma >0$. 

For term $\RNum{2}$, we first fix a cube $Q\subset\Rn$ such that $B\subseteq Q\subseteq \sqrt{n}B$. We then use the cancellation bound (\ref{cancellation bound}), Holder's inequality and the Fefferman--Phong inequality (Lemma~\ref{improved Fefferman--Phong inequality} with $p=2$) to obtain
\begin{align}\begin{split}\label{estimate of the second term using the cancellation bound and FP}
\abs{(g)_B}\left(\dashint_{B}\abs{\nabla_{\mu}\left(1+r^{2}H\right)^{-\omega k}\left(1\right)}^{p_0}\right)^{1/p_0}
    &\lesssim \abs{(g)_B}\min\set{r^{-1},\left(\dashint_{B} V\right)^{1/2}}\\
    &\lesssim \left(\dashint_{Q}\abs{g}^{2}\right)^{1/2}\min\set{\ell(Q)^{-1} , \left(\dashint_{Q}V\right)^{1/2}}\\
&\lesssim\left(\dashint_{Q}\abs{\nabla_{\mu}g}^{2}\right)^{1/2}\lesssim \left(\dashint_{2^{j_{0}+1}B}\abs{\nabla_{\mu}g}^{2}\right)^{1/2},
\end{split}
\end{align}
where the integer $j_{0}\geq 1$ is chosen so that $\sqrt{n}B\subseteq 2^{j_{0}+1}B$.  

We can now combine \eqref{estimate first term} and \eqref{estimate of the second term using the cancellation bound and FP} to obtain
\begin{align*}
\left(\dashint_{B}\abs{\nabla_{\mu}\left(1+r^{2}H\right)^{-\omega k}g}^{p_0}\right)^{1/p_0}
    \lesssim\sum_{j=1}^{\infty}2^{-j\gamma}2^{\frac{3nj}{2}}\left(\dashint_{2^{j+1}B}\abs{\nabla_{\mu}g}^{2}\right)^{1/2}.
\end{align*}
We choose $\gamma >\frac{3n}{2}$ to prove \eqref{claim to be proven to obtain second condition for boundedness of riesz transform}.
\vspace{6pt}
\item It remains to prove \eqref{cancellation bound}. We first prove, for all integers $m\geq\omega$, that 
\begin{equation}\label{elementary bound in cancellation bound}
    \left(\dashint_{B}\abs{\nabla_{\mu}(1+r^{2}H)^{-m}(1)}^{p_0}\right)^{1/{p_0}}\lesssim r^{-1}.
\end{equation}
By definition (recall \cite[Proposition 5.1]{AuscherEgert}), we have 
\begin{equation*}
    t\nabla_{\mu}(1+t^{2}H)^{-m}(1):=\sum_{j=1}^{\infty}t\nabla_{\mu}(1+t^{2}H)^{-m}\left(\ind{C_{j}(B)}\right),
\end{equation*}
with convergence in $\El{p_0}_{\textup{loc}}(\Rn)$. Therefore, using the $\El{2}-\El{p_0}$ off-diagonal estimates of (arbitrarily large) order $\gamma>\frac{n}{2}$ for the family $\{t\nabla_{\mu}(1+t^{2}H)^{-m}\}_{t>0}$, we obtain
\begin{align*}
    \norm{\nabla_{\mu}(1+r^{2}H)^{-m}(1)}_{\El{p_0}(B)}&= r^{-1}\norm{r\nabla_{\mu}(1+r^{2}H)^{-m}(1)}_{\El{p_0}(B)}\\
    &\leq r^{-1}\sum_{j=1}^{\infty}\norm{r\nabla_{\mu}(1+r^{2}H)^{-m}\left(\ind{C_{j}(B)}\right)}_{\El{p_0}(B)}\\
    &\lesssim r^{-1+\frac{n}{p_0}-\frac{n}{2}}\sum_{j=1}^{\infty} 2^{-j\gamma} \meas{2^{j+1}B}^{\frac{1}{2}}\lesssim r^{-1+\frac{n}{p_0}}\sum_{j=1}^{\infty} 2^{j(\frac{n}{2}-\gamma)}\lesssim r^{-1}\meas{B}^{\frac{1}{p_0}},
\end{align*}
and dividing by $\meas{B}^{\frac{1}{p_0}}$ gives (\ref{elementary bound in cancellation bound}).

To prove the estimate (\ref{cancellation bound}), we observe that (\ref{elementary bound in cancellation bound}) reduces matters to proving that 
\begin{align}\label{basically the cancellation bound up to factor}
    \norm{\nabla_{\mu}(1+r^{2}H)^{-m}(1)}_{\El{p_0}(B)}\lesssim \meas{B}^{\frac{1}{p_{0}}}\left(\dashint_{B}V\right)^{\frac{1}{2}}
\end{align}
under the assumption that $r^{-1}\geq \left(\dashint_{B} V\right)^{\frac{1}{2}}$.
We will estimate $\norm{\nabla_{\mu}(1+r^{2}H)^{-m}(1)}_{\El{p_0}(B)}$ by duality. Let us first observe that if $k\geq 1$ is an integer, then it follows from the definition that $r\nabla_{\mu}(1+r^{2}H)^{-k}(1):=\sum_{j=1}^{\infty}r\nabla_{\mu}(1+r^{2}H)^{-k}\left(\ind{C_{j}(B(0,1))}\right)$, with convergence in $\El{2}_{\textup{loc}}(\Rn)$. Therefore, for any $\varphi=(\varphi_{\parallel},\varphi_{\mu})\in C^{\infty}_{c}(\Rn,\C^{n+1})$ we have 
\begin{align*}
\begin{split}
\inner{\nabla_{\mu}(1+r^{2}H)^{-k}(1)}{\varphi}&=\sum_{j=1}^{\infty}\inner{\nabla_{\mu}(1+r^{2}H)^{-k}\left(\ind{C_{j}(B(0,1))}\right)}{\varphi}\\
&=\sum_{j=1}^{\infty}\inner{\ind{C_{j}(B(0,1))}}{(1+r^{2}H^{*})^{-k}\!\left(\nabla_{\mu}^{*}\varphi\right)}\!\!=\!\!\int_{\Rn}\clos{(1+r^{2}H^{*})^{-k}\!\left(\nabla_{\mu}^{*}\varphi\right)},
\end{split}
\end{align*}
by duality and dominated convergence, using similarity and Lemma~\ref{off-diag + cpct supp ==> integrable} applied to $H^{\sharp}$ to deduce that $(1+r^{2}H^{*})^{-k}\left(\nabla_{\mu}^{*}\varphi\right)=A_{\perp\perp}^{*}(1+r^{2}H^{\sharp})^{-k}\left(b^{*}\nabla_{\mu}^{*}\varphi\right)\in \Ell{1}$,
since $b^{*}\nabla_{\mu}^{*}\varphi\in \El{2}_{c}(\Rn)$.
By telescopic summation we have
\begin{align*}
    (1+r^{2}H^{*})^{-m} &= 1+ \sum_{k=1}^{m} \left((1+r^{2}H^{*})^{-k}-(1+r^{2}H^{*})^{-(k-1)}\right)=1-\sum_{k=1}^{m} r^{2}H^{*}(1+r^{2}H^{*})^{-k}.
\end{align*}
As a consequence, we obtain that
\begin{align}
\begin{split}
    \abs{\inner{\nabla_{\mu}(1+r^{2}H)^{-m}(1)}{\varphi}}&=\abs{\int_{\Rn}(1+r^{2}H^{*})^{-m}\left(\nabla_{\mu}^{*}\varphi\right)}\\
    &\leq \abs{\int_{\Rn}\nabla_{\mu}^{*}\varphi} + \sum_{k=1}^{m}\abs{\int_{\Rn}r^{2}H^{*}(1+r^{2}H^{*})^{-k}(\nabla_{\mu}^{*}\varphi)}\\
    &=\abs{\int_{\Rn}\nabla_{\mu}^{*}\varphi} + \sum_{k=1}^{m}\abs{\int_{\Rn}r^{2}A_{\perp\perp}^{*}H^{\sharp}(1+r^{2}H^{\sharp})^{-k}(b^{*}\nabla_{\mu}^{*}\varphi)}\\
    &=\abs{\int_{\Rn}V^{1/2}\varphi_{\mu}} + \sum_{k=1}^{m}\abs{\int_{\Rn}r^{2}a^{*}V(1+r^{2}H^{\sharp})^{-k}(b^{*}\nabla_{\mu}^{*}\varphi)}\label{last two terms to est}\\
    &=\abs{\int_{\Rn}V^{1/2}\varphi_{\mu}} + \sum_{k=1}^{m}\abs{\int_{\Rn}r^{2}a^{*}b^{*}V(1+r^{2}H^{*})^{-k}(\nabla_{\mu}^{*}\varphi)},
    \end{split}
\end{align}
where we have used similarity between $H^{\sharp}$ and $H^{*}$, Lemma~\ref{cancellation lemma} applied to $H^{\sharp}$, and the fact that $\nabla_{\mu}^{*}\varphi=-\textup{div}\varphi_{\parallel}+V^{1/2}\varphi_{\mu}$ with $\int_{\Rn}\textup{div}{\varphi_{\parallel}}=0$.

We first estimate the first term in the last line of \eqref{last two terms to est}. Recall that $V\in \textup{RH}^{\beta}(\Rn)$ for some $\beta>\frac{n}{2}$. Since $\frac{p_0}{2}\leq \beta$, Hölder's inequality yields
\begin{align*}
    \norm{V^{1/2}}_{\El{p_0}(B)}&=\meas{B}^{1/p_{0}}\left(\dashint_{B}V^{\frac{p_{0}}{2}}\right)^{\frac{1}{p_0}}
    \leq \meas{B}^{1/p_{0}}\left(\left(\dashint_{B}V^{\beta}\right)^{\frac{1}{\beta}}\right)^{\frac{1}{2}}
    \lesssim \meas{B}^{1/p_{0}}\left(\dashint_{B}V\right)^{\frac{1}{2}}.
\end{align*}
Hence, another application of Hölder's inequality gives
\begin{align}\label{intermediate step proof riesz trsf1}
    \abs{\int_{\Rn}V^{1/2}\varphi_{\mu}}\leq \norm{V^{1/2}}_{\El{p_0}(B)}\norm{\varphi}_{L^{p'_{0}}(B)}\lesssim\meas{B}^{1/p_{0}}\left(\dashint_{B}V\right)^{\frac{1}{2}}\norm{\varphi}_{L^{p'_{0}}(B)}.
\end{align}

We now estimate the second term in the last line of (\ref{last two terms to est}). We claim that there is an exponent $q_{0}\in (1,2\beta]$ such that 
\begin{equation}\label{conditions to be imposed on the exponent q to make the argument work}
    \frac{1}{p_{0}}\leq \frac{1}{q_{0}} + \frac{1}{2\beta} < \frac{1}{{p_{0}}_{*}}=\frac{1}{p_{0}}+\frac{1}{n}.
\end{equation}
Indeed, it suffices to pick  $\frac{1}{q_{0}}\in[\frac{1}{p_0}-\frac{1}{2\beta} , \frac{1}{p_0}-\frac{1}{2\beta}+\frac{1}{n})\cap [\frac{1}{2\beta},1)$. This intersection is non empty since $\frac{1}{\beta}<\frac{1}{p_0} + \frac{1}{n}$ (as $p_{0}<\beta^{*}$) and $\frac{1}{p_{0}}-\frac{1}{2\beta}<1$ (as $p_{0}>2)$.
Let us now fix an integer $1\leq k\leq m$. We can apply Hölder's inequality with the exponent $q_{0}>1$ to get
\begin{align}\label{intermediate inequality to prove Blunck Kunstmann criterion satisfied}
    \abs{\int_{\Rn}r^{2}a^{*}b^{*}V(1+r^{2}H^{*})^{-k}(\nabla_{\mu}^{*}\varphi)}&\lesssim\sum_{j=1}^{\infty}\int_{C_{j}(B)}\abs{rV^{1/2}}\abs{rV^{1/2}(1+r^{2}H^{*})^{-k}(\nabla_{\mu}^{*}\varphi)}\nonumber\\
    &\leq \sum_{j=1}^{\infty}\norm{rV^{1/2}}_{\El{q_{0}}(C_{j}(B))}\!\norm{rV^{1/2}(1\!+\! r^{2}H^{*})^{-k}(\nabla_{\mu}^{*}\varphi)}_{\El{q_{0}'}(C_{j}(B))}.
\end{align}
First, since $\frac{q_{0}}{2}\leq\beta$ and $V\in\textup{RH}^{\beta}$, we can use Hölder's inequality to estimate
\begin{align}\label{intermediate inequality to prove Blunck Kunstmann criterion satisfied_sub1}
    \norm{rV^{1/2}}_{\El{q_{0}}(C_{j}(B))}&\leq r\left[\left(\int_{2^{j+1}B}V^{\frac{q_{0}}{2}}\right)^{\frac{2}{q_{0}}}\right]^{\frac{1}{2}}=r\meas{2^{j+1}B}^{\frac{1}{q_{0}}}\left[\left(\dashint_{2^{j+1}B}V^{\frac{q_{0}}{2}}\right)^{\frac{2}{q_{0}}}\right]^{\frac{1}{2}}\nonumber\\
    &\leq r\meas{2^{j+1}B}^{\frac{1}{q_{0}}}\left[\left(\dashint_{2^{j+1}B}V^{\beta}\right)^{\frac{1}{\beta}}\right]^{\frac{1}{2}}\lesssim r\meas{2^{j+1}B}^{\frac{1}{q_{0}}-\frac{1}{2}}\left(\int_{2^{j+1}B}V\right)^{\frac{1}{2}}.
\end{align}

Let us now estimate the second factor in \eqref{intermediate inequality to prove Blunck Kunstmann criterion satisfied}.
Let $s_{\beta}$ be defined by $\frac{1}{s_{\beta}}:=\frac{1}{q_{0}}+\frac{1}{2\beta}$, and note that (\ref{conditions to be imposed on the exponent q to make the argument work}) shows that $s_{\beta}\in({p_{0}}_{*} , p_{0}]$. 
Let the conjugate exponent $s'_{\beta}$ be defined by $\frac{1}{s'_{\beta}}:=1-\frac{1}{s_{\beta}}$. Since $\frac{1}{q_{0}'}=\frac{1}{2\beta}+\frac{1}{s_{\beta}'}$, Hölder's inequality and the fact that $V\in\textup{RH}^{\beta}(\Rn)$ yield
\begin{align*}
    \norm{V^{1/2}g}_{\El{q_{0}'}(C_{j}(B))}\leq \left[\left(\int_{C_{j}(B)}V^{\beta}\right)^{\frac{1}{\beta}}\right]^{\frac{1}{2}}\norm{g}_{\El{s'_{\beta}}(C_{j}(B))}\lesssim \!\meas{2^{j+1}B}^{\frac{1}{2\beta}-\frac{1}{2}}\!\!\left(\int_{2^{j+1}B}V\right)^{\frac{1}{2}}\norm{g}_{\El{s'_{\beta}}(C_{j}(B))}\!.
\end{align*}
Applying this to $g:=\left(r\nabla_{\mu}(1+r^{2}H)^{-k}\right)^{*}\varphi$, we obtain
\begin{align}
    &\norm{rV^{1/2}(1+r^{2}H^{*})^{-k}(\nabla_{\mu}^{*}\varphi)}_{\El{q_{0}'}(C_{j}(B))}= \norm{V^{1/2}\left(r\nabla_{\mu}(1+r^{2}H)^{-k}\right)^{*}\varphi}_{\El{q'_{0}}(C_{j}(B))}\nonumber\\
    &\lesssim \meas{2^{j+1}B}^{\frac{1}{2\beta}-\frac{1}{2}}\left(\int_{2^{j+1}B}V\right)^{\frac{1}{2}}\norm{\left(r\nabla_{\mu}(1+r^{2}H)^{-k}\right)^{*}\varphi}_{\El{s'_{\beta}}(C_{j}(B))}.
    \label{intermediate inequality to prove Blunck Kunstmann criterion satisfied_sub2}
\end{align}

 Since $s_{\beta}\in({p_{0}}_{*} , p_{0}]$, Lemma~\ref{ode resolvents lemma for riesz transform bounds}.(ii) implies that the family $\{t\nabla_{\mu}(1+t^{2}H)^{-k}\}_{t>0}$ satisfies $\El{s_{\beta}}-\El{p_0}$ off-diagonal estimates, whence the dual family $\set{\left(t\nabla_{\mu}(1+t^{2}H)^{-k}\right)^{*}}_{t>0}$ satisfies $\El{p'_{0}}-\El{s'_{\beta}}$ off-diagonal estimates of arbitrarily large order.
Since $\frac{1}{s'_{\beta}}-\frac{1}{p'_{0}}=\frac{1}{p_0}-\frac{1}{s_{\beta}}$, this means that we have the following estimate for all $\gamma>0$ and $j\geq 1$:
\begin{align}\label{intermediate inequality to prove Blunck Kunstmann criterion satisfied_sub3}
    \norm{\left(r\nabla_{\mu}(1+r^{2}H)^{-k}\right)^{*}\varphi}_{\El{s'_{\beta}}(C_{j}(B))}&\lesssim r^{\frac{n}{s'_{\beta}}-\frac{n}{p'_{0}}}\left(1+\frac{\dist(B,C_{j}(B))}{r}\right)^{-\gamma}\norm{\varphi}_{\El{p'_{0}}(B)}\nonumber\\
    &\lesssim r^{\frac{n}{p_{0}}-\frac{n}{s_{\beta}}}2^{-j\gamma}\norm{\varphi}_{\El{p'_{0}}(B)}.
\end{align}

We can now combine the estimates \eqref{intermediate inequality to prove Blunck Kunstmann criterion satisfied}, \eqref{intermediate inequality to prove Blunck Kunstmann criterion satisfied_sub1}, \eqref{intermediate inequality to prove Blunck Kunstmann criterion satisfied_sub2} and \eqref{intermediate inequality to prove Blunck Kunstmann criterion satisfied_sub3} and use the doubling property of $V$ (with doubling constant $\kappa$ from Lemma~\ref{self-improvement property of reverse holder weights}.(ii)) to get, for any $k\in\set{1,\ldots,m}$,
\begin{align}
\begin{split}
\label{intermediate step proof riesz trsf2}
    \abs{\int_{\Rn}r^{2}a^{*}b^{*}V(1+r^{2}H^{*})^{-k}(\nabla_{\mu}^{*}\varphi)}&\leq \sum_{j=1}^{\infty}\norm{rV^{1/2}}_{\El{q_{0}}(C_{j}(B))}\norm{rV^{1/2}(1+r^{2}H^{*})^{-k}(\nabla_{\mu}^{*}\varphi)}_{\El{q_{0}'}(C_{j}(B))}\\
    &\lesssim \sum_{j=1}^{\infty} \meas{2^{j+1}B}^{\frac{1}{2\beta}+\frac{1}{q_{0}}-1}r^{1+\frac{n}{p_{0}}-\frac{n}{s_{\beta}}}\left(\int_{2^{j+1}B}V\right)2^{-j\gamma}\norm{\varphi}_{\El{p'_{0}}(B)}\\
    &\lesssim\sum_{j=1}^{\infty} \meas{2^{j+1}B}^{\frac{1}{s_{\beta}}-1}r^{1+\frac{n}{p_{0}}-\frac{n}{s_{\beta}}}\kappa^{j+1}\left(\int_{B}V\right)2^{-j\gamma}\norm{\varphi}_{\El{p'_{0}}(B)}\\
    &\lesssim r^{1+\frac{n}{p_0}}\left(\dashint_{B}V\right)\norm{\varphi}_{\El{p'_{0}}(B)}\left(\sum_{j=1}^{\infty}2^{-j\gamma}2^{jn\left(\frac{1}{s_{\beta}}-1\right)}\kappa^{j+1}\right)\\
    &\lesssim r^{1+\frac{n}{p_0}}\left(\dashint_{B}V\right)\norm{\varphi}_{\El{p'_{0}}(B)},
    \end{split}
\end{align}
provided we choose $\gamma>n\left(1-\frac{1}{s_{\beta}}\right)$ + $\log_{2}(\kappa)$.

Combining \eqref{intermediate step proof riesz trsf1} and \eqref{intermediate step proof riesz trsf2} in \eqref{last two terms to est}, with the assumption $r^{-1}\geq \left(\dashint_{B}V\right)^{\frac{1}{2}}$, we obtain
\begin{align*}
\begin{split}
    \abs{\inner{\nabla_{\mu}(1+r^{2}H)^{-m}(1)}{\varphi}}
    &\lesssim \meas{B}^{1/p_{0}}\left(\dashint_{B}V\right)^{\frac{1}{2}}\norm{\varphi}_{\El{p'_{0}}(B)} +\sum_{k=1}^{m}r^{1+\frac{n}{p_0}}\left(\dashint_{B}V\right)\norm{\varphi}_{\El{p'_{0}}(B)}\\
    &\lesssim \meas{B}^{\frac{1}{p_{0}}}\left[\left(\dashint_{B}V\right)^{\frac{1}{2}} + r\dashint_{B}V\right]\norm{\varphi}_{\El{p'_{0}}(B)}\\
    &\lesssim \meas{B}^{\frac{1}{p_{0}}}\left(\dashint_{B}V\right)^{\frac{1}{2}}\norm{\varphi}_{\El{p'_{0}}(B)}
    \end{split}
\end{align*}
for all $\varphi\in C^{\infty}_{c}(B)$, hence \eqref{basically the cancellation bound up to factor} holds. This proves the cancellation estimate \eqref{cancellation bound}, so the entire proof is complete.\qedhere
\end{enumerate}
\end{proof}

\begin{rem}
    We could have studied the $\El{p}$-boundedness of the components of $R_{H}$ separately by introducing the operators $R_{H}^{\parallel}:\Ell{2}\to\El{2}(\Rn;\C^{n})$ and $R_{H}^{\mu}:\Ell{2}\to\Ell{2}$ such that $R_{H}:=\begin{bmatrix}
        R_{H}^{\parallel}\\
        R_{H}^{\mu}
    \end{bmatrix}$, with $R_{H}^{\parallel}=\nabla_{x}H^{-1/2}$ and $R_{H}^{\mu}=V^{1/2}H^{-1/2}$ on $\ran{H^{1/2}}$. Inspection of the proof above reveals that if $n\geq 3$, $q\geq \max\set{\frac{n}{2},2}$ and $V\in\textup{RH}^{q}(\Rn)$, then for a given $p\in\left(2,q_{+}(H)\right)$, the operator $R_{H}^{\parallel}$ is $\El{p}$-bounded if $p<q^{*}$, whereas $R^{\mu}_{H}$ is $\El{p}$-bounded provided $p\leq 2q$. As was previously noted, since $q\geq \frac{n}{2}$, $p\leq 2q$ automatically implies that $p<q^{*}$. Let us also remark that this is consistent with \cite[Theorems 0.5 and 5.10]{Shen_Schrodinger}.
\end{rem}

\section{Identification of the Hardy spaces $\bb{H}^{p}_{H}$ and $\bb{H}^{1,p}_{H}$}\label{section on identification of H adapted Hardy spaces}
This section is concerned with the identification of the operator-adapted spaces $\bb{H}^{p}_{H}$ and $\bb{H}^{1,p}_{H}$ with “concrete" function spaces. In fact, we shall prove the following theorem. We recall here that $b$ is the bounded accretive function defined as $b(x)=A_{\perp\perp}(x)^{-1}$ for almost every $x\in\Rn$, and that the (quasi)norm on $b\textup{H}^{p}_{V,\textup{pre}}(\Rn)$ is defined as $\norm{f}_{b\textup{H}^{p}_{V}}:=\norm{b^{-1}f}_{\textup{H}^{p}_{V}}$ for all $f\in b\textup{H}^{p}_{V,\textup{pre}}$.
\begin{thm}\label{summary of identifications of H adapted Hardys spaces}
    Let $n\geq 3$, $q>\max\{\frac{n}{2},2\}$, $V\in\textup{RH}^{q}(\Rn)$, and let $\delta=2-\frac{n}{q}$. The following identifications hold with equivalent $p$-(quasi)norms:
\begin{equation*}
\begin{array}{ll}
  \bb{H}^{p}_{H}=b\textup{H}^{p}_{V,\textup{pre}}, &  \textup{ if }p\in(p_{-}(H)\vee \frac{n}{n+\delta/2}\; ,p_{+}(H));\\
 \bb{H}^{1,p}_{H}=\dot{\mathcal{V}}^{1,p}\cap\El{2}, & \textup{ if }p\in (p_{-}(H)_{*}\; , q_{+}(H))\cap(1,2q]; \\  
\bb{H}^{1,p}_{H}\cap\mathcal{V}^{1,2}=\dot{\textup{H}}^{1,p}_{V,\textup{pre}}, & \textup{ if }p\in(p_{-}(H)_{*}\vee \frac{n}{n+\delta/2}\, ,1]\cap \mathcal{I}(V).
    \end{array}
        \end{equation*}
\end{thm}
We shall mostly follow the strategy used in \cite[Chapter 9]{AuscherEgert}. The theorem is obtained by combining Corollary~\ref{characterisation abstract H^p p<2} and Proposition~\ref{identification abstract hardy space H^p p>2}. We note that if $p\in(\frac{n}{n+1},\infty)$ is such that $\bb{H}^{p}_{H}=b\textup{H}^{p}_{V,\textup{pre}}$ with equivalent $p$-(quasi)norms, then $p\in\mathcal{J}(H)$. This follows from the sectorial version of \cite[Proposition 8.10]{AuscherEgert} as in Part 1 of the proof of \cite[Theorem 9.7]{AuscherEgert}. This shows that the identification $\bb{H}^{p}_{H}=b\textup{H}^{p}_{V,\textup{pre}}$ in Theorem~\ref{summary of identifications of H adapted Hardys spaces} is the best possible when $V\in\textup{RH}^{\infty}(\Rn).$

\subsection{A Calderón--Zygmund--Sobolev decomposition adapted to singular potentials}
The main ingredient for the identification of the $H$-adapted Hardy spaces (more specifically the proof of Corollary~\ref{continuous inclusion L^p subset abstract H^p}) is a Calderón--Zygmund--Sobolev decomposition adapted to singular potentials. This was obtained by Auscher--Ben Ali \cite[Lemma 7.1]{Auscher-BenAli}. We recall their construction because we need certain additional properties. In particular, their result only treated smooth compactly supported functions $f$, and properties (v), (viii) and (ix) stated below were not stated. We also point out that there is an inaccuracy in the construction in \cite[Lemma 7.1]{Auscher-BenAli}, that dates back to \cite[Theorem 6]{Auscher2004}, where this Calderón--Zygmund decomposition for Sobolev functions was first introduced, but a correction has been provided in \cite{Auscher_erratum_Calderon_Sobolev}, which we follow.

\begin{lem}\label{lemma V-adapted Calderon--Zygmund--Sobolev decomposition}
    Let $n\geq 1$, $q>1$, $V\in\textup{RH}^{q}(\Rn)$ and $p\in[1,2]$. Let $f\in\mathcal{V}^{1,2}(\Rn)\cap\Dot{\mathcal{V}}^{1,p}(\Rn)$ and $\alpha>0$. There exists a sequence of cubes $\{Q_{k}\}_{k\geq 1}$, a sequence of functions $\{b_k\}_{k\geq 1}\subset \dot{\mathcal{V}}^{1,p}(\Rn)\cap \mathcal{V}^{1,2}(\Rn)$ and a function $g\in \Dot{\mathcal{V}}^{1,p}(\Rn)$ with the following properties:
    \begin{enumerate}[label=\emph{(\roman*)}]
        \item $f=g+\sum_{k\geq 1}b_k$ almost everywhere on $\Rn$, and in $\dot{\mathcal{V}}^{1,p}(\Rn)$;
        \item $b_k$ is supported in $Q_k$ for each $k\geq 1$;
        \item The cubes $\{Q_{k}\}_{k\geq 1}$ have bounded overlap (see Section~\emph{\ref{subsubsection on geometry in Rn}}) and satisfy \[\sum_{k\geq 1}\nolimits\meas{Q_k}\lesssim \meas{\bigcup_{k\geq 1}\nolimits Q_k}\lesssim \alpha^{-p}\norm{\nabla_{\mu}f}_{\Ell{p}}^{p};\]
        \item $\displaystyle\int_{Q_k}\abs{\nabla_{x}b_k}^{p}+\abs{V^{1/2}b_k}^{p}+\ell(Q_k)^{-p}\abs{b_k}^{p}\dd x\lesssim \alpha^{p}\meas{Q_k}$ for all $k\geq 1$;
        \item The series $\sum_{k\geq 1}\nabla_{\mu}b_k$ converges unconditionally in $\Ell{p}$, with \[\norm{\sum_{k\geq 1}\nolimits\nabla_{\mu}b_k}_{\El{p}}\lesssim \norm{\nabla_{\mu}f}_{\El{p}};\]
        \item $\norm{\nabla_{x}g}_{\El{\infty}}\lesssim \alpha;$
        \item $\norm{\nabla_{\mu}g}_{\El{r}}^{r}\lesssim \alpha^{r-p}\norm{\nabla_{\mu}f}_{\El{p}}^{p}$ for all $r\in \left[p,2\right]$;
        \item If $p<n$ and $q\in \left[1,p^{*}\right]$, then $\norm{b_k}_{\El{q}}^{q}\lesssim \alpha^{q}\meas{Q_k}^{1+\frac{q}{n}}$ for all $k\geq 1$. If $p\geq n$, then this holds for all $q\in[1,\infty)$;
        \item If $r\in\left[p,2\right]$ and $f\in\Dot{\mathcal{V}}^{1,r}(\Rn)$ (in addition to $f\in\mathcal{V}^{1,2}(\Rn)\cap\Dot{\mathcal{V}}^{1,p}(\Rn)$) then $b_k\in \mathcal{V}^{1,r}(\Rn)$ for all $k\geq 1$, and the series $\sum_{k\geq 1}\nabla_{\mu}b_k$ converges unconditionally in $\Ell{r}$.
    \end{enumerate}
\end{lem}
\begin{proof}
    Let $\mathcal{M}$ denote the uncentred Hardy--Littlewood maximal operator for cubes. Consider the open  subset $\Omega:=\set{x\in\Rn :  \hlmax{\abs{\nabla_{x}f}^{p}+\abs{V^{1/2}f}^{p}}(x)>\alpha^{p}}$. If $\Omega$ is empty, then we can take $g:=f$ and $b_k:=0$ for all $k\geq 1$. Then, properties (vi) and (vii) follows by Lebesgue's differentiation theorem, and all remaining properties are trivial. 
    
    We henceforth assume that $\Omega\neq\emptyset$. Let $F:=\Rn\setminus\Omega$, and let $\{W_{k}\}_{k\geq 1}$ be a Whitney decomposition of $\Omega$, as constructed in \cite[Chapter \RNum{6}, Theorem 1]{SteinSingularIntegrals}. This means that the $W_{k}$ are essentially disjoint (dyadic) cubes, such that $\Omega=\bigcup_{k\geq 1}W_{k}$ and 
    \begin{equation}\label{properties of whitney decomp}
    \diam{W_{k}}\leq\dist(W_{k},F)\leq 4\diam{W_{k}}.
    \end{equation}
    We can then define the enlarged cubes $Q_k:=(1+\eps)W_{k}$ where $\eps>0$ is small enough to ensure that $\Omega=\bigcup_{k\geq 1}Q_k$, with bounded overlap. We can also find some $c>1$ such that the further enlarged cubes $\widetilde{Q}_{k}:=c Q_k$ satisfy $\widetilde{Q}_{k}\cap F\neq\emptyset$ for all $k\geq 1.$
    The bounded overlap of the cubes $\set{Q_{k}}_{k\geq 1}$ and the weak-type $(1,1)$ bound for the maximal operator yield property (iii).  

As in \cite[p.170]{SteinSingularIntegrals}, we can construct a partition of unity of $\Omega$ subordinate to the covering $\Omega=\cup_{k\geq 1} Q_k$. For each $k\geq 1$, there is $\chi_{k}\in C^{\infty}_{c}(\Rn,\left[0,1\right])$ with support in $Q_k$ and such that $\sum_{k\geq 1}\chi_{k} = \ind{\Omega}$, the sum being locally finite on $\Omega$. Moreover, $\norm{\chi_{k}}_{\El{\infty}} + \ell_{k}\norm{\nabla_{x}\chi_{k}}_{\El{\infty}}\lesssim 1$ for all $k\geq 1$, where $\ell_{k}:=\ell(Q_k)$. 
Moreover, the construction in \cite[p.170]{SteinSingularIntegrals} shows that $\chi_{k}\gtrsim 1$ on $W_{k}$, uniformly for all $k\geq 1$.

Since $f\in\mathcal{V}^{1,2}(\Rn)$ and $p\leq 2$, we know that $f\in \El{p}_{\loc}(\Rn)$. Moreover, since $p/2\in(0,1]$, the third item of Lemma~\ref{self-improvement property of reverse holder weights} shows that $V^{p/2}\in \textup{RH}^{\frac{2}{p}}$ if $p\neq 2$, while $V^{p/2}\in\textup{RH}^{q}$ if $p=2$. The local Fefferman--Phong inequality (Proposition~\ref{improved Fefferman--Phong inequality}) gives
\begin{equation}\label{fefferman-phong ineq}
    \int_{Q_k} \abs{\nabla_{x}f}^{p}+\abs{V^{1/2}f}^{p} \gtrsim\min\set{\av{Q_k}V^{p/2},\ell_{k}^{-p}}\int_{Q_k}\abs{f}^{p}
\end{equation}
uniformly for all $k\geq 1$. We say that a cube $Q_k$ is of type 1 if $\av{Q_k}V^{p/2}\geq \ell_{k}^{-p}$, and of type 2 otherwise.
The functions $b_k$ are then defined as 
\begin{equation*}
b_{k}:=
    \begin{cases}
        f\chi_{k}, & \text{ if } Q_k \text{ is type 1;}\\
        (f-(f)_{Q_k})\chi_{k}, & \text{ if } Q_k \text{ is type 2}.
    \end{cases}
\end{equation*}
Property (ii) follows immediately from this definition.
Let us now prove property (iv). The proof of \cite[Lemma 7.1]{Auscher-BenAli} readily shows that
\begin{equation*}
    \int_{Q_k}\left(\abs{\nabla_{x}b_k}^{p}+\ell_{k}^{-p}\abs{b_k}^{p} \right)\lesssim \alpha^{p}\meas{Q_k}
\end{equation*}
for all $k\geq 1$. In addition, if $Q_k$ is type 2, the Fefferman--Phong inequality (\ref{fefferman-phong ineq}) shows that
\begin{align*}
    \norm{V^{1/2}b_k}_{\El{p}}^{p}&\leq \int_{Q_k}V^{p/2}\abs{f-(f)_{Q_k}}^{p}\lesssim \int_{Q_k}V^{p/2}\abs{f}^{p} + \abs{(f)_{Q_k}}^{p}\int_{Q_k}V^{p/2}\\
    &\leq \int_{Q_k}\abs{V^{1/2}f}^{p} + \left(\int_{Q_k}\abs{f}^{p}\right)\av{Q_k}(V^{p/2})\lesssim \int_{Q_k}\abs{\nabla_{x}f}^{p}+\abs{V^{1/2}f}^{p} \lesssim \alpha^{p}\meas{Q_k},
\end{align*}
by enlarging the cube $Q_k$ to $\widetilde{Q}_{k}$ and using the definition of $F$. On the other hand, if $Q_k$ is type 1, we have $\int_{\Rn}\abs{V^{1/2}b_k}^{p} \leq \int_{Q_k} V^{p/2}\abs{f}^{p} \lesssim \alpha^{p}\meas{Q_k}$. This completes the proof of property (iv).

The proof of \cite[Lemma 7.1]{Auscher-BenAli} shows that the series $\sum_{k\geq 1} b_k$ converges in $\mathcal{D}'(\Rn)$. We can therefore define $g:=f-\sum_{k\geq 1}b_k$ as a distribution on $\Rn$. 
Note that the bounded overlap $\sum_{k\geq 1}\ind{Q_{k}}(x)\leq N$ implies that for any finite subset $K\subset \bb{N}$ it holds that 
\begin{align*}
    \norm{\sum_{k\in K}\nolimits b_k}_{\dot{\mathcal{V}}^{1,p}}^{p}&=\int_{\Rn}\abs{\sum_{k\in K}\nolimits\nabla_{\mu}b_k}^{p}\leq N^{p}\sum_{k\in K}\nolimits\int_{\Rn}\abs{\nabla_{\mu}b_k}^{p}\lesssim \alpha^{p}\sum_{k\in K}\nolimits\meas{Q_k},
\end{align*}
using property (iv) to obtain the last inequality. Since ${\sum_{k\geq 1}\meas{Q_k}<\infty}$ by property (iii), this proves property (v) (see, e.g., \cite[Theorem 3.10]{Heil2011}).

Since $\nabla_{x}g=\nabla_{x}f-\sum_{k\geq 1}\nabla_{x}b_k$ in $\mathcal{D}'(\Rn)$, the distributional gradient $\nabla_{x}g$ coincides with an $\El{p}(\Rn,\C^{n})$ function, and 
\begin{equation}\label{Lp bound gradient g}
    \norm{\nabla_{x}g}_{\El{p}}\lesssim \norm{\nabla_{\mu}f}_{\El{p}}.
\end{equation}
In particular, both distributions $g$ and $\sum_{k\geq 1}b_k$ coincide with locally integrable functions on $\Rn$ (see , e.g., \cite[Théorème 2.1]{DenyLions_BeppoLevi}). Consequently, the representation $f=g+\sum_{k\geq1}b_k$ holds almost everywhere on $\Rn$, as well as in the $\dot{\mathcal{V}}^{1,p}(\Rn)$ norm, which proves property (i). In addition, the proof of \cite[Lemma 7.1]{Auscher-BenAli} shows that $\norm{\nabla_{x}g}_{\El{\infty}}\lesssim \alpha$. This proves property (vi). 

Property (vii) in the case $r=2$ is already contained in the statement of \cite[Lemma 7.1]{Auscher-BenAli}. The case $r=p$ follows immediately from property (v), so the result follows by Hölder's inequality. We will however rewrite the argument for general $r\in[p,2]$, as some steps will be needed for other properties. We first claim that for all cubes $Q$ it holds that  
\begin{equation}\label{averages cubes potential}
    \av{Q}\left(V^{r/2}\right)\lesssim \av{Q}\left(V^{p/2}\right)^{r/p}.
\end{equation}
If $r=p$, this is trivial, so we may assume that $p<r\leq 2$. Since $V\in\textup{RH}^{q}$ and $0<\frac{r}{2}\leq 1$, the third item of Lemma~\ref{self-improvement property of reverse holder weights} implies that $V^{r/2}\in \cup_{q>1}\textup{RH}^{q}$. Then, since $\frac{p}{r}<1$, the same result implies that $V^{p/2}=\left(V^{r/2}\right)^{p/r}\in \textup{RH}^{\frac{r}{p}}$, so $\left(\dashint_{Q}V^{r/2}\right)^{p/r}\lesssim \dashint_{Q} \left(V^{r/2}\right)^{p/r}$ for each cube $Q\subset\Rn$, and \eqref{averages cubes potential} holds.

Clearly, $g=f$ on $F$. Moreover, the sum $\sum_{k\geq 1}b_k$ is locally finite on $\Omega$, and the proof of \cite[Lemma 7.1]{Auscher-BenAli} shows that $g=\displaystyle\sum_{\substack{k\geq 1 \\ \textup{ type 2}}}(f)_{Q_k}\chi_{k}$ almost everywhere on $\Omega$.
Hence, using the bounded overlap of the supports of the $\chi_k$ and \eqref{averages cubes potential}, we obtain
\begin{align}
\begin{split}\label{intermediate step to estimate Lr norm of V1/2 g on omega}
    \int_{\Omega} \left(V^{1/2}\abs{g}\right)^{r} &= \int_{\Omega} V^{r/2}\vert{\sum_{\substack{k\geq 1 \\ \textup{ type 2}}}(f)_{Q_k}\chi_{k}}\vert^{r}\lesssim \sum_{\substack{k\geq 1 \\ \textup{ type 2}}}\int_{\Omega}V^{r/2}\abs{(f)_{Q_k}\chi_{k}}^{r}\\
    &\leq \sum_{\substack{k\geq 1 \\ \textup{ type 2}}}\meas{Q_k}\av{Q_k}(V^{r/2})\abs{(f)_{Q_k}}^{r}\lesssim \sum_{\substack{k\geq 1 \\ \textup{ type 2}}} \meas{Q_k} \left(\av{Q_k}(V^{p/2})\abs{(f)_{Q_k}}^{p}\right)^{r/p}.
    \end{split}
\end{align}
Now, for all cubes $Q_k$ of type 2 we have, by the Fefferman--Phong inequality (\ref{fefferman-phong ineq}),
\begin{align}
\begin{split}\label{estimate cube type 2 from fefferman phong}
    \av{Q_k}(V^{p/2})\abs{(f)_{Q_k}}^{p}&\leq \av{Q_k}(V^{p/2})\av{Q_k}\left(\abs{f}^{p}\right)\lesssim \dashint_{Q_k} \left(\abs{\nabla_{x}f}^{p}+\abs{V^{1/2}f}^{p} \right)\lesssim \alpha^{p},
    \end{split}
\end{align}
by enlarging the cube $Q_k$ to $\widetilde{Q}_{k}$ and using the definition of $F=\Rn\setminus\Omega$ in the last inequality. Combining \eqref{intermediate step to estimate Lr norm of V1/2 g on omega} and \eqref{estimate cube type 2 from fefferman phong}, and using property (iii), we get $\int_{\Omega} \left(V^{1/2}\abs{g}\right)^{r} \lesssim\alpha^{r-p}\norm{\nabla_{\mu}f}_{\El{p}}^{p}$. 
Moreover, since $r\geq p$ and $\abs{V^{1/2}f}\leq \alpha $ almost everywhere in $F$ (by Lebesgue's differentiation theorem), we have
\begin{align*}
    \int_{F} \left(V^{1/2}\abs{g}\right)^{r}\leq \alpha^{r-p}\int_{F} \left(V^{1/2}\abs{f}\right)^{p}
    \leq \alpha^{r-p}\norm{\nabla_{\mu}f}_{\El{p}}^{p}.
\end{align*}
This proves that $\norm{V^{1/2}g}_{\Ell{r}}^{r}\lesssim \alpha^{r-p}\norm{\nabla_{\mu}f}_{\El{p}}^{p}$. We also combine (vi) and \eqref{Lp bound gradient g} to obtain $\norm{\nabla_{x}g}_{\El{r}}^{r}\lesssim \alpha^{r-p}\norm{\nabla_{\mu}f}_{\El{p}}^{p}$ for all $r\geq p$. This proves property (vii).

We now prove property (viii). Let $q\in \left[1,p^{*}\right]$ if $1\leq p<n$, or $q\in[1,\infty)$ if $p\geq n$. Therefore, if $Q_k$ is any type 2 cube, the Sobolev--Poincaré inequality (see, e.g., \cite[Lemmas 7.12 and 7.16]{Gilbarg_Trudinger_Elliptic}) yields
\begin{equation*}
    \norm{b_{k}}_{\El{q}}\leq \meas{Q_k}^{1/q}\left(\dashint_{Q_k}\abs{f-(f)_{Q_k}}^{q}\right)^{1/q}\lesssim \meas{Q_k}^{1/q}\ell_{k}\left(\dashint_{2Q_k}\abs{\nabla_{x}f}^{p}\right)^{1/p}\lesssim \alpha\meas{Q_k}^{\frac{1}{q}+\frac{1}{n}}.
\end{equation*}
Similarly, if $Q_k$ is type 1 we have
\begin{align*}
    \norm{b_k}_{\El{q}}\leq \norm{(f-(f)_{Q_k})\chi_{k}}_{\El{q}}+\abs{(f)_{Q_k}}\norm{\chi_{k}}_{\El{q}}\lesssim \meas{Q_k}^{1/q}\ell_{k}\alpha + \abs{(f)_{Q_k}}\meas{Q_k}^{1/q}\lesssim \meas{Q_k}^{1/q}\ell_{k}\alpha,
\end{align*}
where we have used the fact that $\abs{(f)_{Q_k}}\lesssim l_{k}\alpha$ for all cubes $Q_k$ of type 1. In fact,
\begin{align}\label{estimate type 1 cubes}
    \abs{(f)_{Q_k}}^{p}&\leq \dashint_{Q_k}\abs{f}^{p}\lesssim l_{k}^{p}\dashint_{Q_k}\left(\abs{\nabla_{x}f}^{p}+\abs{V^{1/2}f}^{p}\right)\lesssim l_{k}^{p}\alpha^{p},
\end{align}
by the Fefferman--Phong inequality (\ref{fefferman-phong ineq}). This proves property (viii).

Finally, we prove property (ix). Let $r\in\left[p,2\right]$ and $f\in \Dot{\mathcal{V}}^{1,r}(\Rn)$. We claim that $b_k\in \mathcal{V}^{1,r}(\Rn)$ for all $k\geq 1$. First, $b_k\in \Ell{r}$ is clear from its definition.  Now, if $Q_k$ is type 2, we simply have
\begin{align}\label{intermediate estimate on gradient of bk in the Lr norm}
\begin{split}
    \norm{\nabla_{x}b_k}_{\El{r}}&\lesssim \norm{\chi_{k}\nabla_{x}f}_{\El{r}}+\norm{\left(\nabla_{x}\chi_{k}\right)(f-(f)_{Q_k})}_{\El{r}}\\
    &\lesssim \norm{\nabla_{x}f}_{\El{r}(Q_k)}+\ell_{k}^{-1}\norm{f-(f)_{Q_k}}_{\El{r}(Q_k)}\lesssim \norm{\nabla_{x}f}_{\El{r}(Q_k)}
    \end{split}
\end{align}
by Poincaré's inequality. Moreover, we can crudely estimate, using \eqref{averages cubes potential} and \eqref{estimate cube type 2 from fefferman phong}, 
\begin{align*}
\norm{V^{1/2}b_k}_{\El{r}}^{r}&=\int_{\Rn}V^{r/2}\abs{f-(f)_{Q_k}}^{r}\abs{\chi_{k}}^{r}\lesssim \int_{Q_k}V^{r/2}\abs{f}^{r} + \int_{Q_k}\abs{(f)_{Q_k}}^{r}V^{r/2}\\
    &\lesssim\norm{V^{1/2}f}_{\El{r}(Q_k)}^{r} + \meas{Q_k}\left(\abs{(f)_{Q_k}}^{p}\av{Q_k}(V^{p/2})\right)^{r/p}\lesssim \norm{V^{1/2}f}_{\El{r}(Q_k)}^{r} +\alpha^{r}\meas{Q_k}.
\end{align*}
If $Q_k$ is type 1, then, as $p\leq r$, the Fefferman--Phong inequality (\ref{fefferman-phong ineq}) yields
\begin{align*}
    \abs{(f)_{Q_k}}\leq \left(\dashint_{Q_k}\abs{f}^{p}\right)^{1/p}\lesssim \ell_{k}\left(\dashint_{Q_k}\abs{\nabla_{\mu}f}^{p}\right)^{1/p}\lesssim \ell_{k}\left(\dashint_{Q_k}\abs{\nabla_{\mu}f}^{r}\right)^{1/r}.
\end{align*}
Consequently, using Poincaré's inequality as in \eqref{intermediate estimate on gradient of bk in the Lr norm}, we get
\begin{align*}
    \norm{\nabla_{x}b_{k}}_{\El{r}}&\leq \norm{\nabla_{x}\left(\chi_{k}\left(f-(f)_{Q_k}\right)\right)}_{\El{r}} + \abs{(f)_{Q_k}}\norm{\nabla_{x}\chi_{k}}_{\El{r}}\\
    &\lesssim \norm{\nabla_{x}f}_{\El{r}(Q_k)} + \abs{(f)_{Q_k}}\ell_{k}^{-1}\meas{Q_k}^{1/r}\lesssim \norm{\nabla_{\mu}f}_{\El{r}(Q_k)}.
\end{align*}
 In addition, we have $\norm{V^{1/2}b_k}_{\El{r}}^{r}\leq \norm{V^{1/2}f}_{\El{r}(Q_k)}^{r}$.
 This proves that $b_{k}\in\dot{\mathcal{V}}^{1,r}(\Rn)$, with $\norm{\nabla_{\mu}b_k}_{\El{r}}^{r}\lesssim \norm{\nabla_{\mu}f}_{\El{r}(Q_k)}^{r} + \alpha^{r}\meas{Q_k}$ for all $k\geq 1$. 
 
 Finally, we use the estimate above to prove the unconditional convergence of the series $\sum_{k\geq 1}b_k$ in $\dot{\mathcal{V}}^{1,r}(\Rn)$. For all finite subsets $K\subset \bb{N}$ we can write, by bounded overlap of the supports of the $b_k$'s,
\begin{align*}
    \norm{\sum_{k\in K}\nolimits\nabla_{\mu}b_k}_{\El{r}}^{r}\!=\int_{\Rn}\!\abs{\sum_{k\in K}\nolimits\nabla_{\mu}b_k}^{r}\lesssim \sum_{k\in K}\nolimits\int_{\Rn}\!\abs{\nabla_{\mu}b_k}^{r}\lesssim \sum_{k\in K}\nolimits\!\left(\norm{\nabla_{\mu}f}_{\El{r}(Q_k)}^{r}\!+\alpha^{r}\meas{Q_k}\right).
\end{align*}
By bounded overlap of $\set{Q_{k}}_{k\geq 1}$ and property (iii), we have
\begin{equation*}
    \sum_{k\geq 1}\left(\norm{\nabla_{\mu}f}_{\El{r}(Q_k)}^{r}+\alpha^{r}\meas{Q_k}\right)\lesssim 
 \norm{\nabla_{\mu}f}^{r}_{\El{r}} + \alpha^{r-p}\norm{\nabla_{\mu}f}_{\El{p}}^{p}<\infty.
\end{equation*}
This proves property (ix), and concludes the proof.
\end{proof}
\subsection{Embeddings $\El{p}\cap\El{2}\subseteq \bb{H}^{p}_{H}$ and $\dot{\mathcal{V}}^{1,p}\cap \El{2}\subseteq \bb{H}_{H}^{1,p}$ for $p\in (1,2]$}

In the rest of Section~\ref{section on identification of H adapted Hardy spaces}, we assume that $n\geq 3$, $q\geq \max\set{\frac{n}{2},2}$ and $V\in\textup{RH}^{q}(\Rn)$.
As a first step towards the proof of Theorem~\ref{summary of identifications of H adapted Hardys spaces},  we will use Lemma~\ref{lemma V-adapted Calderon--Zygmund--Sobolev decomposition} to prove the continuous inclusions $\El{p}\cap\El{2}\subseteq \bb{H}^{p}_{H}$ and $\dot{\mathcal{V}}^{1,p}\cap \El{2}\subseteq \bb{H}_{H}^{1,p}$ (for the respective $p$-(quasi)norms) for a subinterval of exponents $p\in(1,2]$ in Corollary~\ref{continuous inclusion L^p subset abstract H^p} below.

For any integer $\alpha\geq 1$, we define $\psi_{\alpha}(z):=z^{\alpha-\frac{1}{2}}(1+z)^{-3\alpha}$ for all $z\in S^{+}_{\pi}$, and consider the square function 
\begin{align}\label{definition of the square function psi alpha H for injection}
    \left(S_{\psi_{\alpha},H}f\right)(x):=\left(\iint_{\abs{x-y}<t}\abs{\psi_{\alpha}(t^2 H)f(y)}^{2}\frac{\dd t\dd y}{t^{n+1}}\right)^{1/2}
\end{align}
for all $f\in \Ell{2}$ and all $x\in\Rn$. We note that, since $\psi_{\alpha}\in\textup{H}^{\infty}_{0}$, it follows from McIntosh's theorem that $S_{\psi_{\alpha},H}f\in\Ell{2}$ with $\norm{S_{\psi_{\alpha},H}f}_{\El{2}}\lesssim \norm{f}_{\El{2}}$ for all $f\in\Ell{2}.$

We shall need the following extrapolation lemma, adapted from \cite[Lemma 9.10]{AuscherEgert}. Recall that $\dom{H^{1/2}}=\mathcal{V}^{1,2}(\Rn)$.
\begin{lem}\label{extrapolation lemma}
    If $s\in\left(p_{-}(H)\vee 1, 2\right]$, $p\in (s_{*}\vee 1, s)$ and there exists an $\alpha_{0}\geq 1$ such that  
    \begin{equation}\label{estimate square function q}
        \norm{S_{\psi_{\alpha},H}\left(H^{1/2}f\right)}_{\El{s}}\lesssim \norm{\nabla_{\mu}f}_{\El{s}}
    \end{equation}
    for all $f\in \dot{\mathcal{V}}^{1,s}(\Rn)\cap\mathcal{V}^{1,2}(\Rn)$ and all $\alpha\geq \alpha_0$, where the implicit constant is allowed to depend on $\alpha$, then there exists $\alpha_{1}\geq 1$, depending on $p$, $s$, $p_{-}(H)$, $n$ and $\alpha_0$, such that 
    \begin{align}\label{extrapolated estimate}
\norm{S_{\psi_{\alpha},H}\left(H^{1/2}f\right)}_{\El{p}}\lesssim \norm{\nabla_{\mu}f}_{\El{p}}
    \end{align}
    for all $f\in \dot{\mathcal{V}}^{1,p}(\Rn)\cap\mathcal{V}^{1,2}(\Rn)$ and all $\alpha\geq \alpha_{1}$.
\end{lem}
\begin{proof}
Let $p\in (s_{*}\vee 1, s)$. The proof of \eqref{extrapolated estimate} relies on the following weak-type estimate:
    \begin{align}\label{weak type estimate}
        \meas{\{x\in\Rn : S_{\psi_{\alpha},H}\left(H^{1/2}f\right) (x)>3\lambda\}} \lesssim \lambda^{-p_{0}}\norm{\nabla_{\mu}f}_{\Ell{p_{0}}}^{p_{0}}
    \end{align}
    for all $p_{0}\in(s_{*}\vee 1, s)$, $f\in \dot{\mathcal{V}}^{1,p_{0}}(\Rn)\cap\mathcal{V}^{1,2}(\Rn)$ and all $\lambda>0$. 
    Indeed, let us assume that (\ref{weak type estimate}) holds for some fixed $\alpha\geq \alpha_{0}$. We shall deduce that the extrapolated estimate (\ref{extrapolated estimate}) holds for the same $\alpha$. As a first step we pick some $p_{0}\in (s_{*}\vee 1, s)$ such that $p_{0}<p$. We proceed as in the start of the proof of Lemma~9.10 in \cite{AuscherEgert}, but our argument is more involved as we don't have a suitable analogue of the space $\mathcal{Z}(\Rn)$ used therein. We rely on the space $\smcp$ instead.

    The fractional power $H_{0}^{-1/2}=\left(-\Delta +V\right)^{-1/2}$ extends to a bounded invertible operator $\invH :\Ell{2}\to\dot{\mathcal{V}}^{1,2}(\Rn)$, as described in Section~\ref{subsection on reverse riesz transform bounds on Dziubanski hardy space for no coefficients case}.
    By \cite[Section~4.7]{Ma_etAl}, for all $1<p<n$, it holds that
    \begin{equation}\label{Lp-Lq boundedness inverse square root H_0}
        \norm{\invH f}_{\El{p^{*}}}\lesssim \norm{f}_{\El{p}}
    \end{equation}
    for all $f\in\Ell{2}\cap\Ell{p}.$
    Since $n\geq 3$, the estimate (\ref{Lp-Lq boundedness inverse square root H_0}) with $p=2_{*}>1$ implies that $\invH f\in\El{2}\cap\dot{\mathcal{V}}^{1,2}=\mathcal{V}^{1,2}=\dom{H^{1/2}}$ for all $f\in\El{2_{*}}\cap\El{2}$. This allows us to define a positive, sub-additive operator $T: \El{2}\cap\El{2_{*}}\to \El{2}$ by
    \begin{align*}
        Tf:= S_{\psi_{\alpha},H}\left(H^{1/2}\invH f\right)
    \end{align*}
    for all $f\in \El{2}\cap\El{2_{*}}$. The fact that $T$ maps into $\Ell{2}$ follows from the remark following (\ref{definition of the square function psi alpha H for injection}). 
    %Notice that $T$ clearly satisfies $Tf\geq 0$, $T(f+g)\leq Tf + Tg$, $T(af)=\abs{a}Tf$, and consequently $\abs{Tf-Tg}\leq T(f-g)$, for all $f,g\in \El{2}\cap\El{2_{*}}$ and $a\in\C$, pointwise on $\Rn$. 
    By \cite[Thm. 0.5, Thm. 5.10]{Shen_Schrodinger} (or Theorem~\ref{thm boundedness of Riesz transforms on Lp full range}), the Riesz transform $\nabla_{\mu}\invH $ is $\El{r}$-bounded for all $r\in(1,2]$. Let us denote $\textup{X}:=\bigcap_{1<r\leq 2}\Ell{r}$ and note that if $f\in \textup{X}$, then $\invH f\in\dot{\mathcal{V}}^{1,r}\cap\mathcal{V}^{1,2}$ for all $r\in(1,2]$. Since $s\in(1,2]$, we can use assumption \eqref{estimate square function q} to get
    \begin{align*}
\norm{Tf}_{\El{s}}=\norm{S_{\psi_{\alpha},H}\left(H^{1/2}\invH f\right)}_{\El{s}}\lesssim \norm{\nabla_{\mu}\invH f}_{\El{s}}\lesssim \norm{f}_{\El{s}}
    \end{align*}
    for all $f\in \textup{X}$.
    Similarly, since $1<p_{0}\leq 2$, the weak-type estimate (\ref{weak type estimate}) implies that
    \begin{align*}
        \meas{\set{x\in\Rn : (Tf)(x)> 3\lambda}}&\lesssim \lambda^{-p_{0}}\norm{\nabla_{\mu}\invH f}_{\El{p_{0}}}^{p_{0}}\lesssim \lambda^{-p_{0}}\norm{f}_{\El{p_{0}}}^{p_{0}}
    \end{align*}
    for all $f\in \textup{X}$. By sub-additivity of $T$ and the Marcinkiewicz interpolation theorem, we deduce that for all $p_{0}<r\leq s$ and $f\in \textup{X}$, it holds that 
    \begin{align}\label{extrap estimate}
        \norm{Tf}_{\El{r}}\lesssim \norm{f}_{\El{r}}.
    \end{align}
    Now, by \cite[Theorem 1.2]{Auscher-BenAli}, we know that for all $f\in C^{\infty}_{c}\left(\Rn\right)\subseteq \dom{H_{0}^{1/2}}$ it holds that  $H_{0}^{1/2}f \in \textup{X}$, with $\norm{H_{0}^{1/2}f}_{\El{r}}\lesssim\norm{\nabla_{\mu}f}_{\El{r}}$ for all $r\in (1,2]$. Moreover, since $\invH =H_{0}^{-1/2}$ on $\ran{H_{0}^{1/2}}$, for all $f\in C^{\infty}_{c}(\Rn)$ and all $r\in(p_{0},s]$ we obtain by \eqref{extrap estimate} that
    \begin{align}\label{estimate on square function of square root for dense subclass}
        \norm{S_{\psi_{\alpha},H} (H^{1/2}f)}_{\El{r}}=\norm{T(H_{0}^{1/2}f)}_{\El{r}}\lesssim \norm{H_{0}^{1/2}f}_{\El{r}}\lesssim \norm{\nabla_{\mu}f}_{\El{r}}.
    \end{align}
     Finally, for $r\in(p_{0},s]$ and $f\in \dot{\mathcal{V}}^{1,r}(\Rn)\cap \mathcal{V}^{1,2}(\Rn)$, we use Corollary~\ref{density in mixed V adapted sobolev spaces} to find a sequence $\left(f_k\right)_{k\geq 1}\subset C_{c}^{\infty}(\Rn)$ such that $f_k\to f$ in both $\mathcal{V}^{1,2}(\Rn)$ and $\dot{\mathcal{V}}^{1,r}\left(\Rn\right)$. Let us define $g_{k}:=H^{1/2}f_k$ and $g:=H^{1/2}f$. Since $g_{k}\to g$ in $\El{2}$, hence $\psi_{\alpha}(t^{2}H)g_{k}\to \psi_{\alpha}(t^{2}H)g$ in $\El{2}$ for all $t>0$, we have
    \begin{align*}
        \int_{B(x,t)}\abs{\psi_{\alpha}\left(t^{2}H\right)g}^{2}=\lim_{k\to\infty} \int_{B(x,t)}\abs{\psi_{\alpha}\left(t^{2}H\right)g_{k}}^{2}
    \end{align*}
    for all $t>0$ and all $x\in\Rn$. By Fatou's lemma, this implies that
    \begin{align*}
        \int_{0}^{\infty}\int_{\abs{x-y}<t}\abs{\psi_{\alpha}\left(t^{2}H\right)g(y)}^{2}\frac{\dd y\dd t}{t^{n+1}}\leq \liminf_{k\to\infty} \int_{0}^{\infty}\int_{B(x,t)}\abs{\psi_{\alpha}\left(t^{2}H\right)g_{k}(y)}^{2}\frac{\dd y\dd t}{t^{n+1}}.
    \end{align*}
    Raising this inequality to the power $r/2$ and using \eqref{estimate on square function of square root for dense subclass} and Fatou's lemma once more yields
    \begin{align*}
        \norm{S_{\psi_{\alpha},H}\left(g\right)}_{\El{r}}^{r}&\leq \liminf_{k\to\infty}\norm{S_{\psi_{\alpha},H}\left(g_k\right)}_{\El{r}}^{r}\lesssim \liminf_{k\to\infty}\norm{\nabla_{\mu}f_k}_{\El{r}}^{r}=\norm{\nabla_{\mu}f}_{\El{r}}^{r}.
    \end{align*}
    Since $p_{0}<p\leq s$, applying this with $r=p$ concludes this part of the proof.

    It remains to prove the weak-type estimate (\ref{weak type estimate}).
    Let $f\in \dot{\mathcal{V}}^{1,p_{0}}(\Rn)\cap\mathcal{V}^{1,2}(\Rn)$. Note that this implies that $f\in \dot{\mathcal{V}}^{1,s}(\Rn)$. Fix $\lambda>0$, and consider the Calderón--Zygmund decomposition of $f$ at height $\lambda$, from Lemma~\ref{lemma V-adapted Calderon--Zygmund--Sobolev decomposition}. We obtain a sequence of cubes $\{Q_k\}_{k\geq 1}$, functions $g\in\dot{\mathcal{V}}^{1,p_{0}}(\Rn)$, and $\{b_{k}\}_{k\geq 1}\subset \dot{\mathcal{V}}^{1,p_{0}}\cap\mathcal{V}^{1,2}$, such that $f=g+\sum_{k\geq 1}b_k$ with convergence in $\dot{\mathcal{V}}^{1,p_{0}}(\Rn)$. We let $\ell_{k}:=\ell(Q_k)$ for all $k\geq 1$.

    As in the proof of \cite[Lemma 9.10]{AuscherEgert}, we consider the bounded functions $\varphi_{\alpha}(z)=z^{\alpha}(1+z)^{-\alpha}$ for integers $\alpha\geq 1$. We can decompose $f$ further as $f=g+\tilde{g}+b$ where
    \begin{equation*}
        \tilde{g}:=\sum_{k\geq 1}\left(1-\varphi_{\alpha}(\ell_{k}^{2}H)\right)b_k \hspace{1cm}\textup{ and }\hspace{1cm}b:=\sum_{k\geq 1}\varphi_{\alpha}(\ell_{k}^{2}H)b_k.
    \end{equation*}
    This decomposition can be justified by proving that the series $\sum_{k\geq 1}\left(1-\varphi(\ell_{k}^{2}H)\right)b_k$ converges unconditionally in $\dot{\mathcal{V}}^{1,s}(\Rn)$, with the estimate
    \begin{equation*}
        \norm{\sum_{k\geq 1}\nolimits\nabla_{\mu}\left(1-\varphi_{\alpha}(\ell_{k}^{2}H)\right)b_k}_{\El{s}}^{s}\lesssim \lambda^{s-p}\norm{\nabla_{\mu}f}_{\El{p}}^{p}.
    \end{equation*}
    This is obtained by following the argument of step 4 in the proof of \cite[Lemma 9.10]{AuscherEgert}, replacing $L$ by $H$ and $\nabla_{x}$ by $\nabla_{\mu}$. In addition, Proposition~\ref{properties of critical numbers proposition} should be used instead of \cite[Theorem 6.2]{AuscherEgert}, and property (viii) of the Calderón--Zygmund--Sobolev decomposition from Lemma~\ref{lemma V-adapted Calderon--Zygmund--Sobolev decomposition} should be used instead of property (i') stated in the proof of \cite[Lemma 9.10]{AuscherEgert}.
    This justifies defining $\tilde{g}:=\sum_{k\geq 1}\left(1-\varphi_{\alpha}(\ell_{k}^{2}H)\right)b_k$ as an element of $\dot{\mathcal{V}}^{1,s}(\Rn)$ satisfying $\norm{\widetilde{g}}_{\dot{\mathcal{V}}^{1,s}}^{s}\lesssim\lambda^{s-p}\norm{f}^{p}_{\dot{\mathcal{V}}^{1,p}}$.

    Since $p_{0}\in\left(s_{*}\vee 1 , s\right)$, we have $s\in \left(p_{0},2\right]$ and $f\in \dot{\mathcal{V}}^{1,s}(\Rn)$. Property (ix) of the Calderón--Zygmund decomposition implies that $b_k\in \mathcal{V}^{1,s}(\Rn)$ for all $k\geq 1$, and that the series $\sum_{k\geq 1}b_k$ converges unconditionally in $\dot{\mathcal{V}}^{1,s}$. Since for any finite subset $K\subset \bb{N}$ we can write
    \begin{align*}
        \sum_{k\in K}\varphi_{\alpha}(\ell_{k}^{2}H)b_k=\sum_{k\in K}b_k - \sum_{k\in K} \left(1-\varphi_{\alpha}(\ell_{k}^{2}H)\right)b_k,
    \end{align*}
    this proves the convergence of the series $b:=\sum_{k\geq 1}\varphi_{\alpha}(\ell_{k}^{2}H)b_k$ in $\dot{\mathcal{V}}^{1,s}$. This justifies the decomposition $f=g+\tilde{g}+b$ in $\dot{\mathcal{V}}^{1,s}$.
    Now, since it is assumed that $f\in\mathcal{V}^{1,2}(\Rn)$, the previous argument can also be carried out for $s=2$, and yields convergence of all the series above in the $\dot{\mathcal{V}}^{1,2}$ norm. Therefore, $g,\Tilde{g}$ and $b$ are in $\dot{\mathcal{V}}^{1,s}(\Rn)\cap\dot{\mathcal{V}}^{1,2}(\Rn)$. 
    
    Let us now remark that, by density of $C^{\infty}_{c}(\Rn)$ in the intersection $\dot{\mathcal{V}}^{1,s}\cap\dot{\mathcal{V}}^{1,2}$ (Proposition~\ref{density lemma in homogeneous V adapted Sobolev spaces}), the extension of the square root $\dot{H}^{1/2}:\dot{\mathcal{V}}^{1,2}(\Rn)\to \Ell{2}$ still verifies the estimate (\ref{estimate square function q}) for all $f\in\dot{\mathcal{V}}^{1,2}(\Rn)\cap\dot{\mathcal{V}}^{1,s}(\Rn)$. By linearity of $\dot{H}^{1/2}$ and subadditivity of the square function operator $S_{\psi_{\alpha},H}$, we deduce that 
    \begin{align*}
    \set{x\in\Rn : S_{\psi_{\alpha},H}(H^{1/2}f)(x)>3\lambda}&\subseteq \set{x: S_{\psi_{\alpha},H}(\dot{H}^{1/2}g)(x)>\lambda} \cup \set{x : S_{\psi_{\alpha},H}(\dot{H}^{1/2}\tilde{g})(x)>\lambda}\\
    &\cup \set{x : S_{\psi_{\alpha},H}(\dot{H}^{1/2}b)(x)>\lambda}=:A_{1}\cup A_{2}\cup A_{3}.
    \end{align*}
It therefore remains to bound the measures of $A_1$, $A_2$ and $A_3$. First, by Markov's inequality, the assumption \eqref{estimate square function q} and property (vii) of the Calderón--Zygmund decomposition, we have
\begin{align*}
    \meas{A_1}&\leq \lambda^{-s}\norm{S_{\psi_{\alpha},H}(\dot{H}^{1/2}g)}_{\El{s}}^{s}\lesssim \lambda^{-s}\norm{\nabla_{\mu}g}_{\El{s}}^{s}\lesssim\lambda^{-p_{0}}\norm{\nabla_{\mu}f}_{\El{p_{0}}}^{p_{0}}.
\end{align*}
Similarly, we also obtain $\meas{A_2}\lesssim \lambda^{-p_{0}}\norm{\nabla_{\mu}f}_{\Ell{p_{0}}}^{p_{0}}$. To prove the remaining estimate for $\meas{A_3}$, as explained in step 6 of the proof of  \cite[Lemma 9.10]{AuscherEgert}, it suffices to prove the estimate
\begin{align*}
    \meas{\set{x\in\Rn | S_{\psi_{\alpha},H}(\dot{H}^{1/2}b_{F})(x)>\lambda}}\lesssim \lambda^{-p}\norm{\nabla_{\mu}f}_{\El{p}}^{p}
\end{align*}
uniformly for all finite subsets $F\subset \bb{N}$, where $b_F:=\sum_{k\in F}\varphi_{\alpha}\left(\ell_{k}^{2}H\right)b_k$. This estimate is obtained by following the argument in Steps $7$ and $8$ of the proof of \cite[Lemma 9.10]{AuscherEgert}, by replacing $L$ by $H$ and $\nabla$ by $\nabla_{\mu}$ appropriately, and by using Lemma~\ref{Obtaining L^2-L^q or L^p-L^2 odes for powers of resolvents} instead of \cite[Lemma 7.4]{AuscherEgert}. \end{proof}

We obtain the following important consequence.
\begin{cor}\label{continuous inclusion L^p subset abstract H^p}

    For $p\in\left(p_{-}(H)\vee 1 , 2\right]$, there is a continuous inclusion $\El{p}\cap\El{2}\subseteq \bb{H}^{p}_{H}$ for the respective $p$-(quasi)norms. For $p\in(p_{-}(H)_{*}\vee 1 ,2]$, there is a continuous inclusion $\dot{\mathcal{V}}^{1,p}\cap \El{2}\subseteq \bb{H}_{H}^{1,p}$ for the respective $p$-(quasi)norms.
\end{cor}
\begin{proof}
    It follows from McIntosh's theorem and the Kato square root estimate \eqref{Kato estimate} that the estimate \eqref{estimate square function q} holds for $s=2$ and $\alpha\in\N$. As explained in \cite[Remark 9.11]{AuscherEgert}, we can iterate Lemma~\ref{extrapolation lemma} to obtain that 
\begin{equation}\label{extrapolation estimate for induction}
\norm{S_{\psi_{\alpha},H}\left(H^{1/2}f\right)}_{\El{p}}\lesssim \norm{\nabla_{\mu}f}_{\El{p}}
    \end{equation}
    for all $p\in\left(p_{-}(H)_{*}\vee 1, 2\right]$, all sufficiently large $\alpha$ (possibly depending on $p$) and all $f\in \dot{\mathcal{V}}^{1,p}(\Rn)\cap\mathcal{V}^{1,2}(\Rn)$.
    The first continuous inclusion then follows from the argument of the proof of \cite[Proposition 9.12]{AuscherEgert}, upon replacing $L$ by $H$ and $\nabla_{x}$ by $\nabla_{\mu}$. In particular, we rely on the $\El{p}$-boundedness of $R_{H}$ established in Theorem~\ref{thm boundedness of Riesz transforms on Lp full range}. We also need to use Lemma~\ref{density lemma range dom L^p} instead of \cite[Lemma 7.2]{AuscherEgert}. The second inclusion is proved as in \cite[Proposition 9.13]{AuscherEgert}, using that $\dot{\mathcal{V}}^{1,p}\cap\mathcal{V}^{1,2}$ is dense in $\dot{\mathcal{V}}^{1,p}\cap\El{2}$ by Corollary~\ref{density in mixed V adapted sobolev spaces}.
\end{proof}

\subsection{Embeddings $\bb{H}^{p}_{H}\subseteq \El{p}\cap\El{2}$ and $\bb{H}^{1,p}_{H}\subseteq \dot{\mathcal{V}}^{1,p}\cap\El{2}$ for $p\in(1,2]$}
The main result of this section is Proposition~\ref{two continuous inclusions abstract H^p and H^1,p} below. We shall need the following lemma, which follows from the machinery used to treat general bisectorial operators $T$ satisfying the standard assumptions of \cite[Section~8.1]{AuscherEgert}. Here we use the operators $B$ and $D$ from Section~\ref{subsection on schrodinger operators H and DB BD}, and recall that $n\geq 3$, $q\geq \max\set{\frac{n}{2},2}$ and $V\in\textup{RH}^{q}(\Rn)$.
\begin{lem}\label{continuous inclusions q>2 abstract bisectorial adapted hardy spaces}
    Let $T\in\set{BD,DB}$. The embeddings below hold for the respective $p$-(quasi)norms:
    \begin{enumerate}[label=\emph{(\roman*)}]
    \item For all $p\in [2,\infty)$, there is a continuous inclusion $\clos{\ran{T}}\cap\El{p}(\Rn;\C^{n+2})\subseteq \bb{H}^{p}_{T}$.
    \item For all $p\in(1,2]$, there is a continuous inclusion $\bb{H}^{p}_{T}\subseteq \El{p}(\Rn;\C^{n+2})\cap\El{2}(\Rn;\C^{n+2})$. 
    \end{enumerate}
\end{lem}
\begin{proof}
   The first item is a bisectorial version of Part 2 in the proof of \cite[Theorem 9.7]{AuscherEgert}. In fact, if $\alpha>\frac{n}{2}$, $p\in [2,\infty)$ and ${\varphi(z)=z^{\alpha}(1+iz)^{-2\alpha}}$ for all $z\in S_{\pi/2}$, then it shows that  
    \begin{align}\label{intermediate tent space estimate of general extension operator for q geq 2}
        \norm{\bb{Q}_{\varphi,T}f}_{\tent{p}}\lesssim \norm{f}_{\El{p}}
    \end{align}
    for all $f\in\El{2}\cap\El{p}$. Therefore, if $f\in\clos{\ran{T}}\cap\El{p}(\Rn;\C^{n+2})$, we obtain that $f\in\bb{H}^{p}_{T}$ with $\norm{f}_{\bb{H}^{p}_{T}}\eqsim \norm{\bb{Q}_{\varphi,T}f}_{\tent{p}}\lesssim \norm{f}_{\El{p}}$.
    
To prove (ii), we cannot rely on interpolation as in the proof of \cite[Lemma 9.14]{AuscherEgert}, and we shall therefore rely on (i) and duality.
Let $p\in(1,2]$ and $f\in\bb{H}^{p}_{T}\subseteq\clos{\ran{T}}$. 
    Pick a function $\psi\in \textup{H}^{\infty}_{0}$ which is admissible in the definition of $\bb{H}^{p}_{T}$. If $\varphi\in  \Psi^{\infty}_{\infty}$ is chosen so that $\varphi\psi$ satisfies the Calderón--McIntosh reproducing formula, then for all $g\in\El{2}\cap\El{p'}$, it holds that 
    \begin{align*}
        \langle f,g\rangle_{\El{2}}=\lim_{\eps\to 0^{+}}\int_{\eps}^{1/\eps} \langle (\varphi\psi)(tT)f,g\rangle_{\El{2}}\frac{\dd t}{t}=\lim_{\eps\to 0^{+}}\int_{\eps}^{1/\eps} \langle \psi(tT)f,\varphi^{*}(tT^{*})g\rangle_{\El{2}}\frac{\dd t}{t}.
    \end{align*}
    
    In the formula above, we do not restrict the operator $T$ to the closure of its range. We can still make sense of $(\varphi\psi)(tT)\in\mathcal{L}(\El{2}(\Rn ; \C^{n+2}))$ since $\varphi\psi\in\textup{H}^{\infty}_{0}$. 
    Consequently, we can use \eqref{estimate duality in tent spaces} to obtain  \begin{align*}
        \abs{\langle f,g\rangle_{\El{2}}}&\leq \int_{0}^{\infty}\abs{\langle \psi(tT)f , \varphi^{*}(tT^{*})g\rangle_{\El{2}}}\frac{\dd t}{t}\leq\iint_{\Hn} \abs{(\bb{Q}_{\psi,T}f)(t,y)}\abs{(\bb{Q}_{\varphi^{*},T^{*}}g)(t,y)}\frac{\dd t\dd y}{t}\\
        &\lesssim \norm{\bb{Q}_{\psi, T}f}_{\tent{p}}\norm{\bb{Q}_{\varphi^{*},T^{*}}g}_{\tent{p'}}\lesssim \norm{f}_{\bb{H}^{p}_{T}} \norm{\bb{Q}_{\varphi^{*},T^{*}}g}_{\tent{p'}},
    \end{align*}
     by definition of the adapted Hardy space $\bb{H}^{p}_{T}$. Here it is important that we haven't restricted the operator $T$ to the closure of its range. Now, since $p'\geq 2$ and $T^{*}\in\set{DB^{*},B^{*}D}$ has the same properties as $T$, \eqref{intermediate tent space estimate of general extension operator for q geq 2} implies that $\abs{\langle f,g\rangle_{\El{2}}}\lesssim \norm{f}_{\bb{H}^{p}_{T}}\norm{g}_{\El{p'}}$. Since $g\in\El{2}\cap\El{p'}$ was arbitrary, this implies that $f\in\El{p}$ with the estimate $\norm{f}_{\El{p}}\lesssim \norm{f}_{\bb{H}^{p}_{T}}$.
\end{proof}
The previous lemma has the following consequence.
\begin{prop}\label{two continuous inclusions abstract H^p and H^1,p}
    Let $p\in(1,2]$. The following are continuous inclusions for the respective $p$-(quasi)norms:
    \begin{equation*}
        \bb{H}^{p}_{H}\subseteq \Ell{p}\cap\Ell{2} \quad\text{ and } \quad\bb{H}^{1,p}_{H}\subseteq \dot{\mathcal{V}}^{1,p}(\Rn)\cap\Ell{2}.
    \end{equation*} 
\end{prop}
\begin{proof}
    The first inclusion follows from $\bb{H}^{p}_{BD}=\bb{H}^{p}_{H}\oplus\bb{H}_{M}^{p}$ (see Section~\ref{subsubsection on mapping properties ofa dapted hardy spaces}) and Lemma~\ref{continuous inclusions q>2 abstract bisectorial adapted hardy spaces}.(ii).
    To prove the second inclusion, we first consider $f\in\bb{H}^{1,p}_{H}\cap\dom{H^{1/2}}$. 
    Then, Figure \ref{figure relations between adapted Hardy spaces} shows that $\begin{bmatrix}
        0\\
        \nabla_{\mu}f
    \end{bmatrix}\in\bb{H}^{p}_{DB}$, with $\norm{f}_{\bb{H}^{1,p}_{H}}\eqsim\norm{\begin{bmatrix}
            f\\
            0
        \end{bmatrix}}_{\bb{H}^{1,p}_{BD}}\eqsim \norm{\begin{bmatrix}
            0\\
            \nabla_{\mu}f
        \end{bmatrix}}_{\bb{H}^{p}_{DB}}$.
    Hence, Lemma~\ref{continuous inclusions q>2 abstract bisectorial adapted hardy spaces}.(ii) implies that $\nabla_{\mu}f\in\El{p}$, with
    \begin{equation}\label{intermediate step estimate inclusions hardy spaces}
        \norm{\nabla_{\mu}f}_{\El{p}}\lesssim \norm{\begin{bmatrix}
            0\\
            \nabla_{\mu}f
        \end{bmatrix}}_{\bb{H}^{p}_{DB}}\eqsim \norm{f}_{\bb{H}^{1,p}_{H}}.
    \end{equation}
    Now, if $f\in\bb{H}^{1,p}_{H}$ is arbitrary, then we can find (see \cite[Lemma 8.7]{AuscherEgert}) a sequence $(f_{k})_{k\geq 1}$ in $ \bb{H}^{1,p}_{H}\cap\dom{H^{1/2}}$ such that $f_{k}\to f$ in both $\El{2}$ and $\bb{H}^{1,p}_{H}$ as $k\to\infty$. Applying the estimate (\ref{intermediate step estimate inclusions hardy spaces}) to $f_{k}-f_{j}$, $k,j\geq 1$, shows that the sequence $(\nabla_{\mu}f_{k})_{k\geq 1}$ is Cauchy in $\Ell{p}$. Therefore, we can apply \cite[Lemma 2.1]{MorrisTurner} to deduce that $f$ is weakly differentiable, with $\nabla_{\mu}f\in\Ell{p}$, and that $\nabla_{\mu}f_k\to\nabla_{\mu}f$ in $\El{p}$ as $k\to\infty$. As a consequence, $f\in\dot{\mathcal{V}}^{1,p}\cap\El{2}$, and we can pass to the limit $k\to\infty$ in the estimate (\ref{intermediate step estimate inclusions hardy spaces}) applied to $f_k$ to conclude the proof.
\end{proof}
\subsection{Identification of $\bb{H}^{p}_{H}$ and $\bb{H}^{1,p}_{H}$ for $p\leq 2$}
In this section (see Corollary~\ref{characterisation abstract H^p p<2} below), we obtain identifications $\bb{H}^{p}_{H}=b\textup{H}^{p}_{V,\textup{pre}}(\Rn)$ and $\bb{H}^{1,p}_{H}\cap\mathcal{V}^{1,2}(\Rn)=\dot{\textup{H}}^{1,p}_{V,\textup{pre}}(\Rn)$ (up to equivalent $p$-quasinorms) for exponents $p$ in certain subintervals of $(\frac{n}{n+1},2]$. 

The following lemma is a substitute for \cite[Corollary 8.29]{AuscherEgert} that will be needed in the proof of Proposition~\ref{lemma continuous inclusions abstract hardy spaces for p<=1}. The definition of $\left(\bb{H}^{p}_{T},\eps,1\right)$-molecule is recalled in the proof.
\begin{lem}\label{abstract molecules are concrete dziubanski molecules}
    Let $\delta=2-\frac{n}{q}$ and $p\in (\frac{n}{n+ \frac{\delta}{2}},1]$. There exists $\eps_{0}>0$ such that all $\left(\bb{H}^{p}_{D},\eps_{0},1\right)$-molecules $m$ belong to $\textup{H}^{p}_{V,\textup{pre}}(\Rn;\C^{n+2})$, with $\norm{m}_{\Dz{p}}\lesssim 1$. 
\end{lem}
\begin{proof}
    By definition, if $m$ is a $\left(\bb{H}^{p}_{D},\eps_{0},1\right)$-molecule, there is a cube $Q\subset\Rn$ and a function $b\in\dom{D}$ such that $m=Db$ and
    \begin{align*}
        \norm{m}_{\El{2}(C_{j}(Q))} + \ell(Q)^{-1}\norm{b}_{\El{2}(C_{j}(Q))}\lesssim (2^{j}\ell{(Q)})^{\frac{n}{2}-\frac{n}{p}}2^{-j\eps_{0}}
    \end{align*}
    for all $j\geq 1$. This implies that $m\in\El{1}(\Rn;\C^{n+2})$. In fact,
    \begin{align}\label{molecules are in L1 with cube estimate}
        \norm{m}_{\El{1}}&\leq \sum_{j=1}^{\infty} \meas{C_{j}(Q)}^{1/2}\norm{m}_{\El{2}(C_{j}(Q))}\lesssim \sum_{j=1}^{\infty}2^{\frac{nj}{2}}\ell(Q)^{\frac{n}{2}}(2^{j}\ell{(Q)})^{\frac{n}{2}-\frac{n}{p}}2^{-j\eps_{0}}\lesssim \meas{Q}^{1-\frac{1}{p}}.
    \end{align}
    Writing $b=\begin{bmatrix}
        b_{\perp}\\
        (b_{\parallel},b_{\mu})
    \end{bmatrix}$ and $m=\begin{bmatrix}
        -\nabla_{\mu}^{*}(b_{\parallel},b_{\mu})\\
        -\nabla_{\mu}b_{\perp}
    \end{bmatrix},$ where $b_{\perp}\in\mathcal{V}^{1,2}(\Rn)$ and $(b_{\parallel},b_{\mu})\in\dom{\nabla_{\mu}^{*}}$, gives 
    \begin{align}\label{decomposition of m molecule}
        \abs{\int_{\Rn}m}&= \abs{\int_{\Rn}\nabla_{\mu}^{*}(b_{\parallel},b_{\mu})}+\abs{\int_{\Rn}\nabla_{x}b_{\perp}}+ \abs{\int_{\Rn}V^{1/2}b_{\perp}}=\abs{\int_{\Rn}V^{1/2}b_{\mu}} + \abs{\int_{\Rn}V^{1/2}b_{\perp}}.
    \end{align}
    To see the second equality, observe that since $b_{\perp}\in\dot{\textup{W}}^{1,1}(\Rn)$ and $n\geq 3$, we have $\int_{\Rn}\nabla_{x}b_{\perp}=0$ (this general property of $\dot{\textup{W}}^{1,1}(\Rn)$ functions for $n\geq 2$ follows from \cite[Theorem 4]{Hajłasz1995} with $m=1$). In our case, this can also be seen directly. Indeed, let $\varphi\in \smcp$ be such that $\supp{\varphi}\subseteq 2Q$ and $\varphi =1$ on $Q$. Then, define $\varphi_{j}(x)=\varphi\left(\frac{x}{2^{j}}\right)$ for all $x\in\Rn$ and $j\geq 1$. Since $b_{\perp}\in\mathcal{V}^{1,2}(\Rn)$ we are justified in writing
    \begin{align*}
        \abs{\int_{\Rn}\varphi_{j}\nabla_{x}b_{\perp}}&=\abs{\int_{\Rn}b_{\perp}\nabla_{x}\varphi_{j}}\lesssim 2^{-j}\int_{2^{j+1}Q\setminus 2^{j}Q}\abs{b_{\perp}}\leq 2^{-j}\norm{b_{\perp}}_{\El{2}(C_{j}(Q))}\meas{2^{j+1}Q}^{1/2}\\
        &\lesssim 2^{-j}\ell{(Q)}(2^{j}\ell{(Q)})^{\frac{n}{2}-\frac{n}{p}}2^{-j\eps_{0}}2^{\frac{jn}{2}}\meas{Q}^{1/2}\lesssim 2^{j(n-\frac{n}{p}-\eps_{0}-1)},
    \end{align*}
    where the implicit constants depend on $Q$. Since $\nabla_{x}b_{\perp}\in\Ell{1}$, it follows from dominated convergence that $\abs{\int_{\Rn}\nabla_{x}b_{\perp}}=\lim_{j\to \infty}\abs{\int_{\Rn}\varphi_{j}\nabla_{x}b_{\perp}}=0$.
    Similarly, as $\nabla_{\mu}^{*}(b_{\parallel},b_{\mu})\in\El{1}$, we have 
    \begin{align}\label{limits intermediate step proof of inclusion molecules in dziubanski}
        \int_{\Rn}\nabla_{\mu}^{*}(b_{\parallel},b_{\mu}) &= \lim_{j\to \infty}\int_{\Rn} \varphi_{j}\nabla_{\mu}^{*}(b_{\parallel},b_{\mu})=\lim_{j\to \infty}\langle (b_{\parallel},b_{\mu}) , \nabla_{\mu}\varphi_{j}\rangle_{\El{2}}\nonumber\\
        &=\lim_{j\to\infty}\int_{\Rn} b_{\parallel}\cdot \nabla_{x}\varphi_{j} + \lim_{j\to\infty} \int_{\Rn}b_{\mu}V^{1/2}\varphi_{j}=\int_{\Rn}V^{1/2}b_{\mu},
    \end{align}
    where we have used the same reasoning as above to show that the penultimate limit is equal to zero. In order to make \eqref{limits intermediate step proof of inclusion molecules in dziubanski} rigorous, we need to choose $\eps_{0}>0$ suitably large to ensure that $V^{1/2}b_{\mu}\in\Ell{1}$. 
    In fact, if $\hat{b}\in\set{b_{\mu},b_{\perp}}$ and $\kappa$ is the doubling constant from Lemma~\ref{self-improvement property of reverse holder weights}, we have
    \begin{align*}
        \int_{\Rn}V^{1/2}\abs{\hat{b}}&\leq \sum_{j\geq 1}\norm{V^{1/2}}_{\El{2}(C_{j}(Q))}\norm{b}_{\El{2}(C_{j}(Q))}\leq \sum_{j\geq 1}\left(\int_{2^{j+1}Q}V\right)^{1/2}\ell(Q)(2^{j}\ell(Q))^{\frac{n}{2}-\frac{n}{p}}2^{-j\eps_{0}}\\
        &\lesssim \ell(Q)^{1+\frac{n}{2}-\frac{n}{p}}\sum_{j\geq 1}2^{j\left(\frac{n}{2}-\frac{n}{p}-\eps_{0}\right)}\kappa^{\frac{j+1}{2}}\left(\int_{Q} V\right)^{1/2}\lesssim \meas{Q}^{1-\frac{1}{p}}\left(\ell(Q)^{2}\dashint_{Q} V\right)^{1/2},
    \end{align*}
    provided we choose $\eps_{0}> \frac{n}{2}-\frac{n}{p}+\frac{1}{2}\log_{2}{\kappa}$.
    Consequently, it follows from \eqref{molecules are in L1 with cube estimate} and \eqref{decomposition of m molecule} that
    \begin{align*}
        \abs{\int_{\Rn}m}\lesssim \meas{Q}^{1-\frac{1}{p}}\min\set{1 , \left(\ell(Q)^{2}\dashint_{Q} V\right)^{1/2}}.
    \end{align*}
    It follows that $m$ is (a generic multiple of) a $(p,\eps_{0})$-molecule associated to the cube $Q$ for $\Dz{p}$, in the sense of Definition \ref{molecule Dziubanski hardy space}. The desired conclusion follows from Theorem~\ref{molecules are in Dziubanski hardy space with uniform bound}.
\end{proof}
Using Lemma~\ref{abstract molecules are concrete dziubanski molecules}, we follow the proof of \cite[Lemma 9.14]{AuscherEgert} to obtain the next result.
\begin{prop}\label{lemma continuous inclusions abstract hardy spaces for p<=1}
    Let $\delta=2-\frac{n}{q}$ and $p\in (\frac{n}{n+\frac{\delta}{2}},1].$ There is a continuous inclusion $\bb{H}^{p}_{H}\subseteq b\textup{H}^{p}_{V,\textup{pre}}$ for the respective $p$-quasinorms.
    If in addition $p\in\mathcal{I}(V)$, then there is a continuous inclusion $\bb{H}^{1,p}_{H}\cap\mathcal{V}^{1,2}\subseteq \dot{\textup{H}}^{1,p}_{V,\textup{pre}}$ for the respective $p$-quasinorms.
\end{prop}
\begin{proof}
    We claim that there is a continuous inclusion $\bb{H}^{p}_{DB}\subseteq \textup{H}^{p}_{V,\textup{pre}}(\Rn;\C^{n+2})$ for the respective $p$-quasinorms.
    Let $\eps_{0}>0$ be the constant found in Lemma~\ref{abstract molecules are concrete dziubanski molecules}. Let $f\in\bb{H}^{p}_{DB}$ and let us choose some $M>\frac{n}{p}-\frac{n}{2}$. The molecular decomposition \cite[Theorem 8.17]{AuscherEgert} implies that 
 $f=\sum_{i=1}^{\infty}\lambda_{i}m_{i}$
    with (unconditional) convergence in $\El{2}$, where $m_i$ is a $(\bb{H}_{DB}^{p},\eps_{0},M)$-molecule for all $i\geq 1$, and such that $\norm{(\lambda_{i})_{i\geq 1}}_{\ell^{p}}\lesssim \norm{f}_{\bb{H}^{p}_{DB}}$. Writing $m_{i}:=DBb_{i}$ for $b_{i}\in\dom{DB}$, and using boundedness of $B$ we see that $m_{i}$ is a multiple of some $(\bb{H}^{p}_{D}, \eps_{0} ,1)$-molecule. Lemma~\ref{abstract molecules are concrete dziubanski molecules} implies that $m_{i}\in\textup{H}^{p}_{V,\textup{pre}}$ with ${\norm{m_i}_{\Dz{p}}\lesssim 1}$ uniformly for all $i\geq 1$. The claim follows from \eqref{reduction to uniform atomic estimate for Hardy Dziubanski spaces}. 

    Let us now prove the first inclusion. Let $f\in\bb{H}^{p}_{H}\cap\ran{b\nabla_{\mu}^{*}}$. Using Figure \ref{figure relations between adapted Hardy spaces}, going from the third to the fourth row, we obtain that $\begin{bmatrix}
        b^{-1}f\\
        0
    \end{bmatrix}\in\bb{H}^{p}_{DB}$, with $\norm{\begin{bmatrix}
        b^{-1}f\\
        0
    \end{bmatrix}}_{\bb{H}^{p}_{DB}}\eqsim \norm{\begin{bmatrix}
        f\\
        0
    \end{bmatrix}}_{\bb{H}^{p}_{BD}}\eqsim \norm{f}_{\bb{H}^{p}_{H}}$. The previous claim thus implies that $b^{-1}f\in\textup{H}^{p}_{V,\textup{pre}}(\Rn)$, with $\norm{b^{-1}f}_{\Dz{p}}\lesssim \norm{f}_{\bb{H}^{p}_{H}}$.
    
    Let us now consider an arbitrary $f\in\bb{H}^{p}_{H}$. We know from Figure \ref{figure relations between adapted Hardy spaces} that there is a sequence $(f_{k})_{k\geq 1}\subset \bb{H}^{p}_{H}\cap\ran{b\nabla_{\mu}^{*}}$ such that $f_{k}\to f$ in both $\bb{H}^{p}_{H}$ and $\El{2}$. 
    We claim that $g:=b^{-1}f\in\textup{H}^{p}_{V,\textup{pre}}(\Rn)$. 
    Indeed, since $g_{k}\to g$ in $\El{2}$ as $k\to \infty$, up to a subsequence we get $\mathcal{M}_{V}g(x)\leq \liminf_{k\to\infty}\mathcal{M}_{V}g_{k}(x)$ for almost every $x\in\Rn$ (see Section~\ref{subsubsection with definition of completion H^{1}_{V}}). Raising this inequality to the power $p$ and applying Fatou's lemma, we obtain
    \begin{align*}
        \norm{\mathcal{M}_{V}g}_{\El{p}}\leq \liminf_{k\to\infty}\norm{\mathcal{M}_{V}g_{k}}_{\El{p}}=\liminf_{k\to\infty}\norm{g_{k}}_{\Dz{p}}\lesssim \liminf_{k\to \infty}\norm{f_{k}}_{\bb{H}^{p}_{H}}=\norm{f}_{\bb{H}^{p}_{H}}<\infty.
    \end{align*}
    This means that $g=b^{-1}f\in\mathrm{H}^{p}_{V,\textup{pre}}(\Rn)$ with $\norm{b^{-1}f}_{\Dz{p}}\lesssim \norm{f}_{\bb{H}^{p}_{H}}$, as claimed. 

    We now prove the second inclusion for $p\in\mathcal{I}(V)$. Let $f\in\bb{H}^{1,p}_{H}\cap\mathcal{V}^{1,2}(\Rn)=\bb{H}^{1,p}_{H}\cap\dom{H^{1/2}}$. It follows from Figure \ref{figure relations between adapted Hardy spaces} that $\begin{bmatrix}
        0\\
        \nabla_{\mu}f
    \end{bmatrix}\in\bb{H}_{DB}^{p}\cap\ran{D}$, with 
    \begin{align*}
        \norm{f}_{\bb{H}^{1,p}_{H}}\eqsim \norm{\begin{bmatrix}
            f\\
            0
        \end{bmatrix}}_{\bb{H}^{1,p}_{BD}}\eqsim \norm{\begin{bmatrix}
            0\\
            \nabla_{\mu}f
        \end{bmatrix}}_{\bb{H}^{p}_{DB}}\gtrsim \norm{\nabla_{\mu}f}_{\Dz{p}}\gtrsim \norm{f}_{\dot{\textup{H}}^{1,p}_{V}},
    \end{align*}
     where we have used the earlier claim and the fact that $p\in\mathcal{I}(V)$ (see Section~\ref{subsection on reverse riesz transform bounds on Dziubanski hardy space for no coefficients case}).
\end{proof}
We now study the converse inclusions. The first result is an analogue of \cite[Lemma 9.17]{AuscherEgert}.
\begin{lem}\label{lemma continuous inclusion concrete hardy space for p<=1}
    For all $p\in(p_{-}(H),1]$, there is a continuous inclusion $b\textup{H}^{p}_{V,\textup{pre}}\subseteq \bb{H}^{p}_{H}$ for the respective $p$-quasinorms. 
\end{lem}
\begin{proof}
    Let $\alpha\geq 1$ be an integer, and consider the auxilliary function $\eta_{\alpha}(z)=z^{\alpha}(1+z)^{-2\alpha}$ for all $z\in S^{+}_{\pi}$. Then, $\eta_{\alpha}\in\Psi^{\alpha}_{\alpha}$ on all sectors. For $f\in\Ell{2},$ consider the square function $S_{\eta_{\alpha},H}f$ defined as 
    \begin{equation*}
        S_{\eta_{\alpha},H}f(x)=\left(\iint_{\abs{x-y}<t}\abs{(\eta_{\alpha}(t^{2}H)f)(y)}^{2}\frac{\dd t\dd y}{t^{1+n}}\right)^{1/2}.
    \end{equation*}
    By Table \ref{decay parameters conditions bisectorial Hardy adapted space}, we know that $\norm{S_{\eta_{\alpha},H}(\cdot)}_{\Ell{p}}$ is an equivalent norm on $\bb{H}^{p}_{H}$ provided we choose $\alpha>\frac{n}{2p}-\frac{n}{4}$. 
    If $f\in b\textup{H}^{p}_{V,\textup{pre}}(\Rn)$, then Theorem~\ref{abstract atomic decomposition DKKP} yields an $\El{2}$-convergent atomic decomposition $b^{-1}f=\sum_{i}\lambda_{i}m_{i}$, where $\norm{\set{\lambda_{i}}_{i}}_{\ell^{p}}\lesssim \norm{b^{-1}f}_{\Dz{p}}$ and the $m_i$ are abstract $\textup{H}^{p}_{V}$-atoms for $H_{0}$ (see Definition \ref{definition abstract atoms DKKP}). By the Fatou's lemma argument used in the proof of Lemma~\ref{extrapolation lemma}, we have
    \begin{align*}
        \norm{S_{\eta_{\alpha},H}(f)}_{\El{p}}^{p}\leq \sum_{i}\abs{\lambda_{i}}^{p}\norm{S_{\eta_{\alpha},H}(bm_{i})}_{\El{p}}^{p}.
    \end{align*}
    Consequently, it suffices to prove the existence of some $\alpha\geq 1$ large enough (depending on $n, p$ and $p_{-}(H)$) such that $\norm{S_{\eta_{\alpha},H}(bm)}_{\El{p}}\lesssim 1$ for all abstract $\textup{H}^{p}_{V}$-atoms $m$ for $H_{0}$. This is obtained by following the proof of \cite[Lemma 9.17]{AuscherEgert}, applying Lemma~\ref{property J(H)} instead of \cite[Lemma 6.3]{AuscherEgert} and Proposition~\ref{bound on power of family} instead of \cite[Lemma 4.4]{AuscherEgert}. Let us also note that for $\frac{n}{n+1}<r<p\leq 1$, it holds that  $\lfloor \frac{n}{2p}\rfloor +1 = \lfloor \frac{n}{2r}\rfloor +1$, and it is therefore clear from Definition \ref{definition abstract atoms DKKP} that if $m$ is an abstract $\textup{H}^{p}_{V}$-atom associated to a ball $B\subset\Rn$, then the multiple $\meas{B}^{\frac{1}{p}-\frac{1}{r}}m$ is an abstract $\textup{H}^{r}_{V}$-atom associated to the same ball. 
\end{proof}

The following result is proved by adapting \cite[Lemma 9.16]{AuscherEgert}.
\begin{lem}\label{Lemma continuous inclusion V adapted hardy sobolev space for p<=1}
    For $p\in(p_{-}(H)_{*}\vee \frac{n}{n+1}, 1]$ there is a continuous inclusion $\dot{\textup{H}}^{1,p}_{V,\textup{pre}}\subseteq \bb{H}^{1,p}_{H}\cap\mathcal{V}^{1,2}$ for the respective $p$-quasinorms. 
\end{lem}
\begin{proof}
    Let $\alpha\geq 1$ be an integer, and consider the auxilliary function $\eta_{\alpha}(z)=z^{\alpha}(1+z)^{-2\alpha}$. Then, $\eta_{\alpha}\in\Psi^{\alpha}_{\alpha}$ on all sectors. For $f\in\Ell{2},$ consider the square function $S_{\eta_{\alpha},H}^{(1)}f$ defined as 
    \begin{equation*}
        S_{\eta_{\alpha},H}^{(1)}f(x):=\left(\iint_{\abs{x-y}<t}\abs{t^{-1}(\eta_{\alpha}(t^{2}H)f)(y)}^{2}\frac{\dd t\dd y}{t^{1+n}}\right)^{1/2}.
    \end{equation*}
    We know that $\norm{S_{\eta_{\alpha},H}^{(1)}(\cdot)}_{\El{p}}$ is an equivalent norm on $\bb{H}^{1,p}_{H}$ provided we choose $\alpha>\frac{n}{2p}-\frac{n}{4}$ (see Table \ref{decay parameters conditions bisectorial Hardy adapted space}). Let $f\in\dot{\textup{H}}^{1,p}_{V}\cap\mathcal{V}^{1,2}$. By Theorem~\ref{thm atomic decomposition of hardy sobolev spaces} there is an $\El{2}$-convergent atomic decomposition $H_{0}^{1/2}f=\sum_{i}\lambda_{i}H_{0}^{1/2}m_i$,
    where the $m_{i}\in\dom{H_{0}^{1/2}}$ are $\El{2}$-atoms for $\dot{\textup{H}}^{1,p}_{V,\textup{pre}}(\Rn)$ and ${\norm{\set{\lambda_{i}}_{i}}_{\ell^{p}}\lesssim \norm{f}_{\dot{\textup{H}}^{1,p}_{V}}}$. By the Kato square root estimate \eqref{Kato estimate}, we have $\dom{H^{1/2}}=\dom{H_{0}^{1/2}}$ with $\norm{H^{1/2}f}_{\El{2}}\eqsim \norm{H_{0}^{1/2}f}_{\El{2}}$ for all $f\in\dom{H_{0}^{1/2}}=\mathcal{V}^{1,2}(\Rn)$. Consequently, we obtain an $\El{2}$-convergent decomposition $H^{1/2}f=\sum_{i}\lambda_{i}H^{1/2}m_i$.
    
    Since $H^{-1/2}\eta_{\alpha}(t^{2}H)\in\mathcal{L}(\El{2})$ for all $t>0$, we obtain $\eta_{\alpha}(t^{2}H)f=\sum_{i}\lambda_{i}\eta_{\alpha}(t^{2}H)m_{i}$, where the sum converges in $\El{2}$. As in the proof of Lemma~\ref{lemma continuous inclusion concrete hardy space for p<=1}, this implies that $\norm{S_{\eta_{\alpha},H}^{(1)}(f)}_{\El{p}}^{p}\leq \sum_{i}\abs{\lambda_{i}}^{p}\norm{S_{\eta_{\alpha},H}^{(1)}(m_{i})}_{\El{p}}^{p}$,
    and this shows that it suffices to obtain a uniform estimate of the form $\norm{S_{\eta_{\alpha},H}^{(1)}(m)}_{\Ell{p}}\lesssim 1$ for all $\El{2}$-atoms $m$ for $\dot{\textup{H}}^{1,p}_{V,\textup{pre}}(\Rn)$, for all sufficiently large $\alpha$.
    
    Let $m$ be an $\El{2}$-atom for $\dot{\textup{H}}^{1,p}_{V,\textup{pre}}(\Rn)$, associated to a cube $Q\subset \Rn$. For all $x\in\Rn$, it holds that 
$(S_{\eta_{\alpha},H}^{(1)}m)(x)=S_{\varphi_{\alpha},H}(H^{1/2}m)(x)$, where $\varphi_{\alpha}(z):=z^{\alpha-1/2}(1+z)^{-2\alpha}$. Using Hölder's inequality followed by McIntosh's theorem, we obtain the local bound
    \begin{align*}
        \norm{S_{\eta_{\alpha},H}^{(1)}(m)}_{\El{p}(4Q)}&\leq \meas{4Q}^{\frac{1}{p}-\frac{1}{2}}\norm{S_{\eta_{\alpha},H}^{(1)}(m)}_{\El{2}(4Q)}
        =\meas{4Q}^{\frac{1}{p}-\frac{1}{2}}\norm{S_{\varphi_{\alpha},H}(H^{1/2}m)}_{\El{2}(4Q)}\\
        &\lesssim \meas{4Q}^{\frac{1}{p}-\frac{1}{2}}\norm{H^{1/2}m}_{\Ell{2}}
        \eqsim\meas{4Q}^{\frac{1}{p}-\frac{1}{2}}\norm{H_{0}^{1/2}m}_{\El{2}}\lesssim 1.
    \end{align*}
 To obtain the global bound, we can follow the proof of \cite[Lemma 9.16]{AuscherEgert}, by choosing some $q\in (p_{-}(H),p^{*})\cap (1,2]$ and applying Lemma~\ref{Obtaining L^2-L^q or L^p-L^2 odes for powers of resolvents} instead of \cite[Lemma 7.4]{AuscherEgert}. This yields the global bound $\norm{S_{\eta_{\alpha},H}^{(1)}(m)}_{\El{p}(\Rn\setminus 4Q)}\lesssim \ell(Q)^{\frac{n}{p}-\frac{n}{q}-1}\norm{m}_{\El{q}}$ for suitably large $\alpha$.
    Since $m$ is supported in $Q$ and $q\leq 2<2^{*}$, the conclusion follows from Hölder's inequality and a Sobolev embedding as follows:
    \begin{align*}
        \norm{S_{\eta_{\alpha},H}^{(1)}(m)}_{\El{p}(\Rn\setminus 4Q)}&\lesssim \ell(Q)^{\frac{n}{p}-\frac{n}{q}-1}\norm{m}_{\El{q}}\leq \ell(Q)^{\frac{n}{p}-\frac{n}{2}}\norm{m}_{\El{2^{*}}(Q)}\lesssim \ell(Q)^{\frac{n}{p}-\frac{n}{2}}\norm{\nabla_{x}m}_{\El{2}}\\
        &\leq \ell(Q)^{\frac{n}{p}-\frac{n}{2}}\norm{\nabla_{\mu}m}_{\El{2}}
        =\ell(Q)^{\frac{n}{p}-\frac{n}{2}}\norm{H_{0}^{1/2}m}_{\El{2}}\lesssim 1.\qedhere
    \end{align*}
    \end{proof}

    We can now combine Corollary~\ref{continuous inclusion L^p subset abstract H^p}, Propositions \ref{two continuous inclusions abstract H^p and H^1,p} and \ref{lemma continuous inclusions abstract hardy spaces for p<=1}, and Lemmas \ref{lemma continuous inclusion concrete hardy space for p<=1} and \ref{Lemma continuous inclusion V adapted hardy sobolev space for p<=1}, to obtain the following consequence.
\begin{cor}\label{characterisation abstract H^p p<2}
Let $\delta=2-\frac{n}{q}$. The following identifications hold:
\begin{enumerate}[label=\emph{(\roman*)}]
    \item If $p\in(p_{-}(H)\vee \frac{n}{n+\delta/2}, 2]$, then $\bb{H}^{p}_{H}=b\textup{H}^{p}_{V,\textup{pre}}$ with equivalent $p$-(quasi)norms. 
    \item If $p\in(p_{-}(H)_{*}\vee \frac{n}{n+\delta/2} , 1]\cap\mathcal{I}(V)$, then $\bb{H}^{1,p}_{H}\cap\mathcal{V}^{1,2}=\dot{\textup{H}}^{1,p}_{V,\textup{pre}}$ with equivalent $p$-quasinorms.  
    \item If $p\in(p_{-}(H)_{*}\vee 1 , 2]$, then $\bb{H}^{1,p}_{H}=\dot{\mathcal{V}}^{1,p}\cap\El{2}$ with equivalent $p$-norms.
    \end{enumerate}
\end{cor}

\subsection{Identification of $\bb{H}^{p}_{H}$ and $\bb{H}^{1,p}_{H}$ for $p\geq 2$}
The following lemma provides an identification of the $H$-adapted spaces $\bb{H}^{p}_{H}$ and $\bb{H}^{1,p}_{H}$ in a subinterval of exponents $p\in(2,\infty)$ to conclude the proof of Theorem~\ref{summary of identifications of H adapted Hardys spaces}.
\begin{prop}\label{identification abstract hardy space H^p p>2}
    If $p\in(2,p_{+}(H))$, then $\bb{H}^{p}_{H}=\El{p}\cap\El{2}$ with equivalent $p$-norms. 
    If $p\in(2,q_{+}(H))$ and $p\leq 2q$, then $\bb{H}^{1,p}_{H}= \dot{\mathcal{V}}^{1,p}\cap\El{2}$ with equivalent $p$-norms.
\end{prop}
\begin{proof}
    The first identification follows both from the first component of the continuous inclusion of Lemma~\ref{continuous inclusions q>2 abstract bisectorial adapted hardy spaces}.(i) (with $T=BD$),  and by duality and similarity as in Part 8 of the proof of \cite[Theorem 9.7]{AuscherEgert}, by applying the first item of Corollary~\ref{characterisation abstract H^p p<2} to the operator $H^{\sharp}$ (recall Section~\ref{subsubsection on adjoints and similarity}).

    To prove the second identification, we need to establish two inclusions. First, if $p\in(2,p_{+}(H))$, we claim that $\dot{\mathcal{V}}^{1,p}\cap\El{2}\subseteq \bb{H}^{1,p}_{H}$.
    We follow the proof of (9.34) in \cite{AuscherEgert}. Let us first pick $f\in\dot{\mathcal{V}}^{1,p}\cap\mathcal{V}^{1,2}$. By duality and similarity, this implies that $p'\in (p_{-}(H^{\sharp})\vee 1 ,2)$. Since $H^{\sharp}$ satisfies the same hypotheses as $H$, Theorem~\ref{thm boundedness of Riesz transforms on Lp full range} implies that the Riesz transform $R_{H^{\sharp}}$ is $\El{p'}$-bounded. Now, for any $g\in\El{p'}\cap\ran{H^{*}}\cap\dom{H^{*}}$, we can write, using the similarity between $H^{*}$ and $H^{\sharp}$ (see Section~\ref{subsubsection on adjoints and similarity}),
    \begin{align*}
        \langle H^{1/2}f,g\rangle_{\El{2}}&=\langle f, {H^{*}}^{1/2}g\rangle_{\El{2}}=\langle f ,{b^{*}}^{-1}H^{\sharp}\left(b^{*}{H^{*}}^{-1/2}g\right)\rangle_{\El{2}}\\
        &=\langle \nabla_{x}f, A_{\parallel\parallel}^{*}\nabla_{x}\left(b^{*}{H^{*}}^{-1/2}g\right)\rangle_{\El{2}}+\langle V^{1/2}f, a^{*}V^{1/2}b^{*}{H^{*}}^{-1/2}g\rangle_{\El{2}}\\
        &=\left\langle \begin{bmatrix}
            A_{\parallel\parallel} & 0\\
            0 & a
        \end{bmatrix}\nabla_{\mu}f,\nabla_{\mu}{H^{\sharp}}^{-1/2}(b^{*}g)\right\rangle_{\El{2}}=\left\langle \begin{bmatrix}
            A_{\parallel\parallel} & 0\\
            0 & a
        \end{bmatrix}\nabla_{\mu}f,R_{H^{\sharp}}(b^{*}g)\right\rangle_{\El{2}},
    \end{align*}
    where the last equality follows from the characterisation of $R_{H^{\sharp}}$ on $\ran{{H^{\sharp}}}=b^{*}\ran{{H^{*}}}$.
    Following the argument of Part 9 of the proof of \cite[Theorem 9.7]{AuscherEgert}, using boundedness of $R_{H^{\sharp}}$, Lemma~\ref{density lemma range dom L^p},  Proposition~\ref{identification abstract hardy space H^p p>2} and Figure \ref{figure relations between adapted Hardy spaces}, we obtain the estimate $\norm{f}_{\bb{H}^{1,p}_{H}}\lesssim \norm{\nabla_{\mu}f}_{\El{p}}$. 
    Finally, we can use Corollary~\ref{density in mixed V adapted sobolev spaces} and the usual Fatou's lemma argument to extend to the general case $f\in\dot{\mathcal{V}}^{1,p}\cap\El{2}$.

    Next, we claim that if $p\in(2,q_{+}(H))$ and $p\leq 2q$, then $\bb{H}^{1,p}_{H}\subseteq \dot{\mathcal{V}}^{1,p}\cap\El{2}$. This is the analogue of (9.33) in Step 9 of the proof of \cite[Theorem 9.7]{AuscherEgert}. It follows from Figure \ref{figure relations between adapted Hardy spaces}, Proposition~\ref{properties of critical numbers proposition}, Proposition~\ref{identification abstract hardy space H^p p>2}, and Theorem~\ref{thm boundedness of Riesz transforms on Lp full range}, following the argument in the aforementioned reference. 
    Finally, one should conclude using \cite[Lemma 2.1]{MorrisTurner} as in the proof of Proposition~\ref{two continuous inclusions abstract H^p and H^1,p}. Since $p_{+}(H)>q_{+}(H)$ by Proposition~\ref{properties of critical numbers proposition}, the two inclusions conclude the proof.
\end{proof}

\section{Properties of weak solutions}\label{section on properties of weak solutions and interior estimates}
In this section, we collect some important features of the elliptic equation $\mathscr{H}u=g$ on open subsets of $\R^{1+n}$ (recall \eqref{definition of weak solutions in general open subset of R1+n}). We work with $n\geq 1$ and a general non-negative $V\in\El{1}_{\loc}(\Rn)$. The potential $V$ is extended to $\R^{1+n}$ by $t$-independence by setting $V(t,x):=V(x)$ for all $t\in\R$ and $x\in\Rn$. 
With that in mind, the space $\mathcal{V}^{1,2}(\R^{1+n})$ is defined as in \eqref{definition of inhomogeneous V adapted sobolev spaces}, and now we distinguish the components  $\nabla_{\mu}u:=\begin{bmatrix}
    \partial_{t}u\\
    \nabla_{x}u\\
    V^{1/2}u
\end{bmatrix}$ for a weakly differentiable function $u:\R^{1+n}\to\C$.

We first note that the $t$-independence of $V$ and of the coefficients $A$ implies that the ellipticity assumption (\ref{ellipticity assumption}) can be extended to $\bb{R}^{n+1}$. In fact, for $u\in C^{\infty}_{c}(\bb{R}^{n+1})$, we can integrate \eqref{ellipticity assumption} with respect to $t$ to obtain $\Re{\langle \mathcal{A}\nabla_{\mu}u , \nabla_{\mu}u \rangle_{\El{2}(\R^{n+1};\C^{n+2})}}\geq \lambda \norm{\nabla_{\mu}u}_{\El{2}(\R^{n+1}; \C^{n+2})}^{2}$, where $\mathcal{A}\in\El{\infty}\left(\R^{1+n};\mathcal{L}\left({\C^{n+2}}\right)\right)$ is defined as 
\begin{equation}\label{definition of mathcal A coefficients on the full space}
\mathcal{A}(t,x):=\begin{bmatrix}
        A(x) & 0\\
        0 & a(x)
    \end{bmatrix}:=\begin{bmatrix}
        A_{\perp \perp}(x) & 0 & 0\\
        0& A_{\parallel\parallel}(x) & 0\\
        0 & 0& a(x)
    \end{bmatrix}
        \end{equation}
    for all $(t,x)\in\R^{1+n}$. This ellipticity inequality extends to all $u\in\mathcal{V}^{1,2}(\R^{1+n})$ by Lemma~\ref{density lemma adapted sobolev spaces Morris Turner}. 

    Even though only the homogeneous norm appears here, we may not be able to extend the ellipticity to $\dot{\mathcal{V}}^{1,2}(\R^{1+n})$, because the $t$-extension $V(t,x):=V(x)$ may only belong to $\textup{RH}^{\frac{n}{2}}(\R^{1+n})$, which is not enough to apply Proposition~\ref{density lemma in homogeneous V adapted Sobolev spaces}.

We now record a Caccioppoli-type inequality and reverse Hölder estimates for weak solutions, as defined in \eqref{definition of weak solutions in general open subset of R1+n}.
If $\Omega\subset \bb{R}^{n+1}$ is open, $g\in\El{2}_{\textup{loc}}(\Omega)$, and $u\in\mathcal{V}^{1,2}_{\textup{loc}}(\Omega)$ is a weak solution of $\mathscr{H}u=g$ in $\Omega$, and $\alpha>1$, then
\begin{align}\label{Caccioppoli's inequality}
    \iint_{Q}\abs{\nabla_{\mu}u}^{2}\dd t\dd x\lesssim 
    \ell(Q)^{-2}\iint_{\alpha Q}\abs{u}^{2}\dd t\dd x + \left(\iint_{\alpha Q}\abs{g}^{2}\right)^{1/2}\left(\iint_{\alpha Q}\abs{u}^{2}\right)^{1/2}
\end{align}
for all cubes $Q\subset\R^{1+n}$ such that $\alpha \clos{Q}\subseteq \Omega$,  where the implicit constant depends on $\alpha.$
This type of Caccioppoli inequality can be proven using the same argument as in the proof of \cite[Proposition 5.1]{MorrisTurner}.
 Moreover, if $g\equiv 0$, and $\delta\in(0,\infty)$, then the reverse Hölder estimate
 \begin{equation}\label{reverse holder property weak solutions}
         \left(\fiint_{Q}\abs{u}^{2^{*}_{+}}\dd x\dd t\right)^{1/2^{*}_{+}}\lesssim \left(\fiint_{\alpha Q}\abs{u}^{\delta}\dd x\dd t\right)^{1/\delta}
     \end{equation}
holds (see \cite[Proposition 5.2]{MorrisTurner} or \cite[Lemma 16.7]{AuscherEgert}), where the implicit constant depends on $\alpha$ and $\delta$, and $2^{*}_{+}=\frac{2(n+1)}{(n+1)-2}$ is the upper Sobolev exponent in dimension $n+1$.
 
 The following result will be crucial for constructing solutions in Section~\ref{section on existence for the dirichlet and regularit problems}. The proof shows how the argument outlined in \cite[Lemma 16.8]{AuscherEgert} can be adapted to the case of the Schrödinger equations considered here.
 \begin{lem}\label{limit of weak solutions is weak solution to elliptic equation}
     Let $\Omega\subset \R^{n+1}$ be open, and let $g\in\El{2}_{\textup{loc}}(\Omega)$. Let $(u_k)_{k\geq 1}$ be a sequence of weak solutions of $\mathscr{H}u_{k}=g$ in $\Omega.$ Let $u\in\El{1}_{\textup{loc}}(\Omega)$ such that $u_k\to u$ in $\El{1}_{\textup{loc}}(\Omega)$ as $k\to\infty$. Then $u\in\mathcal{V}^{1,2}_{\textup{loc}}(\Omega)$ is a weak solution of $\mathscr{H}u=g$ in $\Omega$, and $u_k\to u$ in $\mathcal{V}^{1,2}_{\textup{loc}}(\Omega)$.
 \end{lem}
 \begin{proof}
 Let $k,j\geq 1$ be integers. By linearity, $u_k -u_j\in\mathcal{V}_{\textup{loc}}^{1,2}(\Omega)$ is a weak solution of $\mathscr{H}(u_k-u_j)=0$ in $\Omega$.
     Let $Q\subset \Omega$ be any cube such that $\clos{Q}\subset \Omega$. Then, there is $\alpha>1$ such that $\alpha \clos{Q}\subseteq \Omega$. By Caccioppoli's inequality \eqref{Caccioppoli's inequality} and Hölder's inequality, we have (where $2^{*}_{+}=\frac{2n+2}{n-1}$)
     \begin{align*}
         \ell(Q)\left(\fiint_{Q}\abs{\nabla_{\mu}u_k - \nabla_{\mu}u_j}^{2}\right)^{1/2}\lesssim \left(\fiint_{\alpha ^{1/2}Q}\abs{u_k -u_j}^{2}\right)^{1/2}\lesssim\left(\fiint_{\alpha ^{1/2}Q}\abs{u_k -u_j}^{2^{*}_{+}}\right)^{1/2^{*}_{+}}.
     \end{align*}
     Hence, the reverse Hölder estimate \eqref{reverse holder property weak solutions} with $\delta =1$ implies that
     \begin{align*}
        \left(\fiint_{Q}\abs{u_k -u_j}^{2}\right)^{1/2} +\left(\fiint_{Q}\abs{\nabla_{\mu}u_k - \nabla_{\mu}u_j}^{2}\right)^{1/2} \lesssim \left(1+\ell(Q)^{-1}\right)\fiint_{\alpha Q}\abs{u_k-u_j}.
     \end{align*}
     By a covering argument, this shows that for any open set $O\subset \Omega$ with compact closure $\clos{O}\subseteq \Omega$ (henceforth denoted $O\Subset \Omega$), there exists $C>0$, depending on $O$, and an open set $O'$ such that $O\subset O'\Subset \Omega$ and
     \begin{equation*}
         \norm{u_k -u_j}_{\mathcal{V}^{1,2}(O)}\leq C\norm{u_k-u_j}_{\El{1}(O')}.
     \end{equation*}
     This holds for any $j,k\geq 1$, and since $(u_k)_{k\geq 1}$ converges in $\El{1}_{\loc}(\Omega)$, this shows that $(u_k)_{k\geq 1}$ is Cauchy in $\mathcal{V}^{1,2}_{\loc}(\Omega)$. Let $v\in\mathcal{V}^{1,2}_{\loc}(\Omega)$ be such that $u_k\to v$ in $\mathcal{V}^{1,2}_{\loc}(\Omega)$ as $k\to\infty$. Since $u_k\to u$ in $\El{1}_{\loc}(\Omega)$, we deduce that $u=v$ almost everywhere in $\Omega$, thus $u\in\mathcal{V}^{1,2}_{\loc}(\Omega)$. Finally, we can take the limit as $k\to\infty$ in the weak formulation of $\mathscr{H}u_{k}=g$ to conclude that $u$ is a weak solution of $\mathscr{H}u=g$ in $\Omega$.
 \end{proof}
 \subsection{Poisson semigroup solutions}
 As stated in the introduction, block coefficients allow us to rewrite equation~(\ref{general div form schrodinger equ}) in the form $\partial_{t}^{2}u=Hu$, which allows us to construct solutions to boundary value problems for $\El{2}$ boundary data by considering the Poisson semigroup $\set{e^{-tH^{1/2}}}_{t>0}$. We now make this precise.

 Since $H$ is an injective sectorial operator of type $S_{2\omega(B)}^{+}$, we know that its square root $H^{1/2}$ is an injective sectorial operator of type $S_{\omega(B)}^{+}$ (see, e.g., \cite[Propositions 3.1.1 and 3.1.2]{HaaseFunctionalCalculusBook}) which generates a holomorphic semigroup $\set{e^{-\lambda H^{1/2}}}_{\lambda\in S_{\frac{\pi}{2}-\omega(B)}^{+}}$, 
 satisfying the properties stated in \cite[Proposition 3.4.1]{HaaseFunctionalCalculusBook}. Moreover, \cite[Proposition 3.1.4]{HaaseFunctionalCalculusBook} shows that for all $\lambda\in S^{+}_{\frac{\pi}{2}-\omega(B)}$, it holds that  $e^{-\lambda H^{1/2}}=f_{\lambda}(H^{1/2})=g_{\lambda}(H)$, where $f_{\lambda}(z)=e^{-\lambda z}$ and $g_{\lambda}(z)=e^{-\lambda \sqrt{z}}$ for all $z\in S^{+}_{\pi}$.
 
These properties show that for fixed $f\in\Ell{2}$, the mapping 
\begin{equation}\label{Poisson semigroup extension of f is continuous mapping srtongly}
\mathrm{P}:(0,\infty)\to\Ell{2}: t\mapsto e^{-tH^{1/2}}f
\end{equation}
is continuous and satisfies $\norm{P(t)}_{\Ell{2}}\lesssim \norm{f}_{\Ell{2}}$ uniformly for all $t>0$. A slight variation of the argument of the proof of \cite[Lemma 2.2]{Arendt_Bukhvalov_1994} shows that there exists a (measurable) $u\in\El{2}_{\loc}(\Hn)$ such that $u(t,x)=\mathrm{P}(t)(x)$ for almost every $t\in(0,\infty)$ and almost every $x\in\Rn$. This is how the notation $u(t,x):=(e^{-tH^{1/2}}f)(x)$ should be understood in Theorem~\ref{Poisson semigroup extension is a weak solution} below, and in the rest of the paper. In particular, $u(t,\cdot):=\mathrm{P}(t)$ for all $t>0$.
 
Note that in what follows, for a normed space $X$, we let $C_{0}([0,\infty); X)$ denote the space of continuous $X$-valued functions on $[0,\infty)$ vanishing at $\infty$, equipped with the supremum norm. 
To avoid confusion, we also mention that for any measurable function $v:\Hn\to \C$, the notation $v\in C_{0}\left([0,\infty); X\right)$, for a subspace $X$ of $\C$-valued measurable functions on $\Rn$, means that there exists some $F\in C_{0}\left([0,\infty); X\right)$ such that $v(t,x)=F(t)(x)$ for almost every $t>0$ and almost every $x\in\Rn$. This notation extends to the spaces different from $C_{0}$ used below.
\begin{thm}\label{Poisson semigroup extension is a weak solution}
    Let $f\in\Ell{2}$ and $u:\Hn\to \C$ be defined as $u(t,x):=(e^{-tH^{1/2}}f)(x)$
for all $t>0$ and $x\in\Rn$. Then, $u\in\mathcal{V}^{1,2}_{\loc}(\Hn)$, and it is a weak solution of $\mathscr{H}u=0$ in $\Hn$. Moreover, it is of class $C_{0}([0,\infty);\El{2})\cap C^{\infty}((0,\infty);\El{2})$.
\end{thm}
\begin{proof}
 The general properties of holomorphic semigroups imply that $u\in C_{0}([0,\infty);\El{2})\cap C^{\infty}((0,\infty);\El{2})$, with $\partial_{t}^{k}u(t,\cdot)=(-1)^{k}H^{k/2}u(t,\cdot)$ in the strong sense in $\El{2}$ for all $t>0$.
Moreover, for all $t>0$, we can use \eqref{decomposition property of BD into H and M} to write
\begin{align*}
    \begin{bmatrix}
        u(t,\cdot)\\
        0
    \end{bmatrix}=\begin{bmatrix}
        e^{-tH^{1/2}}f\\
        0
    \end{bmatrix}
    =B\begin{bmatrix}
        e^{-t\widetilde{H}^{1/2}}(b^{-1}f)\\
        0
    \end{bmatrix}=Be^{-t\sqrt{(DB)^{2}}}\begin{bmatrix}
        b^{-1}f\\
        0
    \end{bmatrix}
    =B e^{-t[DB]}\begin{bmatrix}
        b^{-1}f\\
        0
    \end{bmatrix},
\end{align*}
where we have defined $[z]:=\sqrt{z^{2}}$ for all $z\in S_{\pi/2}$. Consequently, since $\ran{e^{-t[DB]}}\subseteq\dom{[DB]}=\dom{DB}$, we may apply the operator $D$ to obtain 
\begin{align*}
    \begin{bmatrix}
        0\\
        t\nabla_{\mu}^{\parallel}u(t,\cdot)
    \end{bmatrix}=-tDBe^{-t[DB]}\begin{bmatrix}
        b^{-1}f\\
        0
    \end{bmatrix}.
\end{align*}
Since $t\partial_{t}u(t,\cdot)=-tH^{1/2}e^{-tH^{1/2}}f$, we may write for all $t>0$,
\begin{equation}\label{functional calculus representation of the nabla mu of u(t,)}
    t\nabla_{\mu}u(t,\cdot)=\begin{bmatrix}
        t\partial_{t}u(t,\cdot)\\
        t\nabla_{\mu}^{\parallel}u(t,\cdot)
    \end{bmatrix}=\begin{bmatrix}
        \psi(t^{2}H)f\\
        \left(\varphi(tDB)g\right)_{\parallel,\mu}
    \end{bmatrix},
\end{equation}
where $\psi(z)=-\sqrt{z}e^{-\sqrt{z}}$, $\varphi(z)=-ze^{-[z]}$, and $g:=\begin{bmatrix}
    b^{-1}f\\
    0
\end{bmatrix}$. It follows from \eqref{functional calculus representation of the nabla mu of u(t,)} and the (baby) convergence lemma \cite[Lemma 5.12]{HaaseFunctionalCalculusBook} that the map
\begin{equation*}
    (0,\infty)\to \El{2}(\Rn;\C^{n+2}): t\mapsto \nabla_{\mu}u(t,\cdot)
\end{equation*}
is continuous, and each of its $n+2$ components can therefore be identified with an element of $\El{2}_{\loc}(\Hn)$, as follows from the argument following \eqref{Poisson semigroup extension of f is continuous mapping srtongly} above. It follows that $u\in\El{2}_{\loc}(\Hn)$ has weak derivatives in $\El{2}_{\loc}(\Hn)$, thus $u\in\mathcal{V}^{1,2}_{\loc}(\Hn)$. 

We now prove that $u$ is a weak solution of $\mathscr{H}u=0$ in $\Hn$. Let $\phi\in C^{\infty}_{c}(\Hn)$ and observe that the slices $\phi(t,\cdot)$ belong to $\smcp$ for all $t>0$. Since $\partial_{t}^{2}u(t,\cdot)=Hu(t,\cdot)$ in the strong sense in $\El{2}$ for all $t>0$, and $b^{-1}=A_{\perp\perp}$ is independent of $t$ and bounded, we may use the form definition \eqref{form method def of H} of the operator $H$ to obtain
\begin{align}\label{intermediate equation form definition weak solution poisson semigroup}
    \langle \partial_{t}^{2}\left(A_{\perp\perp}u\right), \phi(t,\cdot)\rangle_{\Ell{2}}&=\langle b^{-1}Hu(t,\cdot), \phi(t,\cdot)\rangle_{\Ell{2}}=\langle A \nabla_{\mu}^{\parallel}u(t,\cdot), \nabla_{\mu}^{\parallel}\phi(t,\cdot)\rangle_{\El{2}(\Rn;\C^{n+1})}
\end{align}
for each $t>0$. Since both sides of \eqref{intermediate equation form definition weak solution poisson semigroup} are continuous functions of the variable $t$, and since $\phi$ has compact support, we may integrate with respect to $t$ to obtain
\begin{equation}\label{intermediate equation weak solution poisson semigroup}
    \int_{0}^{\infty}\langle \partial_{t}^{2}\left(A_{\perp\perp}u\right), \phi(t,\cdot)\rangle_{\El{2}} \dd t = \int_{0}^{\infty}\langle A \nabla_{\mu}^{\parallel}u(t,\cdot), \nabla_{\mu}^{\parallel}\phi(t,\cdot)\rangle_{\El{2}(\Rn;\C^{n+1})}\dd t.
\end{equation}
Using Fubini's theorem, we may write the integral on the left-hand side as
\begin{equation*}
    \int_{0}^{\infty}\langle \partial_{t}^{2}\left(A_{\perp\perp}u\right), \phi(t,\cdot)\rangle_{\El{2}} \dd t = \int_{\Rn}A_{\perp\perp}(x)\left(\int_{0}^{\infty}\partial_{t}^{2}u(t,x)\clos{\phi(t,x)}\dd t\right)\dd x.
\end{equation*}
We may integrate by parts the $\El{2}$-valued Bochner integral to obtain
$\int_{0}^{\infty}\partial_{t}^{2}u(t,\cdot)\clos{\phi(t,\cdot)}\dd t=-\int_{0}^{\infty}\partial_{t}u(t,\cdot)\clos{\partial_{t}\phi(t,\cdot)}\dd t$.
Applying Fubini's theorem again, \eqref{intermediate equation weak solution poisson semigroup} becomes
\begin{equation*}
    -\int_{0}^{\infty} \langle A_{\perp\perp}\partial_{t}u(t,\cdot) , \partial_{t}\phi(t,\cdot)\rangle_{\El{2}}\dd t = \int_{0}^{\infty}\langle A \nabla_{\mu}^{\parallel}u(t,\cdot), \nabla_{\mu}^{\parallel}\phi(t,\cdot)\rangle_{\El{2}(\Rn;\C^{n+1})}\dd t,
\end{equation*}
which means that $u$ is a weak solution of $\mathscr{H}u=0$ in $\Hn$.
\end{proof}
\section{Estimates for Poisson semigroup solutions}\label{section estimates towards the dirichelt and regularity problems}
Henceforth we assume that $n\geq 3$, $q\geq \max\{\frac{n}{2},2\}$, $V\in\textup{RH}^{q}(\Rn)$, and as always, let $\delta=2-\frac{n}{q}$. In this section, we shall derive the important square function and non-tangential maximal function bounds on Poisson semigroup solutions which will be central in the existence part of the proofs of Theorems \ref{existence and uniqueness result Dirichlet Lp}, \ref{existence and uniqueness result Regularity Lp}.
\subsection{Dirichlet data}
We first consider estimates for the Dirichlet problem. 
\subsubsection{Square function estimates}
The following result is adapted from \cite[Proposition 17.1]{AuscherEgert}.
\begin{prop}\label{SQuare function estimate dirichlet poisson semigroup}
     Let $p\in(p_{-}(H)\vee \frac{n}{n+\delta/2}\,,\,p_{+}(H)^{*})$. If $f\in b\textup{H}^{p}_{V,\textup{pre}}$, then $u(t,\cdot)=e^{-tH^{1/2}}f$ satisfies 
    \begin{equation}\label{equation square function estimate Poisson extension p< 2}
        \norm{S\left(t\nabla_{\mu}u\right)}_{\El{p}}\eqsim \norm{b^{-1}f}_{\textup{H}^{p}_{V}}.
    \end{equation}
\end{prop}
\begin{proof}
    We first treat the case $p_{-}(H)\vee \frac{n}{n+\delta/2}<p\leq 2$. We note that the functions $\psi$ and $\varphi$ in \eqref{functional calculus representation of the nabla mu of u(t,)} are such that $\psi\in\Psi^{\infty}_{1/2}$ on any sector, and $\varphi\in\Psi^{\infty}_{1}$ on any bisector. Moreover, the functions $\psi$ and $\varphi$ are admissible for the definition of the adapted spaces $\bb{H}^{p}_{H}$ and $\bb{H}_{DB}^{p}$, respectively. By Corollary~\ref{characterisation abstract H^p p<2}, the assumption on $p$ implies that $\bb{H}^{p}_{H}=b\textup{H}^{p}_{V,\textup{pre}}$ and thus for $f\in b\textup{H}^{p}_{V,\textup{pre}}$ we have
\begin{equation*}
\norm{b^{-1}f}_{\textup{H}^{p}_{V}}\eqsim\norm{f}_{\bb{H}^{p}_{H}}\eqsim\norm{S\left(\psi(t^{2}H)f\right)}_{\El{p}}= \norm{S\left(t\partial_{t}u(t,\cdot)\right)}_{\El{p}}\leq \norm{S\left(t\nabla_{\mu}u\right)}_{\El{p}},
\end{equation*}
and if $h=\begin{bmatrix}
        b^{-1}f\\
        0
    \end{bmatrix}$, then
\begin{align*}
\norm{S\left(t\nabla_{\mu}^{\parallel}u(t,\cdot)\right)}_{\El{p}}&
    \leq \norm{S\left(\varphi(tDB)h\right)}_{\El{p}}\eqsim \norm{\begin{bmatrix}
        b^{-1}f\\
        0
    \end{bmatrix}}_{\bb{H}^{p}_{DB}}\eqsim \norm{b^{-1}f}_{\bb{H}^{p}_{\widetilde{H}}}
    \eqsim \norm{f}_{\bb{H}^{p}_{H}}\eqsim \norm{b^{-1}f}_{\textup{H}^{p}_{V}},
\end{align*}
where we have used the decomposition $\bb{H}^{p}_{DB}=\bb{H}^{p}_{\widetilde{H}}\oplus\bb{H}^{p}_{\widetilde{M}}$ from Section~\ref{subsubsection on mapping properties ofa dapted hardy spaces}, and the similarity between $H$ and $\widetilde{H}$ from Section~\ref{subsubsection on adjoints and similarity}.
Those two estimates show that (\ref{equation square function estimate Poisson extension p< 2}) holds.

Next, we prove the upper bound in \eqref{equation square function estimate Poisson extension p< 2} for $2<p<p_{+}(H)^{*}$. Consider the function $\phi(z):=e^{-\sqrt{z}}-(1+z)^{-2}$ and define $v(t,\cdot):=\phi(t^{2}H)f$ for all $t>0$, so we have 
\begin{equation*}
    A_{\perp\perp}\partial_{t}^{2}v-\mathscr{H}v(t,\cdot)=t^{-2}A_{\perp\perp}\eta(t^{2}H)f,
\end{equation*}
where $\eta(z):=z(1+z)^{-2}+4z(1+z)^{-3}-24z^{2}(1+z)^{-4}$. Note that $\eta\in\Psi^{1}_{1}$ on any sector. 
If $g(t,x):=-t^{-2}A_{\perp\perp}(x)\left(\eta(t^{2}H)f\right)(x)$, then $g\in\El{2}_{\loc}(\Hn)$, and the argument of the proof of Theorem~\ref{Poisson semigroup extension is a weak solution} implies that $v\in\mathcal{V}^{1,2}_{\loc}(\Hn)$ is a weak solution of $\mathscr{H}v=g$ in $\Hn$.  Following the argument of the proof of \cite[Proposition 17.1]{AuscherEgert}, applying Caccioppoli's inequality \eqref{Caccioppoli's inequality} gives
\begin{equation*}
    \norm{S\left(t\nabla_{\mu}u\right)}_{\El{p}}\lesssim \norm{S\left(\phi(t^{2}H)f\right)}_{\El{p}}+\norm{S\left(\psi(t^{2}H)f\right)}_{\El{p}} + \norm{S\left(t\nabla_{\mu}(1+t^{2}H)^{-2}f\right)}_{\El{p}}.
\end{equation*}
Since $\phi\in\Psi^{2}_{1/2}$ and $\eta\in\Psi^{2}_{1}$ on any sector, applying \cite[Theorem 9.22]{AuscherEgert} with $\sigma =1/2$ gives
\begin{equation*}
\norm{S\left(\phi(t^{2}H)f\right)}_{\El{p}}+\norm{S\left(\psi(t^{2}H)f\right)}_{\El{p}}\lesssim \norm{f}_{\El{p}},
\end{equation*}
for all $2\leq p<\frac{np_{+}(H)}{n-p_{+}(H)}=p_{+}(H)^{*}$, with the understanding that $p_{+}(H)^{*}=\infty$ if $p_{+}(H)\geq n$. 
The remaining estimate, 
$\norm{S\left(t\nabla_{\mu}(1+t^{2}H)^{-2}f\right)}_{\El{p}}\lesssim \norm{f}_{\El{p}}$,
holds for any $p\in [2,\infty)$ and is obtained as in the proof of \cite[Proposition 17.1]{AuscherEgert}, by applying \cite[Remark 9.9]{AuscherEgert} and (\ref{Kato estimate}).

Finally, we prove the lower bound for $2<p<p_{+}(H)^{*}$. We follow the duality argument used in Step 3 of the proof of \cite[Proposition 17.1]{AuscherEgert} by introducing the adapted Schrödinger operator $\widetilde{H}_{0}:=b^{*}H_{0}$. Let $f\in\El{p}\cap\El{2}$ and $g\in\El{p'}\cap\El{2}$, and consider the infinitely differentiable function 
\begin{equation*}
    \phi : (0,\infty)\to \C : t\mapsto \langle A_{\perp\perp}e^{-tH^{1/2}}f , e^{-t\widetilde{H}_{0}^{1/2}}g\rangle_{\El{2}}.
\end{equation*}
The smoothness of $\phi$ follows from the properties of holomorphic semigroups on $\El{2}$.
Moreover, properties of the holomorphic functional calculus, and in particular the bounded $\textup{H}^{\infty}$-calculus of $H$ and $\widetilde{H}_{0}$, justify the identity 
\begin{align*}
    \langle A_{\perp\perp}f,g\rangle_{\El{2}}=\lim_{\eps\to 0^{+}}\int_{\eps}^{1/\eps}t^{2}\phi''(t)\frac{\dd t}{t}.
\end{align*}
Following the proof of \cite[Proposition 17.1]{AuscherEgert}, replacing $\nabla_x$ with $\nabla_{\mu}$ and $L$ with $H$, and using that $\nabla_{\mu}=\begin{bmatrix}
    \partial_{t}\\
    \nabla_{\mu}^{\parallel}
\end{bmatrix}$, we get $\abs{\langle A_{\perp\perp}f,g\rangle_{\El{2}}}\lesssim \norm{S\left(t\nabla_{\mu}e^{-tH^{1/2}}f\right)}_{\El{p}}\norm{S\left(t\nabla_{\mu}e^{-t\widetilde{H}_{0}^{1/2}}g\right)}_{\El{p'}}$.
Now, note that $p_{-}(\widetilde{H}_{0})\vee 1 =p_{-}(H_0)\vee 1=1$. Indeed, the first equality is proved exactly as in \cite[Theorem 6.9]{AuscherEgert}, and the second equality follows from the $\El{p}$-boundedness of the heat semigroup $\set{e^{-tH_{0}}}_{t>0}$ in \eqref{uniform boundedness heat semigroup schrodinger on Lp for all p} for all $p\in[1,\infty]$.  We can therefore apply \eqref{equation square function estimate Poisson extension p< 2} to $\widetilde{H}_{0}$ and $p'\in(1,2]$ to get
\begin{equation*}
    \abs{\langle A_{\perp\perp}f,g\rangle_{\El{2}}}\lesssim \norm{S\left(t\nabla_{\mu}e^{-tH^{1/2}}f\right)}_{\El{p}}\norm{g}_{\El{p'}}.
\end{equation*}
Since $g\in\El{p'}\cap\El{2}$ was arbitrary, this concludes the proof.
\end{proof}
\subsubsection{Non-tangential maximal functions estimates}
We next turn to the non-tangential maximal function bounds. We need the following result from the $\El{2}$-theory, which is an elaboration of \cite[Theorem 5.8, Proposition 5.9]{MorrisTurner}.
\begin{thm}\label{non tangential function L2 estimates Morris Turner}
   Let $T\in\set{BD,DB}$. Then, for all $f\in\clos{\ran{T}}$ it holds that  
    \begin{equation}\label{non tangential maximal function estimate abstract bisectorial functional calculus}
        \norm{N_{*}\left(e^{-t[T]}f\right)}_{\El{2}}\eqsim \norm{f}_{\El{2}}
    \end{equation}
    and 
    \begin{equation}\label{non tangential convergence of Whitney averages towards the boundary data in first order bisectorial functional calculus}
        \lim_{t\to 0^{+}}\fiint_{W(t,x)}\abs{e^{-s[T]}f - f(x)}^{2}\dd s\dd y =0
    \end{equation}
    for almost every $x\in\Rn$.
\end{thm}
\begin{proof}
Let us remark that the reference \cite[Theorem 5.8]{MorrisTurner} only proves the non-tangential maximal function estimate (\ref{non tangential maximal function estimate abstract bisectorial functional calculus}) for $T=DB$ and $f\in \mathcal{H}^{+}_{DB}$, a spectral subspace of $\clos{\ran{DB}}$. Let us first discuss how the proof can be adapted to cover the general case $f\in\clos{\ran{DB}}$. 

First, we note that the lower bound in (\ref{non tangential maximal function estimate abstract bisectorial functional calculus}) follows from exactly the same argument as in the start of the proof of \cite[Proposition 5.8]{MorrisTurner}. In fact, let  $f\in\clos{\ran{DB}}$ and $F(t)=e^{-t[DB]}f$ for all $t>0$. Since $\lim_{t\to 0^{+}}\norm{F(t)-f}_{\El{2}}=0$, it follows from \cite[Lemma 5.7]{MorrisTurner} that
\begin{align*}
    \norm{f}_{\El{2}}^{2}=\lim_{t\to 0^{+}}\dashint_{t}^{2t}\norm{F(s)}^{2}_{\El{2}}\dd s\lesssim \norm{N_{*}F}_{\El{2}}^{2}
\end{align*}
as claimed. Let us now focus on the reverse estimate. We recall that the spectral subspaces $\mathcal{H}^{+}_{DB}$,$\mathcal{H}^{-}_{DB}\subset \clos{\ran{DB}}$ are defined using the bounded $\textup{H}^{\infty}$-calculus of $DB$ as $\mathcal{H}^{\pm}_{DB}:=\set{\chi^{\pm}(DB)f \;:\; f\in\clos{\ran{DB}}}$,
where $\chi^{\pm}: S_{\pi/2}\to \C$ are defined as 
\begin{equation*}
    \chi^{\pm}(z)=\begin{cases}
        1, & \textup{ if }\pm\Re(z)>0;\\
        0, & \textup{ if }\pm \Re(z)<0,
    \end{cases}
\end{equation*}
for all $z\in S_{\pi/2}$. As in \cite[Section~4.2]{MorrisTurner} we introduce the Hilbert space $\mathcal{H}:=\clos{\ran{DB}}=\clos{\ran{D}}$. There is a topological direct sum decomposition $\mathcal{H}=\mathcal{H}^{+}_{DB}\oplus\mathcal{H}^{-}_{DB}$. 

For $f\in\mathcal{H}^{-}_{DB}$, the function $F\in \El{2}_{\loc}([0,\infty); \mathcal{H}^{-}_{DB})\cap C^{\infty}((0,\infty);\mathcal{H})$ defined as $F(t):=e^{-t[DB]}f$ satisfies the first-order equation $\partial_{t}F(t)-DBF(t)=0$, for all $t>0$, in the strong sense in $\El{2}(\Rn ; \C^{n+2})$ (see the proof of \cite[Proposition 4.4]{MorrisTurner}). Let us consider the reflection $G\in \El{2}_{\loc}((-\infty , 0]); \mathcal{H}^{-}_{DB})\cap C^{\infty}((-\infty , 0);\mathcal{H})$ defined as $G(t)=F(-t)$ for all $t\leq 0$. This new function $G$ satisfies the first-order equation $\partial_{t}G(t)+DBG(t)=0$ for all $t<0$, in the strong sense. Using $t$-independence of the coefficients $A$, $a$ and $V$, and by using the argument of the proof of \cite[Proposition 4.3 (2)]{MorrisTurner}, we can obtain a function $u\in\mathcal{V}^{1,2}_{\loc}(\R^{1+n}_{-})$, defined on the lower half-space $\R^{1+n}_{-}$, such that $G=\overline{\mathcal{A}}\nabla_{\mu}u$ and $\mathscr{H}u=0$ in $\R^{1+n}_{-}.$ The notation $\overline{\mathcal{A}}$ is introduced at the start of \cite[Section~4.1]{MorrisTurner}. 

Let us now consider an arbitrary Whitney cube $W:=(t,2t)\times Q(x,t)$ with $(t,x)\in \Hn$, and its reflection $-W:=(-2t,-t)\times Q(x,t)$. Let us now choose $p<2$ as in the proof of \cite[Theorem 5.8]{MorrisTurner}. We can therefore apply the reverse Hölder inequality \cite[Proposition 5.3]{MorrisTurner} in $\R^{1+n}_{-}$ to obtain, 
\begin{align*}
    \fiint_{W}\abs{F}^{2} &= \fiint_{-W}\abs{G}^{2} =\fiint_{-W} \abs{\overline{\mathcal{A}}\nabla_{\mu}u}^{2}\lesssim \fiint_{-W} \abs{\nabla_{\mu}u}^{2}\lesssim \left(\fiint_{-2W}\abs{\nabla_{\mu}u}^{p}\right)^{\frac{2}{p}}\\
&\lesssim \left(\fiint_{-2W}\abs{\overline{\mathcal{A}}\nabla_{\mu}u}^{p}\right)^{\frac{2}{p}}=\left(\fiint_{-2W}\abs{G}^{p}\right)^{\frac{2}{p}}=\left(\fiint_{2W}\abs{F}^{p}\right)^{\frac{2}{p}}.
\end{align*}
We can then follow the rest of the proof of \cite[Theorem 5.8]{MorrisTurner} to obtain the estimate $\norm{N_{*}F}_{\El{2}}^{2}\lesssim \norm{f}_{\El{2}}^{2}$. The desired conclusion \eqref{non tangential maximal function estimate abstract bisectorial functional calculus} for a general $f\in\mathcal{H}$ then follows from the topological decomposition $\mathcal{H}=\mathcal{H}^{+}_{DB}\oplus \mathcal{H}^{-}_{DB}$.

The result for $T=BD$ (more precisely, its injective part, i.e. its restriction to $\clos{\ran{BD}}$) and $f\in\clos{\ran{BD}}$ can be obtained as follows. First, note that the lower bound $\norm{N_{*}\left(e^{-t[T]}f\right)}_{\El{2}}\gtrsim \norm{f}_{\El{2}}$ follows from exactly the same argument as in the $DB$ case. Next, we can observe that if $f\in \clos{\ran{BD}}$, then $g:=(B|_{\clos{\ran{DB}}})^{-1}f\in\clos{\ran{DB}}$ with $\norm{f}_{\El{2}}\eqsim\norm{g}_{\El{2}}$, and the similarity (\ref{similarity DB and BD}) between $DB$ and $BD$ implies that $e^{-t[BD]}f=Be^{-t[DB]}g$ for all $t>0$. Consequently, 
\begin{align*}
    \norm{N_{*}\left(e^{-t[BD]}f\right)}_{\El{2}}&=\norm{N_{*}\left(Be^{-t[DB]}g\right)}_{\El{2}}\lesssim\norm{N_{*}\left(e^{-t[DB]}g\right)}_{\El{2}}\lesssim \norm{g}_{\El{2}}\eqsim \norm{f}_{\El{2}}
\end{align*}
follows from the boundedness of $B$ and the result for $DB$.

Similarly, we should note that the reference \cite[Proposition 5.9]{MorrisTurner} only proves the non-tangential convergence (\ref{non tangential convergence of Whitney averages towards the boundary data in first order bisectorial functional calculus}) for $T=DB$. However, the similarity argument above allows us to deduce the result for $BD$ from that of $DB$. Indeed, we can use boundedness of the multiplication operator $B$ and the triangle inequality to obtain, for arbitrary $t>0$ and $x\in\Rn$,
\begin{align*}
    &\fiint_{W(t,x)}\abs{e^{-s[BD]}f-f(x)}^{2}\dd s\dd y=\fiint_{W(t,x)}\abs{Be^{-s[DB]}g-(Bg)(x)}^{2}\dd s\dd y\\
    &\lesssim \fiint_{W(t,x)}\abs{(e^{-s[DB]}g)(y)-g(x)}^{2}\dd s\dd y + \dashint_{B(x,t)}\abs{g(x)-g(y)}^{2}\dd y \\
    &+ \dashint_{B(x,t)}\abs{(Bg)(y)-(Bg)(x)}^{2}\dd y.
\end{align*}
Since $g$ and $ Bg=f\in\El{2}(\Rn;\C^{n+2})$, the conclusion follows from Lebesgue's differentiation theorem and the result for $DB$.
\end{proof}
In order to prove the last part of Proposition~\ref{estimate non tangential maximal function on poisson extension solution} below, we shall need the following lemma.
\begin{lem}\label{lemma L^infty estimate nabla mu from L^infty estimate on H_0}
    Let $V\in\textup{RH}^{\infty}(\Rn)$, $x\in\Rn$ and $t>0$. If $g\in\dom{H_{0}}$ and $\supp{g}\subseteq B(x,t)$ with $\norm{H_{0}g}_{\El{\infty}}\lesssim t^{-n-2}$, then $\norm{\nabla_{\mu}g}_{\El{\infty}}\lesssim t^{-n-1}$.
\end{lem}
\begin{proof}
    Let $L=\set{(x,y)\in\Rn\times\Rn : x=y}$ and let $\Gamma^{V}:(\Rn\times\Rn)\setminus L \to \R$ denote the fundamental solution of the operator $H_{0}=-\Delta +V$ on $\Rn$. We refer to \cite[Section~5]{MayborodaPoggi_ExponDecayEstimates} for its construction and important properties, including that $\Gamma^{V}(\cdot,y)\in\El{1}_{\loc}(\Rn)\cap\El{\infty}_{\loc}(\Rn\setminus\set{y})$ for all $y\in\Rn$. Since $H_{0}g\in\El{\infty}_{c}(\Rn)$, it follows from \cite[Theorem 5.17]{MayborodaPoggi_ExponDecayEstimates} that 
    \begin{align*}
        g(z)=\int_{\Rn}\Gamma^{V}(z,y)\left(H_{0}g\right)(y)\dd y
    \end{align*}
    for almost every $z\in\Rn$. Also, as $V\in\textup{RH}^{n}(\Rn)$, \cite[Theorem 2.7, Remark 4.9]{Shen_Schrodinger} shows that 
    \begin{align*}
        \abs{\Gamma^{V}(z,y)}\lesssim (1+\abs{z-y}m(z))^{-1}\abs{z-y}^{-(n-2)},\hspace{1cm}
        \abs{\nabla_{z}\Gamma^{V}(z,y)}\lesssim\abs{z-y}^{-(n-1)},
    \end{align*}
    for all $y,z\in\Rn$ with $y\neq z$. Since $V\in\textup{RH}^{\infty}(\Rn)$, the first estimate above can be combined with the inequality $V(z)\lesssim \rho(z)^{-2}$ (see \eqref{Shen eqn control of integral of potential by the critical radius function}) to obtain $\abs{V^{1/2}(z)\Gamma^{V}(z,y)}\lesssim \abs{z-y}^{-(n-1)}$
    for all $y,z\in\Rn$ with $y\neq z$. Altogether, for all $z\in B(x,t)$, we have
    \begin{align*}
        \abs{\nabla_{\mu}g(z)}
        \lesssim \int_{B(x,t)} \abs{z-y}^{-(n-1)}t^{-n-2}\dd y
        \leq t^{-n-2}\int_{B(z,2t)}\abs{z-y}^{-(n-1)}\dd y
        \lesssim t^{-n-2}t=t^{-n-1}.
    \end{align*}
    Since $\nabla_{\mu}g$ is supported in $B(x,t)$, this finishes the proof.
\end{proof}
We have the following proposition, the proof of which is adapted from \cite[Proposition 17.3]{AuscherEgert}.
\begin{prop}\label{estimate non tangential maximal function on poisson extension solution}
    Let $p\in (p_{-}(H)\vee \frac{n}{n+\delta/2}\,,\,p_{+}(H)^{*})$. If $f\in b\textup{H}^{p}_{V,\textup{pre}}$, then $u(t,\cdot):=e^{-tH^{1/2}}f$ satisfies the estimate $\norm{N_{*}u}_{\El{p}}\lesssim \norm{b^{-1}f}_{\textup{H}^{p}_{V}}$
    and the non-tangential limit 
    \begin{equation}\label{non tangential limit for sirichlet problem in L2 POisson extension}
    \lim_{t\to 0^{+}}\fiint_{W(t,x)}\abs{u(s,y)-f(x)}^{2}\dd s\dd y=0
    \end{equation}
    for almost every $x\in\Rn$. If $p>1$, or if $V\in\textup{RH}^{\infty}$ and $\frac{n}{n+1}<p\leq 1$, then the reverse estimate also holds.
\end{prop}
\begin{proof}
     For $p\leq 2$, these are obtained as in the proof of \cite[Proposition 17.3]{AuscherEgert}, using Theorem~\ref{non tangential function L2 estimates Morris Turner} (instead of \cite[Theorem 17.2]{AuscherEgert}), \cite[Proposition 8.27]{AuscherEgert}, and Corollary~\ref{characterisation abstract H^p p<2}.

For $2<p<p_{+}(H)^{*}$, the estimate is proved exactly as in \cite[Proposition 17.3]{AuscherEgert}. 
In addition, when $1<p<p_{+}(H)^{*}$, the reverse estimate follows from the fact that the non-tangential limit implies that $\abs{f(x)}\leq N_{*}(u)(x)$ for almost every $x\in\Rn$. 

Finally, let us assume that $V\in\textup{RH}^{\infty}$. We shall prove that $\norm{b^{-1}f}_{\textup{H}^{p}_{V}}\lesssim \norm{N_{*}u}_{\El{p}}$ for $\frac{n}{n+1}<p\leq 1$ and $f\in b\textup{H}^{p}_{V,\textup{pre}}(\Rn)$. Let $\chi:[0,\infty)\to [0,1]$ be smooth, such that $\chi(s)=1$ for $s\in[0,1/2]$, and $\chi(s)=0$ for $s\geq 2$. For all $t>0$, we consider the function $\chi_{t}\in C^{\infty}_{c}([0,\infty))$ defined as $\chi_{t}(s):=\chi(\frac{s}{t})$ for all $s\geq 0$. 
Let us also pick $\varphi\in \mathcal{S}(\R)$ real-valued, even, and such that $\varphi(0)\neq 0$, with Fourier transform $\hat{\varphi}$ supported near the origin, so that the integral operators $\varphi(tH_{0}^{1/2})\in\mathcal{L}(\El{2})$, $t>0$, have integral kernels $K_{t}(x,y)$ satisfying $K_{t}(x,y)=0$ for $\abs{x-y}\geq t$, by \cite[Lemma 3.5]{HLMMY_HardySpacesDaviesGaffney}.
It follows from \cite[Proposition 5.3]{KP_HeatKernelBasedDecomposition} that for all $x\in\Rn$ and $t>0$, the function $K_{t}(x,\cdot)$ belongs to $\cap_{m\geq 1}\dom{H_{0}^{m}}\subseteq \dom{H_{0}}$.

For fixed $t>0$, let us introduce the function $\Phi^{t}:[0,\infty)\times\Rn\times\Rn\to \R$ defined as $\Phi^{t}(s,x,y):=K_{t}(x,y)\chi_{t}(s)$. Following the proof of \cite[Proposition 17.3]{AuscherEgert} we obtain for fixed $x\in\Rn$:
\begin{align*}
    \abs{\left[\varphi(tH_{0}^{1/2})(b^{-1}f)\right](x)}&\lesssim \iint_{\Hn}\abs{u(s,y)}\abs{\partial_{s}\Phi^{t}(s,x,y)}\dd s\dd y \\
    &+\iint_{\Hn}\abs{s\partial_{s}\Phi^{t}(s,x,y)}\abs{\partial_{s}u(s,y)}\dd s\dd y\\
    &+\iint_{\Hn}\abs{s\nabla_{\mu}^{\parallel}u(s,y)}\abs{\nabla_{\mu,y}^{\parallel}\Phi^{t}(s,x,y)}\dd s\dd y=:\RNum{1}+\RNum{2}+\RNum{3}.
\end{align*}
The notation $\nabla_{\mu,y}^{\parallel}$ refers to the operator $\nabla_{\mu}^{\parallel}$ acting with respect to the variable $y$. To estimate terms $\RNum{1}$ and $\RNum{2}$, we can use the support properties of $K_{t}(x,\cdot)$ and the derivative $\chi'$, as well as the uniform bound $\abs{K_{t}(x,y)}\lesssim t^{-n}$ (see \cite[Corollary 3.5]{KP_HeatKernelBasedDecomposition}), to obtain (recall $W(t,x):=(t/2,2t)\times B(x,t)$)
\begin{align*}
    \abs{\partial_{s}\Phi^{t}(s,x,y)}&=\abs{t^{-1}\chi'\left(\frac{s}{t}\right)K_{t}(x,y)}\lesssim t^{-n-1}\ind{B(x,t)}(y)\ind{(\frac{t}{2},2t)}(s)\lesssim \frac{1}{\meas{W(t,x)}}\ind{W(t,x)}(s,y)
\end{align*}
for all $s>0$ and $y\in\Rn$. To estimate term $\RNum{3}$, we note that the function $F_{t}:\Rn\to\C$, $F(y):=K_{t}(x,y)$, satisfies $F_{t}\in\dom{H_{0}}$, $\supp{F_{t}}\subseteq B(x,t)$ and $\norm{H_{0}F_{t}}_{\El{\infty}}\lesssim t^{-n-2}$ by \cite[Corollary 3.5]{KP_HeatKernelBasedDecomposition}. It follows from Lemma~\ref{lemma L^infty estimate nabla mu from L^infty estimate on H_0} that $\norm{\nabla_{\mu}F_{t}}_{\El{\infty}}\lesssim t^{-n-1}$. 
Since $\nabla_{\mu}F_{t}$ has compact support in $B(x,t)$, this implies the uniform estimate
\begin{align}\label{uniform estimate on Phi for reverse non tangential max function estimates}
    \abs{\nabla_{\mu,y}^{\parallel}\Phi^{t}(s,x,y)}\lesssim t^{-n-1}\ind{B(x,t)}(y)\ind{(0,2t)}(s)
\end{align}
for all $y\in\Rn$ and $s>0$. By Lemma~\ref{lemma equivalent characterisation Hp_V by DKKP}, these estimates allow us to conclude the proof as in \cite[Proposition 17.3]{AuscherEgert}, using Caccioppoli's inequality \eqref{Caccioppoli's inequality} in the final step.
\end{proof}

\subsubsection{Uniform estimates and strong continuity}
Next, we establish uniform bounds in $t>0$, and strong continuity at $t=0$ in $\El{p}$. As in \cite[Proposition 17.4]{AuscherEgert}, the following result follows from the identification $\bb{H}_{H}^{p}=b\textup{H}^{p}_{V,\textup{pre}}$ from Theorem~\ref{summary of identifications of H adapted Hardys spaces}.
\begin{prop}\label{uniform bounds Lp dirichlet problem poisson semigroup}
    Let $p\in (p_{-}(H)\vee \frac{n}{n+\delta/2}\,,\,p_{+}(H))$. If $f\in b\textup{H}^{p}_{V,\textup{pre}}$ and $u(t,\cdot)=e^{-tH^{1/2}}f$, then $b^{-1}u\in C_{0}([0,\infty);\textup{H}^{p}_{V,\textup{pre}})\cap C^{\infty}((0,\infty);\textup{H}^{p}_{V,\textup{pre}})$ with $u(0,\cdot)=f$, and satisfies 
    \begin{equation*}
\sup_{t>0}\norm{b^{-1}u(t,\cdot)}_{\textup{H}^{p}_{V}}\eqsim \norm{b^{-1}f}_{\textup{H}^{p}_{V}} \quad \textup{and}\quad \sup_{t>0}\norm{t^{k}\partial_{t}^{k}\left(b^{-1}u(t,\cdot)\right)}_{\textup{H}^{p}_{V}}\lesssim k^{k}e^{-k}\norm{b^{-1}f}_{\textup{H}^{p}_{V}}
    \end{equation*}
for all integers $k\geq 1$.
\end{prop}

For exponents $p\in[p_{+}(H), p_{+}(H)^{*})$, we don't have the identification $\bb{H}^{p}_{H}=\El{p}\cap\El{2}$, but we can still obtain some control of the solution in $\El{2}_{\loc}(\Rn)$. The relevant estimate is obtained by following the method outlined in \cite[Lemma 17.5]{AuscherEgert}. This yields the following result. 
\begin{prop}\label{local estimate time and space Poisson solution p>p+}
    If $p\in[p_{+}(H),p_{+}(H)^{*})$, then for all $f\in\El{p}\cap\El{2}$, all balls $B\subset\Rn$ and all $t>0$, it holds that 
    \begin{equation*}
        \norm{e^{-tH^{1/2}}f-f}_{\El{2}(B)}\lesssim \meas{B}^{\frac{1}{2}-\frac{1}{p}}\left(1+\frac{t}{r_{B}}
        \right)\norm{f}_{\El{p}},
    \end{equation*}
    and $\norm{e^{-tH^{1/2}}f}_{\El{2}(B)}\lesssim \meas{B}^{\frac{1}{2}-\frac{1}{p}}\left(2+\frac{t}{r_B}\right)\norm{f}_{\El{p}}$.
\end{prop}
\subsection{Regularity data}
We now consider estimates for the Regularity problem. 
\subsubsection{Square function estimates}
The following result is adapted from \cite[Proposition 17.6]{AuscherEgert}.
\begin{prop}\label{proposition  L^p square function estimate regularity problem}
Let $f\in\El{2}$ and $u(t,\cdot):=e^{-tH^{1/2}}f$. The following properties hold:
\begin{enumerate}[label=\emph{(\roman*)}]
\item If $p\in (p_{-}(H)_{*}, q_{+}(H))\cap(1,2q]$ and $f\in\mathcal{V}^{1,2}\cap\dot{\mathcal{V}}^{1,p}$, then $\norm{S\left(t\nabla_{\mu}\partial_{t}u\right)}_{\El{p}}\eqsim \norm{\nabla_{\mu}f}_{\El{p}}$. 
\item If $p\in\left( p_{-}(H)_{*}\vee \frac{n}{n+1} , 1\right]$ and $f\in\dot{\textup{H}}^{1,p}_{V,\textup{pre}}$, then $\norm{S\left(t\nabla_{\mu}\partial_{t}u\right)}_{\El{p}}\lesssim\norm{f}_{\dot{\textup{H}}^{1,p}_{V}}$. 
\item If $p\in\left( p_{-}(H)_{*}\vee \frac{n}{n+\delta/2} , 1\right]\cap\mathcal{I}(V)$ and $f\in\dot{\textup{H}}^{1,p}_{V,\textup{pre}}$, then $\norm{S\left(t\nabla_{\mu}\partial_{t}u\right)}_{\El{p}}\gtrsim\norm{f}_{\dot{\textup{H}}^{1,p}_{V}}$. 
\end{enumerate}
\end{prop}
\begin{proof}
    For the first item, we note that the assumptions on $p$ imply that $\dot{\mathcal{V}}^{1,p}\cap\El{2}=\bb{H}^{1,p}_{H}$ and $\bb{H}^{p}_{H}=\El{p}\cap\El{2}$ with equivalent $p$-norms (see Theorem~\ref{summary of identifications of H adapted Hardys spaces}).   
We can now follow the proof of \cite[Proposition 17.6]{AuscherEgert}, using Figure \ref{figure relations between adapted Hardy spaces} and the intertwining relations of Lemma~\ref{intertwining relations} to obtain the desired conclusion.
For the remaining items, we use the continuous inclusion $\dot{\textup{H}}^{1,p}_{V,\textup{pre}}\subseteq\bb{H}^{1,p}_{H}\cap\dom{H^{1/2}}$ from Lemma~\ref{Lemma continuous inclusion V adapted hardy sobolev space for p<=1} and the converse inclusion from Proposition~\ref{lemma continuous inclusions abstract hardy spaces for p<=1} along with the argument of the proof of \cite[Proposition 17.6]{AuscherEgert}.
\end{proof}

\subsubsection{Non-tangential maximal function estimates}
We now turn to the non-tangential maximal function bounds for the Regularity problem.
\begin{prop}\label{proposition estimate non tangential ùax function regularity problem p>1}
    Let $p\in(p_{-}(H)_{*}\vee 1 \,,\, q_{+}(H))$ with $p\leq 2q$. If $f\in\mathcal{V}^{1,2}\cap\dot{\mathcal{V}}^{1,p}$, then $u(t,\cdot):=e^{-tH^{1/2}}f$ satisfies $\norm{N_{*}\left(\nabla_{\mu}u\right)}_{\El{p}}\eqsim \norm{\nabla_{\mu}f}_{\El{p}}$
    and 
    \begin{equation*}
        \lim_{t\to 0^{+}}\fiint_{W(t,x)}\abs{\begin{bmatrix}
            A_{\perp\perp}\partial_{t}u\\
            \nabla_{\mu}^{\parallel}u
        \end{bmatrix}-\begin{bmatrix}
            -(A_{\perp\perp}H^{1/2}f)(x)\\
            \nabla_{\mu}f(x)
        \end{bmatrix}}^{2}\dd s\dd y=0
    \end{equation*}
    for almost every $x\in\Rn$.
\end{prop}
\begin{proof}
    We follow the proof of \cite[Proposition 17.7]{AuscherEgert}, using Lemma~\ref{intertwining relations}, Theorem~\ref{summary of identifications of H adapted Hardys spaces} and \eqref{identities square BD; H, M operators}. We also use Theorem~\ref{non tangential function L2 estimates Morris Turner} instead of \cite[Theorem 17.2]{AuscherEgert}.
\end{proof}
For $p\leq 1$, we have the following analogous result.
\begin{prop}\label{proposition non tangential max function estimate regularity problem for p<1}
    Let $p\in\left(p_{-}(H)_{*}\vee\frac{n}{n+1}\,,\,1\right]$. If $f\in\dot{\textup{H}}^{1,p}_{V,\textup{pre}}$, then $u(t,\cdot):=e^{-tH^{1/2}}f$ satisfies the estimate
   $\norm{N_{*}\left(\nabla_{\mu}u\right)}_{\El{p}}\lesssim \norm{f}_{\dot{\textup{H}}^{1,p}_{V}}$, and
    \begin{equation*}
        \lim_{t\to 0^{+}}\fiint_{W(t,x)}\abs{\begin{bmatrix}
            A_{\perp\perp}\partial_{t}u\\
            \nabla_{\mu}^{\parallel}u
        \end{bmatrix}-\begin{bmatrix}
            -(A_{\perp\perp}H^{1/2}f)(x)\\
            \nabla_{\mu}f(x)
        \end{bmatrix}}^{2}\dd s\dd y=0
    \end{equation*}
    for almost every $x\in\Rn$. If $V\in\textup{RH}^{\infty}$ and $p\in\mathcal{I}(V)$, then the reverse estimate also holds.
\end{prop}
\begin{proof}
    We follow the proof of the previous proposition, using \cite[Proposition 8.27]{AuscherEgert} in the range $p\leq 1$, as well as the norm equivalences of Figure \ref{figure relations between adapted Hardy spaces}. It then remains to prove the estimate $\norm{f}_{\dot{\textup{H}}^{1,p}_{V}}\lesssim \norm{N_{*}\left(\nabla_{\mu}u\right)}_{\El{p}}$, which we claim holds for all $\frac{n}{n+1}<p\leq 1$, under the assumption that $V\in\textup{RH}^{\infty}$ and $p\in\mathcal{I}(V)$. 
    We follow and adapt the third step of the proof of \cite[Proposition 17.7]{AuscherEgert}. Recall the function $\varphi\in\mathcal{S}(\R)$ and, for fixed $t>0$, the function $\Phi^{t}$ from the proof of Proposition~\ref{estimate non tangential maximal function on poisson extension solution}.
For $t>0$, we let $\varphi(tH_{0}^{1/2})$ act componentwise on $\nabla_{\mu}f\in\El{2}(\Rn;\C^{n+1})$. By the intertwining relations of Proposition~\ref{intertwining relations}, we note that $u(t,\cdot)\in\mathcal{V}^{1,2}(\Rn)$ and $\nabla_{\mu}^{\parallel}u(t,\cdot)=e^{-t\tilde{M}^{1/2}}\left(\nabla_{\mu}f\right)$, and therefore $\nabla_{\mu}^{\parallel}u(t,\cdot)=e^{-t\tilde{M}^{1/2}}\left(\nabla_{\mu}f\right)\to\nabla_{\mu}f$ as $t\to 0$, in the $\El{2}$ norm. We can therefore follow the proof of \cite[Proposition 17.7]{AuscherEgert} to obtain that for all fixed $t>0$ and $x\in\Rn$ it holds that  
\begin{align*}
    \abs{\varphi(tH_{0}^{1/2})\left(\nabla_{\mu}f\right)(x)}\leq \iint_{\Hn}\abs{\partial_{s}\Phi^{t}\; \nabla_{\mu}^{\parallel}u(s,y)}\dd s\dd y + \iint_{\Hn}\abs{\partial_{s}u \;\nabla_{\mu,y}^{\parallel} \Phi^{t}(x,y,s)}\dd s\dd y.
\end{align*}
Since $V\in\textup{RH}^{\infty}$, we can apply the pointwise estimate (\ref{uniform estimate on Phi for reverse non tangential max function estimates}) and conclude the proof as in \cite[Proposition 17.7]{AuscherEgert} to obtain (recall Lemma~\ref{lemma equivalent characterisation Hp_V by DKKP})
\begin{align*}
\norm{\nabla_{\mu}f}_{\textup{H}^{p}_{V}}\eqsim\norm{\mathcal{M}_{V}(\nabla_{\mu}f\, ,\, \varphi)}_{\El{p}}\lesssim \norm{N_{*}\left(\nabla_{\mu}u\right)}_{\El{p}}.
\end{align*}
The conclusion follows from the assumption $p\in\mathcal{I}(V)$ (see Section~\ref{subsection on reverse riesz transform bounds on Dziubanski hardy space for no coefficients case}).
\end{proof}
\subsubsection{Uniform estimates and strong continuity}
Finally, we follow \cite[Proposition 17.8]{AuscherEgert} to obtain uniform boundedness and strong continuity at $t=0$.
\begin{prop}\label{proposition strong continuity of regularity poisson solution p>1}
    Let $p\in (p_{-}(H)_{*}\vee 1 \,,\, q_{+}(H))$ with $p\leq 2q$. If $f\in\dot{\mathcal{V}}^{1,p}\cap\mathcal{V}^{1,2}$, then $u(t,\cdot):=e^{-tH^{1/2}}f$ has the following properties:
    \begin{enumerate}[label=(\roman*)]
        \item It holds that  $\nabla_{\mu}^{\parallel}u\in C_{0}([0,\infty) ; \El{p})\cap C^{\infty}((0,\infty);\El{p})$, with
        \begin{equation*}
            \sup_{t>0}\norm{\nabla_{\mu}^{\parallel}u(t,\cdot)}_{\El{p}}\eqsim \norm{\nabla_{\mu}f}_{\El{p}}\quad \textup{and}\quad
\sup_{t>0}\norm{t^{k}\partial_{t}^{k}\nabla_{\mu}^{\parallel}u(t,\cdot)}_{\El{p}}\lesssim k^{k}e^{-k}\norm{\nabla_{\mu}f}_{\El{p}}
        \end{equation*}
        for all integers $k\geq 1$.
        \item If $p<n$ and $V\not\equiv 0$, then $f\in\El{p^{*}}$ and $u\in C_{0}([0,\infty) ; \El{p^{*}})\cap C^{\infty}((0,\infty);\El{p^{*}})$, with 
        \begin{equation*}
             \sup_{t>0}\norm{u(t,\cdot)}_{\El{p^{*}}}\eqsim\norm{f}_{\El{p^{*}}}.
        \end{equation*}
    \end{enumerate}
\end{prop}
\begin{proof} 
    Since $f\in\mathcal{V}^{1,2}(\Rn)$, the intertwining relations (Lemma~\ref{intertwining relations}) imply that $u(t,\cdot)\in\mathcal{V}^{1,2}(\Rn)$ for each $t\geq 0$, with
    \begin{equation*}
        \nabla_{\mu}^{\parallel}u(t,\cdot)=\nabla_{\mu}\left(e^{-tH^{1/2}}f\right)=e^{-t\widetilde{M}^{1/2}}(\nabla_{\mu}f).
    \end{equation*}
    The argument of the proof of Proposition~\ref{proposition  L^p square function estimate regularity problem} shows that, in the range of exponents $p$ considered, it holds that  $\norm{g}_{\bb{H}^{p}_{\widetilde{M}}}\eqsim \norm{g}_{\El{p}}$ for all $g\in\bb{H}^{p}_{\widetilde{M}}\cap\ran{\nabla_{\mu}}$. Moreover, it shows that $\nabla_{\mu}f\in\bb{H}^{p}_{\widetilde{M}}$. We can therefore apply the abstract theory of \cite[Propositions 8.10 and 8.13]{AuscherEgert} to the sectorial operator $\widetilde{M}$, as we did in the proof of Proposition~\ref{uniform bounds Lp dirichlet problem poisson semigroup}, in order to obtain the first item.
    
    Next, if we assume that $p<n$ and $V\not\equiv 0$, then Proposition~\ref{embedding homogeneous V spaces in L ^p* for p<n} shows that there is a continuous embedding $\dot{\mathcal{V}}^{1,p}\subseteq \El{p^{*}}$. 
    The first item in the proposition means that $u\in C_{0}([0,\infty);\dot{\mathcal{V}}^{1,p})\cap C^{\infty}((0,\infty) ; \dot{\mathcal{V}}^{1,p})$ with the boundary limit $\lim_{t\to 0^{+}}u(t,\cdot)= f$ in the $\dot{\mathcal{V}}^{1,p}$ norm. Together with the embedding this implies the required regularity.
    Moreover, since $u(t,\cdot)\to f $ in $\El{p^{*}}$ as $t\to 0$, we obtain $\norm{f}_{\El{p^{*}}}\leq \sup_{t>0}\norm{u(t,\cdot)}_{\El{p^{*}}}$.
    In addition, it follows from Proposition~\ref{properties of critical numbers proposition} and Theorem~\ref{summary of identifications of H adapted Hardys spaces} that $\bb{H}^{p^{*}}_{H}=\El{p^{*}}\cap\El{2}$ with equivalent $p^{*}$-norms. Therefore, the sectorial counterpart of \cite[Proposition 8.10]{AuscherEgert} shows that $\norm{u(t,\cdot)}_{\El{p^{*}}}\lesssim \norm{f}_{\El{p^{*}}}$, uniformly for all $t>0$. We also note that the embedding $\dot{\mathcal{V}}^{1,p}\subseteq \El{p^{*}}$ implies that $\norm{u(t,\cdot)}_{\El{p^{*}}}\lesssim \norm{\nabla_{\mu}^{\parallel}u(t,\cdot)}_{\El{p}}\lesssim\norm{\nabla_{\mu}f}_{\El{p}}$ for all $t>0$.
\end{proof}
In the case $p\leq 1$, we have the following analogous result.
\begin{prop}\label{proposition strong continuity estimate regularity problem for p<1}
    Let $p\in(p_{-}(H)_{*}\vee \frac{n}{n+1},1]$. If $f\in\dot{\textup{H}}^{1,p}_{V,\textup{pre}}$, then $u(t,\cdot):=e^{-tH^{1/2}}f$ has the following properties:
    \begin{enumerate}[label=(\roman*)]
    \item It holds that $f\in\El{p^{*}}$ and $u\in C_{0}([0,\infty) ; \El{p^{*}})\cap C^{\infty}((0,\infty);\El{p^{*}})$, with 
        \begin{equation*}
             \sup_{t>0}\norm{u(t,\cdot)}_{\El{p^{*}}}\eqsim \norm{f}_{\El{p^{*}}}.
        \end{equation*}
        \item If $p\in(p_{-}(H)_{*}\vee \frac{n}{n+\delta/2},1]\cap\, \mathcal{I}(V)$, then $u\in C_{0}([0,\infty) ; \dot{\textup{H}}^{1,p}_{V,\textup{pre}})\cap C^{\infty}((0,\infty);\dot{\textup{H}}^{1,p}_{V,\textup{pre}})$, with
        \begin{equation*}
\sup_{t>0}\norm{u(t,\cdot)}_{\dot{\textup{H}}^{1,p}_{V}}\eqsim \norm{f}_{\dot{\textup{H}}^{1,p}_{V}}\quad \textup{and}\quad \sup_{t>0}\norm{t^{k}\partial_{t}^{k}u(t,\cdot)}_{\dot{\textup{H}}^{1,p}_{V}}\lesssim k^{k}e^{-k}\norm{f}_{\dot{\textup{H}}^{1,p}_{V}}
        \end{equation*}
        for all integers $k\geq 1$.
    \end{enumerate}
\end{prop}
\begin{proof}
The first item follows from the identification $\bb{H}^{p^{*}}_{H}=\El{p^{*}}\cap\El{2}$, as in the proof of the second part of Proposition~\ref{proposition strong continuity of regularity poisson solution p>1} above. The second item uses the identification $\dot{\textup{H}}^{1,p}_{V,\textup{pre}}=\bb{H}^{1,p}_{H}\cap\mathcal{V}^{1,2}$ with equivalent $p$-(quasi)norms (see Corollary~\ref{characterisation abstract H^p p<2}), along with the abstract theory of \cite[Propositions 8.10 and 8.13]{AuscherEgert} for the sectorial operator $H$, as in the proof of Proposition~\ref{uniform bounds Lp dirichlet problem poisson semigroup}. 
\end{proof}

\section{Solvability and estimates for Theorems~\ref{existence and uniqueness result Dirichlet Lp} and \ref{existence and uniqueness result Regularity Lp}}\label{section on existence for the dirichlet and regularit problems}
With the estimates of Section~\ref{section estimates towards the dirichelt and regularity problems} at hand, we shall now be able to prove the existence results for the Dirichlet and Regularity problems. Throughout this section, we assume that $n\geq 3$, $q\geq \max\set{\frac{n}{2},2}$ and $V\in\textup{RH}^{q}(\Rn)$.

\subsection{The Dirichlet problem}
The following result contains solvability and estimates for Theorem~\ref{existence and uniqueness result Dirichlet Lp}. We follow the proof of \cite[Theorem 1.1]{AuscherEgert} outlined in \cite[Section~17.3]{AuscherEgert}. For notational convenience, we introduce $\textup{H}^{p}_{V}(\Rn):=\Ell{p}$, with the corresponding norm, for $p\in(1,\infty)$. This is in accordance with Lemma~\ref{identification of Dziubanski hardy spaces for p>=1}.
\begin{thm}\label{existence result for Dirichlet problem again}
    Let $n\geq 3$, $q\geq\max\{\frac{n}{2},2\}$ and $V\in\textup{RH}^{q}(\Rn)$. Suppose that $p\geq 1$ and $p\in (p_{-}(H), p_{+}(H)^{*})$. If either $p>1$ and $f\in \El{p}$, or $p=1$ and $f\in b\textup{H}^{1}_{V}$, then the Dirichlet problem $\left(\mathcal{D}\right)^{\mathscr{H}}_{p}$ for data $f\in b\textup{H}^{p}_{V}$ admits a weak solution $u\in\mathcal{V}^{1,2}_{\textup{loc}}(\Hn)$, with the following additional properties:
    \begin{enumerate}[label=\emph{(\roman*)}]
\item There is comparability: $\norm{N_{*}(u)}_{\El{p}}\lesssim \norm{b^{-1}f}_{\textup{H}^{p}_{V}}\eqsim \norm{S\left(t\nabla_{\mu}u\right)}_{\El{p}}$.
    If either $p>1$, or $p=1$ and $V\in\textup{RH}^{\infty}$, then the reverse non-tangential maximal function bound holds.
\item The non-tangential convergence improves to $\El{2}$ averages:
\begin{equation*}
        \lim_{t\to 0^{+}}\fiint_{W(t,x)}\abs{u(s,y) - f(x)}^{2}\dd s\dd y =0 \text{ for almost every  }x\in\Rn.
    \end{equation*}
\item If $p<p_{+}\left(H\right)$, then $b^{-1}u\in C_{0}\left(\left[0,\infty \right) ;\textup{H}^{p}_{V} \right)\cap C^{\infty}\left((0,\infty);\textup{H}^{p}_{V}\right)$, where $t$ is the distinguished variable, and $u(0,\cdot)=f$, with $\sup_{t\geq 0}\norm{b^{-1}u(t,\cdot)}_{\textup{H}^{p}_{V}}\eqsim \norm{b^{-1}f}_{\textup{H}^{p}_{V}}$.
\item If $p\geq p_{+}(H)$, then for all\hspace{0.1cm} $T>0$ and all compact subsets $K\subset\Rn$, it holds that  $u\in C\left([0,T];\El{2}(K)\right)$, with $u(0,\cdot)=\left.f\right|_{K}$, and $\sup_{0<t\leq T}\norm{u(t,\cdot)}_{\El{2}(K)}\lesssim \norm{f}_{\Ell{p}}$,
where the implicit constant depends on $T$ and $K$.
\end{enumerate}
\end{thm}
\begin{proof}
    Let $f\in b\textup{H}^{p}_{V}$. If $1<p<\infty$, then $f\in\Ell{p}$ and we pick a sequence $(f_{k})_{k}\subset \El{p}\cap\El{2}$ such that $f_k \to f$ in $\El{p}$. If $p=1$, we choose a sequence $(f_k)_k\subset b\textup{H}^{1}_{V,\textup{pre}}(\Rn)$ such that $b^{-1}f_{k}\to b^{-1}f$ in $\textup{H}^{1}_{V}(\Rn)$, hence also in $\Ell{1}$ (see the definition of $\textup{H}^{1}_{V}$ in Section~\ref{subsubsection with definition of completion H^{1}_{V}}). For each $k\geq 1$, we obtain a weak solution $u_{k}\in\mathcal{V}^{1,2}_{\loc}(\Hn)$ corresponding to the boundary data $f_k$, given by the Poisson extension $u_{k}(t,\cdot):=e^{-tH^{1/2}}f_k$ (see Theorem~\ref{Poisson semigroup extension is a weak solution}). For all $j,k\geq 1$, it follows from linearity of $e^{-tH^{1/2}}$ that 
    \begin{equation*}
        u_{k,j}:=u_k(t,\cdot)-u_j(t,\cdot)=e^{-tH^{1/2}}(f_k -f_j).
    \end{equation*}
    Consequently, it follows from Propositions \ref{SQuare function estimate dirichlet poisson semigroup} and \ref{estimate non tangential maximal function on poisson extension solution} that for all $j,k\geq 1$ it holds that 
    \begin{align*}\label{estimate non tangential max function Cauchy sequence Poisson solutions}
        \norm{u_k-u_j}_{\tent{0,p}_{\infty}}&=\norm{N_{*}(u_k-u_j)}_{\El{p}}\lesssim\norm{b^{-1}(f_k-f_j)}_{\textup{H}^{p}_{V}}\eqsim \norm{t\nabla_{\mu}(u_k-u_j)}_{\tent{p}}.
    \end{align*}
    It follows that the sequence $(u_k)_{k\geq 1}$ is Cauchy in $\tent{0,p}_{\infty}$ and that the sequence $(t\nabla_{\mu}u_k)_{k\geq 1}$ is Cauchy in $\tent{p}$. By completeness, there are $u\in\tent{0,p}_{\infty}$ and $v\in\tent{p}$ such that $u_k\to u$ in $\tent{0,p}_{\infty}$ and $t\nabla_{\mu}u_k\to v$ in $\tent{p}$. Both convergences also hold in $\El{2}_{\loc}(\Hn)$. Lemma~\ref{limit of weak solutions is weak solution to elliptic equation} shows that $u\in\mathcal{V}^{1,2}_{\loc}(\Hn)$ is a weak solution in $\Hn$, and that $u_k\to u$ in the $\mathcal{V}^{1,2}_{\loc}(\Hn)$ topology. Consequently, $v=t\nabla_{\mu}u$. We can therefore pass to the limit in the estimates to obtain 
    \begin{align*}
\norm{u}_{\tent{0,p}_{\infty}}=\norm{N_{*}u}_{\El{p}}\lesssim\norm{b^{-1}f}_{\textup{H}^{p}_{V}}\eqsim \norm{S\left(t\nabla_{\mu}u\right)}_{\El{p}}=\norm{t\nabla_{\mu}u}_{\tent{p}}.
    \end{align*}
    The last statement of Proposition~\ref{estimate non tangential maximal function on poisson extension solution} shows that the reverse estimate $\norm{b^{-1}f}_{\textup{H}^{p}_{V}}\lesssim \norm{N_{*}u}_{\El{p}}$ holds if $p>1$, and it also holds for $p=1$ if $V\in\textup{RH}^{\infty}$. This proves (i).

    To prove (iii), let us now assume that $p<p_{+}(H)$. By Proposition~\ref{uniform bounds Lp dirichlet problem poisson semigroup}, for each $k,j\geq 1$, it holds that  $b^{-1}(u_k-u_j)\in C_{0}([0,\infty);\textup{H}^{p}_{V,\textup{pre}})\cap \, C^{\infty}((0,\infty);\textup{H}^{p}_{V,\textup{pre}})$ in the variable $t$, with 
    \begin{align*}
        \norm{b^{-1}(u_k-u_j)}_{C_{0}\left([0,\infty);\,\textup{H}^{p}_{V,\textup{pre}}\right)}:=\sup_{t\geq 0}\norm{b^{-1}(u_{k}(t,\cdot)-u_{j}(t,\cdot))}_{\textup{H}^{p}_{V}}\eqsim \norm{b^{-1}(f_k-f_j)}_{\textup{H}^{p}_{V}}.
    \end{align*}
    Since $\textup{H}^{p}_{V}(\Rn)$ is complete, the space $C_{0}([0,\infty);\textup{H}^{p}_{V})$ is a Banach space, hence there is $w\in C_{0}([0,\infty);\textup{H}^{p}_{V})$ such that $b^{-1}u_k\to w$ in $C_{0}([0,\infty);\textup{H}^{p}_{V})$. Since $p\geq 1$, this implies that $b^{-1}u_k\to w$ in the $\El{1}_{\loc}(\Hn)$ topology as well, hence $w=b^{-1}u\in C_{0}([0,\infty);\textup{H}^{p}_{V})$.
Moreover, taking the limit as $k\to\infty$ in the equivalence $\norm{b^{-1}u_k}_{C_{0}([0,\infty);\;\textup{H}^{p}_{V})}\eqsim \norm{b^{-1}f_{k}}_{\textup{H}^{p}_{V}}$ yields that $\norm{b^{-1}u}_{C_{0}([0,\infty);\textup{H}^{p}_{V})}\eqsim \norm{b^{-1}f}_{\textup{H}^{p}_{V}}$, and that $b^{-1}u(0,\cdot)=b^{-1}f$ in $\textup{H}^{p}_{V}$ (by Proposition~\ref{uniform bounds Lp dirichlet problem poisson semigroup}).
    We obtain in a similar manner that $u\in C^{\infty}\left((0,\infty);\textup{H}^{p}_{V}\right)$, and this proves (iii).

    To prove (iv), let us now assume that $p\geq p_{+}(H)$. Let $T>0$ and $K\subset \Rn$ be compact. It follows from Proposition~\ref{local estimate time and space Poisson solution p>p+} that $u_{k}\in C([0,T];\El{2}(K))$ for all $k\geq 1$, and that
    \begin{equation*}
        \norm{u_k-u_j}_{C([0,T];\;\El{2}(K))}:=\sup_{0\leq t\leq T}\norm{u_{k}(t,\cdot)-u_{j}(t,\cdot)}_{\El{2}(K)}\leq c_{K,T}\norm{f_k -f_j}_{\El{p}}
    \end{equation*}
    for all $j,k\geq 1$, where $c_{K,T}$ is a constant depending on $K$ and $T$. It follows from completeness of $C([0,T];\El{2}(K))$ that there exists some $w\in C([0,T];\El{2}(K))$ such that $\norm{u_k- w}_{C([0,T];\El{2}(K))}\to 0$ as $k\to\infty$. Since $u_k\to u $ in $\El{2}_{\loc}(\Hn)$, this implies that $u|_{(0,T]\times K}=w$ and therefore $u\in C((0,T] ; \El{2}(K))$. We can pass to the limit as $k\to\infty$ to obtain the estimate
    \begin{equation*}
        \sup_{0\leq t\leq T}\norm{u(t,\cdot)}_{\El{2}(K)}\leq c_{K,T}\norm{f}_{\El{p}}.
    \end{equation*}
    This proves (iv). Finally, we obtain the non-tangential almost everywhere convergence in (ii) using the well-known density argument (recalled in \cite[Section~17.3]{AuscherEgert}) relying on the estimate
    \begin{equation*}
        \norm{N_{*}(u-u_k)}_{\El{p}}\lesssim \norm{b^{-1}(f-f_k)}_{\textup{H}^{p}_{V}}
    \end{equation*}
and the identity \eqref{non tangential limit for sirichlet problem in L2 POisson extension} applied to $u_k$.
\end{proof}
\subsection{The Regularity problem}\label{subsection on existence for the regularity problem}
Before proving the solvability and estimates for Theorem~\ref{existence and uniqueness result Regularity Lp}, we provide justification for the statement made in the introduction that \cite[Theorem 1.1]{MorrisTurner} already implies that the Regularity problem $\left(\mathcal{R}\right)^{\mathscr{H}}_{2}$ is well-posed if $n\geq 3$, $q\geq \max\set{\frac{n}{2},2}$, $V\in\textup{RH}^{q}(\Rn)$ and $V\not\equiv 0$. The reason we need an extra argument is because the type of convergence of the solution towards the boundary data that was obtained in \cite[Theorem 1.1]{MorrisTurner} is \emph{a priori} different from the one considered in the problem $\left(\mathcal{R}\right)^{\mathscr{H}}_{2}$ stated in the introduction. 

To this end, we first note that Proposition~\ref{embedding homogeneous V spaces in L ^p* for p<n} shows that any $f\in\dot{\mathcal{V}}^{1,2}(\Rn)$ belongs to $\Ell{2^{*}}$ and therefore belongs to the space $\dot{\mathcal{V}}^{1,2}(\Rn)$ defined in \cite[Section~2.1]{MorrisTurner} (recall Remark \ref{remark about different definition of V^{11,2} in Morris Turner}). Now, let $u\in\mathcal{V}^{1,2}_{\loc}(\Hn)$ be the solution to the regularity problem provided by \cite[Theorem 1.1]{MorrisTurner} for data $f\in\dot{\mathcal{V}}^{1,2}(\R^n)$. Since $N_{*}(\nabla_{t,x}u)\in\Ell{2}$, we claim that a.e. convergence of $\El{2}$-Whitney averages $\scriptstyle{N_{*}}$-$\lim_{t\to 0^{+}}\nabla_{\mu}^{\parallel}u(t,\cdot)=\nabla_{\mu}f$ and the $\El{2}$ convergence $\lim_{t\to 0^{+}}\norm{\nabla_{\mu}^{\parallel}u(t,\cdot)-\nabla_{\mu}f}_{\El{2}}=0$ obtained in \cite[Theorem 1.1]{MorrisTurner} together imply the convergence 
\begin{equation}\label{non tangential convergence as a consequence of convergence of nabla mu of solution}
\lim_{t\to 0^{+}}\fiint_{W(t ,x)}\abs{u(s,y)-f(x)}^{2}\dd s\dd y = 0,
\end{equation}
for almost every $x\in\Rn$. To see this, we note that \cite[Proposition A.5]{AuscherEgert} with $q=r=p=2$ shows that there exists some $u_{0}\in\dot{\textup{W}}^{1,2}(\Rn)$ such that 
\begin{equation}\label{non tangential convergence of a function u to a trace u_0}
    \left(\fiint_{W(t,x)}\abs{u(s,y)-u_{0}(x)}^{2}\dd s\dd y\right)^{1/2}\lesssim t N_{*}(\nabla_{t,x}u)(x)
\end{equation}
for almost every $x\in\Rn$ and all $t>0$. In addition, $\lim_{t\to 0^{+}}\dashint_{t/2}^{2t}u(s,\cdot)\dd s = u_{0}$ in $\mathcal{D}'(\Rn)$. The later identity implies that $\lim_{t\to 0^{+}}\dashint_{t/2}^{2t}\langle\nabla_x u(s,\cdot),\varphi\rangle\dd s =  \langle\nabla_x u_0,\varphi\rangle$ for all $\varphi \in C_c^\infty(\R^n;\C^n)$, which combined with the convergence $\lim_{t\to 0^{+}}\norm{\nabla_{x}u(t,\cdot)-\nabla_{x}f}_{\El{2}}=0$ shows that $\nabla_{x}u_{0}=\nabla_{x}f$ so there is $c\in\C$ such that $u_{0}=f+c$. To see that $c=0$, we note that \eqref{non tangential convergence of a function u to a trace u_0} and Lebesgue's differentiation theorem imply that 
    \begin{align*}
        &\fiint_{W(t,x)}\abs{V^{1/2}(y)u(s,y)-V^{1/2}(x)u_{0}(x)}\dd y\dd s\\
        &\lesssim\left(\dashint_{B(x,t)}V\right)^{1/2}tN_{*}\left(\nabla_{t,x}u\right)(x)+ \abs{u_{0}(x)}\left(\dashint_{B(x,t)}\abs{V^{1/2}(y)-V^{1/2}(x)}\dd y\right)\to 0
    \end{align*}
as $t\to 0^{+}$, for almost every $x\in\Rn$. When combined with the non-tangential convergence $\scriptstyle{N_{*}}$-$\lim_{t\to 0^{+}}V^{1/2}u(t,\cdot)=V^{1/2}f$, this shows that $V^{1/2}(x)u_{0}(x)=V^{1/2}(x)f(x)$ for almost every $x\in\Rn$. Since $V\not\equiv 0$, it follows that $c=0$, and this proves \eqref{non tangential convergence as a consequence of convergence of nabla mu of solution}, so $\left(\mathcal{R}\right)^{\mathscr{H}}_{2}$ is well-posed.

The following result contains the solvability and estimates for Theorem~\ref{existence and uniqueness result Regularity Lp} in the range $p>1$. We follow the proof of \cite[Theorem 1.2]{AuscherEgert} outlined in \cite[Section~17.3]{AuscherEgert}.
\begin{thm}\label{existence and uniqueness result Regularity Lp again}
    Let $n\geq 3$, $q\geq\max\{\frac{n}{2},2\}$ and $V\in\textup{RH}^{q}(\Rn)$. Let ${p\in (p_{-}(H)_{*}\vee 1,q_{+}(H))}$ and $p\leq 2q$. If $f\in\dot{\mathcal{V}}^{1,p}\left(\Rn\right)$, then the Regularity problem $\left(\mathcal{R}\right)^{\mathscr{H}}_{p}$ for data $f\in\dot{\mathcal{V}}^{1,p}(\Rn)$ admits a weak solution $u\in\mathcal{V}^{1,2}_{\textup{loc}}(\Hn)$ with the following additional properties:
    \begin{enumerate}[label=\emph{(\roman*)}]
        \item The following estimates hold:
       \begin{equation*}
        \norm{S(t\nabla_{\mu}\partial_{t}u)}_{\El{p}}\eqsim\norm{\nabla_{\mu}f}_{\El{p}}\eqsim \norm{N_{*}\left(\nabla_{\mu}u\right)}_{\El{p}}\gtrsim  \norm{A_{\perp\perp}H^{1/2}f}_{\El{p}},
    \end{equation*}
    where the square root $H^{1/2}$ extends by density to ${H^{1/2}:\dot{\mathcal{V}}^{1,p}\left(\Rn\right)\to\Ell{p}}$.
    \item For almost every $x\in\Rn$ and all $t>0$ it holds that 
    \begin{equation*}
        \left(\fiint_{W(t,x)}\abs{u(s,y) - f(x)}^{2}\dd s\dd y\right)^{1/2}\lesssim t N_{*}\left(\nabla_{t,x}u\right)(x),
    \end{equation*}
    hence $\lim_{t\to 0^{+}}\fiint_{W(t,x)}\abs{u(s,y) - f(x)}^{2}\dd s\dd y=0$ for almost every $x\in\Rn$. Moreover, 
    \begin{equation*}
        \lim_{t\to 0^{+}}u(t,\cdot)=f \text{ in }\mathcal{D}'\left(\Rn\right).
    \end{equation*}
    \item For almost every $x\in\Rn$ it holds that 
    \begin{equation*}
        \lim_{t\to 0^{+}}\fiint_{W(t,x)}\abs{\begin{bmatrix}
            A_{\perp\perp}\partial_{t}u\\
            \nabla_{\mu}^{\|} u
        \end{bmatrix}\left(s,y\right) - \begin{bmatrix}
            (-A_{\perp\perp}H^{1/2}f)(x)\\
            \left(\nabla_{\mu}f\right)(x)
        \end{bmatrix}}^{2}\dd s\dd y =0.
    \end{equation*}
    \item It holds that  $\nabla_{\mu}^{\parallel} u\in C_{0}\left(\left[0,\infty \right) ;\El{p} \right)\cap C^{\infty}\left((0,\infty);\El{p}\right)$, where $t$ is the distinguished variable, and $\nabla_{\mu}^{\parallel}u\left(0,\cdot\right)=\nabla_{\mu}f$, with
    \begin{equation*}
        \sup_{t\geq 0}\norm{\nabla_{\mu}^{\parallel}u\left(t,\cdot\right)}_{\El{p}}\eqsim \norm{\nabla_{\mu}f}_{\El{p}}. 
    \end{equation*}
    \item If $p<n$, then $u\in C_{0}\left(\left[0,\infty \right) ;\El{p^{*}} \right)\cap C^{\infty}\left((0,\infty);\El{p^{*}}\right)$, and $u(0,\cdot)=f$, with 
    \begin{equation*}
        \norm{f}_{\El{p^{*}}}\leq \sup_{t\geq 0}\norm{u(t,\cdot)}_{\El{p^{*}}}\lesssim \norm{f}_{\El{p^{*}}}.
    \end{equation*}
    \item If $p>p_{-}(H)$, then $\partial_{t}u\in C_{0}\left([0,\infty); \El{p}\right)$ and 
    \begin{equation*}
        \norm{N_{*}\left(\partial_{t}u\right)}_{\El{p}}\eqsim \sup_{t\geq 0}\norm{\partial_{t}u\left(t,\cdot\right)}_{\El{p}}\eqsim \norm{A_{\perp\perp}H^{1/2}f}_{\El{p}}\eqsim \norm{\nabla_{\mu}f}_{\El{p}}.
    \end{equation*}
    \end{enumerate}
\end{thm}
\begin{proof}
    If $V\equiv 0$, then this is \cite[Theorem 1.2]{AuscherEgert}. We can therefore assume that $V\not\equiv 0$. Let $p\in\left(p_{-}(H)_{*}\;,q_{+}(H)\right)\cap(1,2q]$. We first treat the case when $f\in\dot{\mathcal{V}}^{1,p}(\Rn)\cap\mathcal{V}^{1,2}(\Rn)$. The identifications of Theorem~\ref{summary of identifications of H adapted Hardys spaces} (and $\dom{H^{1/2}}=\mathcal{V}^{1,2}$) then imply that $f\in\bb{H}^{1,p}_{H}\cap\dom{H^{1/2}}$. It follows from Figure \ref{figure relations between adapted Hardy spaces} that $H^{1/2}f\in\bb{H}^{p}_{H}\cap\ran{H^{1/2}}$ with the estimates 
    \begin{align}\label{intermediate estimates on the square root of H in Lp}
\norm{f}_{\dot{\mathcal{V}}^{1,p}}=\norm{\nabla_{\mu}f}_{\El{p}}\eqsim\norm{f}_{\bb{H}^{1,p}_{H}}\eqsim \norm{H^{1/2}f}_{\bb{H}^{p}_{H}}\gtrsim \norm{H^{1/2}f}_{\El{p}}.
    \end{align}
    The last inequality follows from the continuous inclusion $\bb{H}^{p}_{H}\subseteq \El{p}\cap\El{2}$ for $p\in (1,p_{+}(H))$, which follows from Proposition~\ref{two continuous inclusions abstract H^p and H^1,p} and Proposition~\ref{identification abstract hardy space H^p p>2}, as well as the fact that $p<q_{+}(H)<p_{+}(H)$ (Proposition~\ref{properties of critical numbers proposition}). Since $1<p\leq 2q$, Proposition~\ref{density lemma in homogeneous V adapted Sobolev spaces} and the estimate \eqref{intermediate estimates on the square root of H in Lp} show that $H^{1/2}$ extends by density to a bounded mapping from $\dot{\mathcal{V}}^{1,p}(\Rn)$ to $\Ell{p}$.
As before, Theorem~\ref{Poisson semigroup extension is a weak solution} shows that $u(t,\cdot):=e^{-tH^{1/2}}f$ is a weak solution of $\mathscr{H}u=0$ in $\Hn$. Also, note that 
    \begin{equation*}
        \partial_{t}u=-H^{1/2}e^{-tH^{1/2}}f=e^{-tH^{1/2}}(-H^{1/2}f)=e^{-tH^{1/2}}g,
    \end{equation*}
    where $g:=-H^{1/2}f\in\Ell{p}\cap\Ell{2}$ with $\norm{g}_{\El{p}}\lesssim\norm{\nabla_{\mu}f}_{\El{p}}$ by \eqref{intermediate estimates on the square root of H in Lp}.

    Now let $f\in\dot{\mathcal{V}}^{1,p}(\Rn)$ be arbitrary. By Proposition~\ref{density lemma in homogeneous V adapted Sobolev spaces} again, we can find a sequence $(f_{k})_{k\geq1}$ in $C^{\infty}_{c}(\Rn)\subset \dot{\mathcal{V}}^{1,p}\cap\mathcal{V}^{1,2}$ such that $f_k\to f$ in $\dot{\mathcal{V}}^{1,p}(\Rn)$ as $k\to\infty.$ 
    We now consider two cases:
    \begin{enumerate}[wide, labelwidth=!, labelindent=0pt, label={\textit{{Case} \arabic*.}}]
    \item Let us first assume that $p<n$. Then it follows from Proposition~\ref{embedding homogeneous V spaces in L ^p* for p<n} that $f\in\Ell{p^{*}}$ and the sequence $(f_{k})_{k\geq 1}\subset\El{p^{*}}$ also converges to $f$ in the $\El{p^{*}}$ norm. For each $k\geq 1$, let $u_k(t,\cdot):=e^{-tH^{1/2}}f_k$ be the associated weak solution. For all $j,k\geq 1$, we have that ${(u_k-u_j)(t,\cdot)=e^{-tH^{1/2}}(f_k-f_j)}$, and it follows from Proposition~\ref{proposition strong continuity of regularity poisson solution p>1} that 
    \begin{equation*}
        \norm{u_k-u_j}_{C_{0}([0,\infty);\, \El{p^{*}})}\eqsim\norm{f_k-f_j}_{\El{p^{*}}}.
    \end{equation*}
    This implies that $(u_k)_{k\geq 1}$ is Cauchy in $C_{0}([0,\infty);\El{p^{*}})$, and consequently also in $\El{1}_{\loc}(\Hn)$. By Lemma~\ref{limit of weak solutions is weak solution to elliptic equation}, there is a weak solution $u\in\mathcal{V}^{1,2}_{\loc}(\Hn)$ such that $u_k\to u$ in $\mathcal{V}^{1,2}_{\loc}(\Hn)$. Moreover, $u\in C_{0}([0,\infty);\El{p^{*}})$ with $\norm{u}_{C_{0}([0,\infty);\, \El{p^{*}})}\eqsim \norm{f}_{\El{p^{*}}}$.
    Next, we observe that for all $j,k\geq 1$, by Proposition~\ref{proposition estimate non tangential ùax function regularity problem p>1}, it holds that
    \begin{align*}
        \norm{\nabla_{\mu}u_k-\nabla_{\mu}u_j}_{\tent{0,p}_{\infty}}=\norm{N_{*}(\nabla_{\mu}(u_k-u_j))}_{\El{p}}\eqsim\norm{f_k-f_j}_{\dot{\mathcal{V}}^{1,p}}.
    \end{align*}
    Consequently, the sequence $(\nabla_{\mu}u_k)_{k\geq 1}$ is Cauchy in $\tent{0,p}_{\infty}$ and there is $v\in\tent{0,p}_{\infty}$ such that $\nabla_{\mu}u_k\to v$ in $\tent{0,p}_{\infty}$. In particular, convergence also holds in $\El{2}_{\loc}(\Hn)$. Since $u_k\to u$ in $\mathcal{V}^{1,2}_{\loc}(\Hn)$, we obtain that $v=\nabla_{\mu}u\in\tent{0,p}_{\infty}$, with $\norm{N_{*}(\nabla_{\mu}u)}_{\Ell{p}}\eqsim\norm{f}_{\dot{\mathcal{V}}^{1,p}}$.
    Next, we note that \cite[Lemma 4.2]{MorrisTurner} shows that $\partial_{t}u\in\mathcal{V}^{1,2}_{\loc}(\Hn)$ since it is also a weak solution (by $t$-independence of the coefficients). Proposition~\ref{proposition  L^p square function estimate regularity problem} shows that the sequence $(t\nabla_{\mu}\partial_{t}u_k)_{k\geq1}$ is Cauchy in $\tent{p}$, so there is $w\in\tent{p}$ such that $t\nabla_{\mu}\partial_{t}u_k\to w$ in $\tent{p}$ as $k\to\infty$. Convergence also holds in the $\El{2}_{\loc}(\Hn)$ topology. The estimate of \cite[Lemma 4.2]{MorrisTurner} can be applied to obtain that 
    \begin{align*}
        \iint_{W}\abs{\nabla_{\mu}\partial_{t}(u_k-u)}^{2}\dd s\dd y\lesssim \frac{1}{\ell(W)^{2}}\iint_{2W}\abs{\nabla_{x,t}(u_k-u)}^{2}\dd s\dd y
    \end{align*}
    for all Whitney cubes $W\subset \Hn$. The convergence $u_k\to u$ in $\mathcal{V}^{1,2}_{\loc}(\Hn)$ therefore implies that $\nabla_{\mu}\partial_{t}u_k\to \nabla_{\mu}\partial_{t}u$ in $\El{2}_{\loc}(\Hn)$. We conclude that $w=t\nabla_{\mu}\partial_{t}u\in\tent{p}$ and we obtain the estimate
    \begin{equation*}
        \norm{S(t\nabla_{\mu}\partial_{t}u)}_{\El{p}}\eqsim \norm{f}_{\dot{\mathcal{V}}^{1,p}}.
    \end{equation*}
    We now prove strong continuity of the solutions in the variable $t$. It follows from Proposition~\ref{proposition strong continuity of regularity poisson solution p>1} that $\nabla_{\mu}^{\parallel}u_k-\nabla_{\mu}^{\parallel}u_j\in C_{0}([0,\infty) ; \Ell{p})$ with $\norm{\nabla_{\mu}^{\parallel}u_k -\nabla_{\mu}^{\parallel}u_j}_{C_{0}([0,\infty) ; \El{p})}\eqsim \norm{f_k-f_j}_{\dot{\mathcal{V}}^{1,p}}$, and $\nabla_{\mu}^{\parallel}u_k-\nabla_{\mu}^{\parallel}u_j\in C^{\infty}((0,\infty) ; \El{p})$
    with $\norm{\nabla_{\mu}^{\parallel}u_k-\nabla_{\mu}^{\parallel}u_j}_{C^{m}((0,\infty);\El{p})}\lesssim_{m}\norm{f_k-f_j}_{\dot{\mathcal{V}}^{1,p}}$ for all $m\geq 1$. There is $v\in C_{0}([0,\infty) ; \El{p})$ such that $\nabla_{\mu}u_k\to v $ in $C_{0}([0,\infty) ; \El{p})$. Convergence also holds in $\El{1}_{\loc}(\Hn)$ and therefore $v=\nabla_{\mu}^{\parallel}u\in C_{0}([0,\infty) ; \El{p})$, with $\nabla_{\mu}^{\parallel}u(0,\cdot)=\nabla_{\mu}f$ and $\norm{\nabla_{\mu}^{\parallel}u}_{C_{0}([0,\infty) ; \El{p})}\eqsim \norm{f}_{\dot{\mathcal{V}}^{1,p}}$.
    Similarly, we obtain that $\nabla_{\mu}^{\parallel}u\in C^{m}((0,\infty);\El{p})$ for all $m\geq 1$, with the corresponding estimates, whence $\nabla_{\mu}^{\parallel}u\in C^{\infty}((0,\infty);\El{p})$. The continuous embedding $\dot{\mathcal{V}}^{1,p}\subseteq \Ell{p^{*}}$ implies that $u\in C_{0}([0,\infty) ; \El{p^{*}})\cap C^{\infty}((0,\infty);\El{p^{*}})$. 
    Since $\lim_{t\to 0^{+}} u(t,\cdot)=f$ in $\Ell{p^{*}}$, this trivially implies that $\lim_{t\to 0^{+}} u(t,\cdot)=f$ in the distributional sense. 

    Now, to prove (iii), we use the same argument as in the proof of Theorem~\ref{existence result for Dirichlet problem again}, relying instead on the second part of Proposition~\ref{proposition estimate non tangential ùax function regularity problem p>1} and the estimates 
    \begin{equation}\label{intermediate step to prove step (iii) in regularity problem existence}
        \norm{N_{*}(\nabla_{\mu}u-\nabla_{\mu}u_k)}_{\El{p}}\eqsim \norm{f-f_k}_{\dot{\mathcal{V}}^{1,p}}\gtrsim \norm{H^{1/2}f- H^{1/2}f_k}_{\El{p}}
    \end{equation}
    for all $k\geq 1$. Item (ii) is proved as in \cite[Section~17.3]{AuscherEgert}, once we note that $u\in\mathcal{V}^{1,2}_{\loc}(\Hn)\subseteq \textup{W}^{1,2}_{\loc}(\Hn)$, and that 
    \begin{equation*}
        \norm{N_{*}(\nabla_{t,x}u)}_{\El{p}}\leq \norm{N_{*}(\nabla_{\mu}u)}_{\El{p}}\lesssim \norm{f}_{\dot{\mathcal{V}}^{1,p}}<\infty.
    \end{equation*}
Let us now treat item (vi) and assume that $p>p_{-}(H)$. Then, we can use Propositions \ref{estimate non tangential maximal function on poisson extension solution} and \ref{uniform bounds Lp dirichlet problem poisson semigroup} to obtain that $\partial_{t}u_k=e^{-tH^{1/2}}(-H^{1/2}f_k)\in C_{0}([0,\infty) ; \El{p})$ for all $k\geq 1$, with 
    \begin{align*}
        \norm{N_{*}(\partial_{t}u_k - \partial_{t}u_j)}_{\El{p}}&\eqsim \sup_{t\geq 0}\norm{\partial_{t}u_k(t,\cdot) - \partial_{t}u_j(t,\cdot)}_{\El{p}}\eqsim \norm{H^{1/2}f_k - H^{1/2}f_j}_{\El{p}}\eqsim\norm{f_k-f_j}_{\dot{\mathcal{V}}^{1,p}}
    \end{align*}
    for all $k,j\geq 1$. We have used Theorem~\ref{thm boundedness of Riesz transforms on Lp full range} in addition to \eqref{intermediate step to prove step (iii) in regularity problem existence} to obtain the last equivalence. The sequence $(\partial_{t}u_k)_{k\geq 1}$ is Cauchy in $\tent{0,p}_{\infty}$, so there is $v\in\tent{0,p}_{\infty}$ such that $\partial_{t}u_k\to v $ in $\tent{0,p}_{\infty}$. Convergence also holds in $\El{2}_{\loc}(\Hn)$, and thus $v=\partial_{t}u\in\tent{0,p}_{\infty}$, with $\norm{N_{*}(\partial_{t}u)}_{\El{p}}\eqsim \norm{f}_{\dot{\mathcal{V}}^{1,p}}$. In addition, the sequence $(\partial_{t}u_k)_{k\geq 1}$ is Cauchy in $C_{0}([0,\infty) ; \El{p})$, hence there exists $w\in C_{0}([0,\infty) ; \El{p})$ such that $\partial_{t}u_k\to w $ in $C_{0}([0,\infty) ; \El{p})$. Convergence also holds in $\El{1}_{\loc}(\Hn)$, whence $w=\partial_{t}u\in C_{0}([0,\infty) ; \El{p})$, with the estimate $\norm{\partial_{t}u}_{C_{0}([0,\infty) ; \El{p})}\eqsim \norm{f}_{\dot{\mathcal{V}}^{1,p}}\eqsim \norm{g}_{\El{p}}$.

    \item Let us now assume that $p\geq n$. We can still find a sequence $(f_k)_{k\geq 1}\subset C^{\infty}_{c}(\Rn)\subset \mathcal{V}^{1,2}\cap\dot{\mathcal{V}}^{1,p}$ such that $f_k\to f$ in $\dot{\mathcal{V}}^{1,p}$, but, in contrast with Case 1, convergence no longer holds in the $\Ell{p^{*}}$ norm. As usual, we let $u_k(t,\cdot):=e^{-tH^{1/2}}f_k$ be the weak solution associated to the data $f_k$. As in Case 1, it follows from the first estimate of Proposition~\ref{proposition estimate non tangential ùax function regularity problem p>1} that the sequence $(\nabla_{\mu}u_k)_{k\geq 1}$ is Cauchy in $\tent{0,p}_{\infty}$, hence we can find some $v\in\tent{0,p}_{\infty}\subseteq \El{2}_{\loc}(\Hn)$ such that $\nabla_{\mu}u_k\to v$ in $\tent{0,p}_{\infty}$, whence also in the $\El{2}_{\loc}(\Hn)$ topology.
    Let $Q\subset\Rn$ and $0<a<b$ such that $b-a=\ell(Q)$.
    For arbitrary $u\in\mathcal{V}^{1,2}_{\loc}(\Hn)$, the local Fefferman--Phong inequality for the $t$-independent extension $V\in\textup{RH}^{2}(\R^{1+n})$ and the cube $(a,b)\times Q\subset \Hn$ yields
    \begin{equation*}
        \min\set{\ell(Q)^{-1},\left(\dashint_{Q}V\right)^{1/2}}\norm{u}_{\El{2}((a,b)\times Q)}\lesssim \norm{\nabla_{\mu}u}_{\El{2}((a,b)\times Q)}. 
    \end{equation*}
    Since $V$ is not identically zero, for all sufficiently large cubes $Q\subset \Rn$, the constant $C(Q,V):=\min\set{\ell(Q)^{-1},\left(\dashint_{Q}V\right)^{1/2}}$ is strictly positive. Since $(\nabla_{\mu}u_k)_{k\geq 1}$ converges in $\El{2}_{\loc}(\Hn)$, this shows that $(u_k)_{k\geq 1}$ is Cauchy in $\El{2}_{\loc}(\Hn)$. Let $u\in\El{2}_{\loc}(\Hn)$ be such that $u_k\to u$ in $\El{2}_{\loc}(\Hn)$ as $k\to \infty$. It follows from Lemma~\ref{limit of weak solutions is weak solution to elliptic equation} that $u$ is a weak solution and that $u_k\to u$ in $\mathcal{V}^{1,2}_{\loc}(\Hn)$, whence $v=\nabla_{\mu}u\in\tent{0,p}_{\infty}$, with the estimate $\norm{N_{*}(\nabla_{\mu}u)}_{\El{p}}\eqsim \norm{f}_{\dot{\mathcal{V}}^{1,p}}$.
    All properties of the solution $u$ follow as in the case $p<n$, except for (ii), where we have explicitly used the fact that $u\in C_{0}([0,\infty),\El{p^{*}})$. We must proceed differently. Since $p\geq n\geq 2>p_{-}(H)$, (vi) can be applied to obtain that $\partial_{t}u\in C_{0}([0,\infty) ; \Ell{p})$. Moreover, it follows from \cite[Lemma 4.2]{MorrisTurner} that $u\in C^{\infty}((0,\infty);\El{2}_{\loc}(\Rn))$, hence for a fixed compact subset $K\subset \Rn$ and for all $0<t_0<t_1\leq 1$ we can use the fundamental theorem of calculus to estimate
    \begin{align*}
        \norm{u(t_1,\cdot) - u(t_0,\cdot)}_{\El{2}(K)}&\leq \int_{t_0}^{t_1}\norm{\partial_{t}u(s,\cdot)}_{\El{2}(K)}\dd s
        \leq \meas{K}^{\frac{1}{2}-\frac{1}{p}}\int_{t_0}^{t_1}\norm{\partial_{t}u(s,\cdot)}_{\El{p}}\dd s\\
        &\leq \meas{K}^{\frac{1}{2}-\frac{1}{p}}\abs{t_1-t_0}^{1/2}\left(\int_{0}^{1}\norm{\partial_{t}u(s,\cdot)}_{\El{p}}^{2}\dd s\right)^{1/2},
    \end{align*}
    where we have used Minkowski's inequality. This shows that $\set{u(t,\cdot)}_{0<t\leq 1}$ is a Cauchy net in $\El{2}_{\loc}(\Rn)$. Consequently, there is some $F\in\El{2}_{\loc}(\Rn)$ such that $\lim_{t\to 0^{+}}u(t,\cdot)=F$ in the $\El{2}_{\loc}(\Rn)$ topology, so $\lim_{t\to 0^{+}}\nabla_{x}u(t,\cdot)=\nabla_{x}F$ in the distributional sense. Since we know that $\nabla_{x}u(t,\cdot)\to \nabla_{x}f$ in $\Ell{p}$ as $t\to 0$, we find that $\nabla_{x}F=\nabla_{x}f\in\El{p}$. In particular, there is a $c\in\C$ such that $F=f+c$. For all $x\in\Rn$ and $t>0$, we can therefore decompose 
    \begin{align*}
        cV(x)^{1/2}=V(x)^{1/2}(F(x)-u(t,x))+V(x)^{1/2}(u(t,x)-f(x)).
    \end{align*}
    Thus, by Hölder's inequality, we obtain for all $t>0$ and all compact subsets $K\subset \Rn$:
    \begin{equation*}
        \abs{c}\int_{K}V^{1/2} \leq \left(\int_{K}V\right)^{1/2}\norm{u(t,\cdot)-F}_{\El{2}(K)}+\meas{K}^{1-\frac{1}{p}}\norm{\nabla_{\mu}^{\parallel}u(t,\cdot)-\nabla_{\mu}f}_{\El{p}}.
    \end{equation*}
    Letting $t\to 0$ implies that $\abs{c}\int_{K}V^{1/2}=0$. Since $V$ is not identically zero, choosing $K\subset\Rn$ sufficiently large implies that $c=0$. Consequently, $F=f$ and $\lim_{t\to 0}u(t,\cdot)=f$ in $\El{2}_{\loc}(\Rn)$.
    We can finally apply \cite[Proposition A.5]{AuscherEgert} as in the case $p<n$ to find a non-tangential trace $u_{0}\in\El{1}_{\loc}(\Rn)$ such that $\lim_{t\to 0^{+}}\dashint_{t/2}^{2t}u(s,\cdot)\dd s=u_0$ in the distributional sense. The $\El{2}_{\loc}(\Rn)$ convergence implies that the limit above also holds in the $\El{2}_{\loc}(\Rn)$ topology. It follows that $u_0=f$ almost everywhere on $\Rn$, and this finishes the proof.\qedhere
    \end{enumerate}
\end{proof}
We conclude this section with the solvability and estimates for Theorem~\ref{existence and uniqueness result Regularity Lp} in the range $p\in(\frac{n}{n+1},1]$.
\begin{thm}\label{existence and uniqueness result Regularity for p<1}
    Let $n\geq 3$, $q\geq\max\{\frac{n}{2},2\}$ and $V\!\in\!\textup{RH}^{q}(\Rn)$. Suppose that $p\!\in\!\left(p_{-}(H)_{*}\!\vee\! \frac{n}{n+1}, 1\right]$. If $f\in\dot{\textup{H}}^{1,p}_{V}(\Rn)$, then the Regularity problem $\left(\mathcal{R}\right)^{\mathscr{H}}_{p}$ for data $f\in\dot{\textup{H}}^{1,p}_{V}$ admits a weak solution $u\in\mathcal{V}^{1,2}_{\textup{loc}}(\Hn)$ with the following additional properties:
    \begin{enumerate}[label=\emph{(\roman*)}]
        \item The following estimates hold:
        \begin{align*}
\norm{S(t\nabla_{\mu}\partial_{t}u)}_{\El{p}}\lesssim\norm{f}_{\dot{\textup{H}}^{1,p}_{V}}\quad \textup{and}\quad
\norm{N_{*}\left(\nabla_{\mu}u\right)}_{\El{p}}\lesssim\norm{f}_{\dot{\textup{H}}^{1,p}_{V}}.
        \end{align*}
        Moreover, if ${p\in\left(p_{-}(H)_{*}\vee \frac{n}{n+\delta/2}, 1\right]\cap\mathcal{I}(V)}$, then the reverse square function estimate holds,
    and if $V\in\textup{RH}^{\infty}$ and $p\in\mathcal{I}(V)$, then the reverse non-tangential maximal function estimate holds.
    \item Properties \emph{(ii)} and \emph{(v)} from Theorem~\ref{existence and uniqueness result Regularity Lp again} hold.
    \item If $p\in\left(p_{-}(H)_{*}\vee \frac{n}{n+\delta/2}, 1\right]\cap\,\mathcal{I}(V)$, then $u\in C_{0}\left(\left[0,\infty \right) ;\dot{\textup{H}}^{1,p}_{V} \right)\cap C^{\infty}\left((0,\infty);\dot{\textup{H}}^{1,p}_{V}\right)$, where $t$ is the distinguished variable, and $u\left(0,\cdot\right)=f$, with $\sup_{t\geq 0}\norm{u\left(t,\cdot\right)}_{\dot{\textup{H}}^{1,p}_{V}}\eqsim \norm{f}_{\dot{\textup{H}}^{1,p}_{V}}$. 
    \end{enumerate}
\end{thm}
\begin{proof}
    Again, if $V\equiv 0$, this follows from \cite[Theorem 1.2]{AuscherEgert}, and we can therefore assume that $V\not\equiv 0$. Let $p\in\left(p_{-}(H)_{*}\vee \frac{n}{n+1}, 1\right]$ and $f\in\dot{\textup{H}}^{1,p}_{V}(\Rn)$. By definition (recall Section~\ref{subsection definition of the completion of DZiubanski H^{1,p} spaces}),  we have $f\in\Ell{p^{*}}$ and there exists a sequence $(f_k)_{k\geq 1}\subset \dot{\textup{H}}^{1,p}_{V,\textup{pre}}(\Rn)$ such that $f_k\to f$ in both $\dot{\textup{H}}^{1,p}_{V}(\Rn)$ and $\Ell{p^{*}}$. We can then follow the proof of Theorem~\ref{existence and uniqueness result Regularity Lp again} (in Case 1, i.e. $p<n$), using Propositions~\ref{proposition  L^p square function estimate regularity problem}, \ref{proposition non tangential max function estimate regularity problem for p<1} and \ref{proposition strong continuity estimate regularity problem for p<1} in the range $\frac{n}{n+1}<p\leq 1$.
\end{proof}

\section{Uniqueness for Theorems~\ref{existence and uniqueness result Dirichlet Lp} and \ref{existence and uniqueness result Regularity Lp}}\label{section on Uniqueness for the Dirichlet and Regularity problems}
In this last section, we prove uniqueness results for solutions of the Dirichlet problem $\left(\mathcal{D}\right)^{\mathscr{H}}_{p}$ and of the Regularity problem $\left(\mathcal{R}\right)^{\mathscr{H}}_{p}$. We apply the strategy used in \cite[Chapter 21]{AuscherEgert}, which originates in \cite{Auscher_Egert_Uniqueness}, to prove the following two theorems, which complete the proofs of Theorems \ref{existence and uniqueness result Dirichlet Lp} and \ref{existence and uniqueness result Regularity Lp}. Throughout this section, we assume that $n\geq 3$, $q\geq \max\set{\frac{n}{2},2}$ and $V\in\textup{RH}^{q}(\Rn)$.
\begin{thm}\label{uniqueness result for dirichlet problem}
    Suppose that $V\not\equiv 0$ and $p\in(p_{-}(H)_{*}\vee \frac{n}{n+1}\,, p_{+}(H))$. 
    If either $p\leq 1$ and $f\in\dot{\textup{H}}^{1,p}_{V}(\Rn)$, or $p>1$ and $f\in\dot{\mathcal{V}}^{1,p}(\Rn)$, then the modified Regularity problem for data $f$
    \begin{equation*}
\left(\widetilde{\mathcal{R}}\right)^{\mathscr{H}}_{p}\hspace{0.5cm}\left\{\begin{array}{l}
        \mathscr{H}u=0  \text{ weakly in } \Hn;\\
        N_{*}(\nabla_{t,x}u)\in\Ell{p};\\
        \lim_{t\to 0^{+}}\fiint_{W(t ,x)}\abs{u(s,y)-f(x)}\dd s\dd y = 0 \text{ for a.e. } x\in\Rn;
    \end{array}\right.
\end{equation*}
admits at most one solution in $\mathcal{V}^{1,2}_{\loc}(\Hn)$. 
\end{thm}
Note that only the non-tangential estimate on the gradient of the solution is required to obtain uniqueness, and not the full control $N_{*}(\nabla_{\mu}u)\in\El{p}$.
\begin{thm}\label{uniqueness result for regularity problem}
    Suppose that $p\in[1,\infty)\cap(p_{-}(H), p_{+}(H)^{*})$. If either $p>1$ and $f\in\Ell{p}$, or $p=1$ and $f\in b\textup{H}^{1}_{V}(\Rn)$, then the Dirichlet problem $\left(\mathcal{D}\right)^{\mathscr{H}}_{p}$ admits at most one solution in $\mathcal{V}^{1,2}_{\loc}(\Hn)$.
\end{thm}
\subsection{Further estimates and representations for solutions}
In order to prove these theorems, first note that since the equation $\mathscr{H}u=0$ is linear, it suffices to prove that a solution $u$ to $\mathscr{H}u=0$ on $\Hn$ with zero boundary data vanishes identically (almost everywhere) on $\Hn$. The first basic result in this direction is the following.
\begin{lem}\label{first elementary basic uniqueness result}
    Let $u\in\mathcal{V}^{1,2}_{\loc}(\Hn)$ be a weak solution of $\mathscr{H}u=0$ on $\Hn$ such that
    \begin{align}\label{zero boundary condition in non tangential almost everywhere sense}
        \lim_{t\to 0^{+}}\fiint_{W(t,x)}\abs{u(s,y)}\dd s\dd y=0
    \end{align}
    for almost every $x\in\Rn$. If $\langle u, G\rangle_{\El{2}(\Hn)}=0$ for all $G=\partial_{t}\widetilde{G}$, $\widetilde{G}\in C^{\infty}_{c}(\Hn)$, then $u=0$ almost everywhere.
\end{lem}
\begin{proof}
    Since $u\in\textup{W}^{1,2}_{\loc}(\Hn)$, we can integrate by parts to find that $\langle \partial_{t}u , \tilde{G}\rangle_{\El{2}(\Hn)}=0$ for all $\tilde{G}\in C^{\infty}_{c}(\Hn)$. It follows that $\partial_{t}u=0$ almost everywhere on $\Hn$. Since $u\in C^{\infty}((0,\infty);\El{2}_{\loc}(\Rn))$ (iterate \cite[Lemma 4.2]{MorrisTurner} as in the proof of \cite[Lemma 16.9]{AuscherEgert}), there exists $g\in\El{2}_{\loc}(\Rn)$ such that $u(t,\cdot)=g$ for all $t>0$. It follows from (\ref{zero boundary condition in non tangential almost everywhere sense}) and the equality 
    \begin{align*}
        \fiint_{W(t,x)}\abs{u(s,y)}\dd s\dd y=\dashint_{B(x,t)}\abs{g(y)}\dd y
    \end{align*}
    that $g(x)=0$ for almost every $x\in\Rn$, by Lebesgue's differentiation theorem. 
\end{proof}

We now introduce the Hilbert space 
\begin{equation*}
    \dot{\mathcal{X}}^{1,2}(\R^{1+n}):=\dot{\mathcal{V}}^{1,2}(\R^{1+n})\cap \El{2^{*}_{+}}(\R^{1+n}),
\end{equation*}
with the inner-product $\langle u,v\rangle_{\dot{\mathcal{X}}^{1,2}}:={\langle \nabla_{\mu}u,\nabla_{\mu}v\rangle}_{\El{2}}$. Again,  $2^{*}_{+}=\frac{2(n+1)}{(n+1)-2}$ is the upper Sobolev exponent in dimension $n+1$. Note that $\dot{\mathcal{X}}^{1,2}(\R^{1+n})$ coincides with the space $\dot{\mathcal{V}}^{1,2}(\R^{1+n})$ defined in \cite[Section~2]{MorrisTurner}, which shows that $C^{\infty}_{c}(\R^{1+n})$ is dense in $\dot{\mathcal{X}}^{1,2}(\R^{1+n})$, so using the argument outlined at the start of Section~\ref{section on properties of weak solutions and interior estimates}, the following ellipticity estimate holds for all $u\in\dot{\mathcal{X}}^{1,2}(\R^{1+n})$:
\begin{equation}\label{ellipticity extended to whole space for u in Y^1,2}
    \Re{\langle \mathcal{A}\nabla_{\mu}u , \nabla_{\mu}u \rangle_{\El{2}(\R^{1+n};\C^{n+2})}}\geq \lambda \norm{\nabla_{\mu}u}_{\El{2}(\R^{1+n}; \C^{n+2})}^{2}.
\end{equation}
Recall the definition of $\mathcal{A}$ in \eqref{definition of mathcal A coefficients on the full space}, and that $V$ and $A$ are extended to $\R^{1+n}$ by $t$-independence.

We now consider the sesquilinear form $\dot{h}: \dot{\mathcal{X}}^{1,2}(\R^{1+n})\times \dot{\mathcal{X}}^{1,2}(\R^{1+n})\to \C$,
defined as $\dot{h}(u,v):=\langle \mathcal{A}\nabla_{\mu}u,\nabla_{\mu}v\rangle_{\El{2}(\R^{1+n};\C^{n+2})}$ for all $u,v\in\dot{\mathcal{X}}^{1,2}(\R^{1+n})$. 
For fixed $u\in\dot{\mathcal{X}}^{1,2}(\R^{1+n})$, we define $\dot{\mathscr{H}}u\in \dot{\mathcal{X}}^{1,2}(\R^{1+n})^{*}$ by 
\begin{equation}\label{definition of mathscrH by the form dot h}
\langle \dot{\mathscr{H}}u , v\rangle:=\dot{h}(u,v)
\end{equation}
for all $v\in\dot{\mathcal{X}}^{1,2}(\R^{1+n})$, where the angled brackets represent the action of the dual space $\dot{\mathcal{X}}^{1,2}(\R^{1+n})^{*}$.
By the ellipticity \eqref{ellipticity extended to whole space for u in Y^1,2}, we have $\Re{\dot{h}(u,u)}\gtrsim \norm{\nabla_{\mu}u}_{\El{2}}^{2}=\norm{u}_{\dot{\mathcal{X}}^{1,2}(\R^{1+n})}^{2}$, so the Lax--Milgram theorem shows that the mapping $\dot{\mathscr{H}} : \dot{\mathcal{X}}^{1,2}(\R^{1+n})\to \dot{\mathcal{X}}^{1,2}(\R^{1+n})^{*}$ is a (bounded) linear isomorphism. In particular, we can consider the inverse operator
$\dot{\mathscr{H}}^{-1}: \dot{\mathcal{X}}^{1,2}(\R^{1+n})^{*} \to \dot{\mathcal{X}}^{1,2}(\R^{1+n})$. 

Let us note that any test function $G\in C^{\infty}_{c}(\R^{1+n})$ is naturally identified with an element of the dual $\dot{\mathcal{X}}^{1,2}(\R^{1+n})^{*}$, as the linear functional
\begin{equation}\label{definition identification of smcp function with element of dual of V1,2dot}
    \dot{\mathcal{X}}^{1,2}(\R^{1+n})\to \C : \varphi \mapsto \langle \varphi, G\rangle_{\El{2}(\R^{1+n})}.
\end{equation}
Continuity of the functional follows by a Sobolev embedding in $1+n>2$ dimensions as follows:  
\begin{align*}
    \abs{\langle \varphi, G\rangle_{\El{2}}}&\leq \norm{G}_{\El{\frac{2n+2}{n+3}}}\norm{\varphi}_{\El{\frac{2n+2}{n-1}}}\lesssim \norm{G}_{\El{\frac{2n+2}{n+3}}}\norm{\nabla_{t,x}\varphi}_{\El{2}}\lesssim\norm{\varphi}_{\dot{\mathcal{X}}^{1,2}}.
\end{align*}
We shall need the following analogue of \cite[Lemma 3.1]{Auscher_Egert_Uniqueness} for Schrödinger operators.
\begin{lem}\label{smoothness of weak solutions in the t variable}
    Let $n\geq 3$ and $G\in C^{\infty}_{c}(\R^{1+n})$. If $u:=\dot{\mathscr{H}}^{-1}(G)\in\dot{\mathcal{X}}^{1,2}(\R^{1+n})$, then the following properties hold for all integers $k\geq 1$:
    \begin{enumerate}[label=\emph{(\roman*)}]
        \item $\partial_{t}^{k}u\in\mathcal{V}^{1,2}(\R^{1+n})$, and $\dot{\mathscr{H}}(\partial_{t}^{k}u)=\partial_{t}^{k}G$. 
        \item $\partial_{t}^{k}u\in C_{0}(\R ; \mathcal{V}^{1,2}(\Rn))$, with $t\in\R$ being the distinguished variable.
    \end{enumerate}
\end{lem}
\begin{proof}
    We first prove (i). This is similar to the proof of \cite[Lemma 4.2]{MorrisTurner}, except that $G$ is not identically zero.
    Let $e_{0}\in\R^{1+n}$ be the unit vector in the positive $t$ direction. We consider the difference quotient $\bb{D}_{0}^{h}u:\R^{1+n}\to\C$, defined as
    \begin{align*}
        \bb{D}_{0}^{h}u(s,y):=h^{-1}\left[(u((s,y)+h e_{0})-u((s,y))\right]
    \end{align*}
    for all $(s,y)\in\R^{1+n}$ and $h\in\R\setminus\set{0}$.
By $t$-independence of $\mathcal{A}$, we get $\dot{\mathscr{H}}\left(\bb{D}^{h}_{0}u\right)=\bb{D}^{h}_{0}G$ for all $h\in\R\setminus\set{0}$.
It follows from Caccioppoli's inequality \eqref{Caccioppoli's inequality} that for all cubes $Q\subset\R^{1+n}$ we have
    \begin{align*}
        \iint_{Q}\abs{\nabla_{\mu}(\bb{D}^{h}_{0}u)}^{2}&\lesssim \ell(Q)^{-2}\iint_{2Q}\abs{\bb{D}^{h}_{0}u}^{2} + \left(\iint_{2Q}\abs{\bb{D}^{h}_{0}u}^{2}\right)^{1/2}\left(\iint_{2Q}\abs{\bb{D}^{h}_{0}G}^{2}\right)^{1/2}\\
        &\lesssim\ell(Q)^{-2}\norm{\partial_{t}u}_{\El{2}}^{2} + \norm{\partial_{t}u}_{\El{2}}\norm{\partial_{t}G}_{\El{2}},
    \end{align*}
    where we have used \cite[Lemma 7.23]{Gilbarg_Trudinger_Elliptic}. Letting $\ell(Q)\to\infty$, this yields 
    \begin{equation}\label{uniform estimate on difference quotients}
    \iint_{\R^{1+n}}\abs{\nabla_{\mu}(\bb{D}_{0}^{h}u)}^{2}\lesssim\norm{\partial_{t}u}_{\El{2}}\norm{\partial_{t}G}_{\El{2}}
    \end{equation}
    uniformly for all $h\in\R\setminus\set{0}$. 
    Now, since $u\in\dot{\mathcal{X}}^{1,2}(\R^{1+n})$ and $V$ is $t$-independent, we clearly have that $\bb{D}^{h}_{0}u\in\dot{\mathcal{X}}^{1,2}(\R^{1+n})$, with $\nabla_{\mu}\left(\bb{D}^{h}_{0}u\right)=\bb{D}^{h}_{0}(\nabla_{\mu}u)$.
    It follows from this, \eqref{uniform estimate on difference quotients} and the Sobolev difference quotient characterisation (see, e.g., \cite[Lemma 7.24]{Gilbarg_Trudinger_Elliptic}) that $\nabla_{\mu}u$ has a weak $\partial_{t}$-derivative in $\El{2}(\R^{1+n})$, denoted $\partial_{t}\nabla_{\mu}u\in\El{2}(\R^{1+n})$. 
    It follows from the proof of \cite[Lemma 4.2]{MorrisTurner} that $\partial_{t}u\in\mathcal{V}^{1,2}(\R^{1+n})$, with
    \begin{align}\label{interchange derivatives nabla mu partial t for weak solutions}
\nabla_{\mu}\left(\partial_{t}u\right)=\partial_{t}\left(\nabla_{\mu}u\right)\in\El{2}(\R^{1+n}),
    \end{align}
    from which it follows that $\dot{\mathscr{H}}(\partial_{t}u)=\partial_{t}G$. 
    Since $\partial_{t}^{k}G\in C^{\infty}_{c}(\R^{1+n})$ for all $k\geq 1$, we can proceed by induction to obtain that $\partial_{t}^{k}u\in\mathcal{V}^{1,2}(\R^{1+n})$ with $\nabla_{\mu}(\partial_{t}^{k}u)=\partial_{t}^{k}(\nabla_{\mu}u)$ and $\dot{\mathscr{H}}(\partial_{t}^{k}u)=\partial_{t}^{k}G$.
    The proof of (ii) follows as in the proof of \cite[Lemma 3.1]{Auscher_Egert_Uniqueness} from the vector-valued Sobolev embedding $\textup{W}^{1,2}(\R;X)\subseteq C_{0}(\R ; X)$ for the Banach space $X=\Ell{2}$. 
    \end{proof}

We shall now use the qualitative information provided by the previous lemma to derive a representation formula for solutions $u$ of $\dot{\mathscr{H}}u=G\in C^{\infty}_{c}(\R^{1+n})$, which is the analogue of \cite[Corollary 20.3]{AuscherEgert} in our context. Recall that $b=A_{\perp\perp}^{-1}$.
\begin{lem}\label{representation formula of slices of weak solution by Poisson semigroup}
    If $\widetilde{G}\in C^{\infty}_{c}(\R^{1+n})$ and $G=\partial_{t}\widetilde{G}$, then for all $t\in\R$ it holds that  
    \begin{equation*}\label{representation formula for weak solutions of schrodinger by poisson semigroup}
        \dot{\mathscr{H}}^{-1}(G)(t,\cdot)=-\frac{1}{2}\int_{\R}\sgn{t-s}e^{-\abs{t-s}H^{1/2}}\left(b\widetilde{G}(s,\cdot)\right)\dd s,
    \end{equation*}
    as an absolutely convergent $\Ell{2}$-valued Bochner integral.
\end{lem}
\begin{proof}
Let $u:=\dot{\mathscr{H}}^{-1}(G)$, where $G=\partial_{t}\widetilde{G}$ for some $\widetilde{G}\in C^{\infty}_{c}(\R^{1+n})$, and consider $F:=\begin{bmatrix}
        A_{\perp\perp}\partial_{t}u\\
        \nabla_{\mu}^{\parallel}u
    \end{bmatrix}\in\El{2}(\R^{1+n} ; \C^{n+2})$. The proof of Lemma~\ref{smoothness of weak solutions in the t variable} shows that $F$ can identified with an element of $C^{\infty}(\R ;\El{2}(\Rn;\C^{n+2}))$, where $t$ is the distinguished variable. We claim that $F_{t}:=F(t,\cdot)\in\clos{\ran{D}}\cap\dom{DB}$ for all $t\in\R$, and that it solves the first-order system
    \begin{equation}\label{F on slices solves 1st order ode}
        \partial_{t}F_{t}+DBF_{t}=-\begin{bmatrix}
            G(t,\cdot)\\
            0
        \end{bmatrix}
    \end{equation}
    in the strong sense for all $t\in\R$ (recall that $B\in\El{\infty}(\Rn; \mathcal{L}(\C^{n+2}))$ is defined in \eqref{definition of multiplication operator B associated to elliptic coefficients A}). 
    
    Let us first prove that $F_{t}\in\clos{\ran{D}}$ for all $t\in\R$. To this end, note that $u\in\El{2^{*}_{+}}(\R^{1+n})$ can be identified with an element of $\El{2^{*}_{+}}(\R;\El{2^{*}_{+}}(\Rn))$ via Fubini's theorem and therefore $u\in\El{2}_{\loc}(\R;\El{2}_{\loc}(\Rn))$. By Lemma~\ref{smoothness of weak solutions in the t variable} (and by one-dimensional Sobolev embeddings, see also \cite[Corollary 16.9]{AuscherEgert}), we deduce that $u\in C^{\infty}(\R;\El{2}_{\loc}(\Rn))$. It then follows from the smoothness (or just the continuity) of both $u$ and $F$ in the variable $t$ that $u(t,\cdot)\in\dot{\mathcal{V}}^{1,2}(\Rn)$ with $\nabla_{\mu}\left(u(t,\cdot)\right)=\left(\nabla_{\mu}^{\parallel}u\right)(t,\cdot)$ for all $t\in\R$.
    Since our assumptions on $V$ imply that $\dot{\mathcal{V}}^{1,2}(\Rn)\subseteq \El{2^{*}}(\Rn)$ by Proposition~\ref{embedding homogeneous V spaces in L ^p* for p<n}, it follows from \eqref{characterisation range D} that $F_{t}\in\clos{\ran{D}}$ for all $t\in\R$, as claimed.

We now show that $BF_{t}\in\dom{D}=\mathcal{V}^{1,2}(\Rn)\times \dom{\nabla_{\mu}^{*}}$ for all $t\in\R$. Note that the argument from the previous paragraph shows that
$BF_{t}=\begin{bmatrix}
        \partial_{t}u(t,\cdot)\\
        A_{\parallel\parallel}\nabla_{x}u(t,\cdot)\\
        aV^{1/2}u(t,\cdot)
    \end{bmatrix}.$
We already know from Lemma~\ref{smoothness of weak solutions in the t variable} that $\partial_{t}u(t,\cdot)\in\mathcal{V}^{1,2}(\Rn).$ For the other components, we let $A_{\parallel,\mu}:=\begin{bmatrix}
    A_{\parallel\parallel} & 0\\
    0 & a
\end{bmatrix}\in\El{\infty}(\Rn;\mathcal{L}(\C^{n+1}))$,
and we note that for $\phi\in C^{\infty}_{c}(\R)$ and $\varphi\in \mathcal{V}^{1,2}(\Rn)$, the function $(\phi\varphi)(t,x):=\phi(t)\varphi(x)$ satisfies $\nabla_{\mu}(\varphi\phi)(t,x)=\begin{bmatrix}
    \varphi(x)\phi'(t)\\
    \phi(t)\nabla_{\mu}\varphi(x)
\end{bmatrix}$. It follows from an integration by parts in $t$ that
\begin{align*}
    &\int_{\R}\left(\langle A_{\parallel,\mu}\nabla_{\mu}^{\parallel}u(t,\cdot) , \nabla_{\mu}\varphi\rangle_{\Ell{2}} - \langle A_{\perp\perp}\partial_{t}^{2}u(t,\cdot) , \varphi\rangle_{\El{2}}\right)\clos{\phi(t)}\dd t\\
    &=\int_{\R}\langle A_{\parallel,\mu}\nabla_{\mu}^{\parallel}u(t,\cdot) , \nabla_{\mu}\varphi\rangle_{\Ell{2}}\clos{\phi(t)}\dd t + \int_{\R}\langle A_{\perp\perp}\partial_{t}u(t,\cdot), \varphi\rangle_{\Ell{2}}\clos{\phi'(t)}\dd t\\
    &=\iint_{\R^{1+n}}\mathcal{A}\nabla_{\mu}u \cdot \clos{\nabla_{\mu}(\varphi\phi)}\dd x\dd t=\langle \dot{\mathscr{H}}u,\phi\varphi\rangle=\int_{\R}\left(\int_{\R^{n}}G(t,\cdot)\clos{\varphi} \dd x\right) \clos{\phi(t)} \dd t.
\end{align*}
By the fundamental lemma of the calculus of variations, and smoothness in the variable $t$, it follows that 
\begin{align*}
    \langle A_{\parallel,\mu}\nabla_{\mu}^{\parallel}u(t,\cdot) , \nabla_{\mu}\varphi\rangle_{\El{2}} = \langle G(t,\cdot) , \varphi\rangle_{\El{2}} + \langle A_{\perp\perp}\partial_{t}^{2}u(t,\cdot),\varphi\rangle_{\El{2}}
\end{align*}
for all $t\in\R$. This means that $A_{\parallel,\mu}\nabla_{\mu}^{\parallel}u(t,\cdot)\in\dom{\nabla_{\mu}^{*}}$, with $\nabla_{\mu}^{*}\left(A_{\parallel,\mu}\nabla_{\mu}^{\parallel}u(t,\cdot)\right)=G(t,\cdot)+A_{\perp\perp}\partial_{t}^{2}u(t,\cdot)$. Consequently, $BF_{t}\in\dom{D}$ and by \eqref{interchange derivatives nabla mu partial t for weak solutions} we are justified in writing 
\begin{align*}
    DBF_{t}&=\begin{bmatrix}
        0 & -\nabla_{\mu}^{*}\\
        -\nabla_{\mu} & 0
    \end{bmatrix}\begin{bmatrix}
        \partial_{t}u\\
        A_{\parallel,\mu}\nabla_{\mu}^{\parallel}u
    \end{bmatrix}_{t}=\begin{bmatrix}
        -A_{\perp\perp}\partial_{t}^{2}u(t,\cdot) - G(t,\cdot)\\
        -\nabla_{\mu}^{\parallel}(\partial_{t}u)
    \end{bmatrix}=-\partial_{t}F_{t}-\begin{bmatrix}
        G(t,\cdot)\\
        0
    \end{bmatrix}
    \end{align*}
     for all $t\in\R$, which proves \eqref{F on slices solves 1st order ode}. We can then follow the Proof of Proposition~4.5 in \cite[Section~7]{Auscher_Egert_Uniqueness}, along with that of \cite[Proposition 20.2]{AuscherEgert} to obtain \eqref{representation formula for weak solutions of schrodinger by poisson semigroup}.
\end{proof}

Our next goal is to be able to relate $\dot{\mathscr{H}}^{-1}(G)(t,\cdot)$ to the boundary value $\dot{\mathscr{H}}^{-1}(G)(0,\cdot)$. This requires the following boundedness properties of the Poisson semigroup $\set{e^{-tH^{1/2}}}_{t>0}$, analogous to \cite[Proposition 12.5]{AuscherEgert}.
\begin{lem}
    If $p_{-}(H)\vee \frac{n}{n+\delta/2}<p\leq r<p_{+}(H)$, then the Poisson semigroup $\set{e^{-tH^{1/2}}}_{t>0}$ is $b\textup{H}^{p}_{V}-b\textup{H}^{r}_{V}$-bounded. If in addition $p_{-}(H^{\sharp})<1$, then it is $b\textup{H}^{p}_{V}-\El{\infty}$-bounded. 
\end{lem}
\begin{proof}
    We follow the proof of \cite[Proposition 12.5]{AuscherEgert}, using the fact that the Riesz transform $R_{H}$ is $\El{2_{*}}$-bounded by Theorem~\ref{thm boundedness of Riesz transforms on Lp full range} and Corollary~\ref{improvement on property of interval J(H)}. We also rely on Corollary~\ref{characterisation abstract H^p p<2} and Proposition~\ref{bound on power of family} to prove the first part. 
    For the second part, we proceed as in the aforementioned reference, except that for $0<\alpha<n\left(\frac{1}{p_{-}(H^{\sharp})}-1\right)$, we pick $\rho\in (p_{-}(H^{\sharp})\vee \frac{n}{n+\delta/2},1)$ such that $\alpha=n\left(\frac{1}{\rho}-1\right)$. Then, we proceed by similarity and duality by using Lemma~\ref{dualisation of boundedness of families in Lp spaces and Hp spaces-campanato} (noting that $\frac{n}{n+1\wedge \delta}\leq \frac{n}{n+\delta/2}$). We conclude with the interpolation inequality of \cite[Lemma 12.6]{AuscherEgert} and the fact that $\Lambda^{\alpha}_{V}$ embeds continuously into the space of $\alpha$-Hölder continuous functions on $\Rn$.
\end{proof}
The previous two lemmas allow us to obtain the following analogue of \cite[Lemma 20.4]{AuscherEgert}. 
\begin{lem}\label{quantitative comparison between solution and Poisson extension of its trace}
    Let $\widetilde{G}\in C^{\infty}_{c}(\R^{1+n})$, and consider $G:=\partial_{t}\widetilde{G}\in C^{\infty}_{c}(\R^{1+n})$. Let $F:=\dot{\mathscr{H}}^{-1}(G)\in\dot{\mathcal{X}}^{1,2}(\R^{1+n}).$ Assume that $\supp{\widetilde{G}}\subseteq [\beta^{-1},\beta]\times\Rn$ for some $\beta>1$. Let $f:=F(0,\cdot)\in\Ell{2}$. Then, for $r\in (p_{-}(H)\vee 1 , p_{+}(H))$, there exists $\gamma>0$ such that the following estimates hold for all $t>0$:
    \begin{equation*}
        \norm{F(t,\cdot)-e^{-tH^{1/2}}f}_{\El{r}}\lesssim t\wedge t^{-\gamma}\quad\textup{and}\quad
        \norm{\partial_{t}\left(F(t,\cdot)-e^{-tH^{1/2}}f\right)}_{\El{r}}\lesssim 1\wedge t^{-1-\gamma},
    \end{equation*}
    where the implicit constants also depend on $\beta$. If in addition $p_{-}(H^{\sharp})<1$, then both estimates also hold for $r=\infty.$
\end{lem}
Finally, we shall need a version of the technical lemma \cite[Lemma 21.3]{AuscherEgert}. In our setting, it trivially extends to the following result.
\begin{lem}\label{technical lemma for uniqueness of solutions}
    Let $0<p<r\leq 2$ and let $u\in\mathcal{V}^{1,2}_{\loc}(\Hn)$ be a weak solution of $\dot{\mathscr{H}}u=0$ in $\Hn$ such that $\norm{N_{*}(\nabla_{t,x}u)}_{\El{p}}<\infty$. Then, 
    \begin{equation*}
        \left(\iint_{\Hn}\abs{\nabla_{t,x}u}^{r}t^{n(\frac{r}{p}-1)}\frac{\dd x\dd t}{t}\right)^{1/r}\lesssim \norm{N_{*}(\nabla_{t,x}u)}_{\El{p}}.
    \end{equation*}
If in addition $p>\frac{n}{n+1}$ and the boundary limit \eqref{zero boundary condition in non tangential almost everywhere sense} holds, then also
    \begin{equation*}
        \left(\iint_{\Hn}\abs{u}^{r}t^{n(\frac{r}{p}-1)-r}\frac{\dd x\dd t}{t}\right)^{1/r}\lesssim \norm{N_{*}(\nabla_{t,x}u)}_{\El{p}}.
    \end{equation*}
\end{lem}

\begin{proof}
    Apply \cite[Proposition A.5]{AuscherEgert} and \cite[Lemma A.3]{AuscherEgert} along with the reverse Hölder estimates \eqref{reverse holder property weak solutions} for weak solutions.
\end{proof}

\subsection{The Dirichlet and Regularity problems}\label{section on proof of uniqueness}
We now follow the method described in \cite[Chapter 21]{AuscherEgert} (or, equivalently, \cite{Auscher_Egert_Uniqueness}) to simultaneously establish uniqueness of solutions for the problems $(\mathcal{D})^{\mathscr{H}}_{p}$ and $(\mathcal{R})^{\mathscr{H}}_{p}$, using the results collected in the previous section. 

\begin{proof}[Proof of Theorems \ref{uniqueness result for dirichlet problem} and \ref{uniqueness result for regularity problem}]
Let $u\in\mathcal{V}^{1,2}_{\loc}(\Hn)$ be a weak solution of $\mathscr{H}u=0$ on $\Hn$. Let $\widetilde{G}\in C_{c}^{\infty}(\Hn)$, and consider $G:=\partial_{t}\widetilde{G}$. By linearity we only need to show that $\langle u ,G\rangle_{\El{2}(\Hn)}=0$ provided $u$ has some interior control (either $N_{*}(u)\in\El{p}$ or $N_{*}(\nabla_{t,x}u)\in\El{p}$ for some $p$) and vanishes at the boundary in the sense of (\ref{zero boundary condition in non tangential almost everywhere sense}), since then Lemma~\ref{first elementary basic uniqueness result} shows that $u=0$ almost everywhere. 

As in \cite[Chapter 21]{AuscherEgert}, in order to compute $\langle u,G\rangle_{\El{2}(\Hn)}$, we introduce a real-valued Lipschitz continuous cut-off function $\theta$ with compact support in $\Hn$, such that $\theta=1$ on $\supp{G}$. We consider the operator $\dot{\mathscr{H}}^{*}:\dot{\mathcal{X}}^{1,2}(\R^{1+n})\to \dot{\mathcal{X}}^{1,2}(\R^{1+n})^{*}$ defined as in \eqref{definition of mathscrH by the form dot h} above with $(A^{*},a^{*})$ in place of $(A,a)$. By definition and \eqref{definition identification of smcp function with element of dual of V1,2dot}, the function $F:=(\dot{\mathscr{H}}^{*})^{-1}(G)\in\dot{\mathcal{X}}^{1,2}(\R^{1+n})$ satisfies 
\begin{equation}\label{intermediate step F is a weak solution with rhs equal to G}
    \langle \mathcal{A}^{*}\nabla_{\mu}F,\nabla_{\mu}\varphi\rangle_{\El{2}}=\langle G,\varphi\rangle_{\El{2}}
\end{equation}
for all $\varphi\in\dot{\mathcal{X}}^{1,2}(\R^{1+n})$. 
Observe that $\mathcal{V}^{1,2}_{\loc}(\Hn)\subset \El{\frac{2n+2}{n-1}}_{\loc}(\Hn)$ by Sobolev embeddings on bounded smooth domains (see, e.g., \cite[Theorem 7.26]{Gilbarg_Trudinger_Elliptic}). As a consequence, $u\theta\in\dot{\mathcal{X}}^{1,2}(\R^{1+n})$, and it can be used as a test function in \eqref{intermediate step F is a weak solution with rhs equal to G} to obtain 
\begin{align}\label{expression for inner product u with G}
    \langle u,G\rangle_{\El{2}}&=\langle u ,G\theta\rangle_{\El{2}}=\langle u\theta ,G\rangle_{\El{2}}=\clos{\langle \mathcal{A}^{*}\nabla_{\mu}F,\nabla_{\mu}(u\theta)\rangle_{\El{2}}}=\langle \mathcal{A}\nabla_{\mu}(u\theta),\nabla_{\mu}F\rangle_{\El{2}}.
\end{align}
Similarly, note that $F\theta\in\dot{\mathcal{X}}^{1,2}(\R^{1+n})$, with
\begin{equation}\label{product rule for weak gradient nabla mu}
\nabla_{\mu}(F\theta)=\begin{bmatrix}
        F\nabla_{t,x}\theta\\
        0
    \end{bmatrix} +\theta\nabla_{\mu}F.
    \end{equation}
    
Since $F\theta$ has compact support, there is a sequence of suitably supported functions $(F_{k})_{k\geq 1}\subset C^{\infty}_{c}(\Hn)$ such that $\nabla_{\mu}F_{k}\to\nabla_{\mu}(F\theta)$ in $\El{2}(\Hn)$ (see the proof of \cite[Lemma 2.2]{MorrisTurner}). As $\mathscr{H}u=0$ in $\Hn$ (in the sense of \ref{definition of weak solutions in general open subset of R1+n}), we get
\begin{align*}
    \langle \mathcal{A}(\nabla_{\mu}u)\theta ,\nabla_{\mu}F\rangle_{\El{2}}&=\langle \mathcal{A}\nabla_{\mu}u,\nabla_{\mu}(F\theta)\rangle_{\El{2}}-\langle \mathcal{A}\nabla_{\mu}u, \begin{bmatrix}
        F\nabla_{t,x}\theta \\
        0
    \end{bmatrix}\rangle_{\El{2}}=-\langle \mathcal{A}\nabla_{\mu}u, \begin{bmatrix}
        F\nabla_{t,x}\theta \\
        0
    \end{bmatrix}\rangle_{\El{2}}.
\end{align*}
Using \eqref{expression for inner product u with G}, and applying \eqref{product rule for weak gradient nabla mu} with $u$ instead of $F$, this implies that 
\begin{align}
\begin{split}
\label{intermediate step towards identity for u against G}
    \langle u,G\rangle_{\El{2}}&=\langle \mathcal{A}\begin{bmatrix}
        u\nabla_{t,x}\theta\\
        0
    \end{bmatrix} , \nabla_{\mu}F\rangle_{\El{2}}-\langle \mathcal{A}\nabla_{\mu}u,\begin{bmatrix}
        F\nabla_{t,x}\theta\\
        0
    \end{bmatrix}\rangle_{\El{2}}\\
    &=\langle Au\nabla_{t,x}\theta ,\nabla_{t,x}F\rangle_{\El{2}}-\langle A\nabla_{t,x}u , F\nabla_{t,x}\theta\rangle_{\El{2}}.
    \end{split}
\end{align}
By Lemma~\ref{representation formula of slices of weak solution by Poisson semigroup}, we can consider the trace $f:=F(0,\cdot)\in\Ell{2}$, and its Poisson extension $F_{1}(t,\cdot):=e^{-t(H^{\sharp})^{1/2}}f$ for $t>0$. We know from Theorem~\ref{Poisson semigroup extension is a weak solution} that $F_{1}\in C_{0}([0,\infty);\El{2})\cap C^{\infty}((0,\infty);\El{2})$, with $F_{1}(0,\cdot)=f$, and that $F_{1}\in\mathcal{V}^{1,2}_{\loc}(\Hn)$ is a weak solution of $\mathscr{H}^{*}F_{1}=0$ in $\Hn$ (in the sense of \ref{definition of weak solutions in general open subset of R1+n}). As above, we can use the fact that both $u\theta$ and $F_{1}\theta\in \dot{\mathcal{X}}^{1,2}(\Hn)$ have compact support and that $u$ and $F_{1}$ are solutions to write 
\begin{align*}
    0&=\langle \nabla_{\mu}(u\theta) , \mathcal{A}^{*}\nabla_{\mu}F_1\rangle_{\El{2}}=\langle \mathcal{A}\nabla_{\mu}(u\theta),\nabla_{\mu}F_{1}\rangle_{\El{2}}\\
    &=\langle \mathcal{A}\begin{bmatrix}
        u\nabla_{t,x}\theta\\
        0
    \end{bmatrix},\nabla_{\mu}F_{1}\rangle_{\El{2}}+\langle \mathcal{A}\nabla_{\mu}u , \nabla_{\mu}(F_{1}\theta)\rangle_{\El{2}}-\langle \mathcal{A}\nabla_{\mu}u , \begin{bmatrix}
        F_{1}\nabla_{t,x}\theta\\
        0
    \end{bmatrix}\rangle_{\El{2}}\\
    &=\langle Au\nabla_{t,x}\theta , \nabla_{t,x}F_{1}\rangle_{\El{2}}-\langle A\nabla_{t,x}u,F_{1}\nabla_{t,x}\theta\rangle_{\El{2}},
\end{align*}
where we have used \eqref{product rule for weak gradient nabla mu} twice. We can therefore subtract this quantity from $\langle u ,G\rangle_{\El{2}}$ in \eqref{intermediate step towards identity for u against G} to obtain the following:
\begin{align*}
    \langle u,G\rangle_{\El{2}}=\langle Au\nabla_{t,x}\theta ,\nabla_{t,x}(F-F_1)\rangle_{\El{2}}-\langle A\nabla_{t,x}u , (F-F_1)\nabla_{t,x}\theta\rangle_{\El{2}}.
\end{align*}

This is precisely the same identity as in \cite[Equation (21.4)]{AuscherEgert}. We can proceed exactly as in \cite[Chapter 21]{AuscherEgert}, assuming that $\widetilde{G}\in C^{\infty}_{c}(\Hn)$ is supported on $[\beta^{-1} ,\beta]\times B(0,\beta)$ for some $\beta>0$ and picking a cut-off function $\theta=\theta_{\eps, M,R}$ for suitable parameters $\eps,M,R>0$. 
In particular, using Lemma~\ref{technical lemma for uniqueness of solutions}, the reverse Hölder estimate \eqref{reverse holder property weak solutions} and Caccioppoli's inequality \eqref{Caccioppoli's inequality} for weak solutions, as well as the estimates of Lemma~\ref{quantitative comparison between solution and Poisson extension of its trace} instead of \cite[Lemma 20.4]{AuscherEgert}, we obtain uniqueness of solutions to the Regularity problem $(\mathcal{R})^{\mathscr{H}}_{p}$ in the range $p_{-}(H)_{*}\vee \frac{n}{n+1} <p<p_{+}(H)$, and uniqueness of solutions to the Dirichlet problem $(\mathcal{D})^{\mathscr{H}}_{p}$ in the range $1\leq p <p_{+}(H)^{*}$ if $p_{-}(H)<1$, or in the range $p_{-}(H)<p<p_{+}(H)^{*}$ if $p_{-}(H)\geq 1$. Let us observe that we only need the non tangential control of the gradient $\nabla_{t,x}u$ of the solution $u$, i.e. $\norm{N_{*}\left(\nabla_{t,x}u\right)}_{\El{p}}<\infty$, and not the full control $\norm{N_{*}\left(\nabla_{\mu}u\right)}_{\El{p}}<\infty$ to obtain uniqueness.
This concludes the proof.
\end{proof}

\section{The Neumann problem and Theorem \ref{existence and uniqueness result Neumann Lp}}\label{section proof well posedness neumann}
In this final section, we combine Theorems~\ref{existence and uniqueness result Dirichlet Lp} and \ref{existence and uniqueness result Regularity Lp} with identifications from Theorem~\ref{summary of identifications of H adapted Hardys spaces} to prove well-posedness of the Neumann problem $\left(\mathcal{N}\right)^{\mathscr{H}}_{p}$ and Theorem \ref{existence and uniqueness result Neumann Lp}. The cornerstone of our argument is the following lemma, which enables us to go from Neumann-type boundary data to Regularity-type boundary data. Note that the inverse of the mapping $T_p$ below is related to the so-called Neumann to Dirichlet map (see, e.g., \cite[Section 3]{AMM_solvability_BVP_elliptic_systems}).
\begin{lem}\label{lemma extension of square root of H to isomorphism}
    Let $n\geq 3$, $q\geq\max\{\frac{n}{2},2\}$, and $V\in\textup{RH}^{q}(\Rn)$. If $p\in(p_{-}(H),q_{+}(H))\cap(1,2q]$, then the linear mapping $-A_{\perp\perp}H^{1/2} : \mathcal{V}^{1,2}(\Rn)\to \Ell{2}$ extends to an isomorphism \[T_{p}: \dot{\mathcal{V}}^{1,p}(\Rn)\to \Ell{p}.\] If $p_{-}(H)<1$ and $1\in \mathcal{I}(V)$, then it extends to an isomorphism \[T_{1}: \dot{\textup{H}}^{1,1}_{V}(\Rn)\to \textup{H}^{1}_{V}(\Rn).\] 
\end{lem}
\begin{proof}
    Let us first consider the case when $p\in(p_{-}(H),q_{+}(H))\cap(1,2q]$. It then follows from Theorem \ref{summary of identifications of H adapted Hardys spaces} and Figure~\ref{figure relations between adapted Hardy spaces} that for all $f\in\dot{\mathcal{V}}^{1,p}(\Rn)\cap\mathcal{V}^{1,2}(\Rn)$ there holds
    \begin{equation*}
        \norm{f}_{\dot{\mathcal{V}}^{1,p}}\eqsim \norm{f}_{\bb{H}^{1,p}_{H}}\eqsim \norm{H^{1/2}f}_{\bb{H}^{p}_{H}}\eqsim \norm{-A_{\perp\perp}H^{1/2}f}_{\El{p}}.
    \end{equation*}
    This estimate and Proposition~\ref{density lemma in homogeneous V adapted Sobolev spaces} show that $-A_{\perp\perp}H^{1/2}$ can be extended to a bounded mapping $T_{p}: \dot{\mathcal{V}}^{1,p}(\Rn)\to \Ell{p}$ satisfying $\norm{T_{p}f}_{\El{p}}\eqsim \norm{f}_{\dot{\mathcal{V}}^{1,p}}$ for all $f\in\dot{\mathcal{V}}^{1,p}(\Rn)$. The invertibility of $T_{p}$ follows from the density of $\ran{H^{1/2}}\cap \Ell{p}$ in $\Ell{p}$ (see Lemma~\ref{density lemma range dom L^p}). 
    
    Let us now suppose that $p_{-}(H)<1$ and $1\in\mathcal{I}(V)$. The self-improvement property from Lemma \ref{self-improvement property of reverse holder weights}.(i) shows that there exists $q_0> \frac{n}{2}$ such that $V\in\textup{RH}^{q_{0}}(\Rn)$. As a consequence, we have that $1\in\mathcal{I}(V)\cap(p_{-}(H)\vee \frac{n}{n+\delta_{0}/2},1]$, where $\delta_{0}:=2-\frac{n}{q_{0}}>0$, and it again follows from Theorem~\ref{summary of identifications of H adapted Hardys spaces} and Figure~\ref{figure relations between adapted Hardy spaces} that for all $f\in\dot{\textup{H}}^{1,1}_{V,\textup{pre}}(\Rn)$ there holds
    \begin{equation*}
        \norm{f}_{\dot{\textup{H}}^{1,1}_{V}}\eqsim \norm{f}_{\bb{H}^{1,1}_{H}}\eqsim \norm{H^{1/2}f}_{\bb{H}^{1}_{H}}\eqsim \norm{-A_{\perp\perp}H^{1/2}f}_{\textup{H}^{1}_{V}}.
    \end{equation*}
    By completeness of $\textup{H}^{1}_{V}(\Rn)$ and density of $\dot{\textup{H}}^{1,1}_{V,\textup{pre}}(\Rn)$ in $\dot{\textup{H}}^{1,1}_{V}(\Rn)$ (see Sections \ref{subsubsection with definition of completion H^{1}_{V}} and \ref{subsection definition of the completion of DZiubanski H^{1,p} spaces}), this allows us to extend $-A_{\perp\perp}H^{1/2}$ to a bounded mapping $T_{1}: \dot{\textup{H}}^{1,1}_{V}(\Rn) \to \textup{H}^{1}_{V}(\Rn)$ satisfying $\norm{T_{1}f}_{\textup{H}^{1}_{V}}\eqsim \norm{f}_{\dot{\textup{H}}^{1,1}_{V}}$ for all $f\in\dot{\textup{H}}^{1,1}_{V}(\Rn)$. As above, the invertibility of $T_{1}$ follows from the density of $\ran{H^{1/2}}\cap \left(-A_{\perp\perp}^{-1}\textup{H}^{1}_{V,\textup{pre}}(\Rn)\right)$ in $-A_{\perp\perp}^{-1}\textup{H}^{1}_{V}(\Rn)$ (when equipped with the natural norm mentioned at the start of Section \ref{section on identification of H adapted Hardy spaces}), which itself follows from the fact that $\bb{H}^{1}_{H}=-A_{\perp\perp}^{-1}\textup{H}^{1}_{V,\textup{pre}}(\Rn)$ with equivalent norms (see Theorem~\ref{summary of identifications of H adapted Hardys spaces}) and that $\ran{H^{1/2}}\cap \bb{H}^{1}_{H}$ is dense in $\bb{H}^{1}_{H}$ (see \cite[Lemma~8.7]{AuscherEgert}).
\end{proof}
We may now prove the existence and uniqueness of a solution to the Neumann problem. 
\begin{proof}[Proof of Theorem \ref{existence and uniqueness result Neumann Lp}]
    %We may assume that $V\not\equiv 0$, in view of \cite[Theorem 22.4]{AuscherEgert}. 
    Let us first prove the existence of a solution with the desired properties. We first consider the case when $p\in(p_{-}(H),q_{+}(H))\cap(1,2q]$. Let $g\in\Ell{p}$, and let $f:=T_{p}^{-1}g$, where $T_{p}:\dot{\mathcal{V}}^{1,p}(\Rn)\to\Ell{p}$ is as in Lemma~\ref{lemma extension of square root of H to isomorphism} above. Applying Theorem~\ref{existence and uniqueness result Regularity Lp again} with the boundary data $f\in\dot{\mathcal{V}}^{1,p}(\Rn)$ yields a weak solution $u\in\mathcal{V}^{1,2}_{\loc}(\Hn)$ of $\mathscr{H}u=0$ such that
  \begin{equation*}
        \norm{S(t\nabla_{\mu}\partial_{t}u)}_{\El{p}}\eqsim \norm{N_{*}\left(\nabla_{\mu}u\right)}_{\El{p}}\eqsim\norm{f}_{\dot{\mathcal{V}}^{1,p}}\eqsim \norm{g}_{\El{p}}.
    \end{equation*}
    Items (iii) and (vi) of the same theorem also yield the non-tangential convergence
    \begin{equation*}
        \lim_{t\to 0^{+}}\left(\fiint_{W(t,x)}\abs{(A_{\perp\perp}\partial_{t}u)(s,y)-g(x)}^{2}\dd s\dd y\right)^{1/2}=0
    \end{equation*}
    for almost every $x\in\Rn$, and the strong continuity $A_{\perp\perp}\partial_{t}u\in C_{0}([0,\infty); \El{p})$, where $t$ is the distinguished variable, along with the estimate
    \begin{equation*}
    \norm{N_{*}(\partial_{t}u)}_{\El{p}}\eqsim \sup_{t>0}\norm{A_{\perp\perp}\partial_{t}u(t,\cdot)}_{\El{p}}\eqsim \norm{g}_{\El{p}}.
    \end{equation*}
    We also obtain the convergence $A_{\perp\perp}\partial_{t}u(t,\cdot)\to g$ in the $\El{p}$ norm, as $t\to 0^{+}$.

    We now treat the case when $p=1$. Let $g\in\textup{H}^{1}_{V}(\Rn)$, and let us apply Theorem~\ref{existence and uniqueness result Regularity for p<1} with the boundary data $f:=T_{p}^{-1}g\in\dot{\textup{H}}^{1,1}_{V}(\Rn)$, where $T_{p}$ has the same meaning as in the previous lemma. This provides us with a weak solution $u\in\mathcal{V}^{1,2}_{\loc}(\Hn)$ of $\mathscr{H}u=0$, satisfying 
    \begin{equation*}
        \norm{S(t\nabla_{\mu}\partial_{t}u)}_{\El{1}}\eqsim\norm{f}_{\dot{\textup{H}}^{1,1}_{V}}\eqsim \norm{g}_{\textup{H}^{1}_{V}}\gtrsim \norm{N_{*}\left(\nabla_{\mu}u\right)}_{\El{1}}.
    \end{equation*}
    Moreover, we note that the reverse estimate on the non-tangential maximal function also holds provided that $V\in\textup{RH}^{\infty}(\R^n)$. 
    We observe that the proof of item (iii) of Theorem~\ref{existence and uniqueness result Regularity Lp again} now goes through to show that the solution $u$ satisfies the non-tangential convergence 
    \begin{equation*}
        \lim_{t\to 0^{+}}\left(\fiint_{W(t,x)}\abs{(A_{\perp\perp}\partial_{t}u)(s,y)-(T_{1}f)(x)}^{2}\dd s\dd y\right)^{1/2}=0
    \end{equation*}
for almost every $x\in\Rn$, which is the desired conclusion.

Let us now consider the question of uniqueness. Let $p\in (p_{-}(H),q_{+}(H))\cap [1,2q]$ and let $u\in\mathcal{V}^{1,2}_{\loc}(\Hn)$ be a solution to the Neumann problem $\left(\mathcal{N}\right)^{\mathscr{H}}_{p}$ with zero boundary data. In other words, $u$ is a weak solution of $\mathscr{H}u=0$ with the nontangential control $N_{*}(\nabla_{\mu}u)\in\Ell{p}$, and it satisfies the nontangential convergence 
\begin{equation}\label{non tangential convergence for conormal gradient of solution neumann}
        \lim_{t\to 0^{+}}\fiint_{W(t,x)}\abs{(A_{\perp\perp}\partial_{t}u)(s,y)}\dd s\dd y=0
    \end{equation}
    for almost every $x\in\Rn$.
It follows from the $t$-independence of the coefficients that $\mathscr{H}(\partial_{t}u)=0$ (see, e.g., \cite[Lemma 4.2]{MorrisTurner}). Moreover, we have the nontangential control $N_{*}(\partial_{t}u)\in\Ell{p}$. Since $A_{\perp\perp}^{-1}$ is bounded, these observations and \eqref{non tangential convergence for conormal gradient of solution neumann} imply that $\partial_{t}u$ solves the Dirichlet problem $(\mathcal{D})^{\mathscr{H}}_{p}$ with zero boundary data. It follows from Theorem~\ref{uniqueness result for regularity problem} that $\partial_{t}u=0$. As a consequence, we know that there is some $f\in\El{2}_{\loc}(\Rn)$ such that $u(t,x)=f(x)$ for almost every $t>0$ and $x\in\Rn$ (see the proof of Lemma \ref{first elementary basic uniqueness result}). Since $N_{*}(\nabla_{t,x}u)\in\Ell{p}$, \cite[Proposition A.5]{AuscherEgert} shows that there exists a nontangential trace $u_{0}\in\dot{\textup{H}}^{1,p}(\Rn)$ (i.e. $u_{0}\in\El{1}_{\loc}(\Rn)$ and $\nabla_{x}u_{0}\in\textup{H}^{p}(\Rn)$) such that 
\begin{equation}\label{nontangential trace of u is u_0}
\lim_{t\to 0}\left(\fiint_{W(t,x)}\abs{u(s,y)-u_{0}(x)}^{2}\dd s\dd y\right)^{1/2}=0
\end{equation}
for almost every $x\in\Rn$. It follows from Lebesgue's differentiation theorem that $u_{0}(x)=f(x)$ for almost every $x\in\Rn$, and therefore $f\in\El{2}_{\loc}(\Rn)\cap\dot{\textup{H}}^{1,p}(\Rn).$ In addition, the nontangential control $N_{*}(V^{1/2}u)\in\Ell{p}$ and the $t$-independence of $V$ and $u$ imply that the Hardy--Littlewood maximal operator satisfies $\mathcal{M}(V^{1/2}f)\in\Ell{p}$, and therefore $V^{1/2}f\in\textup{H}^{p}(\Rn)$ (see \cite[Chapter \RNum{2}, Section 2.1 and Chapter \RNum{3}, Theorem 1]{Stein_HarmonicAnalysis_93}). 
If $p>1$, this means that $f\in\dot{\mathcal{V}}^{1,p}(\Rn)$ and that $u$ solves the Regularity problem $(\mathcal{R})^{\mathscr{H}}_{p}$ with boundary data $f$. It follows from Theorem~\ref{existence and uniqueness result Regularity Lp} that $\norm{\nabla_{\mu}f}_{\El{p}}\eqsim \norm{S(t\nabla_{\mu}\partial_{t}u)}_{\El{p}}=0$, hence $f$ is constant and therefore so is $u$, and that constant is zero if $V\not\equiv 0$.
\end{proof}
\begin{rem}\label{remark uniqueness Neumann p=1}
Let us remark that, in the context of the previous proof in the case when $p=1$, we can only infer that $f\in\El{2}_{\loc}(\Rn)$ is such that $\nabla_{\mu}f\in\textup{H}^{1}(\Rn)^{n+1}$. If this was enough to imply that $f\in\dot{\textup{H}}^{1,1}_{V}(\Rn)$, then the same argument as in the case $p>1$ would yield uniqueness for $\left(\mathcal{N}\right)^{\mathscr{H}}_{1}$. In this regard, we note that every $g\in\mathcal{V}^{1,2}(\Rn)$ such that $\nabla_{\mu}g\in\textup{H}^{1}(\Rn)^{n+1}$ satisfies
\begin{equation}\label{embedding of nabla mu in H1 into dotH11V}
    \norm{g}_{\dot{\textup{H}}^{1,1}_{V}}\lesssim \norm{\nabla_{\mu}g}_{\textup{H}^{1}_{V}}\lesssim\norm{\nabla_{\mu}g}_{\textup{H}^{1}},
\end{equation}
which follows from the assumption $1\in\mathcal{I}(V)$ and Corollary~\ref{thm continuous inclusion of classical hardy into Dziubanski hardy space}. Consequently, if one could prove existence of a sequence $(g_{k})_{k\geq 1}\subset \mathcal{V}^{1,2}(\Rn)$ such that $\nabla_{\mu}g_{k}\to\nabla_{\mu}f$ in $\textup{H}^{1}(\Rn)^{n+1}$, then \eqref{embedding of nabla mu in H1 into dotH11V} and Proposition \ref{embedding homogeneous V spaces in L ^p* for p<n} would yield the desired conclusion $f\in\dot{\textup{H}}^{1,1}_{V}(\Rn)$.
\end{rem}
\bibliographystyle{abbrv}
\bibliography{refs}
\end{document}